\documentclass[11pt,reqno]{amsart}
\usepackage[foot]{amsaddr}
\usepackage[margin=1in]{geometry}

\RequirePackage{amsthm,amsmath,amsfonts,amssymb,graphicx,subfigure,enumerate}
\RequirePackage[numbers]{natbib}
\RequirePackage[colorlinks,citecolor=blue,urlcolor=blue, linkcolor=blue]{hyperref}
\RequirePackage[percent]{overpic}
\RequirePackage{tikz}
    \usetikzlibrary{
        calc,
        cd,
        arrows.meta,
        decorations.pathmorphing,
        shapes.geometric,
        }

\theoremstyle{plain}
\newtheorem{theorem}{Theorem}[section]
\newtheorem{lemma}[theorem]{Lemma}
\newtheorem{proposition}[theorem]{Proposition}
\newtheorem{corollary}[theorem]{Corollary}
\theoremstyle{definition}
\newtheorem{example}[theorem]{Example}
\newtheorem{definition}[theorem]{Definition}
\newtheorem{assumption}[theorem]{Assumption}
\newtheorem{remark}[theorem]{Remark}

\numberwithin{theorem}{section}
\allowdisplaybreaks

\usepackage[mathlines]{lineno}

\def\PP{\mathbb{P}}
\def\EE{\mathbb{E}}
\def\RR{\mathbb{R}}

\def\bD{\mathbf{D}}
\def\bE{\mathbf{E}}
\def\bK{\mathbf{K}}
\def\bM{\mathbf{M}}
\def\bO{\mathbf{O}}
\def\bT{\mathbf{T}}
\def\bU{\mathbf{U}}
\def\bW{\mathbf{W}}
\def\bX{\mathbf{X}}
\def\bY{\mathbf{Y}}
\def\bZ{\mathbf{Z}}
\def\be{\mathbf{e}}
\def\bi{\mathbf{i}}
\def\bj{\mathbf{j}}
\def\bk{\mathbf{k}}
\def\br{\mathbf{r}}
\def\bu{\mathbf{u}}
\def\bv{\mathbf{v}}
\def\bx{\mathbf{x}}
\def\by{\mathbf{y}}
\def\bz{\mathbf{z}}
\def\bgamma{\boldsymbol{\gamma}}
\def\btheta{\boldsymbol{\theta}}
\def\bxi{\boldsymbol{\xi}}
\def\bOmega{\boldsymbol{\Omega}}
\def\bSigma{\boldsymbol{\Sigma}}
\def\bTheta{\boldsymbol{\Theta}}
\def\bone{\mathbf{1}}
\def\cA{\mathcal{A}}
\def\cD{\mathcal{D}}
\def\cE{\mathcal{E}}
\def\cF{\mathcal{F}}
\def\cG{\mathcal{G}}
\def\cK{\mathcal{K}}
\def\cL{\mathcal{L}}
\def\cN{\mathcal{N}}
\def\cP{\mathcal{P}}
\def\cQ{\mathcal{Q}}
\def\cS{\mathcal{S}}
\def\cT{\mathcal{T}}
\def\cV{\mathcal{V}}

\def\mat{\operatorname{mat}}
\def\vec{\operatorname{vec}}
\def\Id{{\operatorname{Id}}}
\def\diag{\operatorname{diag}}
\def\sign{\operatorname{sign}}
\def\div{\operatorname{div}}
\def\Fro{\mathrm{F}}
\def\op{\mathrm{op}}
\def\val{\mathrm{val}}
\def\vecval{\mathrm{vec\text{-}val}}
\def\GOE{\mathrm{GOE}}
\def\<{\langle}
\def\>{\rangle}
\def\1{1 \!\! 1}
\def\Oprec{O_{\prec}}
\renewcommand{\ss}[1]{[#1]}

\input{SupplementCommands.tex}

\title{On Universality of Non-Separable Approximate Message Passing Algorithms}
\author{Max Lovig$^*$, Tianhao Wang$^\dagger$, Zhou Fan$^*$}
\email{max.lovig@yale.edu, tianhaowang@ucsd.edu, zhou.fan@yale.edu}
\address{$^*$Department of Statistics and Data Science, Yale University}
\address{$^\dagger$Halicioglu Data Science Institute, University of California, San Diego}

\begin{document}

\maketitle

\begin{abstract}
Mean-field characterizations of first-order iterative algorithms --- including
Approximate Message Passing (AMP), stochastic and proximal gradient descent,
and Langevin diffusions --- have enabled a precise understanding of learning
dynamics in many statistical applications. For algorithms whose non-linearities
have a coordinate-separable form, it is known that such characterizations enjoy
a degree of universality with respect to the underlying data distribution.
However, mean-field characterizations of non-separable algorithm dynamics
have largely remained restricted to i.i.d.\ Gaussian or rotationally-invariant data.

In this work, we initiate a study of universality for non-separable AMP
algorithms. We identify a general condition for AMP with polynomial
non-linearities, in terms of a Bounded Composition Property (BCP) for their
representing tensors, to admit a state evolution that holds universally for
matrices with non-Gaussian entries. We then
formalize a condition of BCP-approximability for Lipschitz AMP algorithms to
enjoy a similar universal guarantee. We demonstrate that many common
classes of non-separable non-linearities are BCP-approximable, including local
denoisers, spectral denoisers for generic signals, and compositions of
separable functions with generic linear maps, implying the universality of state
evolution for AMP algorithms employing these non-linearities.
\end{abstract}

\section{Introduction}

First-order iterative algorithms play a central role in modern optimization and
sampling-based paradigms of statistical learning, where it is increasingly
recognized that algorithm dynamics may be equally important as model
specification in determining the properties and efficacy of trained models.
Motivated by learning applications, in recent years there has been a marked
advance in our understanding of mean-field characterizations of the dynamics of
iterative algorithms applied to high-dimensional and random data. We highlight,
as several examples, precise asymptotic characterizations of the iterates of
Approximate Message Passing (AMP) algorithms
\cite{Donoho2009,Bayati2011,Rangan2012,Javanmard2012,Rangan2018,Feng2021,Gerbelot2022,bao2023leave,hachem2024approximate},
gradient descent and proximal gradient descent
\cite{mannelli2019passed,mannelli2020complex,Celentano2021,han2024entrywise,han2024gradient,montanari2025dynamical}, and
stochastic gradient and stochastic diffusion methods 
\cite{arous1995large,arous1997symmetric,mannelli2020marvels,arous2021online,ben2022high,collins2024hitting,paquette2024homogenization,gerbelot2024rigorous,fan2025dynamicalI,fan2025dynamicalII}.

In the context of an asymmetric data matrix $\bW \in \RR^{m \times n}$, a
general form for first-order iterative algorithms alternates between
multiplication by $\bW$ or $\bW^\top$ and entrywise applications of non-linear
functions \cite{Celentano2020}. As a concrete example, given linear observations
$\bx=\bW\btheta_*+\be \in \RR^m$
of an unknown signal $\btheta_* \in \RR^n$ with noise $\be \in \RR^m$, a well-studied AMP algorithm
\cite{Donoho2009} for estimating $\btheta_*$ takes an iterative form
\begin{equation}\label{eq:AMPintro}
\begin{aligned}
\br_t&=\bx-\bW\btheta_t+b_t \br_{t-1}\\
\btheta_{t+1}&=\eta_t(\btheta_t+\bW^\top \br_t)
\end{aligned}
\end{equation}
with a non-linearity $\eta_t:\RR^n \to \RR^n$ applied in each iteration.
The accompanying \emph{state evolution} theory of AMP
prescribes that, when $\bW$ has i.i.d.\ Gaussian entries, the iterates $\br_t$
and $\btheta_t$ satisfy
\begin{equation}\label{eq:SEintro}
\br_t \approx \bY_t, \qquad
\btheta_t+\bW^\top \br_t \approx \btheta_*+\bZ_t
\end{equation}
where $\bY_t \in \RR^m$ and $\bZ_t \in \RR^n$ are Gaussian vectors with
laws $\bY_t \sim \cN(0,\sigma_t^2\Id)$ and $\bZ_t \sim \cN(0,\omega_t^2\Id)$,
and $\{\sigma_t^2,\omega_t^2\}_{t \geq 1}$ are two recursively defined
sequences of variance parameters. When $\eta_t:\RR^n \to \RR^n$ consists of
a scalar function $\mathring \eta_t:\RR \to \RR$ applied
entrywise --- often called the \emph{separable} setting --- it was shown
in \cite{Bayati2011,Javanmard2012} that the
approximations (\ref{eq:SEintro}) hold in a sense of equality of asymptotic
limits for the empirical distributions of entries, and we refer to
\cite{rush2018finite,li2023approximate,han2024entrywise} for
quantitative and non-asymptotic results. As shown in
\cite{Celentano2021,han2024entrywise,fan2025dynamicalI}, such
guarantees can serve as a basis for analogous state evolution characterizations
(with more complex forms) of broad classes of first-order iterative
algorithms, including commonly used variants of Langevin dynamics and
gradient descent.

The separable setting is most natural from the perspective of mean-field theory,
and is typically motivated in practice by applications where $\btheta_* \in
\RR^n$ has entrywise structure such as sparsity or i.i.d.\ coordinates drawn
from a Bayesian prior. However, it is also
understood from \cite{berthier2017,Ma2019,Gerbelot2022,gerbelot2024rigorous}
that state evolution characterizations of the type (\ref{eq:SEintro}) may hold more
broadly for iterative algorithms where $\eta_t:\RR^n \to \RR^n$ is a more
general \emph{non-separable} function in high dimensions. Such generalizations
have been useful across a variety of applications with more
complex data structure, including:
\begin{itemize}
\item Image reconstruction, where $\btheta_* \in \RR^n \equiv \RR^{M \times N}$
represents a 2D-image
\cite{tan2015compressive,metzler2015bm3d,metzler2016denoising,metzler2017learned}.
\item Matrix sensing, where $\btheta_* \in \RR^n \equiv \RR^{M \times N}$ is a 
matrix of approximately low rank
\cite{donoho2013phase,berthier2017,romanov2018near,xu2025fundamental}.
\item Recovery of signals $\btheta_*$ having sequential structure, such as in
Markov chain or changepoint models
\cite{ma2016approximate,Ma2019,arpino2024inferring}.
\item Recovery of signals $\btheta_*$ described by a graphical model or deep
generative prior
\cite{schniter2010turbo,som2012compressive,tramel2016approximate,aubin2019spiked,Ma2019,amalladinne2021unsourced}.
\item Analyses of proximal gradient methods for convex optimization with
non-separable regularizers
\cite{manoel2018approximate,bu2020algorithmic,bu2023characterizing}.
\item Analyses of iterative algorithms with correlated data matrices $\bW$,
where row/column correlations may be incorporated into $\eta_t(\cdot)$ via
variable reparametrization
\cite{loureiro2021learning,loureiro2022fluctuations,zhang2024matrix,wang2025glamp}.
\end{itemize}
Motivated by this broad range of practical applications, our work seeks to
advance our understanding of mean-field characterizations for non-separable
algorithm dynamics, which is currently substantially more limited than in the
separable setting.

In this work, we initiate a study of \emph{universality} of state evolution
characterizations of the form (\ref{eq:SEintro}) for non-separable AMP
algorithms. In the separable setting, universality was first
studied by \cite{Bayati2015}, who showed that state evolution characterizations
of separable AMP procedures with polynomial non-linearities remain valid when $\bW$ has
independent non-Gaussian entries, and also that AMP algorithms with Lipschitz
non-linearities admit polynomial approximants that enjoy such universal
guarantees. Universality was later shown directly for separable Lipschitz AMP
methods in \cite{Chen2021} and for instances of Langevin-type diffusions in
\cite{dembo2021universality,dembo2021diffusions}, and extended to 
other first-order algorithms in
\cite{Celentano2021,han2024entrywise,fan2025dynamicalI}.
The picture which emerges from these works may be summarized as:

\begin{center}
\emph{Mean-field characterizations of separable first-order algorithms for
i.i.d.\ Gaussian matrices $\bW$ hold universally for matrices $\bW$ with
independent non-Gaussian entries.}
\end{center}

\noindent We note that broader statements of universality for semi-random
matrices beyond the i.i.d.\ universality class have also been
investigated more recently in \cite{dudeja2023universality,Dudeja2022,Wang2024}.

It is tempting to surmise that a statement analogous to the above may hold for
non-separable algorithms. However, the following simple example illustrates
that this cannot be true in full generality:

\begin{example}[Failure of universality]\label{ex:nonuniversalexample}
Let $g:\RR^n \to \RR^n$ be a separable function given by
$g(\bz)[i]=\mathring{g}(\bz[i])$, where $\mathring{g}:\RR \to \RR$ is
Lipschitz and applied entrywise. Let $\bO \in \RR^{n \times n}$ be an orthogonal
matrix, and consider the AMP algorithm (\ref{eq:AMPintro}) where
$\eta_t(\bz) \equiv \eta(\bz)=\bO g(\bz)$ for all $t \geq 1$,
initialized at $\btheta_1=0$ and $\br_0=0$. Let us suppose, for simplicity and
concreteness of discussion, that $\btheta_*=\bone \equiv (1,1,\ldots,1)$ is the
all-1's vector in $\RR^n$, the measurements $\bx=\bW\btheta_*$ are
noiseless, the number of measurements is $m=n$, and
the first row and column of $\bO$ are also given by $n^{-1/2}\bone$.

For any covariance matrix $\bSigma \in \RR^{2 \times 2}$, if
$[\bZ,\bZ'] \in \RR^{n \times 2}$ has i.i.d.\ rows with distribution
$\cN(0,\bSigma)$, then it is readily checked that
\[\lim_{n \to \infty} \frac{1}{n} \btheta_*^\top \btheta_*=1,
\qquad
\lim_{n \to \infty} \frac{1}{n} \EE[\btheta_*^\top \eta(\btheta_*+\bZ)]=0,\]
\[\lim_{n \to \infty} \frac{1}{n} \EE[\eta(\btheta_*+\bZ)^\top \eta(\btheta_*+\bZ')]
=\EE[\mathring g(1+\bZ[1])\mathring g(1+\bZ'[1])].\]
Thus if $\bW \in \RR^{n \times n}$ has i.i.d.\ $\cN(0,\frac{1}{n})$ entries,
then the assumptions of \cite[Theorem 14]{berthier2017} hold, ensuring
that the state evolution approximation (\ref{eq:SEintro}) is valid in the
sense
\[\frac{1}{n}\sum_{i=1}^n \phi(\br_t[i])
-\frac{1}{n}\sum_{i=1}^n \EE \phi(\bY_t[i]) \to 0
\text{ in probability as } n \to \infty\]
for any pseudo-Lipschitz test function $\phi:\RR \to \RR$,
and similarly for $\btheta_t+\bW^\top \br_t$ and $\bZ_t$.

Consider instead a setting where $\sqrt{n}\,\bW$ has i.i.d.\ entries with a
fixed non-Gaussian law having mean 0 and variance 1, and suppose that
$\EE_{\xi \sim \cN(0,1)}[\mathring g(1+\xi)]=c \neq 0$.
Then it follows from the form of the dynamics (\ref{eq:AMPintro})
that the first coordinate of $\btheta_2$ is
\[\btheta_2[1]=\eta(\bW^\top \bW \btheta_*)[1]=\frac{1}{\sqrt{n}}
\sum_{i=1}^n \mathring{g}((\bW^\top \bW \btheta_*)[i]) \approx c\sqrt{n}\]
where the last approximation holds with high probability.
This renders the distribution of coordinates of $\bW\btheta_2$
non-universal even in the large-$n$ limit,
as it depends on the non-Gaussian distribution of coordinates in
the first column of $\sqrt{n}\,\bW$. Hence (\ref{eq:SEintro}) will not 
hold for the second iterate $\br_2$.
\end{example}

The mechanism of non-universality in this example is simple, and
illustrates the more general and central issue that non-separable functions
$\eta_t:\RR^n \to \RR^n$ which satisfy $\ell_2$-boundedness and Lipschitz
conditions need not be bounded in the entrywise sense
\begin{equation}\label{eq:linftybound}
\|\eta_t(\bx)\|_\infty \leq C\|\bx\|_\infty
\end{equation}
for a dimension-free constant $C>0$. This can lead to a strong dependence
of the algorithm's iterates on the distribution of individual entries of $\bW$. 
Thus, $\ell_2$-type conditions on
$\eta_t(\cdot)$ alone are not enough to ensure the universality of state
evolution guarantees such as (\ref{eq:SEintro}). On the other hand, imposing an
assumption such as (\ref{eq:linftybound}) is often too strong, as many
examples of non-separable functions of interest in applications do not satisfy
such an assumption uniformly over $\RR^n$. This motivates
a more refined understanding of the behavior of the non-linearities
$\eta_t(\cdot)$ when restricted to the (random) iterates of the algorithm.

\subsection{Main results}

In this work, we study a class of Approximate Message Passing
(AMP) algorithms which encompasses (\ref{eq:AMPintro}), and develop conditions
under which their state evolutions hold universally for matrices $\bW$ having
independent non-Gaussian entries. Our results are summarized as follows:

\begin{enumerate}
\item For AMP algorithms with polynomial non-linearities, we introduce a general
condition on the polynomial functions --- that they are representable by tensors
satisfying a certain \emph{Bounded Composition Property (BCP)} --- which is
sufficient to guarantee the validity and universality of their state
evolution. Representing the homogeneous degree-$d$ components of each
polynomial function by tensors of order $d+1$, this property is defined as
an abstract bound on certain types of products/contractions between these
tensors.
\item For AMP algorithms with Lipschitz non-linearities, we formally
define a condition for approximability of the Lipschitz functions by
BCP-representable polynomials, so that state evolution for the Lipschitz AMP
is also valid and universal.
\item The above BCP-approximability condition is abstract, and may not be
simple to check for concrete examples. Motivated by many of the aforementioned
applications, and to illustrate methods of verifying this condition, we show
that three classes of non-separable Lipschitz functions are BCP-approximable:
\begin{itemize}
\item Local functions $\eta:\RR^n \to \RR^n$ such as sliding-window filters or
local belief-propagation algorithms on bounded-degree graphs, where each output
coordinate of $\eta(\cdot)$ depends on only $O(1)$ input coordinates, and each
input coordinate of $\eta(\cdot)$ affects only $O(1)$ output coordinates.
\item Anisotropic functions $\eta(\cdot)=h'(g(h(\cdot)))$ that arise in
analyses with data matrices $\bW$ having row or column correlations, where
$h,h':\RR^n \to \RR^n$ are sufficiently generic linear maps
and $g:\RR^n \to \RR^n$ is a separable function.
\item Spectral functions $\eta:\RR^{M \times N} \to \RR^{M \times N}$, where
the input space $\RR^n \equiv \RR^{M \times N}$ is identified with
matrices of dimensions $MN=n$, the true signal $\btheta_* \in \RR^n \equiv
\RR^{M \times N}$
has sufficiently generic singular vectors, and $\eta(\cdot)$ represents a scalar
function applied spectrally to the singular values of its matrix input.
\end{itemize}
\end{enumerate}

\begin{figure}[t]
    \centering
    \newcommand{\imgscale}{0.15}
    \newcommand{\imgvspace}{.1cm}
  
    \begin{minipage}{0.48\linewidth}
\centering
\begin{minipage}{0.48\linewidth}
\centering {\bf Gaussian $\bW$}
\end{minipage}
\begin{minipage}{0.48\linewidth}
\centering
{\bf Rademacher $\bW$}
\end{minipage}
{\bf Iteration 1}\\
\begin{minipage}{0.48\linewidth}
\includegraphics[scale=\imgscale]{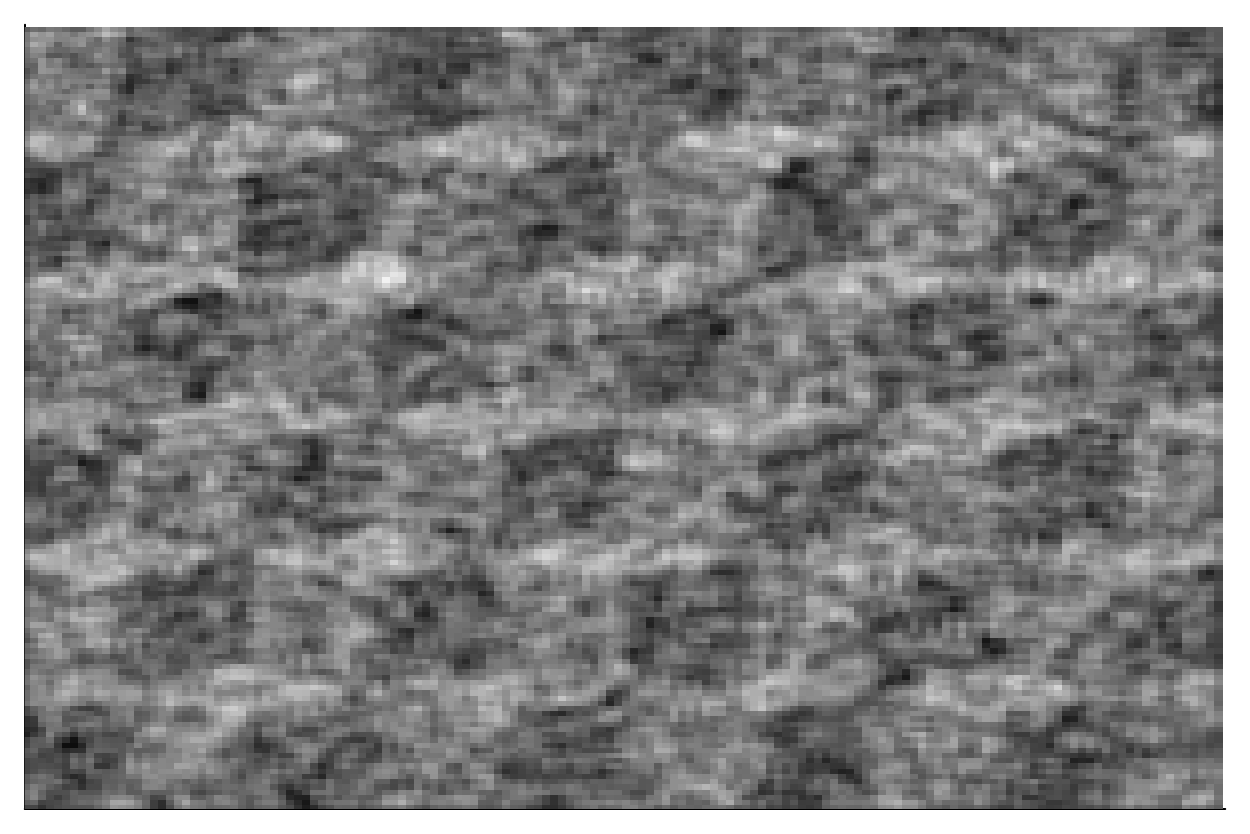}
\end{minipage}
\begin{minipage}{0.48\linewidth}
\includegraphics[scale=\imgscale]{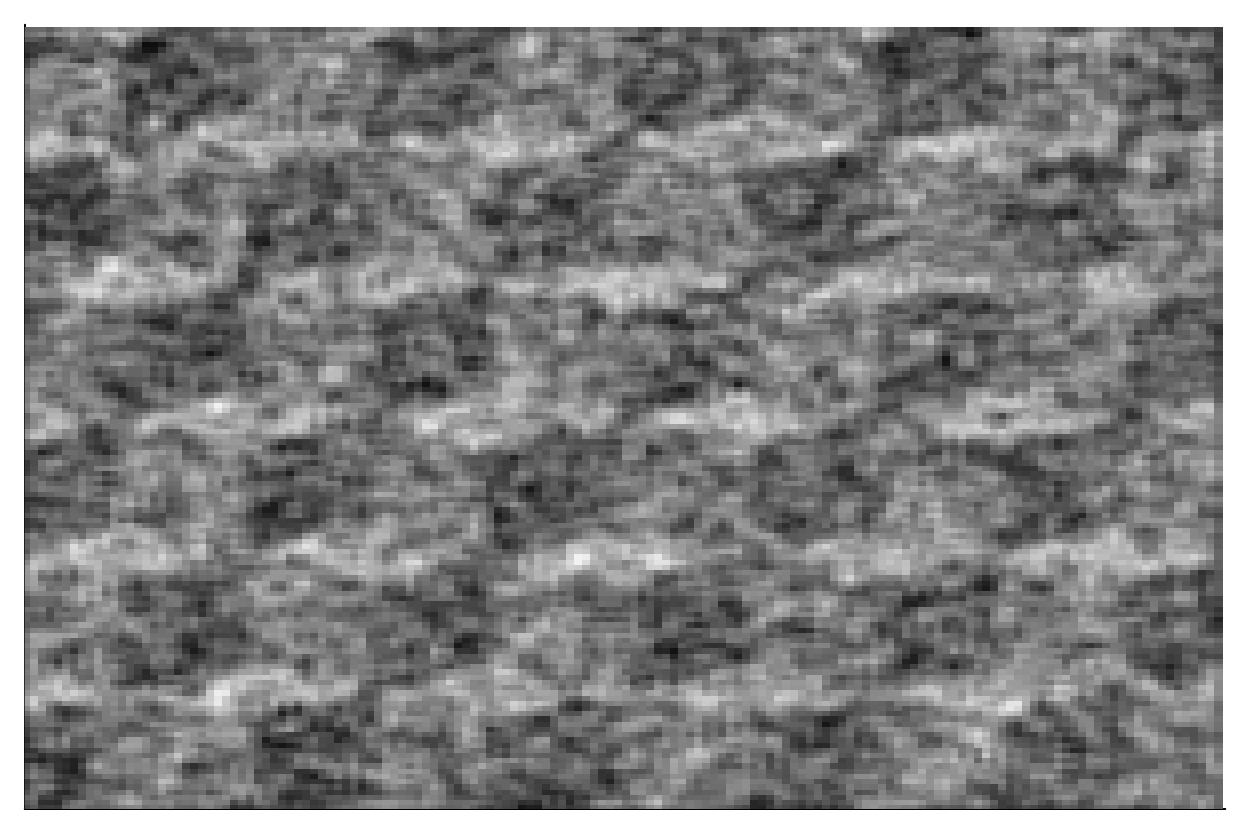}
\end{minipage}
{\bf Iteration 2}\\
\begin{minipage}{0.48\linewidth}
\includegraphics[scale=\imgscale]{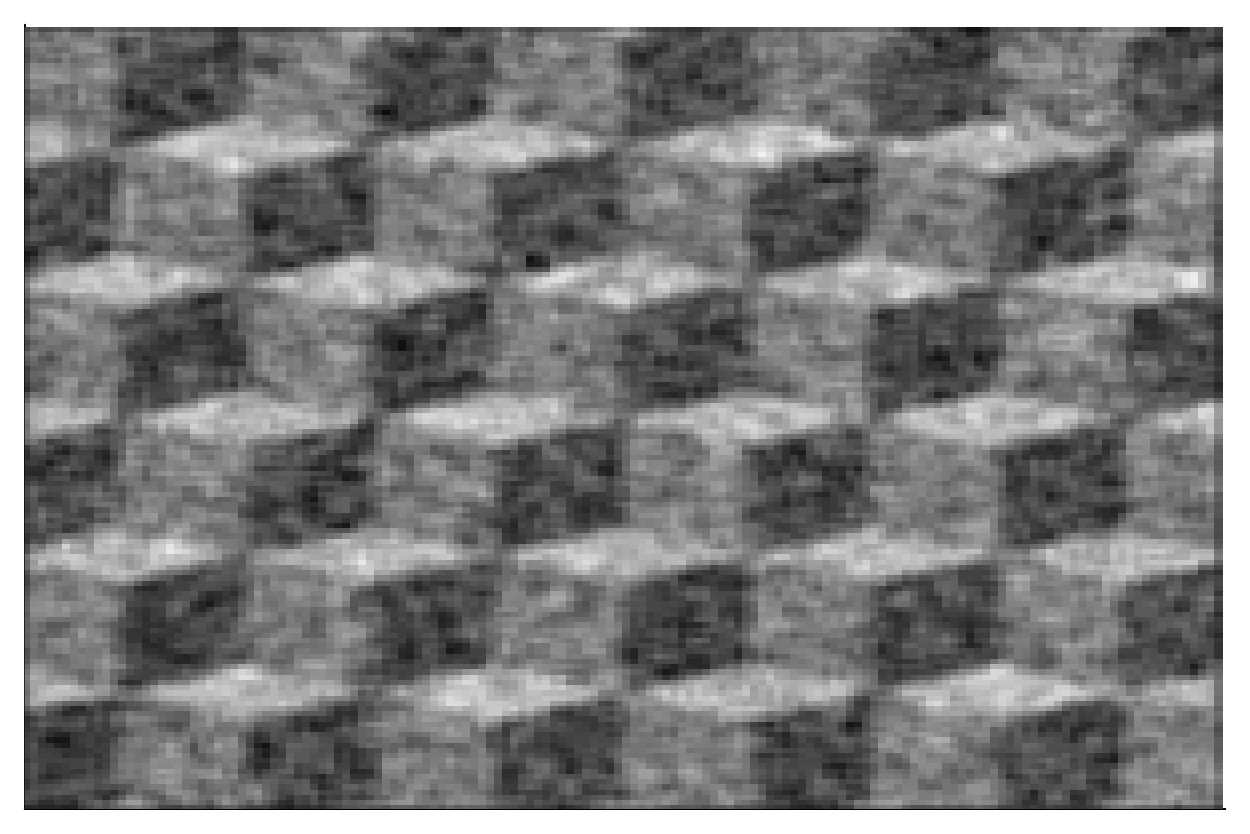}
\end{minipage}
\begin{minipage}{0.48\linewidth}
\includegraphics[scale=\imgscale]{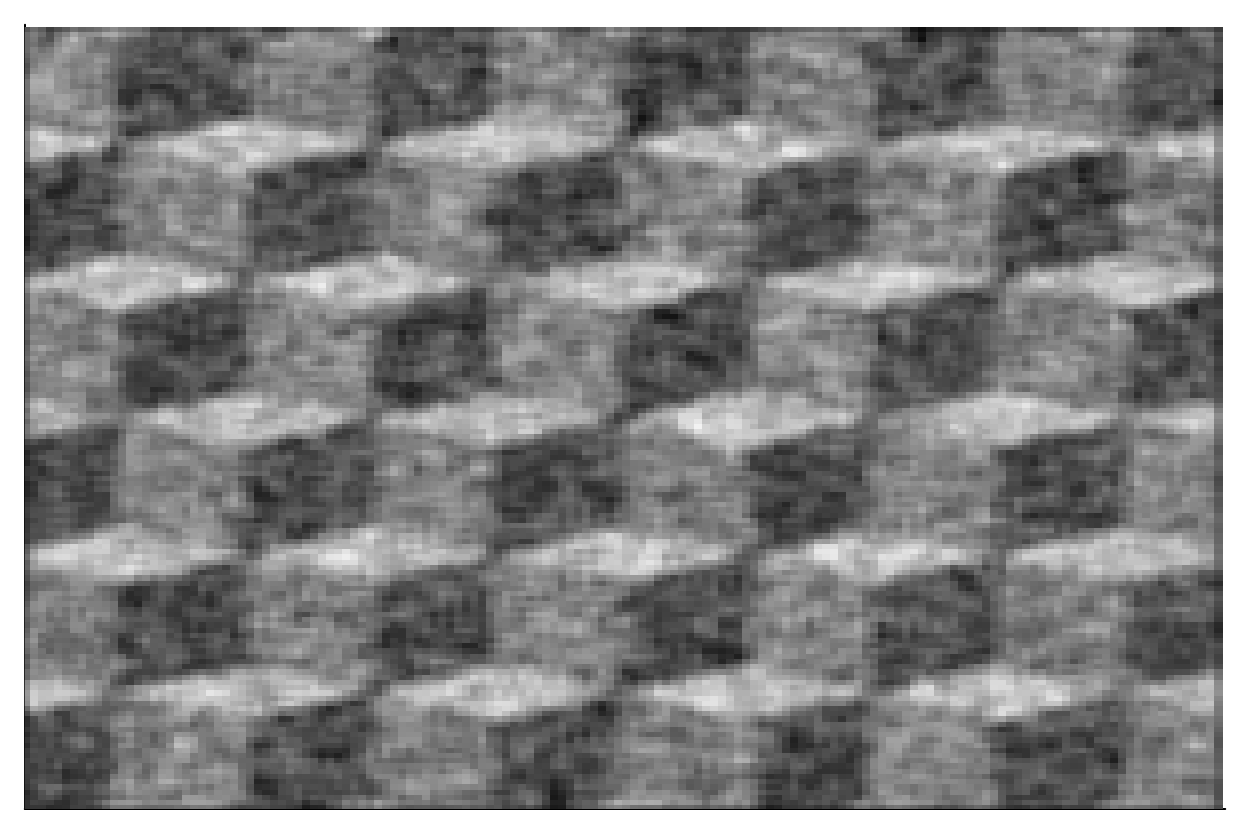}
\end{minipage}
\end{minipage}
\begin{minipage}{0.48\linewidth}
\centering
\begin{minipage}{0.48\linewidth}
\centering {\bf Gaussian $\bW$}
\end{minipage}
\begin{minipage}{0.48\linewidth}
\centering
{\bf Rademacher $\bW$}
\end{minipage}
{\bf Iteration 3}\\
\begin{minipage}{0.48\linewidth}
\includegraphics[scale=\imgscale]{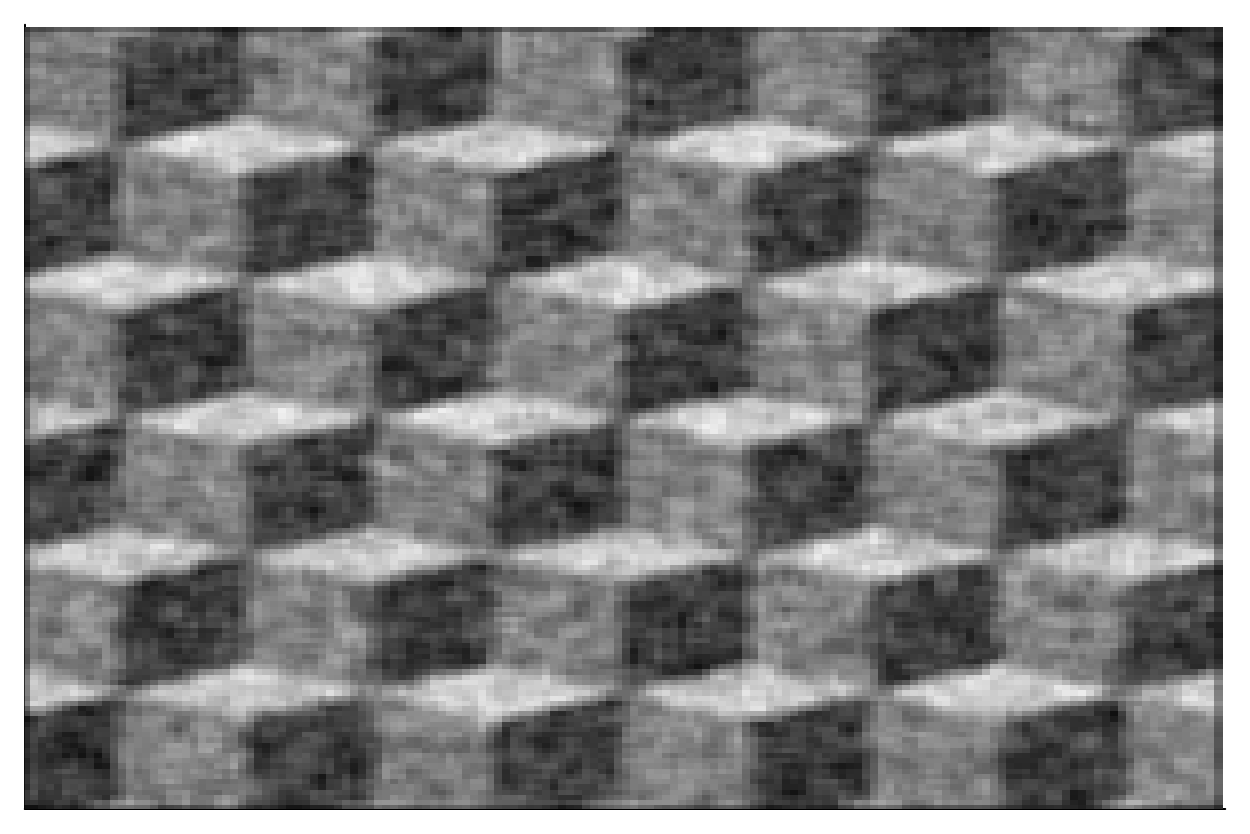}
\end{minipage}
\begin{minipage}{0.48\linewidth}
\includegraphics[scale=\imgscale]{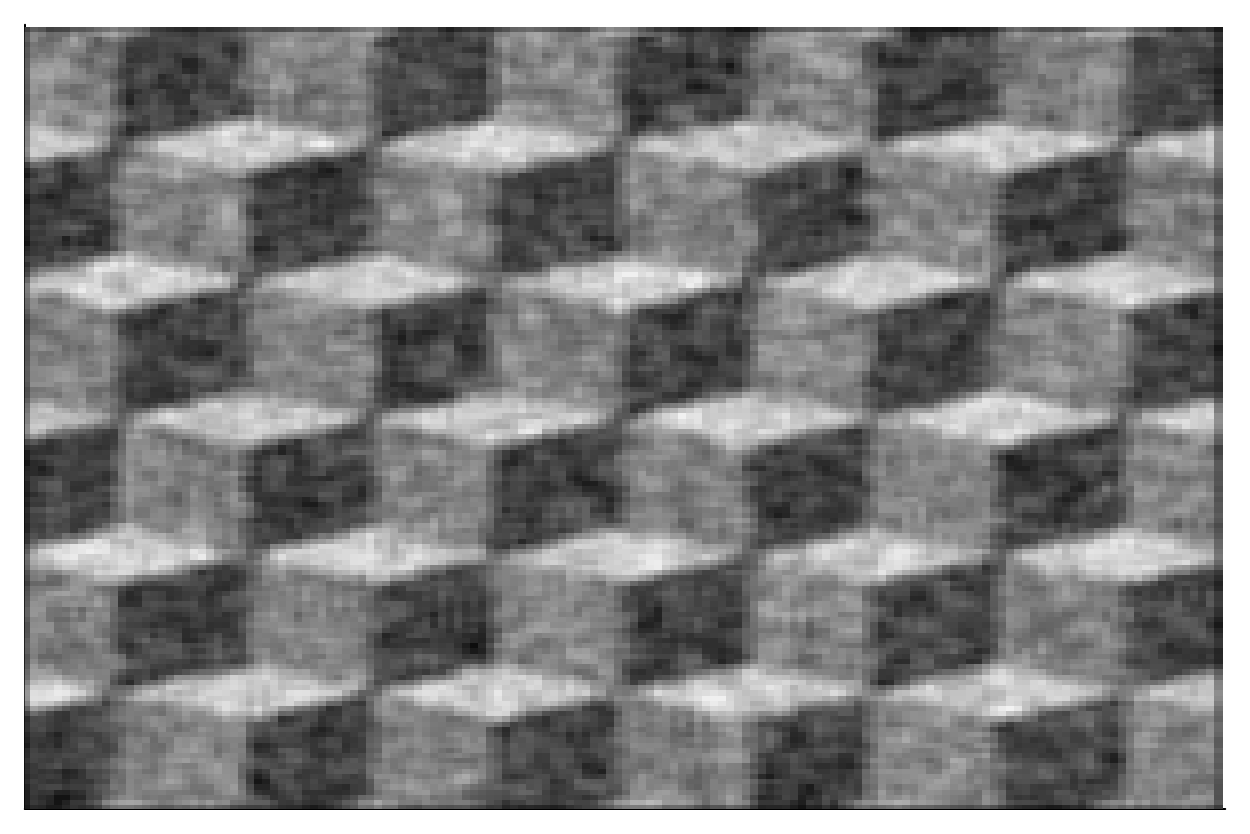}
\end{minipage}
{\bf Iteration 4}\\
\begin{minipage}{0.48\linewidth}
\includegraphics[scale=\imgscale]{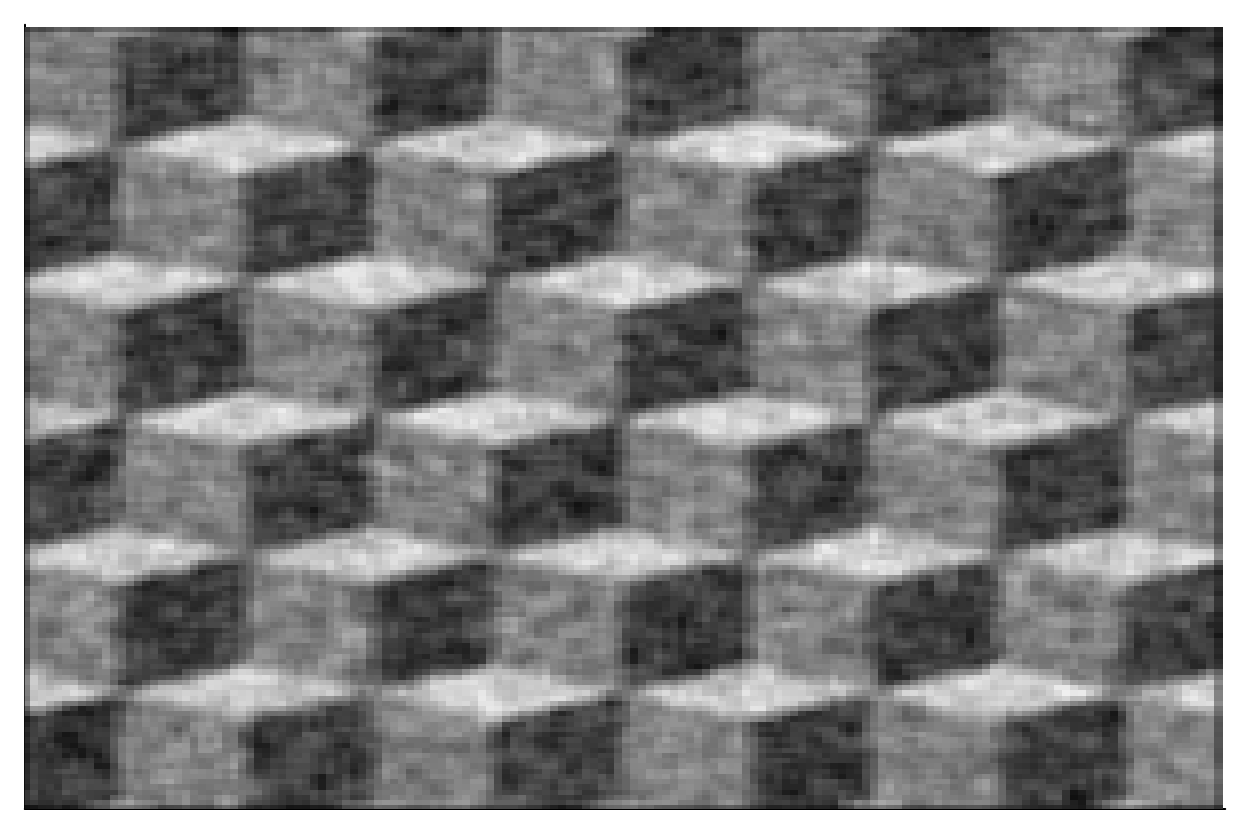}
\end{minipage}
\begin{minipage}{0.48\linewidth}
\includegraphics[scale=\imgscale]{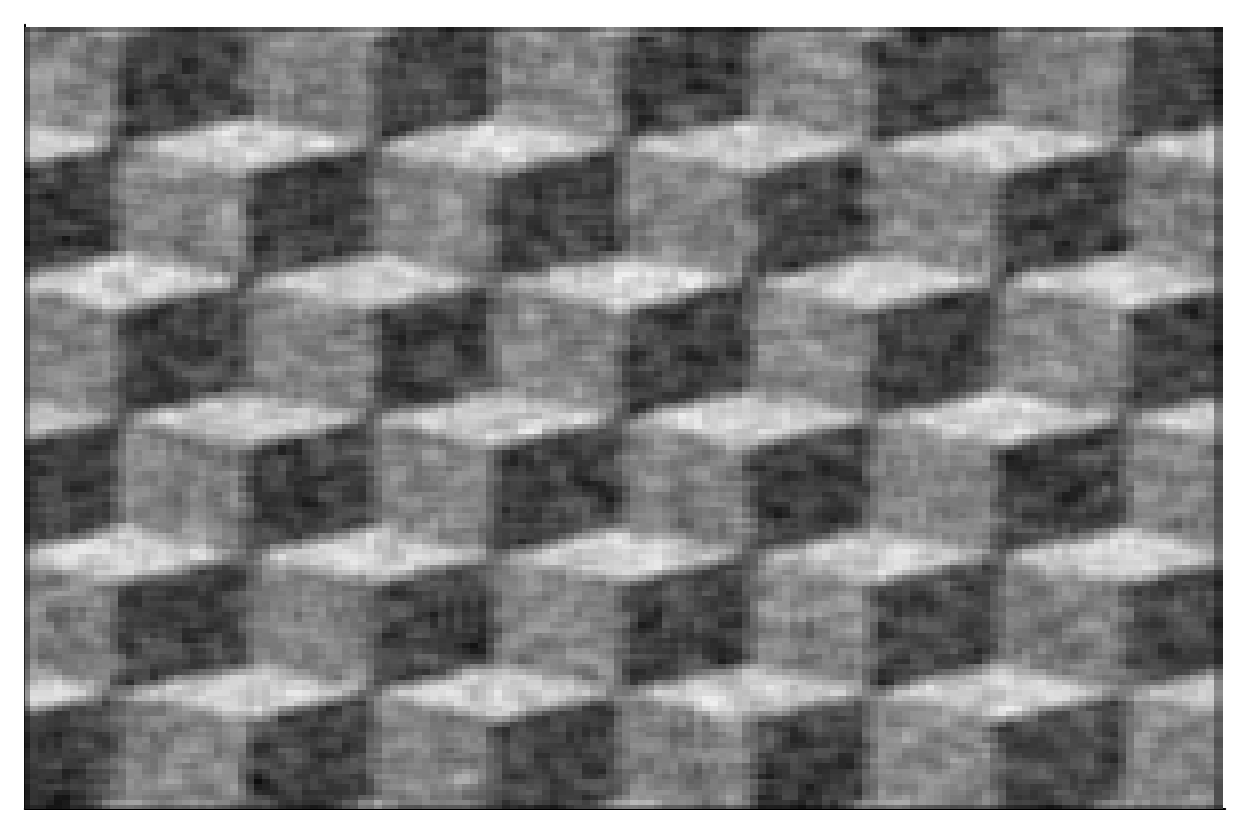}
\end{minipage}
\end{minipage}\\
\vspace{0.3cm}
\begin{minipage}{0.6\linewidth}
\hfill
\rotatebox{90}{\hspace{35pt}$\log_{10}(\text{MSE})$}
\includegraphics[scale = .3]{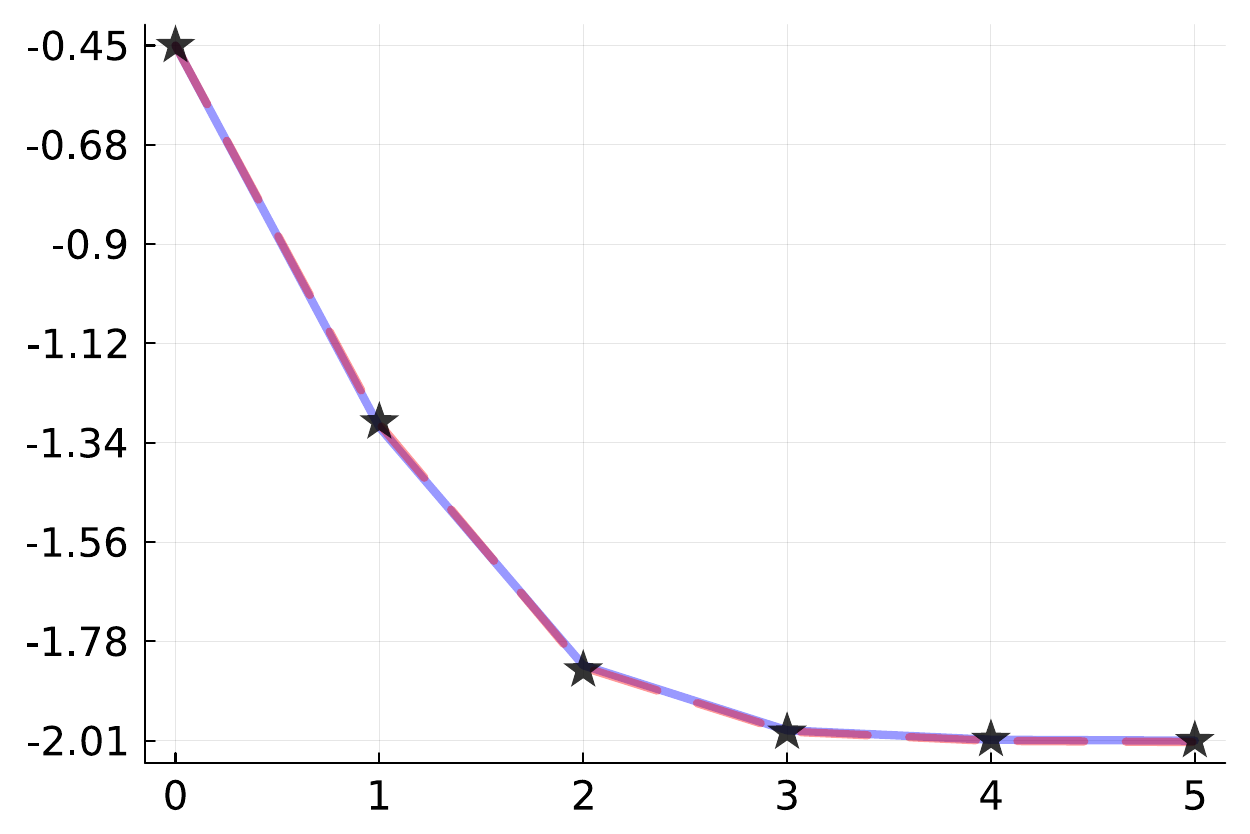}

\hspace{5cm}Iteration
\end{minipage}
\begin{minipage}[l]{.38\linewidth}
\resizebox{5cm}{!}{
    \begin{tikzpicture}
    \draw[thick, black, fill=white] (-0.5,-2.5) rectangle (9,.5);
        \draw[blue, ultra thick, opacity=0.4] (0,0) -- (2,0);
    \node[anchor=west] at (2.3,0) {\Large Gaussian $\bW$};
        \draw[red, dashed, ultra thick, dash pattern=on 20pt off 10pt, opacity=0.4] (0,-1) -- (2,-1);
    \node[anchor=west] at (2.3,-1) {\Large Rademacher $\bW$};
        \draw[white, dashed, thick] (0,-2) -- (2,-2);
    \node[anchor=west] at (2.3,-2) {\Large State Evolution Prediction};
    \node[star, star points=5, star point ratio=2.25, fill=black, draw=black, inner sep=0pt, minimum size=10pt] at (1,-2) {};
    \end{tikzpicture}
}
\end{minipage}
\caption{(Top) AMP iterates $\btheta_t \in \RR^{M \times N}$ of
\eqref{eq:AMPintro} applied with a local kernel-smoothing denoiser and with a
matrix $\bW$ having either i.i.d.\ $\cN(0,1/m)$ or Rademacher $\pm 1/\sqrt{m}$
entries. (Bottom) Mean-squared-errors $\frac{1}{n}\|\btheta_t-\btheta_*\|_2^2$
for the two matrices $\bW$, and the state evolution prediction.
Here $M=N=150$, $n = 22500$, and $m = 0.95\,n$.}
\label{fig:local}
\end{figure}

\begin{figure}[t]
    \centering
    \newcommand{\imgscale}{0.3}
    \newcommand{\imgvspace}{.1cm}
  
    \begin{minipage}{0.48\linewidth}
      \centering
      {\bf Iteration 1}\\
      \rotatebox{90}{\hspace{40pt}Density}
      \begin{overpic}[scale=\imgscale,percent]{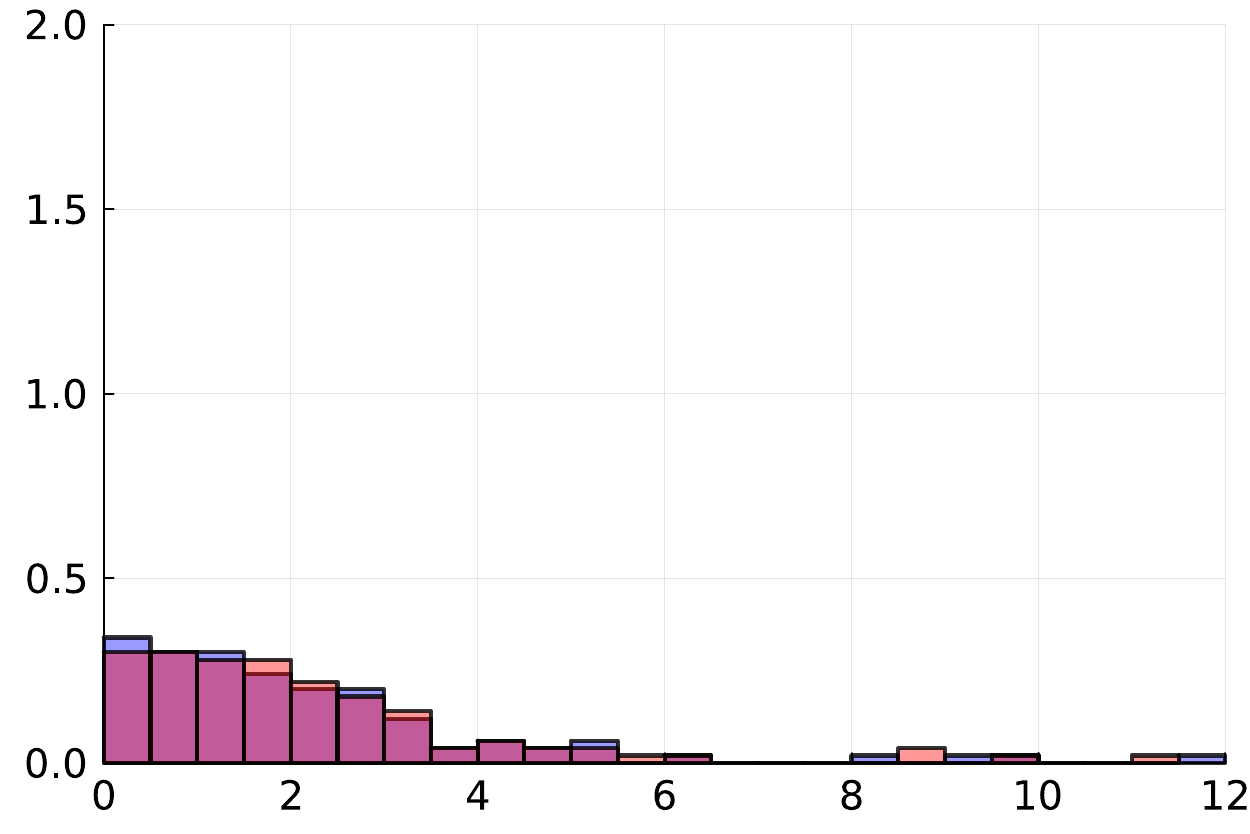}
        \put(19.5,50){%
            \resizebox{5cm}{!}{%
                \begin{tikzpicture}
              \draw[thick,black,fill=white] (-0.5,-1.5) rectangle (10.5,0.5);

              \draw[fill=blue, fill opacity=0.4, draw=black,draw opacity=1, ultra thick]
                (0, 0.15) rectangle (2,-0.15);
              \node[anchor=west] at (2.3,0) {\Large Singular Values for Gaussian $\bW$};

              \draw[fill=red, fill opacity=0.4, draw=black, draw opacity=1, ultra thick]
                (0,-1+0.15) rectangle (2,-1-0.15);
              \node[anchor=west] at (2.3,-1) {\Large Singular Values for Rademacher $\bW$};

                \end{tikzpicture}%
              }
        }
      \end{overpic}\\
Singular Value Profile\\
\vspace{0.3in}
      {\bf Iteration 4}\\
      \rotatebox{90}{\hspace{40pt}Density}
      \includegraphics[scale=\imgscale]{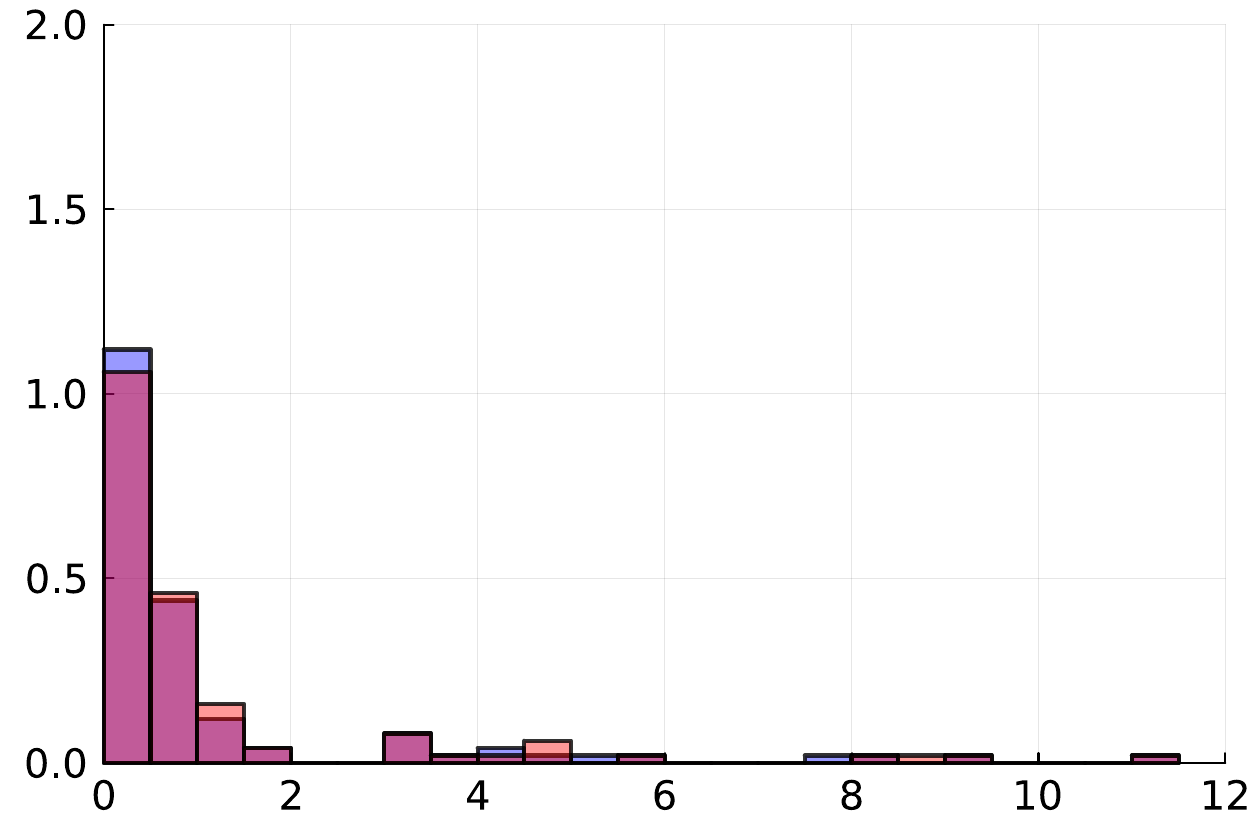}\\
Singular Value Profile
\end{minipage}
\begin{minipage}{0.48\linewidth}
\centering
{\bf Iteration 7}\\
      \rotatebox{90}{\hspace{40pt}Density}
      \includegraphics[scale=\imgscale]{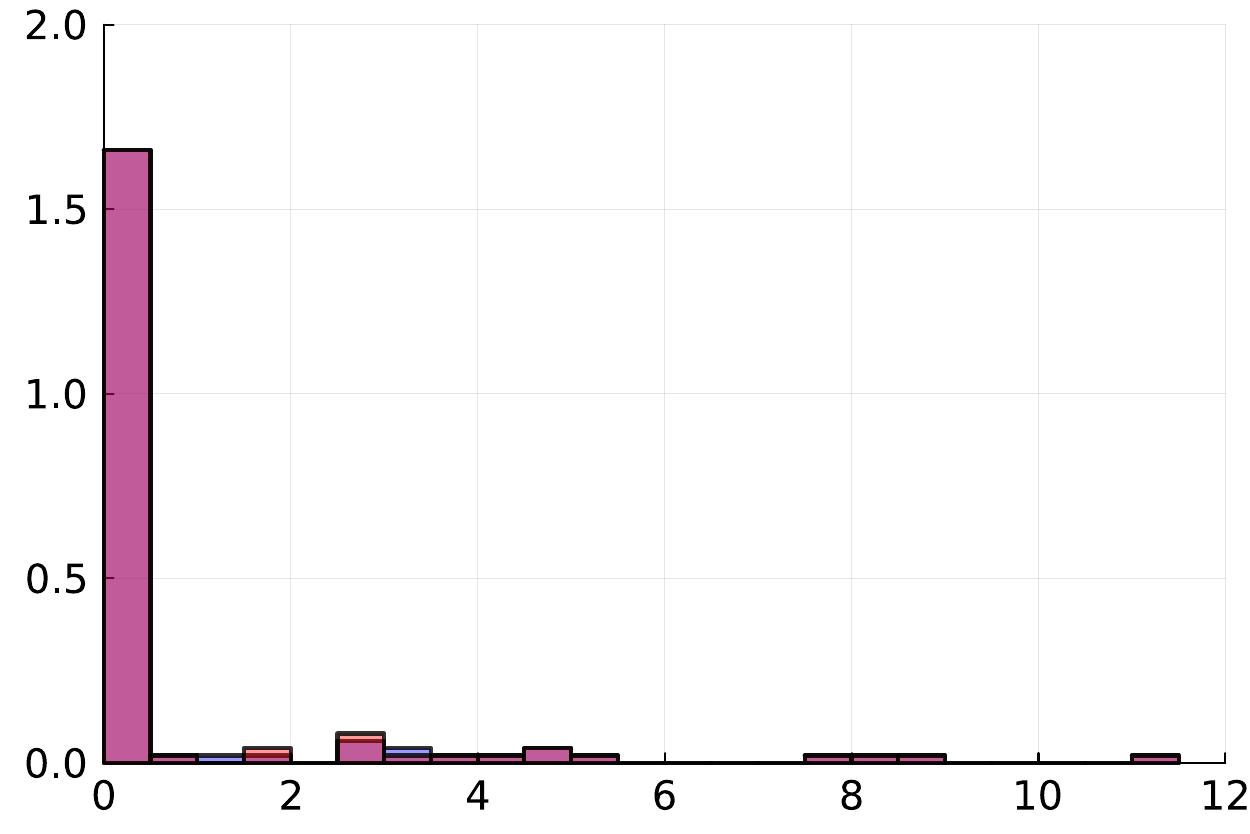}\\
Singular Value Profile\\
\vspace{0.3in}
{\bf Iteration 10}\\
    \rotatebox{90}{\hspace{40pt}Density}
    \includegraphics[scale=\imgscale]{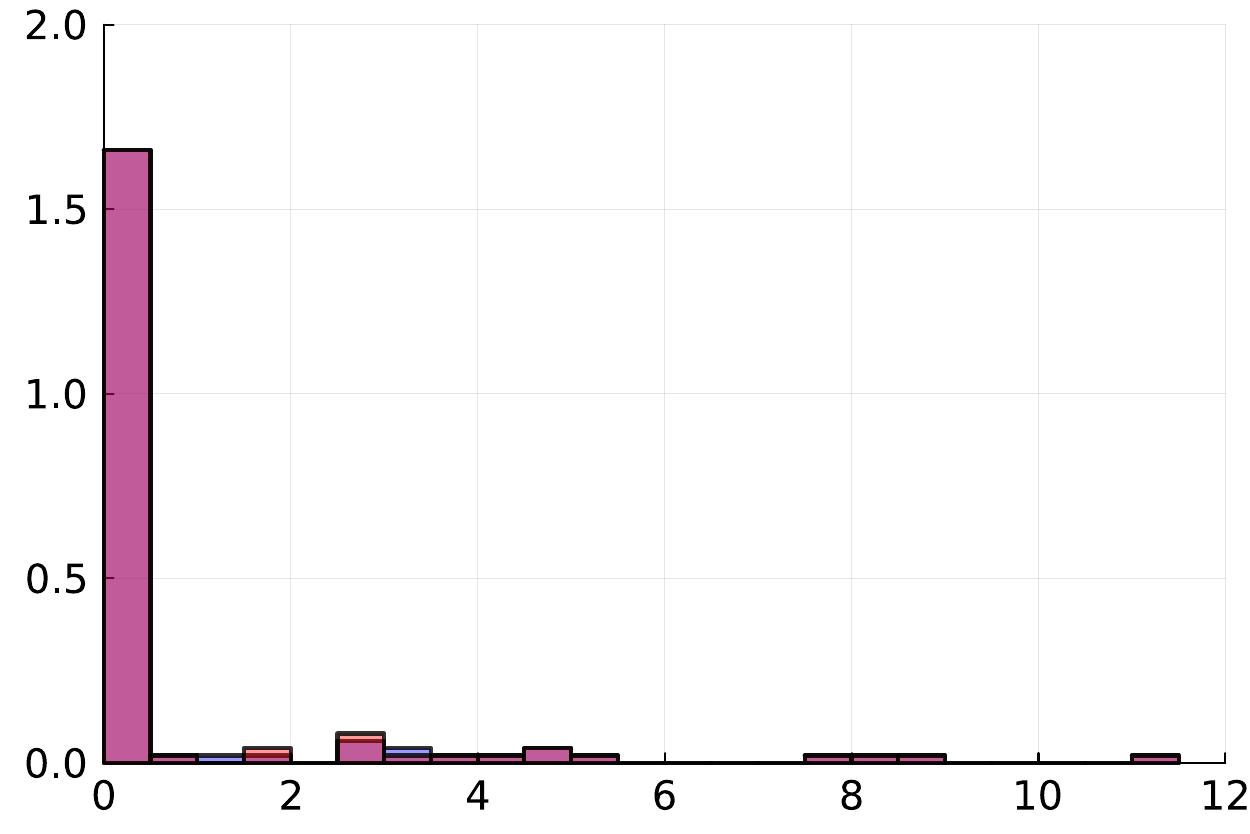}\\
      Singular Value Profile
    \end{minipage}\\
    \vspace{0.3cm}

        \begin{minipage}{0.6\linewidth}

          \hfill
      \rotatebox{90}{\hspace{35pt}$\log_{10}(\text{MSE})$}
      \includegraphics[scale = .3]{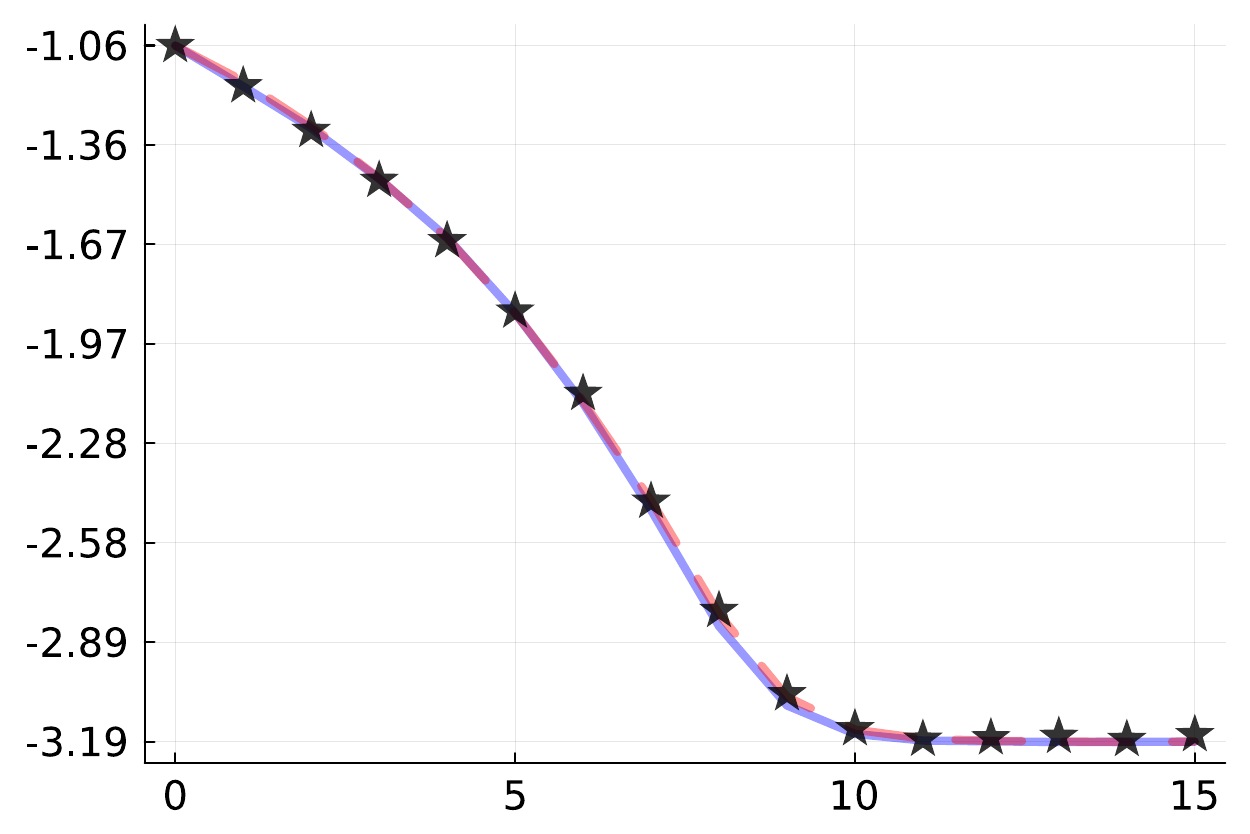}

      \hspace{5cm}Iteration
    \end{minipage}
    \begin{minipage}[c]{.38\linewidth}
  \resizebox{5cm}{!}{%
    \begin{tikzpicture}
      \draw[thick, black, fill=white] (-0.5,-2.5) rectangle (9,.5);
      \draw[blue, ultra thick, opacity=0.4] (0,0) -- (2,0);
      \node[anchor=west] at (2.3,0) {\Large Gaussian $\bW$};

      \draw[red, dashed, ultra thick, dash pattern=on 20pt off 10pt, opacity=0.4] (0,-1) -- (2,-1);
      \node[anchor=west] at (2.3,-1) {\Large Rademacher $\bW$};

      \draw[white, dashed, thick] (0,-2) -- (2,-2);
      \node[anchor=west] at (2.3,-2) {\Large State Evolution Prediction};
      \node[shape=star, star points=5, star point ratio=2.25, fill=black, draw=black, inner sep=0pt, minimum size=10pt] at (1,-2) {};
    \end{tikzpicture}%
  }
\end{minipage}
    
    \caption{(Top) Singular value spectra of the AMP iterates $\btheta_t \in
\RR^{M \times N}$ of \eqref{eq:AMPintro} applied with a singular-value
thresholding denoiser and with a matrix $\bW$ having either i.i.d.\ $\cN(0,1/m)$
or Rademacher $\pm 1/\sqrt{m}$ entries.
(Bottom) Mean-squared-errors $\frac{1}{n}\|\btheta_t-\btheta_*\|_2^2$ for the two
matrices $\bW$, and the state evolution prediction.
Here $M=100$, $N = 150$, and $m=n=15000$.}\label{fig:spectral}
  \end{figure}

Our proofs of the universality results in (1) and (2) above follow a general
strategy of previous works \cite{Bayati2015,dudeja2023universality,Wang2024},
resting on a moment-method comparison of polynomial AMP between Gaussian and
non-Gaussian matrices $\bW$. However, we note that even for Gaussian matrices
$\bW$, the validity of the AMP state evolution for a sufficiently rich class of
non-separable polynomial functions (or more generally, functions with polynomial
growth) is not available in the existing literature. We highlight here a last
contribution that may be of independent interest:

\begin{enumerate}
\setcounter{enumi}{3}
\item For AMP algorithms driven by matrices $\bW \sim \GOE(n)$ with Gaussian
entries, we provide a general condition for non-separable functions
$\eta:\RR^n \to \RR^n$ --- that they are stable with high probability
under $O(\operatorname{polylog} n)$ $\ell_2$-perturbations of
random Gaussian inputs --- which ensures the validity of a state evolution
approximation in a strong quantitative sense.
\end{enumerate}

We discuss the above results (1--3) further in Sections \ref{sec:results}
and \ref{sec:SetupRect}, and defer a discussion of (4) for Gaussian matrices
to Section \ref{sec:proof}.

Figures \ref{fig:local} and \ref{fig:spectral} illustrate our main results in
the context of the AMP algorithm (\ref{eq:AMPintro}) for the linear measurement
model with noise $\bx=\bW\btheta_*+\be$. Figure
\ref{fig:local} depicts an example of local smoothing, where $\RR^n \equiv
\RR^{M \times N}$ is a space of images, and $\eta_t:\RR^{M \times N} \to \RR^{M \times N}$ in
(\ref{eq:AMPintro}) represents the application of a sliding window kernel.
Such a function $\eta_t$ belongs to the class of
local functions for which our universality results apply. We observe that the
denoised AMP iterates with Gaussian and Rademacher sensing matrices are nearly
identical, and that their reconstruction mean-squared-errors both closely match
the theoretical prediction prescribed by the state evolution (\ref{eq:SEintro}).

Figure \ref{fig:spectral} depicts an example of matrix sensing, where
again $\RR^n \equiv \RR^{M \times N}$, and
$\eta_t:\RR^{M \times N} \to \RR^{M \times N}$ in (\ref{eq:AMPintro}) represents
soft-thresholding of the singular values of its matrix input. The true signal
$\btheta_* \in \RR^{M \times N}$
has Haar-orthogonal singular vectors, and this function $\eta_t$
belongs to the class of spectral functions for which our universality results
also apply. We observe that the singular value profiles
of the AMP iterates with Gaussian and Rademacher
sensing matrices are nearly identical, and that their reconstruction
mean-squared-errors again both closely match the theoretical prediction
prescribed by (\ref{eq:SEintro}). Further details of these examples 
are provided in Section \ref{sec:SetupRect}.

\subsection{Notation and conventions}

We use $\bv[i]$, $\bM[i,j]$, and $\bT[i_1,\ldots,i_k]$
for vector, matrix, and tensor indexing. For index subsets $S,S' \subseteq
\{1,\ldots,n\}$, we write $\bv[S] \in \RR^{|S|}$, $\bM[S,S'] \in \RR^{|S| \times
|S'|}$ etc.\ for the rows belonging to $S$ and columns belonging to $S'$.
For vectors $\bz_1,\ldots,\bz_t \in \RR^n$, we will often abbreviate
$\bz_{1:t}=(\bz_1,\ldots,\bz_t) \in \RR^{n \times t}$.

For a function $f:\RR^{n \times t} \to \RR^n$, $f(\cdot)[i]$ denotes the
$i^\text{th}$ coordinate of its output. Function $\div_s f$ is the divergence with
respect to the $s^\text{th}$ column of its input, i.e.\
\[\div_s f(\bz_{1:t})=\sum_{i=1}^n \partial_{\bz_s[i]} f(\bz_{1:t})[i].\]
Functions $f:\RR^{n \times t} \to \RR^n$ for $t=0$ are understood as
constant vectors in $\RR^n$.

Tensor $\Id^k \in (\RR^n)^{\otimes k}$ denotes the order-$k$ diagonal
tensor with diagonal entries equal to 1 and all other entries equal to 0, i.e.\
$\Id^k[i_1,\ldots,i_k]=\1\{i_1=\ldots=i_k\}$. For the identity matrix (i.e.\
$k=2$) we often abbreviate this as $\Id \in \RR^{n \times n}$. We write these as
$\Id^k_n$ and $\Id_n$ if needed to clarify the dimension. For a covariance
matrix $\bSigma \in \RR^{t \times t}$, $\cN(0,\bSigma\otimes \Id_n)$
is the multivariate normal distribution on $\RR^{n \times t}$ having
i.i.d.\ rows with law $\cN(0,\bSigma)$.

We write $\sigma_{\min},\sigma_{\max}$ and $\lambda_{\min},\lambda_{\max}$ for
the minimum and maximum singular value and eigenvalue of a matrix. $\|\cdot\|_2$
is the $\ell_2$-norm for vectors, $\|\cdot\|_\op$ is the $\ell_2$-to-$\ell_2$
operator norm for matrices, and $\|\bT\|_\Fro=(\sum_{i_1,\ldots,i_k}
\bT[i_1,\ldots,i_k]^2)^{1/2}$ is the Frobenius norm for matrices and tensors.

We denote $[n]=\{1,2,\ldots,n\}$. For a set $\cE$, we denote by $[n]^\cE$ the
set of index tuples $(i_e:e \in \cE)$ where $i_e \in [n]$ for each $e \in \cE$.
Given partitions $\pi,\tau$ of $\cE$, we write $\tau \geq \pi$ if $\pi$ refines
$\tau$, i.e.\ every block of $\tau$ is a union of one or more blocks of $\pi$.
The number of blocks in $\pi$ is denoted $|\pi|$.

\section{Universality of symmetric AMP}\label{sec:results}

To illustrate the main ideas, let us consider first the setting of an AMP
algorithm driven by a symmetric random matrix $\bW\in\RR^{n\times n}$.

Let $\bu_1\in\RR^n$ be an initialization, and $f_1,f_2,f_3,\ldots$ a sequence of
non-linear functions where $f_t:\RR^{n \times t} \to \RR^n$. We consider an AMP
algorithm consisting of the iterations, for $t=1,2,3,\ldots$
\begin{equation}\label{eq:AMP}
    \begin{aligned}
        \bz_t
        &= \bW \bu_t - \sum_{s=1}^{t-1}  b_{ts} \bu_s\\
        \bu_{t+1}
        &= f_t(\bz_1, \ldots, \bz_t).
    \end{aligned}
\end{equation}
It will be convenient to identify the initialization
\begin{equation}\label{eq:f1}
\bu_1 \equiv f_0(\cdot)
\end{equation}
as the output of an additional constant function $f_0(\cdot)$ with no inputs,
i.e.\ to understand $f_t(\bz_{1:t})$ for $t=0$ as this initialization.
Our interest will be in applications where $f_1,f_2,f_3,\ldots$ need not be
separable or exchangeable across their $n$ input coordinates.

In the first iteration, we have $\bz_1=\bW\bu_1$. In subsequent iterations,
the scalar Onsager coefficients $\{b_{ts}\}_{s<t}$ are defined so that
$\{\bz_t\}_{t \geq 1}$ admit an asymptotic characterization by a
Gaussian state evolution. These are given by the following definitions.

\begin{definition}[Onsager coefficients and state evolution]
\label{def:non_asymp_se}
Let $\bSigma_1 = \frac{1}{n}\|\bu_1\|_2^2 \in \RR^{1\times 1}$. Iteratively
for each $t\geq 1$, given $\bSigma_t\in\RR^{t\times t}$, let
$\bZ_{1:t} \sim \cN(0,\bSigma_t \otimes \Id_n)$, i.e.\
$\bZ_{1:t} \in\RR^{n \times t}$ has i.i.d.\ rows with distribution
$\cN(0,\bSigma_t)$. Define $\bSigma_{t+1}\in\RR^{(t+1)\times(t+1)}$ entrywise by
\[\bSigma_{t+1}[r+1,s+1]=\frac{1}{n}\,
\EE[f_r(\bZ_{1:r})^\top f_s(\bZ_{1:s})] \text{ for } r,s=0,1,\ldots,t\]
with the above identification $f_0(\cdot) \equiv \bu_1$.
For $t \geq 2$, the Onsager coefficients $\{b_{ts}\}_{s<t}$ in
(\ref{eq:AMP}) are defined\footnote{We will assume in all of our results
that $f_t(\cdot)$ is weakly differentiable and that the minor of $\bSigma_t$
corresponding to iterates $\{\bz_s\}_{s \leq t}$ on which $f_t(\cdot)$ depends
is non-singular, so that (\ref{eq:def_b_ts}) is well-defined.}
as
\begin{equation}\label{eq:def_b_ts}
b_{ts} = \frac{1}{n}\,\EE[\div_s f_{t-1}(\bZ_1,\ldots,\bZ_{t-1})]
\equiv
\frac{1}{n}\sum_{i=1}^n \EE[\partial_{\bZ_s[i]}
f_{t-1}(\bZ_1,\ldots,\bZ_{t-1})[i]].
\end{equation}
The state evolution approximation of the iterates $\{\bz_t\}_{t \geq 1}$ in
(\ref{eq:AMP}) is the sequence of Gaussian vectors $\{\bZ_t\}_{t \geq 1}$.
\end{definition}

We clarify that $\bSigma_s$ is the upper-left $s \times s$ submatrix of
$\bSigma_t$ for any $s \leq t$, and that $\{b_{ts}\}_{s<t}$ and $\bSigma_t$
thus defined are deterministic but $n$-dependent.
Our assumptions will ensure that $b_{ts}$ and $\bSigma_t$
remain bounded as $n \to \infty$ (c.f.\
Lemma \ref{lem:BCP_SEbound}), but we will not require that they 
have asymptotic limits.

When $f_1,f_2,\ldots$ are Lipschitz functions and $\bW \sim \GOE(n)$ is a
symmetric Gaussian matrix, results of \cite{berthier2017} show that the AMP
iterates $\{\bz_t\}_{t \geq 1}$ may be approximated in the large-$n$ limit by
the multivariate Gaussian vectors
$\{\bZ_t\}_{t \geq 1}$ of Definition \ref{def:non_asymp_se}, in the sense
\begin{equation}\label{eq:SEapprox}
\lim_{n \to \infty}
\phi(\bz_1,\ldots,\bz_t)-\EE[\phi(\bZ_1,\ldots,\bZ_t)]=0
\end{equation}
for a class of pseudo-Lipschitz test functions $\phi:\RR^{n \times t} \to \RR$.
Our main results will extend the validity of such an
approximation to certain classes of polynomial and Lipschitz
functions $f_1,f_2,\ldots$ and test functions $\phi$, when $\bW$ is any
non-Gaussian Wigner matrix satisfying the following conditions.
\begin{assumption}\label{assump:Wigner}
$\bW \in \RR^{n\times n}$ is a symmetric random matrix with
independent entries on and above the diagonal
$\{\bW[i,j]\}_{1\leq i\leq j\leq n}$, such that
for some positive constants $C_2,C_3,C_4,\ldots$,
\begin{itemize}
\item $\EE \bW[i,j]=0$ for all $i \leq j$.
\item $\EE \bW[i,j]^2=1/n$ for all $i<j$,
and $\EE \bW[i,i]^2 \leq C_2/n$ for all $i=1,\ldots,n$.
\item $\EE|\bW[i,j]|^k \leq C_kn^{-k/2}$ for each $k \geq 3$ and all $i \leq j$.
\end{itemize}
We write $\bW\sim\GOE(n)$ in the case where $\bW[i,j] \sim \cN(0,1/n)$ for $i<j$
and $\bW[i,i] \sim \cN(0,2/n)$.
\end{assumption}

\subsection{State evolution and universality for polynomial AMP}

We first study the validity and universality of the state evolution
approximation (\ref{eq:SEapprox}) in a setting where (each component of)
$f_t:\RR^{n \times t} \to \RR^n$ is a polynomial function.

We note that the mechanism of non-universality exhibited in
Example \ref{ex:nonuniversalexample} can hold just as well for AMP algorithms
with polynomial non-linearities,
upon replacing $\mathring g:\RR \to \RR$ in that example by a polynomial
function. Thus, universality of the state evolution requires a restriction of
the polynomial function class. We will consider such a restriction
given by polynomials representable by tensors satisfying a Bounded
Composition Property (BCP). This is a boundedness property for certain types
of compositions of such tensors, as formalized below via edge contractions in a
certain class of bipartite tensor networks. An equivalent
combinatorial/algebraic definition of the BCP is discussed in Appendix \ref{sec:BCP}.

\begin{definition}\label{def:TN0}
An {\bf ordered multigraph} $G=(\cV,\cE)$ is an undirected multigraph
with vertices $\cV$ and edges $\cE$, having no self-loops and no isolated
vertices, and with a specified ordering $e_1,\ldots,e_{\deg(v)}$ 
of the edges incident to each vertex $v \in \cV$. Here, $\deg(v)$
is the degree of $v$ (the total number of edges incident to $v$,
counting multiplicity).

Given a set of tensors $\cT=\bigsqcup_{k=1}^K \cT_k$ where $\cT_k \subseteq
(\RR^n)^{\otimes k}$, a
{\bf $\cT$-labeling} $\cL$ of $G$ is an assignment of a tensor $\bT_v \in
\cT_{\deg(v)}$ to each vertex $v \in \cV$, where the order of
$\bT_v$ equals the degree of $v$. We call $(G,\cL)$ a {\bf tensor network}.
The {\bf value} of this tensor network is
\begin{equation}
\val_G(\mathcal{L})=\sum_{\bi \in [n]^\cE}
\prod_{v \in \cV}\bT_v[i_e : e \sim v]
\label{eq:val}\end{equation}
where $[i_e:e \sim v]$ denotes the ordered tuple of indices
$[i_{e_1},\ldots,i_{e_{\deg(v)}}]$, and $e_1,\ldots,e_{\deg(v)}$ are the
ordered edges incident to $v$.
\end{definition}

When $G$ is connected (i.e.\ $(\cV,\cE)$ consists of a single connected
component), $\val_G(\cL)$ may be understood as the scalar value
obtained by contracting the tensor-tensor product associated to
each edge. When $G$ consists of multiple connected components, $\val_G(\cL)$
factorizes as the product of each such value across the components.
We note that specifying an edge ordering is needed to define $\val_G(\cL)$, as
the tensors $\{\bT_v\}_{v \in \cV}$ need not be symmetric.

Recall that $\Id^k \in (\RR^n)^{\otimes k}$ denotes the order-$k$ identity
tensor, with entries $\Id^k\ss{i_1, \dots, i_k} = \1\{i_1 = \cdots = i_k\}$.

\begin{definition}\label{def:BipCon}
An ordered multigraph $G=(\cV_{\Id} \sqcup \cV_T,\cE)$ is {\bf bipartite} if
its vertex set is the disjoint union of two sets $\cV_\Id,\cV_T$, and
each edge of $\cE$ connects a vertex of $\cV_\Id$ with a vertex of $\cV_T$.
A {\bf $(\Id,\cT)$-labeling} $\cL$ of such a multigraph $G$ is
a tensor labeling where each vertex $u \in \cV_\Id$ is labeled
with $\Id^{\deg(u)}$, and each vertex $v \in \cV_T$ has a label
$\bT_v \in \cT$.
\end{definition}

\begin{definition}\label{def:BCP} A set of deterministic tensors
    $\cT=\bigsqcup_{k=1}^K \cT_k$, where $\cT_k \subseteq (\RR^n)^{\otimes k}$
    for each $k=1,\ldots,K$, satisfies the {\bf Bounded Composition Property
(BCP)} if the following holds:\footnote{We clarify that $\cT \equiv \cT(n)$ is an
$n$-dependent set, and the BCP is an asymptotic condition for the
sequence $\{\cT(n)\}_{n=1}^\infty$ as $n \to \infty$. We write
``$\cT$ satisfies the BCP'' rather than ``$\{\cT(n)\}_{n=1}^\infty$ satisfies
the BCP'' for succinctness.}
Let $G=(\cV_{\Id} \sqcup \cV_T,\cE)$ be any fixed bipartite ordered multigraph
not depending on $n$, such that $G$ is connected and all vertices in $\cV_\Id$
have even degree. Then there exists a constant $C \equiv C(G)>0$
independent of $n$ such that
\[\sup_\cL |\val_G(\cL)| \leq Cn\]
for all large $n$, where the supremum is over all $(\Id,\cT)$-labelings $\cL$
of $G$.
\end{definition}

An illustration of a bipartite tensor network $G$ with
$(\Id,\cT)$-labeling is depicted in Figure~\ref{fig:bipartitemultigraph}. For
this example, the value of the network is given by
\[\sum_{i_1,i_2,i_3,i_4=1}^n \bT_1[i_1,i_1,i_2,i_3]
\bT_2[i_2,i_3,i_4,i_4],\]
and Definition~\ref{def:BCP} requires that the absolute value of this quantity is
bounded by $Cn$ for a constant $C>0$ and all large $n$,
uniformly over all choices of order-4 tensors $\bT_1,\bT_2 \in \cT$.
One setting of a tensor class $\cT$ that satisfies the BCP, as
discussed in Example \ref{ex:Sep} below, is that of diagonal tensors with
uniformly bounded entries; we will discuss more general (non-diagonal)
tensor classes in the examples to follow. We will also
show in Appendix \ref{sec:BCP} several elementary properties of
Definition \ref{def:BCP} including closure under tensor contractions and under
the additional inclusion of a finite number of independent Gaussian vectors.

\begin{figure}[t]
 \begin{tikzpicture}[
  rotate=90,
  every node/.style={circle, draw, inner sep=1pt, minimum size=14pt},
  x=1.2cm, y=0.6cm
]
  \node[draw=none, circle=none, inner sep=0pt] (l1) at (0, 4.5) {$\cV_T$};
  \node[draw=none, circle=none, inner sep=0pt] (l2) at (1, 4.5) {$\cV_{\Id}$};
  \node (u1) at (0, 2) {$\bT_1$};
  \node (u2) at (0, -1) {$\bT_2$};
  \node (v1) at (1, 3) {$\Id^2$};
  \node (v2) at (1, 1.5) {$\Id^2$};
  \node (v3) at (1, 0.0) {$\Id^2$};
  \node (v4) at (1, -1.5) {$\Id^2$};
  \draw[double, double distance=4pt] (u1) -- (v1);
  \draw (u1) -- (v2);
  \draw (u1) -- (v3);
  \draw (u2) -- (v2);
  \draw (u2) -- (v3);
  \draw[double, double distance=4pt] (u2) -- (v4);
\end{tikzpicture}
\caption{Example of a bipartite multigraph with $(\Id,\cT)$-labeling, where
edges incident to each vertex are ordered from left to
right.}\label{fig:bipartitemultigraph}
\end{figure}

For any tensor $\bT \in (\RR^n)^{\otimes (d+1)}$, considering its first $d$
dimensions as inputs and the last dimension as output, we may associate $\bT$
with a polynomial function $p:\RR^{n \times d} \to \RR^n$ that is homogeneous of
degree $d$, given by
\begin{equation}\label{eq:polyToTensor}
    p(\bz_1,\ldots,\bz_d)=\bT[\bz_1,\ldots,\bz_d,\,\cdot\,] \in \RR^n.
\end{equation}
The right side denotes the partial contraction of $\bT$ with
$\bz_1,\ldots,\bz_d$ in the first $d$ dimensions, i.e.\ its $j^\text{th}$
coordinate is $\sum_{i_1,\ldots,i_d=1}^n \bT[i_1,\ldots,i_d,j]\bz_1[i_1]\ldots
\bz_d[i_d]$. For $d=0$, this association is given by the constant function
$p(\cdot) =\bT \in \RR^n$ with no inputs. The following then defines a
restricted set of bounded-degree polynomials, representable as a sum of
homogeneous polynomials associated in this way to tensors that satisfy the BCP.

\begin{definition}\label{def:BCPpoly}
Let $\cP=\bigsqcup_{t=0}^T \cP_t$ be a set of polynomials, where
$\cP_t$ consists of polynomials $p:\RR^{n \times t} \to \RR^n$ and $\cP_0$
consists of constant vectors in $\RR^n$. $\cP$ is {\bf BCP-representable}
if there exists a constant $D \geq 0$ independent of $n$ and a set of
tensors $\cT \subseteq \bigsqcup_{k=1}^{D+1} \cT_k$ satisfying the BCP,
such that each $p \in \cP_t$ has a representation
\begin{equation}\label{eq:tensorpolyrepr}
p(\bz_1,\ldots,\bz_t)=\bT^{(0)}+\sum_{d=1}^D
\sum_{\sigma \in \cS_{t,d}}
\bT^{(\sigma)}[\bz_{\sigma(1)},\ldots,\bz_{\sigma(d)},\,\cdot\,]
\end{equation}
where $\cS_{t,d}$ is the set of all maps $\sigma:[d] \to [t]$,
and $\bT^{(0)} \in \cT_1$ and
$\bT^{(\sigma)} \in \cT_{d+1}$ for each $\sigma \in \cS_{t,d}$.
\end{definition}

In the representation (\ref{eq:tensorpolyrepr}),
$D$ denotes the maximum degree of polynomials in $\cP$,
$\bT^{(0)}$ represents the constant term of $p$,
and $\{\bT^{(\sigma)}\}_{\sigma \in \cS_{t,d}}$ represent the terms of
degree $d$. We note that the tensors
$\bT^{(\sigma)}$ in (\ref{eq:tensorpolyrepr}) are, in
general, not symmetric. Given a polynomial $p$,
the representation (\ref{eq:tensorpolyrepr}) also need not be unique due to
reordering of the inputs $\bz_{\sigma(1)},\ldots,\bz_{\sigma(d)}$
and choices of symmetrization for $\bT^{(\sigma)}$;
Definition \ref{def:BCPpoly} requires simply the existence of at least
one such representation.

Although the main focus of our work is in non-separable functions, for clarity
let us illustrate Definitions \ref{def:BCP} and
\ref{def:BCPpoly} in a separable example.

\begin{example}[Separable polynomials are BCP-representable]\label{ex:Sep} Fix
any $D,B>0$, and let $\cP=\bigsqcup_{t=0}^T \cP_t$ be a set of
separable polynomials such that each $p \in \cP_t$ is given by
$p(\bz_1,\ldots,\bz_t)[i]=\mathring{p}(\bz_1[i],\ldots,\bz_t[i])$ for some
scalar-valued polynomial $\mathring{p}:\RR^t \to \RR$ having degree at most $D$
and all coefficients bounded in magnitude by $B$.

Then $\cP$ is BCP-representable via a set of tensors
\[\cT \subseteq
\bigsqcup_{k=1}^{D+1} \Big\{\text{diagonal tensors } \bT \in
(\RR^n)^{\otimes k} \text{ with } \max_{i=1}^n |\bT[i,\ldots,i]| \leq B\Big\}.\]
This set $\cT$ must satisfy the BCP, because for any bipartite multigraph
$G$ satisfying the conditions of Definition \ref{def:BCP},
any $(\Id,\cT)$-labeling $\cL$ of $G$ has a contracted value of the form
\[|\val_G(\cL)|=\left|\sum_{i=1}^n \prod_{a=1}^m \bT_a[i,\ldots,i]\right|\]
for $m=|\cV_T|$ and some $\bT_1,\ldots,\bT_m \in \cT$, and this value
is at most $B^m \cdot n$.
\end{example}

\begin{example}[Counterexample to the BCP]

Consider $\cT=\{\bone,\bO\}$ where $\bone \in \RR^n$ is the all-1's vector
and $\bO$ is an orthogonal matrix whose first row and column are also given by
$n^{-1/2}\bone$. Combined with diagonal tensors,
this class may be used to express functions of the form discussed previously in
Example \ref{ex:nonuniversalexample}.
To see why this class does \emph{not} satisfy the BCP, consider the bipartite
tensor network $G$ with $(\Id,\cT)$-labeling $\cL$ given by
\begin{center}
 \begin{tikzpicture}[
  rotate=90,
  every node/.style={circle, draw, inner sep=1pt, minimum size=14pt},
  x=1.2cm, y=0.6cm
]
  \node[draw=none, circle=none, inner sep=0pt] (l1) at (0, 5) {$\cV_T$};
  \node[draw=none, circle=none, inner sep=0pt] (l2) at (1, 4.5) {$\cV_{\Id}$};
  \node (u1) at (0, -3.5) {$\bO$};
  \node (u2) at (0, -2.5) {$\bO$};
  \node (u3) at (0, -1.5) {$\bO$};
  \node (u4) at (0, -.5) {$\bO$};
  \node (u5) at (0,  .5) {$\bone$};
  \node (u6) at (0,  1.5) {$\bone$};
  \node (u7) at (0,  2.5) {$\bone$};
  \node (u8) at (0,  3.5) {$\bone$};

  \node (v1) at (1,-3) {$\Id^4$};
  \node (v2) at (1,-1.5) {$\Id^2$};
  \node (v3) at (1, 0.0) {$\Id^2$};
  \node (v4) at (1, 1.5) {$\Id^2$};
  \node (v5) at (1, 3) {$\Id^2$};

  \draw (u5) -- (v2);
  \draw (u6) -- (v3);
  \draw (u7) -- (v4);
  \draw (u8) -- (v5);
  \draw (u4) -- (v1);
  \draw (u4) -- (v5);
  \draw (u3) -- (v1);
  \draw (u3) -- (v4);
  \draw (u2) -- (v1);
  \draw (u2) -- (v3);
  \draw (u1) -- (v1);
  \draw (u1) -- (v2);
\end{tikzpicture}
\end{center}
The contracted value of this network satisfies
\[\val_G(\cL)
=\sum_{i=1}^n (\bO \bone)[i]^4 \geq (\bO \bone)[1]^4 \geq
n^2,\]
illustrating that the BCP cannot hold for (any class of tensors containing)
$\cT$.
\end{example}

Our first main result shows that BCP-representability
is sufficient to ensure both the validity and universality of the
state evolution approximation (\ref{eq:SEapprox}) for a corresponding class of
polynomial test functions under polynomial AMP.

\begin{theorem}\label{thm:universality_poly_amp}
Fix any $T \geq 1$, consider an AMP algorithm (\ref{eq:AMP}) defined by
$f_0,f_1,\ldots,f_{T-1}$, and consider a test function
\begin{equation}\label{eq:testfunction}
\phi(\bz_{1:T})=\frac{1}{n}\phi_1(\bz_{1:T})^\top \phi_2(\bz_{1:T})
\end{equation}
where $\phi_1,\phi_2:\RR^{n \times T} \to \RR^n$.
Let $b_{ts}$, $\bSigma_t$, and $\bZ_t$ be as in
Definition \ref{def:non_asymp_se}.

Suppose that $\cP=\{f_0,f_1,\ldots,f_{T-1},\phi_1,\phi_2\}$ is a 
BCP-representable set of polynomial functions, and
$\lambda_{\min}(\bSigma_t)>c$ for all $t=1,\ldots,T$ and a
constant $c>0$. If $\bW$ is any Wigner matrix satisfying
Assumption \ref{assump:Wigner}, then almost surely
\[\lim_{n \to \infty} \phi(\bz_{1:T})-\EE[\phi(\bZ_{1:T})]=0.\]
\end{theorem}

We remark that the polynomials $\cP$ need not be Lipschitz, or even
pseudo-Lipschitz in the sense of \cite[Eq.\ (20)]{berthier2017} or
\cite[Definition 4]{Gerbelot2022}. Such a result is new even in the setting of
a Gaussian matrix $\bW \sim \GOE(n)$, where it is a consequence of a more
general statement that we give
in Section \ref{sec:gaussian} for AMP algorithms defined by 
a general class of functions $f_1,f_2,\ldots$ having polynomial growth.
Theorem \ref{thm:universality_poly_amp} then follows from a
combination of this result for GOE matrices together with a combinatorial
analysis over a class of tensor networks, which we discuss in
Section \ref{sec:tensor}.

\subsection{State evolution and universality for Lipschitz AMP}

To extend the preceding universality guarantee to 
AMP algorithms (\ref{eq:AMP}) defined by Lipschitz functions $f_1,f_2,\ldots$,
we define the following polynomial approximability condition.

\begin{definition}\label{def:BCP_approx}
Let $\cF=\bigsqcup_{t=0}^T \cF_t$ be a set of functions, where $\cF_t$
consists of functions $f:\RR^{n \times t} \to \RR^n$ and $\cF_0$ consists of
constant vectors in $\RR^n$.
$\cF$ is {\bf BCP-approximable} if, for any fixed $C_0,\epsilon>0$:
\begin{enumerate}
\item There exists a BCP-representable set of polynomial functions
$\cP=\bigsqcup_{t=0}^T \cP_t$ such that for each $t=0,1,\ldots,T$, each
$f \in \cF_t$, and each $\bSigma \in \RR^{t \times t}$ with
$\|\bSigma\|_\op<C_0$, there exists $p \in \cP_t$ for which
\begin{equation}\label{eq:BCPpolyapprox}
\frac{1}{n}\,\EE_{\bZ \sim \cN(0,\bSigma \otimes \Id_n)}
\|f(\bZ)-p(\bZ)\|_2^2<\epsilon
\end{equation}
(For $t=0$, this requires $n^{-1}\|f(\cdot)-p(\cdot)\|_2^2<\epsilon$ for the
constant functions $f \in \cF_0$ and $p \in \cP_0$.)
If $f \in \cF_t$ depends only on inputs $\{\bz_s:s \in S_t\}$
for a subset of columns $S_t \subset \{1,\ldots,t\}$, then so does $p$.
\item There exists a set $\cQ=\bigsqcup_{t=0}^T \cQ_t$ of polynomial functions
(typically of unbounded degree) for which the following holds:

Fix any $t=1,\ldots,T$ and any ($n$-indexed sequences of)
$f \in \cF_t$, $\bSigma \in \RR^{t \times t}$ with $\|\bSigma\|_\op<C_0$,
and possibly random $\bz \in \RR^{n \times t}$. Suppose, for any
($n$-indexed sequences of) $q_1,q_2 \in \cQ_t$ with degrees bounded
independently of $n$, that $\cP \cup \{q_1,q_2\}$ remains BCP-representable,
and
\begin{equation}\label{eq:momentconvergence}
\lim_{n \to \infty} \frac{1}{n}q_1(\bz)^\top q_2(\bz)
-\frac{1}{n}\,\EE_{\bZ \sim \cN(0,\bSigma \otimes \Id_n)} q_1(\bZ)^\top q_2(\bZ)=0 \text{ a.s.}
\end{equation}
Then for the above polynomial $p \in \cP_t$ satisfying (\ref{eq:BCPpolyapprox}),
also
\[\limsup_{n \to \infty} \frac{1}{n} \|f(\bz)-p(\bz)\|_2^2<\epsilon
\text{ a.s.}\]
\end{enumerate}
\end{definition}

Condition (1) for BCP-approximability
is a statement about approximability of $\cF$ by $\cP$,
while condition (2) may be understood as an $L^2$-density condition for $\cQ$.
The following illustrates a simple example of both conditions of
this definition for separable Lipschitz functions.

\begin{example}[Separable Lipschitz functions are BCP-approximable]
\label{ex:Sep2} Fix any $L>0$, and let $\cF=\bigsqcup_{t=0}^T \cF_t$ be a set of
separable Lipschitz functions such that each $f \in \cF_t$ is
given by $f(\bz_1,\ldots,\bz_t)[i]=\mathring{f}(\bz_1[i],\ldots,\bz_t[i])$
for some $\mathring{f}:\RR^t \to \RR$ satisfying
\begin{equation}\label{eq:scalarlipschitz}
|\mathring{f}(0)| \leq L,
\quad |\mathring{f}(x)-\mathring{f}(y)| \leq L\|x-y\|_2 \text{ for all } x,y \in
\RR^t.
\end{equation}
We claim that $\cF$ is BCP-approximable.
To see this, note that fixing any $L,C_0,\epsilon>0$, there exist constants
$D,B>0$ such that for any function $\mathring{f}:\RR^t \to \RR$ satisfying
(\ref{eq:scalarlipschitz}), there exists a polynomial
$\mathring{p}:\RR^t \to \RR$ of degree at most $D$ and coefficients
bounded in magnitude by $B$ for which
\[\EE_{Z \sim \cN(0,\bSigma)}|\mathring{f}(Z)-\mathring{p}(Z)|^2<\epsilon\]
for any $\bSigma \in \RR^{t \times t}$ satisfying $\|\bSigma\|_\op \leq C_0$.
(We will verify a more general version of this statement in the proof of
Proposition \ref{prop:local} to follow.) Letting $\cP$ be the set of
corresponding polynomials $p:\RR^{n \times t} \to \RR^n$ given by
$p(\bz_1,\ldots,\bz_t)[i]=\mathring p(\bz_1[i],\ldots,\bz_t[i])$, this set $\cP$ is
BCP-representable by Example \ref{ex:Sep}. Condition (1) of
Definition \ref{def:BCP_approx} holds since
\[\frac{1}{n}\,\EE_{\bZ \sim \cN(0,\bSigma \otimes \Id_n)}
\|f(\bZ)-p(\bZ)\|_2^2
=\EE_{Z \sim \cN(0,\bSigma)} |\mathring f(Z)-\mathring p(Z)|^2<\epsilon.\]

Furthermore, let $\cQ=\bigsqcup_{t=0}^T \cQ_t$ such that $\cQ_t$ is the set of
all separable functions having the form
$q(\bz_1,\ldots,\bz_t)[i]=\mathring q(\bz_1[i],\ldots,\bz_t[i])$ where
$\mathring q:\RR^t \to \RR$ is a monomial function of its arguments with
coefficient 1.
By Example \ref{ex:Sep}, $\cP \cup \{q_1,q_2\}$ is also BCP-representable
for any $q_1,q_2 \in \cQ_t$ of bounded degrees. If
$\bz=\bz_{1:t} \in \RR^{n \times t}$
satisfies (\ref{eq:momentconvergence}) for any such $q_1,q_2$,
then the differences in mixed moments of the empirical distribution
$\frac{1}{n}\sum_{i=1}^n \delta_{\bz[i]}$ and the multivariate normal
distribution $\cN(0,\bSigma)$
converge to 0 a.s. This implies that their Wasserstein-$k$ distance converges to
0 a.s.\ for any fixed order $k \geq 1$, which in turn implies
\begin{align*}
&\lim_{n \to \infty} \frac{1}{n}\|f(\bz)-p(\bz)\|_2^2
-\frac{1}{n}\,\EE\|f(\bZ)-p(\bZ)\|_2^2\\
&=\lim_{n \to \infty} \frac{1}{n}\sum_{i=1}^n (\mathring f(\bz[i])-\mathring
p(\bz[i]))^2-\EE_{Z \sim \cN(0,\bSigma)}
(\mathring{f}(Z)-\mathring{p}(Z))^2 =0 \text{ a.s.}
\end{align*}
since $(\mathring f-\mathring p)^2$ is a continuous function of polynomial
growth.
(We will also carry out a more general version of this argument in the proof
of Proposition \ref{prop:local} to follow.)
Then condition (2) of Definition \ref{def:BCP_approx} also holds,
verifying the BCP-approximability of $\cF$.
\end{example}

Our second main result shows that BCP-approximability for uniformly
Lipschitz functions $f_0,f_1,f_2,\ldots$
is sufficient to ensure the validity and universality of the state
evolution approximation (\ref{eq:SEapprox}) for a corresponding class of
pseudo-Lipschitz test functions.

\begin{theorem}\label{thm:main_universality}
Fix any $T \geq 1$, consider an AMP algorithm (\ref{eq:AMP}) defined by
$f_0,f_1,\ldots,f_{T-1}$, and consider a test function
\begin{equation}\phi(\bz_{1:T})=\frac{1}{n}\phi_1(\bz_{1:T})^\top \phi_2(\bz_{1:T})\label{eq:testInnerLip}\end{equation}
where $\phi_1,\phi_2:\RR^{n \times T} \to \RR^n$. Let
$b_{ts}$, $\bSigma_t$, and $\bZ_t$ be as in Definition \ref{def:non_asymp_se}.

Suppose that $\cF=\{f_0,f_1,\ldots,f_{T-1},\phi_1,\phi_2\}$ is
BCP-approximable, and there exists a constant $L>0$ such that
for each $f \in \cF$ and any arguments $\bx,\by$ of $f(\cdot)$,
\begin{equation}\label{eq:Lipschitz}
\|f(0)\|_2 \leq L\sqrt{n},
\qquad \|f(\bx)-f(\by)\|_2 \leq L\|\bx-\by\|_{\Fro}.
\end{equation}
For each $t=1,\ldots,T-1$, suppose there is a fixed set of preceding iterates
$S_t \subseteq \{1,\ldots,t\}$ for which $f_t(\bz_{1:t})$ depends only on
$\{\bz_s:s \in S_t\}$, and $\lambda_{\min}(\bSigma_t[S_t,S_t])>c$ for a constant $c>0$.
If $\bW$ is any Wigner matrix satisfying Assumption \ref{assump:Wigner}, then
almost surely
\begin{equation}\label{eq:thmSE}
\lim_{n \to \infty} \phi(\bz_{1:T})-\EE[\phi(\bZ_{1:T})]=0.
\end{equation}
\end{theorem}

\begin{remark}
In these results, the test functions $\phi$ of the forms
\eqref{eq:testfunction} or \eqref{eq:testInnerLip} also need not be separable
across coordinates. This allows, for example, $\phi_1,\phi_2$ to represent
functions of $\bz_{1:T}$ that have a similar form as the AMP non-linearities
(e.g.\ sliding window filters or singular value thresholds) or to represent
differences
between such functions and an underlying signal vector $\btheta_* \in \RR^n$.
Then $\phi$ may represent a quantity such as the mean-squared-error of an
AMP iterate
\[\phi(\bz_{1:T})=\frac{1}{n}\|f_T(\bz_{1:T})-\btheta_*\|_2^2.\]
We note that in Theorem \ref{thm:main_universality}, the condition
\eqref{eq:Lipschitz} requires $\phi_1,\phi_2$ to also be Lipschitz, and hence
$\phi$ to be pseudo-Lipschitz of order 2.

However, the result of Theorem \ref{thm:main_universality}
does not encompass general (non-separable) pseudo-Lipschitz test
functions, for the same reason that it cannot encompass general
(non-separable) Lipschitz AMP non-linearities: We require the test function
to be approximable by the same tensor class
satisfying the BCP as the non-linearities of the AMP algorithm, 
in order to avoid the type of non-universality phenomenon exhibited in
Example \ref{ex:nonuniversalexample}. This is captured
naturally by the representation \eqref{eq:testInnerLip}.
\end{remark}

Theorem \ref{thm:main_universality} is proven using the preceding Theorem
\ref{thm:universality_poly_amp} and a polynomial approximation argument that is
similar to that of \cite{dudeja2023universality,Wang2024}, and we carry this out
in Appendix \ref{sec:univ_poly_approx}. We clarify that in common
applications where $f_t(\cdot)$ depends only on the single preceding iterate
$\bz_t$, we have $S_t=\{t\}$, so the
above condition $\lambda_{\min}(\bSigma_t[S_t,S_t])>c$ requires only that the
diagonal entries of $\bSigma_t$ are bounded away from 0. This is a mild
assumption that would hold whenever $n^{-1}\EE\|f_t(\bZ_t)\|_2^2$ is bounded away
from 0 for i.i.d.\ Gaussian inputs $\bZ_t$ with constant
variance. It is the same condition as assumed in \cite[Theorem 1]{berthier2017}, and
constitutes a weakening of the stronger non-singularity requirement
$\lambda_{\min}(\bSigma_t)>c$ in Theorem \ref{thm:universality_poly_amp};
this weakening is obtained via analysis of an AMP algorithm with non-linearities
perturbed by additional random noise (see e.g.\ \cite[Section 5.4]{berthier2017}
and \cite[Appendix D]{fan2021} for similar arguments).

The following corollary helps clarify that Theorem
\ref{thm:main_universality} remains valid under
asymptotically equivalent definitions of the Onsager coefficients and state
evolution covariances.

\begin{corollary}\label{cor:equivalent_SE}
Let $\{b_{ts}\}_{s<t}$ and $\{\bSigma_t\}_{t \geq 1}$
be the ($n$-dependent) quantities of
Definition \ref{def:non_asymp_se}, and let $\{\bar  b_{ts}\}_{s<t}$ and
$\{\bar \bSigma_t\}_{t \geq 1}$ be any (possibly random, $n$-dependent)
quantities satisfying
\[\lim_{n \to \infty} b_{ts}-\bar b_{ts}=0,
\quad \lim_{n \to \infty} \bSigma_t-\bar \bSigma_t=0 \text{ a.s.}\]
for each fixed $s,t$.
Then Theorem \ref{thm:main_universality} continues to hold for
the AMP algorithm defined with $\{\bar b_{ts}\}$ in place of $\{b_{ts}\}$, and
with Gaussian state evolution vectors $\bZ_{1:t}$ defined by
$\bar\bSigma_t$ in place of $\bSigma_t$.
\end{corollary}

For example, if $f_1,f_2,\ldots$ are such that the limits $\lim_{n
\to \infty} b_{ts}$ and $\lim_{n \to \infty} \bSigma_t$ exist, then Corollary
\ref{cor:equivalent_SE} applied with $\bar b_{ts}=\lim_{n \to \infty} b_{ts}$
and $\bar \bSigma_t=\lim_{n \to \infty} \bSigma_t$
shows that Theorem \ref{thm:main_universality} holds
equally with these limits 
in place of $b_{ts}$ and $\bSigma_t$. In practice, one typically uses
data-driven estimates $\hat b_{ts}$ and $\widehat\bSigma_t$, and applying Corollary
\ref{cor:equivalent_SE} instead with
$\bar b_{ts}=\hat b_{ts}$ and $\bar \bSigma_t=\widehat \bSigma_t$ shows
that Theorem \ref{thm:main_universality} also holds with these estimates
in place of $b_{ts}$ and $\bSigma_t$, as long as the estimates are consistent
in the almost-sure sense as $n \to \infty$.

\begin{remark}[Incorporating side information]
Many applications of AMP require the functions $f_1,f_2,\ldots$ to depend on
auxiliary ``side information'' vectors, in order to cast a desired algorithm
for an inference problem into an AMP form. We will discuss several such
examples in Section \ref{sec:SetupRect} to follow, where side information
vectors represent the signal and noise vectors in a statistical model.

The generality of the functions $f_t$ --- which need not be
exchangeable across their $n$ input coordinates --- allows us to incorporate
such side information vectors into the function definitions themselves.
For example, Theorem \ref{thm:main_universality} encompasses the more general
AMP algorithm
\[\bz_t=\bW\bu_t-\sum_{s=1}^{t-1} b_{ts}\bu_s,
\qquad \bu_{t+1}=\tilde f_t(\bz_1,\ldots,\bz_t,\bgamma_1,\ldots,\bgamma_k)\]
for Lipschitz functions $\tilde f_t:\RR^{n \times (t+k)} \to \RR^n$ depending 
on side information vectors $\bgamma_1,\ldots,\bgamma_k \in \RR^n$, upon
identifying $f_t(\cdot) \equiv \tilde f_t(\,\cdot\,,\bgamma_1,\ldots,\bgamma_k)$.

If $\bgamma_j \equiv \bgamma_j(n) \in \RR^n$ for $j=1,\ldots,k$ are random and independent of $\bW$,
then Theorem \ref{thm:main_universality} may be applied in such settings
conditionally on $\bgamma_1,\ldots,\bgamma_k$, provided that
$\cF=\{f_0,f_1,\ldots,f_{T-1},\phi_1,\phi_2\}$ is BCP-approximable
almost surely with respect to the randomness of the infinite sequences
$\{\bgamma_1(n),\ldots,\bgamma_k(n)\}_{n=1}^\infty$. In this context,
\begin{align*}
b_{t+1,s}&=\frac{1}{n}\,\EE[\div_s \tilde
f_t(\bZ_{1:t},\bgamma_{1:k}) \mid \bgamma_{1:k}]\\
\bSigma_{t+1}[r+1,s+1]&=\frac{1}{n}\,
\EE[\tilde f_r(\bZ_{1:r},\bgamma_{1:k})^\top \tilde
f_s(\bZ_{1:s},\bgamma_{1:k}) \mid \bgamma_{1:k}]
\end{align*}
of Definition \ref{def:non_asymp_se} are also defined conditionally on
$\bgamma_{1:k}$. Corollary \ref{cor:equivalent_SE} implies that in such
settings, we may replace these by the deterministic unconditional quantities
\begin{align*}
\bar b_{t+1,s}&=\frac{1}{n}\,\EE[\div_s \tilde
f_t(\bar\bZ_{1:t},\bgamma_{1:k})]\\
\bar \bSigma_{t+1}[r+1,s+1]&=\frac{1}{n}\,
\EE[\tilde f_r(\bar\bZ_{1:r},\bgamma_{1:k})^\top \tilde f_s(\bar\bZ_{1:s},\bgamma_{1:k})]
\end{align*}
defined iteratively with $\bar \bZ_{1:t} \sim \cN(0,\bar\bSigma_t \otimes
\Id_n)$ independent of $\bgamma_{1:k}$, as long as for each fixed $s,t$
we have the almost-sure concentration
\[\lim_{n \to \infty} b_{ts}-\bar b_{ts}=0,
\quad \lim_{n \to \infty} \bSigma_t-\bar \bSigma_t=0,\]
that can often be established inductively on $t$.
\end{remark}

\subsection{Examples}\label{sec:examples}

We next establish that three examples of uniformly Lipschitz non-separable
functions, which arise across a variety of applications,
satisfy this condition of BCP-approximability. Proofs of the results of this
section are given in Appendix \ref{sec:BCPexamples}.

\subsubsection{Local functions}

Consider a natural extension of separable functions, where each output
coordinate of $f_t:\RR^{n \times t} \to \RR^n$ depends on only $O(1)$ rows of
its input, and conversely each row of its input affects only $O(1)$ coordinates
of its output. Such functions include
convolution kernels and sliding-window filters with bounded support, for which
AMP algorithms have been developed and studied previously in
\cite{metzler2016denoising,ma2016approximate,berthier2017,Ma2019}.

We will call such functions ``local'' (although we note that this locality need
not be with respect to any sequential or spatial interpretation of
the coordinates of $\RR^n$). We define formally the following classes.

\begin{definition}\label{def:polylocal}
$\cP=\bigsqcup_{t=0}^T \cP_t$ is a set of {\bf polynomial local functions}
if, for some constants $A,D,B>0$ independent of $n$,
every function $p \in \cP_t$ satisfies the following properties:
\begin{enumerate}
\item (\emph{Locality}) For each $i\in[n]$, there exists a
subset $A_i \subset [n]$ and a function
$\mathring{p}_i:\RR^{|A_i| \times t} \to \RR$ such that
$p(\bz)[i]=\mathring{p}_i(\bz[A_i])$, where $\bz[A_i] \in \RR^{|A_i| \times t}$
are the rows of $\bz$ belonging to $A_i$. For each $i \in [n]$, we have
$|A_i| \leq A$ and $|\{j\in[n]:i\in A_j\}| \leq A$.
\item (\emph{Boundedness}) All such polynomials $\mathring p_i$ have degree at most
$D$ and all coefficients bounded in magnitude by $B$.
\end{enumerate}
\end{definition}

\begin{definition}\label{def:local} 
$\cF=\bigsqcup_{t=0}^T \cF_t$ is a set of {\bf Lipschitz local functions}
if, for some constants $A,L>0$ independent of $n$,
every function $f \in \cF_t$ satisfies the following properties:
\begin{enumerate}
\item (\emph{Locality}) 
Each $f \in \cF_t$ is given by $f=(\mathring f_i)_{i=1}^n$, for functions
$\mathring f_i:\RR^{|A_i| \times t} \to \RR$ satisfying the same locality
condition (1) as in Definition \ref{def:polylocal}.
\item (\emph{Lipschitz continuity}) Each $\mathring{f}_i:\RR^{|A_i| \times t}
\to \RR$ satisfies
\begin{equation}\label{eq:Lidef}
|\mathring{f}_i(0)| \leq L, \qquad |\mathring{f}_i(\bx)-\mathring{f}_i(\by)|
\leq L \|\bx-\by\|_\Fro
\text{ for all } \bx,\by \in \RR^{|A_i| \times t}.
\end{equation}
\end{enumerate}
\end{definition}

These definitions allow the functions $\mathring p_i$ and $\mathring f_i$ to
differ across coordinates, so that they may incorporate differing local
function definitions and also side information vectors.
The following proposition shows that any such function classes $\cP$/$\cF$ are
BCP-representable/BCP-approximable.

\begin{proposition}\label{prop:local}
\phantom{}
\begin{enumerate}[(a)]
\item If $\cP=\{f_0,\ldots,f_{T-1},\phi_1,\phi_2\}$ in
Theorem \ref{thm:universality_poly_amp} is a set of polynomial
local functions, then it is BCP-representable.
\item If $\cF=\{f_0,\ldots,f_{T-1},\phi_1,\phi_2\}$ in
Theorem \ref{thm:main_universality} is a set of Lipschitz local functions,
then it is BCP-approximable.
\end{enumerate}
\end{proposition}

Thus the universality statements of Theorems \ref{thm:universality_poly_amp}
and \ref{thm:main_universality} hold for AMP algorithms where both the driving
functions $f_0,f_1,\ldots,f_{T-1}$ and test function $\phi$ are local in this
sense.

\subsubsection{Anisotropic functions}

A second example is motivated by applications in which a separable AMP
algorithm of the form (\ref{eq:AMP}) is applied to a
matrix having row and column correlation. We consider here an algorithm
\begin{equation}
\label{eq:AMPAniso2}
\begin{split}
    \tilde{\bz}_t
    = \tilde{\bW}\tilde \bu_t - \text{Onsager correction},
\qquad \tilde{\bu}_{t+1} &= \tilde f_t(\tilde \bz_1, \dots, \tilde \bz_t)
\end{split}
\end{equation}
where $\tilde f_1,\tilde f_2,\ldots$ are separable functions, and
$\tilde\bW=\bK^\top \bW\bK$ where $\bW \in \RR^{n \times n}$ is a
Wigner matrix and $\bK \in \RR^{n \times n}$ is a bounded and invertible
linear transform.

To analyze such an algorithm, the following type of reduction to a
non-separable AMP algorithm has been used previously in
e.g.\ \cite{loureiro2021learning,loureiro2022fluctuations,zhang2024matrix,wang2025glamp}, and suggested also for the analysis of SGD in \cite{gerbelot2024rigorous}: Note that the iterations (\ref{eq:AMPAniso2})
are equivalent to the algorithm (\ref{eq:AMP}) applied to $\bW$,
upon identifying
\[\bu_t=\bK \tilde \bu_t,
\qquad \bz_t=(\bK^{-1})^\top \tilde\bz_t,
\qquad f_t(\bz_{1:t})=\bK \tilde f_t(\bK^\top \bz_{1:t}).
\]
(The Onsager correction for $\bz_t$ in (\ref{eq:AMP}) is given by
$\sum_{s=1}^{t-1} b_{ts}\bu_s$, which leads to a form
$\sum_{s=1}^{t-1} b_{ts}\bK^\top\bK \tilde \bu_s$ for the Onsager correction for
$\tilde \bz_t$ in (\ref{eq:AMPAniso2}).) 
 Thus the iterates of (\ref{eq:AMPAniso2}) may be
studied via analysis of \eqref{eq:AMP} for non-separable functions
belonging to the following classes.

\begin{definition}\label{def:polyaniso}
$\cP=\bigsqcup_{t=0}^T \cP_t$ is a set of {\bf polynomial anisotropic
functions} with respect to $\cK \subset \RR^{n \times n}$
if there exist constants $D,B>0$ such that
\begin{itemize}
\item For each $t=0,1,\ldots,T$ and $p \in \cP_t$, 
there is a separable function $q:\RR^{n\times t}\to\RR^n$ given by
$q(\bz_{1:t})[i]=\mathring q_i(\bz_{1:t}[i])$ for some functions $\mathring
q_i:\RR^t \to \RR$, and two matrices $\bK',\bK \in \cK$, such that
\[p(\bz_{1:t})=\bK' q(\bK^\top \bz_{1:t}).\]
\item All components $\mathring q_i:\RR^t \to \RR$ of $q$
have degree at most $D$ and all coefficients bounded in magnitude by $B$.
\end{itemize}
(For $t=0$, this means $q(\cdot)$ is a constant function that is
entrywise bounded by $B$, and $p(\cdot)=\bK'q(\cdot)$ for some $\bK' \in \cK$.)
\end{definition}

\begin{definition}\label{def:aniso}
$\cF=\bigsqcup_{t=0}^T \cF_t$ is a set of {\bf Lipschitz anisotropic functions}
with respect to a set of matrices $\cK \subset \RR^{n \times n}$ if
there exists a constant $L>0$ such that:
\begin{itemize}
\item For each $t=0,1,\ldots,T$ and $f \in \cF_t$, there is
a separable function $g:\RR^{n\times t}\to\RR^n$ given by
$g(\bz_{1:t})[i]=\mathring g_i(\bz_{1:t}[i])$ for some functions $\mathring
g_i:\RR^t \to \RR$, and two matrices $\bK',\bK \in \cK$, such that
\[f(\bz_{1:t})=\bK' g(\bK^\top \bz_{1:t})\]
\item Each function $\mathring g_i:\RR^t\to\RR$ above satisfies
\begin{equation}\label{eq:anisolipschitz}
|\mathring g_i(0)| \leq L, \qquad
|\mathring g_i(\bx)-\mathring g_i(\by)| \leq L\|\bx-\by\|_2
\text{ for all } \bx,\by \in \RR^t.
\end{equation}
\end{itemize}
\end{definition}

We note that $\cP$ may not be BCP-representable (and $\cF$ may not be
BCP-approximable) if rows or columns of matrices
in $\cK$ align with the constant components of $q(\cdot)$ (resp.\ of $g(\cdot)$),
for reasons similar to Example \ref{ex:nonuniversalexample}. The
following proposition shows that if, instead, the matrices $\cK$ are bounded
in $\ell_\infty \to \ell_\infty$ operator norm or have suitably generic shared
singular vectors, then BCP-representability and BCP-approximability hold.

\begin{proposition}\label{prop:anisotropic}
Let $\cK \subset \RR^{n \times n}$ be a set of matrices such that
for a constant $C>0$, either
\begin{enumerate}
\item $\|\bK\|_{\ell_\infty \to \ell_\infty} \equiv \max_i \sum_j
|\bK[i,j]|<C$ and $\|\bK^\top\|_{\ell_\infty \to \ell_\infty}<C$
for all $\bK \in \cK$, or
\item $\cK=\{\bO\bD\bU^\top:\bD \in \cD\}$ for a set of deterministic diagonal
matrices $\cD \subset \RR^{n \times n}$ with $\sup_{\bD \in \cD} \|\bD\|_\op<C$,
and two independent random orthogonal matrices
$\bO \equiv \bO(n) \in \RR^{n \times n}$ and $\bU \equiv
\bU(n) \in \RR^{n \times n}$ (that are also independent of $\bW$ and all other
randomness, and shared by all $\bK \in \cK$) whose laws
have densities with respect to Haar measure uniformly bounded above by $C$.
\end{enumerate}
Then the following hold.
\begin{enumerate}[(a)]
\item Let $\cP=\{f_0,\ldots,f_{T-1},\phi_1,\phi_2\}$ in
Theorem \ref{thm:universality_poly_amp} be a class of polynomial anisotropic
functions with respect to $\cK$.
Then $\cP$ is BCP-representable, almost surely with respect to 
$\{\bO(n),\bU(n)\}_{n=1}^\infty$ under condition (2).
\item Let $\cF=\{f_0,\ldots,f_{T-1},\phi_1,\phi_2\}$ in
Theorem \ref{thm:main_universality} be a class of Lipschitz anisotropic
functions with respect to $\cK$.
Then $\cF$ is BCP-approximable, almost surely with respect to
$\{\bO(n),\bU(n)\}_{n=1}^\infty$ under condition (2).
\end{enumerate}
\end{proposition}

Thus the universality claims of Theorems \ref{thm:universality_poly_amp} and
\ref{thm:main_universality} hold for the analysis of \eqref{eq:AMPAniso2} as
long as $\cK=\{\bK\}$ satisfies one of these two conditions.

\subsubsection{Spectral functions}\label{subsec:spec}

A third example is motivated by matrix sensing applications
\cite{donoho2013phase,berthier2017,romanov2018near,xu2025fundamental},
in which we explicitly identify $\RR^n \equiv \RR^{M \times N}$ as a matrix
space with $n=MN$ and $M \asymp N \asymp \sqrt{n}$.
We consider non-linear functions given by
transformations of singular values on this matrix space.

Formally, consider the vectorization map $\vec:\RR^{M \times N} \to \RR^n$
given by
\[\vec(\bX)
=(\bX[1,1],\ldots,\bX[M,1],\ldots,\bX[1,N],\ldots,\bX[M,N])^\top \in \RR^n\]
and its inverse map $\mat:\RR^n \to \RR^{M \times N}$.
For a scalar function $g:[0,\infty) \to \RR$ and matrix $\bX \in \RR^{M \times N}$
with singular value decomposition $\bX=\bO\bD\bU^\top$ and singular values
$\bD=\diag(d_1,\ldots,d_{\min(M,N)}) \in \RR^{M \times N}$,
we define $g(\bX)$ via the spectral calculus
\begin{equation}\label{eq:spectralcalculus}
g(\bX)=\bO g(\bD)\bU^\top,
\qquad g(\bD)=\diag(g(d_1),\ldots,g(d_{\min(M,N)})) \in \RR^{M \times N}.
\end{equation}
Thus $g(\cdot)$ is applied spectrally to the singular values of $\bX$. We
consider the following class of functions, given by sums of 
Lipschitz spectral maps applied to linear combinations of 
$\mat(\bz_1),\ldots,\mat(\bz_t)$ and a signal matrix $\bTheta_* \in
\RR^{M \times N}$.

\begin{definition}\label{def:spectral}
$\cF=\bigsqcup_{t=0}^T \cF_t$ is a set of {\bf Lipschitz spectral functions}
with shift $\bTheta_* \in \RR^{M \times N}$ if, for some constants $C,K,L>0$:
\begin{itemize}
\item For each $t=0,1,\ldots,T$ and each $f \in \cF_t$, there exist scalar
functions $g_1,\ldots,g_K:[0,\infty) \to \RR$ and coefficients
$\{c_{ks}\}_{k \in [K],s \in [t]}$ with $|c_{ks}|<C$ for which
\begin{equation}\label{eq:def_spectral_func}
f(\bz_1,\ldots,\bz_t)=\sum_{k=1}^K \vec\bigg(g_k\bigg(
\sum_{s=1}^t c_{ks}\mat(\bz_s)+\bTheta_*\bigg)\bigg)
\end{equation}
where $g_k(\cdot)$ is applied spectrally to the singular values of its input
as in (\ref{eq:spectralcalculus}).
\item Each function $g_k$ satisfies
\[g_k(0)=0,
\qquad |g_k(x)-g_k(y)| \leq L|x-y| \text{ for all } x,y \geq 0.\]
\end{itemize}
\end{definition}

In our examples to follow, $\bTheta_* \in \RR^{M \times N}$ will play the role
of a signal matrix, and $g_k(\cdot)$ may represent a singular value
thresholding function such as $g_k(x)=\sign(x)(x-\lambda\sqrt{N})_+$ for some
constant $\lambda>0$. The following proposition ensures that if the singular
vectors of $\bTheta_*$ are suitably generic, then such functions are
BCP-approximable. We defer a discussion of a corresponding class of polynomial
spectral functions that are BCP-representable to Appendix \ref{sec:BCPexamples}.

\begin{proposition}\label{prop:spectral}
Let $\cF=\{f_0,\ldots,f_{T-1},\phi_1,\phi_2\}$ in Theorem
\ref{thm:main_universality} be a set of Lipschitz spectral functions with
shift $\bTheta_* \in \RR^{M \times N}$, where $MN=n$. As $n \to \infty$,
suppose $M/N \to \delta$ for some constant $\delta \in (0,\infty)$,
and $\bTheta_*=\bO\bD\bU^\top$ where
\begin{itemize}
\item $\bD \in \RR^{M \times N}$ is a deterministic diagonal matrix satisfying
$\|\bD\|_\op<C\sqrt{N}$.
\item $\bO \equiv \bO(n) \in \RR^{M \times M}$
and $\bU \equiv \bU(n) \in \RR^{N \times N}$ are independent random
orthogonal matrices (also independent of $\bW$ and
all other randomness) having densities with respect to Haar measure uniformly
bounded above by $C$.
\end{itemize}
Then $\cF$ is BCP-approximable, almost surely with respect to
$\{\bO(n),\bU(n)\}_{n=1}^\infty$.
\end{proposition}

\section{Universality of asymmetric AMP}\label{sec:SetupRect}

The preceding ideas are readily extendable to AMP algorithms beyond the
symmetric matrix setting of (\ref{eq:AMP}). 
We discuss here the extension to asymmetric
matrices, as this encompasses many applications of interest for
non-separable AMP algorithms. We anticipate that similar
extensions may be developed for more general procedures such as the class of
graph-based AMP methods discussed in \cite{Gerbelot2022}.

Let $\bu_1 \in \RR^n$ be an initialization, and let
$f_t:\RR^{m \times t} \to \RR^m$ and $g_t:\RR^{n \times t} \to \RR^n$ be two
sequences of non-linear functions for $t=1,2,3,\ldots$ For a matrix
$\bW \in \RR^{m \times n}$, we consider the AMP algorithm
\begin{equation}
\label{eq:RectAMP}\begin{split} 
    \bz_t &= \bW \bu_t - \sum_{s=1}^{t-1} b_{ts} \bv_s\\
    \bv_t &= f_t(\bz_1,\ldots,\bz_t)\\
    \by_t &= \bW^\top \bv_t - \sum_{s=1}^t a_{ts} \bu_s\\
    \bu_{t+1} &= g_t(\by_1,\ldots,\by_t).
\end{split}
\end{equation}
For convenience, we define the constant function $g_0(\cdot)$ by the
initialization
\[\bu_1 \equiv g_0(\cdot).\]

The Onsager coefficients $b_{ts},a_{ts}$ and corresponding state evolution
are defined analogously to Definition \ref{def:non_asymp_se} as follows.

\begin{definition}\label{def:RectSE}
Let $\bOmega_1 = \frac{1}{m}\|\bu_1\|_2^2 \in \RR^{1\times 1}$. Iteratively
for each $t \geq 1$, given $\bOmega_t \in \RR^{t \times t}$, let
$\bZ_{1:t} \sim \cN(0,\bOmega_t \otimes \Id_m)$, i.e.\ $\bZ_{1:t} \in
\RR^{m \times t}$ has i.i.d.\ rows with distribution
$\cN(0,\bOmega_t)$. Define $\bSigma_t \in \RR^{t \times t}$ entrywise by
\[\bSigma_t[r,s]=\frac{1}{m}\EE[f_r(\bZ_{1:r})^\top f_s(\bZ_{1:s})] \text{ for }
r,s=1,\ldots,t.\]
Then, given $\bSigma_t \in \RR^{t \times t}$, let
$\bY_{1:t} \sim \cN(0,\bSigma_t \otimes \Id_n)$, and define $\bOmega_{t+1} \in
\RR^{(t+1) \times (t+1)}$ entrywise by
\[\bOmega_{t+1}[r+1,s+1] 
=\frac{1}{m}\EE[g_r(\bY_{1:r})^\top g_s(\bY_{1:s})] \text{ for }
r,s=0,\ldots,t\]
where $g_0(\cdot) \equiv \bu_1$.
The Onsager coefficients $\{b_{ts}\}_{s<t}$ and $\{a_{ts}\}_{s \leq t}$
in (\ref{eq:RectAMP}) are defined as
    \[
    b_{ts} = \frac{1}{m} \EE[\div_s g_{t-1}(\bY_{1:(t-1)})],
    \qquad a_{ts} = \frac{1}{m} \EE[\div_s f_t(\bZ_{1:t})].
    \]
    The state evolution approximations of the iterates $\{\by_t,\bz_t\}_{t \geq
1}$ in \eqref{eq:RectAMP} are the sequences of Gaussian vectors
$\{\bY_t,\bZ_t\}_{t \geq 1}$.\end{definition}

We will show the validity and universality of this state evolution
approximation for the following class of asymmetric matrices $\bW \in \RR^{m
\times n}$ having independent entries.

\begin{assumption}\label{assump:Ginibre}
$\bW\in\RR^{m\times n}$ is a matrix with independent entries
$\{\bW[i,j]\}_{i \in [m],\,j \in [n]}$, such that for some positive constants
$C_3,C_4,\ldots$ independent of $n$ and all $i \in [m]$ and $j \in [n]$:
    \begin{itemize}
    \item $\EE \bW[i,j]=0$.
    \item $\EE \bW[i,j]^2=1/m$.
    \item $\EE|\bW[i,j]|^k \leq C_k m^{-k/2}$ for each $k \geq 3$.
    \end{itemize}
\end{assumption}

Our main result is the following
guarantee for the AMP algorithm \eqref{eq:RectAMP} driven
by BCP-representable polynomial functions or BCP-approximable
Lipschitz functions, which parallels Theorems 
\ref{thm:universality_poly_amp} and \ref{thm:main_universality}.

\begin{theorem}\label{thm:RectSE}
Fix any $T \geq 1$, consider an AMP algorithm (\ref{eq:RectAMP}) defined by
$g_0,g_1,\ldots,g_{T-1}$ and $f_1,f_2,\ldots,f_T$, and consider the test
functions
\[\phi(\bz_{1:T})=\frac{1}{m}\phi_1(\bz_{1:T})^\top\phi_2(\bz_{1:T}),
\qquad \psi(\by_{1:T})=\frac{1}{m}\psi_1(\by_{1:T})^\top\psi_2(\by_{1:T})\]
where $\phi_1,\phi_2:\RR^{m \times T} \to \RR^m$
and $\psi_1,\psi_2:\RR^{n \times T} \to \RR^n$. Let
$a_{ts},b_{ts},\bSigma_t,\bOmega_t,\bY_t,\bZ_t$ be as in Definition~\ref{def:RectSE}. Suppose that either:
\begin{enumerate}[(a)]
\item $\cF=\{f_1,\ldots,f_T,\phi_1,\phi_2\}$ and
$\cG=\{g_0,\ldots,g_{T-1},\psi_1,\psi_2\}$ are each a set of BCP-representable
polynomial functions with degrees bounded by a constant $D>0$, and
$\lambda_{\min}(\bOmega_t)>c$ and $\lambda_{\min}(\bSigma_t)>c$ for a constant
$c>0$ and each $t=1,\ldots,T$, or
\item $\cF=\{f_1,\ldots,f_T,\phi_1,\phi_2\}$
and $\cG=\{g_0,\ldots,g_{T-1},\psi_1,\psi_2\}$ are each a set of
BCP-approximable Lipschitz functions for which
there exists a constant $L>0$ such that
for any $f \in \cF$ or $f \in \cG$ and arguments $\bx,\by$ to $f$,
\begin{equation}\label{eq:RectLipschitz}
\|f(0)\|_2 \leq L\sqrt{n},
\qquad \|f(\bx)-f(\by)\|_2 \leq L\|\bx-\by\|_\Fro.
\end{equation}
Furthermore, for each $t=1,\ldots,T$, suppose there is a fixed
set $S_t \subseteq \{1,\ldots,t\}$ of preceding iterates $\{\bz_s:s \in S_t\}$
on which $f_t$ depends, and
$\lambda_{\min}(\bOmega_t[S_t,S_t])>c$ for a constant $c>0$, and 
the same holds for $g_t$ and $\bSigma_t$ for each $t=1,\ldots,T-1$.
\end{enumerate}
If $m,n \to \infty$ such that $c<m/n<C$ for some constants $C,c>0$,
and if $\bW$ is any matrix satisfying Assumption \ref{assump:Ginibre}, then
almost surely
    \[\lim_{m,n \to \infty} \phi(\bz_{1:T})-\EE \phi(\bZ_{1:T})=0, \qquad
    \lim_{m,n \to \infty} \psi(\by_{1:T})-\EE \psi(\bY_{1:T})=0.\]
\end{theorem}

The main assumption of Theorem \ref{thm:RectSE}
is that the sets of functions $\cF$ and $\cG$ are
separately BCP-representable or BCP-approximable as $m,n \to \infty$,
in the sense of Definitions \ref{def:BCPpoly} and \ref{def:BCP_approx}.
This encompasses the three classes of Lipschitz functions
discussed previously in Propositions \ref{prop:local}, \ref{prop:anisotropic},
and \ref{prop:spectral}, where we do not require $\cF$ and $\cG$ to consist of
functions of the same class.
Theorem \ref{thm:RectSE} is proven as a corollary of Theorems
\ref{thm:universality_poly_amp} and \ref{thm:main_universality} using
an embedding argument as introduced in \cite{Javanmard2012}, which
we provide in Appendix \ref{sec:Rectangle}.

To close out our results, let us illustrate three applications of Theorem
\ref{thm:RectSE} to the AMP algorithm (\ref{eq:AMPintro}) for matrix/vector
estimation discussed in the introduction, which parallel the three
function classes discussed in Section \ref{sec:examples}.

\begin{example}[AMP with local averaging]\label{ex:sensinglocal}
We observe measurements
\begin{equation}\label{eq:sensingmodel}
\bx=\bW\btheta_*+\be \in \RR^m
\end{equation}
of an unknown signal $\btheta_* \in \RR^n$, with measurement error/noise
$\be \in \RR^m$. Consider the AMP algorithm (\ref{eq:AMPintro}), whose form
we reproduce here for convenience:
\begin{equation}\label{eq:Berthier}
\begin{aligned}
    \br_t &= \bx - \bW\btheta_t + b_t \br_{t-1},\\
    \btheta_{t+1} &= \eta_t(\btheta_t + \bW^\top \br_t)
\end{aligned}
\end{equation}
This algorithm is initialized at $\btheta_1=\br_0=0$, with Onsager
coefficient $b_t=\frac{1}{m}\div \eta_{t-1}(\btheta_{t-1}+\bW^\top \br_{t-1})$.

Applying the change-of-variables $\bu_t=\btheta_*-\btheta_t$ and
$\bz_t=\br_t-\be$ (see e.g.\ \cite[Section 3.3]{Bayati2011}), this
procedure \eqref{eq:Berthier} is equivalent to the AMP
iterations \eqref{eq:RectAMP} given by
\begin{equation}\label{eq:equivAMP}
\begin{aligned}
    \bz_t &= \bW \bu_t - b_{t,t-1} \bv_{t-1}\\
    \bv_t &= f_t(\bz_t) \equiv \bz_t+\be\\
    \by_t &= \bW^\top \bv_t - \bu_t \quad \text{ (where $a_{tt}=1$)} \\
    \bu_{t+1} &= g_t(\by_t) \equiv \btheta_*-\eta_t(\by_t+\btheta_*)
\end{aligned}
\end{equation}
with $g_0(\cdot)=\bu_1=\btheta_*$. After $T$ iterations, the
reconstruction mean-squared-error of $\btheta_{T+1}$ is
\[\text{MSE}=\frac{1}{n}\|\btheta_{T+1}-\btheta_*\|_2^2
=\frac{1}{n}\|\psi_T(\by_T)\|_2^2, \text{ where } \psi_T=g_T.\]
Defining $\omega_1^2=\frac{1}{m}\|\bu_1\|_2^2=\frac{1}{m}\|\btheta_*\|_2^2$ and
the sequence of variances
\[\sigma_t^2=\frac{1}{m}\EE_{\bZ_t \sim \cN(0,\omega_t^2 \Id)}[\|f_t(\bZ_t)\|_2^2],
\qquad \omega_{t+1}^2=\frac{1}{m}\EE_{\bY_t \sim \cN(0,\sigma_t^2 \Id)}[\|g_t(\bY_t)\|_2^2],\]
state evolution predicts that
\begin{equation}\label{eq:sensingSE}
\lim_{m,n \to \infty} \text{MSE}
-\frac{1}{n}\EE_{\bY_T \sim \cN(0,\sigma_T^2
\Id)}[\|\btheta_*-\eta_T(\bY_T+\btheta_*)\|_2^2]=0.
\end{equation}

Suppose that $\bTheta_*=\mat(\btheta_*) \in \RR^{M \times N}$ is an image,
where we identify $\RR^{M \times N} \equiv \RR^n$ with $n=MN$ via the maps
$\vec:\RR^{M \times N} \to \RR^n$ and
$\mat:\RR^n \to \RR^{M \times N}$ as in Section \ref{sec:examples}.
Motivated by settings where $\bTheta_*$ is locally smooth,
let us consider an instantiation of this algorithm (\ref{eq:Berthier})
where $\eta_t:\RR^{M \times N} \to \RR^{M \times N}$
is given by a local averaging kernel smoother
\[\eta_t(\bz)[j,j']=\frac{1}{|\cS_{j,j'}^t|}
\sum_{(k,k') \in \cS_{j,j'}^t} \bz[k,k']\]
where $\cS_{j,j'}^t=\{(k,k'):|j-k|,|j'-k'| \leq h_t\}$ for a bandwidth parameter
$h_t \geq 0$. For any $\btheta_* \in \RR^n$ and $\be \in
\RR^m$ satisfying $\|\btheta_*\|_\infty,\|\be\|_\infty \leq C$,
the corresponding functions $\{f_1,\ldots,f_T\}$ and $\{g_0,\ldots,g_T\}$ in
(\ref{eq:equivAMP}) constitute two sets of Lipschitz local functions in the
sense of Definition \ref{def:local}. Then Theorem~\ref{thm:RectSE}
and Proposition \ref{prop:local} imply the validity of (\ref{eq:sensingSE}) for 
any i.i.d.\ measurement matrix $\bW$ satisfying
Assumption \ref{assump:Ginibre}. This universality guarantee has been depicted
in Figure \ref{fig:local}, corresponding to $M=N=150$, $n = 22500$, $m =
0.95\,n$, and fixed bandwidth $h_t=1$.
\end{example}

\begin{example}[AMP with spectral denoising]
Consider the same model (\ref{eq:sensingmodel})  and
algorithm (\ref{eq:Berthier}) as in Example \ref{ex:sensinglocal},
with the identification $\RR^{M \times N} \equiv \RR^n$.
Motivated by settings where $\bTheta_*=\mat(\btheta_*) \in \RR^{M \times N}$
is approximately of low rank, consider the
instantiation of (\ref{eq:Berthier}) where $\eta_t:\RR^{M \times N} \to \RR^{M
\times N}$ is given by a soft-thresholding function
\[\mathring \eta_t(x)=\sign(x) \cdot (x-\lambda_t\sqrt{N})_+\] applied
spectrally to the singular values of its input in $\RR^{M \times N}$, and
$\lambda_t>0$ is a $t$-dependent threshold level. Then the corresponding
functions $\{g_0,\ldots,g_T\}$ of (\ref{eq:equivAMP}) constitute a set of
Lipschitz spectral functions in the sense of Definition \ref{def:spectral}.
Suppose that $\bTheta_* \in \RR^{M \times N}$ has singular value decomposition
$\bTheta_*=\bO\bD\bU^\top$ where $\bO \in \RR^{M \times M}$ and $\bU \in \RR^{N
\times N}$ are generic in the sense of Proposition \ref{prop:spectral}, and
$\|\bD\|_\op<C\sqrt{N}$ and $\|\be\|_\infty<C$ for a constant $C>0$. Then
Theorem \ref{thm:RectSE}, Proposition \ref{prop:spectral},
and Proposition \ref{prop:local} again imply the validity of the state
evolution prediction (\ref{eq:sensingSE}) for any matrix $\bW$ satisfying
Assumption \ref{assump:Ginibre}.

This universality guarantee has been depicted in Figure
\ref{fig:spectral}, corresponding to $M=100$, $N = 150$, $m=n=15000$, and
a signal $\bTheta_*=\bO\bD\bU^\top$ where $\bO,\bU$ are Haar-uniform,
the first 20 diagonal elements of $\bD$ are generated uniformly from $[0,\sqrt{N}]$,
and the remaining 80 diagonal elements are zero. The threshold
$\lambda_t=0.05$ is fixed for all $t$, and the Onsager correction term $b_t$
is estimated using the Monte Carlo procedure of \cite{metzler2016denoising}.
\end{example}

\begin{example}[AMP for correlated measurements]\label{ex:sensinganisotropic}
We observe measurements
\[\bx=\tilde\bW\btheta_*+\be \in \RR^m\]
with a signal $\btheta_* \in \RR^n$ that is entrywise sparse, and a
measurement matrix $\tilde \bW$ that is of a colored form
$\tilde \bW=\bW\bK$ where $\bW$ is an i.i.d.\ matrix
satisfying Assumption \ref{assump:Ginibre} and $\bK \in \RR^{n \times n}$ is
an invertible linear map. Consider the AMP algorithm
\begin{equation}\label{eq:Berthieranisotropic}
\begin{aligned}
    \br_t &= \bx - \tilde \bW\btheta_t + b_t\br_{t-1}\\
    \btheta_{t+1} &= \eta_t(\btheta_t + (\bK^\top \bK)^{-1}\tilde \bW^\top \br_t)
\end{aligned}
\end{equation}
with initializations $\btheta_1=\br_0=0$,
where $\eta_t(\cdot)$ consists of a separable soft-thresholding function
$\mathring \eta_t(x)=\sign(x) \cdot (x-\lambda_t)_+$
applied entrywise, and
$b_t=\frac{1}{m}\div \eta_{t-1}(\btheta_{t-1} + (\bK^\top \bK)^{-1}\tilde \bW^\top \br_{t-1})$.

Applying the changes-of-variables $\bu_t=\bK(\btheta_*-\btheta_t)$ and
$\bz_t=\br_t-\be$, this procedure \eqref{eq:Berthieranisotropic}
is equivalent to the AMP iterations \eqref{eq:RectAMP} given by
\begin{align*}
    \bz_t &= \bW \bu_t - b_{t,t-1} \bv_{t-1}\\
    \bv_t &= f_t(\bz_t) \equiv \bz_t+\be\\
    \by_t &= \bW^\top \bv_t - \bu_t \quad \text{ (with $a_{tt}=1$)}\\
    \bu_{t+1} &= g_t(\by_t) \equiv \bK[\btheta_*-\eta_t((\bK^\top\bK)^{-1}\bK^\top
\by_t+\btheta_*)].
\end{align*}
Writing the singular value decomposition $\bK=\bO\bD\bU^\top$,
after $T$ iterations, the reconstruction mean-squared-error of $\btheta_{T+1}$
may be expressed as
\[\text{MSE}=\frac{1}{n}\|\btheta_{T+1}-\btheta_*\|_2^2
=\frac{1}{n}\|\psi_T(\by_T)\|_2^2,\]
where
\[\psi_T(\by)=\bO\bU^\top[\btheta_*-\eta_T((\bK^\top\bK)^{-1}\bK^\top
\by+\btheta_*)].\]
The state evolution predicts
\begin{equation}\label{eq:SEsensinganisotropic}
\lim_{m,n \to \infty} \text{MSE}-\frac{1}{n}\EE_{\bY_T \sim \cN(0,\sigma_T^2\Id)}
[\|\btheta_*-\eta_T((\bK^\top \bK)^{-1}\bK^\top \bY_T+\btheta_*)\|_2^2]=0.
\end{equation}
We note that the functions $\{g_0,\ldots,g_{T-1},\psi_T\}$ constitute a set
of Lipschitz anisotropic functions with respect to
$\cK=\{\bO\bU^\top,\bK,\bK(\bK^\top \bK)^{-1}\}$ in the sense of Definition
\ref{def:aniso}. Thus, assuming that the singular vectors $\bO,\bU$ of $\bK$ are
generic in the sense of condition (2) in Proposition \ref{prop:anisotropic},
and that $\|\bD\|_\op,\|\bD^{-1}\|_\op,\|\be\|_\infty<C$ for a constant $C>0$,
Theorem \ref{thm:RectSE} together with
Propositions \ref{prop:anisotropic} and \ref{prop:local} implies 
the validity of the state evolution
prediction (\ref{eq:SEsensinganisotropic}) for any matrix $\bW$ satisfying
Assumption \ref{assump:Ginibre}.
\end{example}

\section{Proof ideas}\label{sec:proof}

A primary technical contribution of our work is
Theorem \ref{thm:universality_poly_amp} on the validity of the state evolution
approximation for AMP algorithms with BCP-representable polynomial functions.
We summarize in this section the two main steps in the proof of this result.

\subsection{State evolution for Gaussian matrices}\label{sec:gaussian}

The first step establishes Theorem \ref{thm:universality_poly_amp} in the 
Gaussian setting where $\bW \sim \GOE(n)$. This rests on the following more
general result, of independent interest, which establishes a quantitative
version of the state evolution approximation when $f_0,f_1,\ldots,f_{T-1}$ are
general (non-Lipschitz) functions satisfying a certain stability condition.

To simplify notation, for any $n$-dependent random variable $X$ and any
$a \geq 0$, we introduce the shorthand
\begin{equation}\label{eq:precnotation}
X \prec n^{-a} \qquad \text{ or } \qquad X=\Oprec(n^{-a})
\end{equation}
to mean, for any constant $D>0$, there exists a constant $C \equiv C(D)>0$
such that
\[\PP[|X|>(\log n)^C n^{-a}]<n^{-D} \text{ for all large } n.\]
Thus, with high probability, $|X|$ is of size $n^{-a}$ up to a
poly-logarithmic factor. Our stability condition for $f_0,\ldots,f_{T-1}$ is
summarized as the following assumption.

\begin{assumption}\label{assump:GOESE}
Given $f_0,f_1,\ldots,f_{T-1}$, let $\bSigma_t$ and $\bZ_{1:t}$ be as in
Definition \ref{def:non_asymp_se} for each $t=1,\ldots,T$,
and let $\bE_{1:T} \in \RR^{n \times T}$ be any random matrix in the
probability space of $\bZ_{1:T}$ such that
\begin{equation}\label{eq:Econd}
\|\bE_{1:T}\|_\Fro \prec 1.
\end{equation}
Then for all $0 \leq s,t \leq T-1$,
\begin{equation}\label{eq:makeStateEv0}
    \frac{1}{n}\Big|f_t(\bZ_{1:t}+\bE_{1:t})^\top
f_s(\bZ_{1:s}+\bE_{1:s})
    -\EE\big[f_t(\bZ_{1:t})^\top f_s(\bZ_{1:s})\big]\Big| 
    \prec \frac{1}{\sqrt{n}},
\end{equation}
and for all $1 \leq t \leq T$ and $0 \leq s \leq T-1$,
\begin{equation}
    \frac{1}{n}\Big|(\bZ_t+\bE_t)^\top f_s(\bZ_{1:s}+\bE_{1:s})-
    \EE\big[\bZ_t^\top f_s(\bZ_{1:s})\big]\Big|
    \prec \frac{1}{\sqrt{n}}.\label{eq:Stein0}
\end{equation}
\end{assumption}

Informally, this assumption encompasses two separate conditions for the
functions $n^{-1}f_t(\bZ_{1:t})^\top f_s(\bZ_{1:s})$ and $n^{-1}\bZ_t^\top
f_s(\bZ_{1:s})$ evaluated on Gaussian inputs $\bZ_{1:T}$: (a) they
are stable under an arbitrary perturbation of $\bZ_{1:T}$ by an error of size
$\Oprec(1)$ in $\ell_2$, and (b) they concentrate around their expectation.
Such conditions will hold in general settings where $f_t(\cdot)$ satisfies
boundedness and Lipschitz conditions when restricted to sets $\cA_D$ having high
probability under Gaussian measure, as stated in the following proposition.
The proof of this proposition is given in Appendix \ref{sec:StrongSE}.

\begin{proposition}\label{prop:simpleGOESE} 
Let $f \equiv f_n$ and $g \equiv g_n$ be two $n$-dependent functions where
$f,g:\RR^{n \times t} \to \RR^n$. Let $\bSigma \in \RR^{t \times t}$ be a fixed
covariance matrix, let $\bZ=\bZ_{1:t} \sim \cN(0,\bSigma \otimes \Id_n)$, 
and let $\bE=\bE_{1:t}$ satisfy $\|\bE\|_\Fro \prec 1$.
    Suppose, for each $D>0$, there exists a constant $C \equiv C(D)>0$ and
an $n$-dependent set $\cA_D \subseteq \RR^{n \times t}$ where the
following hold:
\begin{enumerate}
        \item\label{it:prob} $\PP[\bZ + \bE \notin \cA_D],\PP[\bZ \notin
        \cA_D]<n^{-D}$,
        \item\label{it:bound} $\|f_n(\bx)\|_2,\|g_n(\bx)\|_2 <
        (\log n)^C\sqrt{n}$ for each $\bx \in \cA_D$,
        \item\label{it:lip} $\|f_n(\bx) - f_n(\by)\|_2, \|g_n(\bx) -
        g_n(\by)\|_2 \leq (\log n)^C\|\bx-\by\|_\Fro$ for each $\bx,\by \in
        \cA_D$,
        \item\label{it:conc} $\PP[\frac{1}{n}|f_n(\bZ)^\top g_n(\bZ)
    -\EE[f_n(\bZ)^\top g_n(\bZ)]| > \frac{(\log n)^{C}}{\sqrt{n}}]<n^{-D}$.
    \end{enumerate}
Then
    \[\frac{1}{n}\Big|f_n(\bZ+\bE)^\top g_n(\bZ+\bE) -\EE\big[f_n(\bZ)^\top
    g_n(\bZ)\big]\Big| \prec \frac{1}{\sqrt{n}}.\]
\end{proposition}

For example, when each coordinate of $f,g$ is locally Lipschitz,
the set $\cA_D$ may be the $\ell_\infty$-bounded domain
$\cA_D=\{\bx \in \RR^{n \times t}:\|\bx\|_\infty<(\log n)^{C(D)}\}$ as discussed in the
remarks following Theorem \ref{thm:StrongSE0} below.

The following theorem shows that when Assumption \ref{assump:GOESE} holds and
$\bW \sim \GOE(n)$, the iterates $\bz_{1:T}$ of the AMP algorithm
(\ref{eq:AMP}) may be approximated by the Gaussian state evolution vectors
$\bZ_{1:T}$ up to $\Oprec(1)$ error. Its proof uses a version of the Gaussian
conditioning arguments of \cite{bolthausen2014iterative,Bayati2011} and is given
in Appendix \ref{sec:StrongSE}.

\begin{theorem}\label{thm:StrongSE0}
Fix any $T \geq 1$, let $\bW \sim \GOE(n)$, and let $f_1,\ldots,f_{T-1}$ be
weakly differentiable. Let $b_{ts}$ and $\bSigma_t$ be as in Definition
\ref{def:non_asymp_se}, and suppose there exist constants $C,c>0$ such that
$\lambda_{\min}(\bSigma_t)>c$, $\|\bSigma_t\|_\op<C$, and
$|b_{ts}|<C$ for all $1 \leq s<t \leq T$.

If Assumption \ref{assump:GOESE} holds,
then the iterates $\bz_{1:T}$ of the AMP algorithm (\ref{eq:AMP})
admit a decomposition
\begin{equation}\label{eq:strong_se_decomposition}
    [\bz_1,\ldots,\bz_T]=[\bZ_1,\ldots,\bZ_T] + [\bE_1,\ldots,\bE_T],
\end{equation}
where $\bZ_{1:T} \sim \cN(0,\bSigma_T \otimes \Id_n) \in\RR^{n\times T}$ and
$\|\bE_{1:T}\|_\Fro\prec 1$.
\end{theorem}

This result strengthens known state evolution statements from
\cite{berthier2017,Gerbelot2022} for non-separable AMP algorithms of the
form (\ref{eq:AMP}) in two ways:

\begin{enumerate}
\item Assumption \ref{assump:GOESE} encompasses a class of
functions that does not satisfy the conditions of these preceding
works. For example, suppose for each $t \geq 1$ and some $L,k>0$, we have that
\begin{equation}
\|f_t(0)\|_2 \leq L\sqrt{n}, \quad
\|f_t(\bx)-f_t(\by)\|_2 \leq L(1+\|\bx\|_\infty^k+\|\by\|_\infty^k)
\cdot \|\bx-\by\|_\Fro,\label{eq:locallipschitz}
\end{equation}
where $\|\bx\|_\infty=\max_{i=1}^n\max_{j=1}^t |\bx[i,j]|$. This includes
Lipschitz functions, as well as separable functions that are uniformly
pseudo-Lipschitz in each coordinate $i \in [n]$, whereas this latter
separable class does not necessarily satisfy the pseudo-Lipschitz condition
$\|f_t(\bx)-f_t(\by)\|_2 \leq
L(1+(\|\bx\|_2/\sqrt{n})^k+(\|\by\|_2/\sqrt{n})^k)\|\bx-\by\|_\Fro$
required in the results of \cite{berthier2017,Gerbelot2022}.

It is direct to check that any functions satisfying
(\ref{eq:locallipschitz}) also satisfy the conditions of Proposition
\ref{prop:simpleGOESE},
where $\cA_D=\{\bx \in \RR^{n \times t}:\|\bx\|_\infty<(\log n)^C\}$ for some $C
\equiv C(D)>0$. Hence this class of functions also
satisfies Assumption \ref{assump:GOESE}.

\item The guarantee $\|\bE_{1:T}\|_\Fro \prec 1$ for the decomposition
(\ref{eq:strong_se_decomposition}) is stronger than the usual statement of
state evolution ensuring that the difference between
$\frac{1}{n}\sum_{i=1}^n \delta_{\bz_{1:T}[i]}$
and $\cN(0,\bSigma_T)$ converges to 0 weakly
in probability. Indeed, a result of the form
$\PP[\|\bE_{1:T}\|_\Fro>cn^{1/2}] \to 0$ for every fixed $c>0$ would
suffice to show such a statement.
Our guarantee $\|\bE_{1:T}\|_\Fro \prec 1$ is stronger
in both the size of the established bound and in its probability statement,
ensuring for example (by the Borel-Cantelli lemma) the convergence of
the aforementioned difference in distributions to 0
in an almost-sure sense rather than in probability.
\end{enumerate}

Importantly for our purposes, Assumption \ref{assump:GOESE} is sufficiently
general to include all BCP-representable polynomial functions. We show this
also in Appendix \ref{sec:StrongSE}, by using the BCP to bound the means
and variances of $n^{-1}f_t(\bZ_{1:t})^\top f_s(\bZ_{1:s})$ and
$n^{-1}\bZ_t^\top f_s(\bZ_{1:s})$ when $\bZ_{1:T}$ are Gaussian inputs
and $f_s(\cdot),f_t(\cdot)$ are BCP-representable. Combined with Theorem
\ref{thm:StrongSE0}, this will show Theorem \ref{thm:universality_poly_amp} in
the Gaussian setting of $\bW \sim \GOE(n)$.

\subsection{Moment-method analysis of tensor networks}\label{sec:tensor}

The second step then establishes Theorem~\ref{thm:universality_poly_amp} for
general Wigner matrices using a moment-method analysis.
Since $f_1,\ldots,f_{T-1}$ and the test functions $\phi_1,\phi_2$ 
in Theorem \ref{thm:universality_poly_amp} are polynomials, it is clear that
the value
\[\phi(\bz_{1:T})=\frac{1}{n}\,\phi_1(\bz_{1:T})^\top \phi_2(\bz_{1:T})\]
may be expressed as a polynomial function of the entries of $\bW$.
We will represent this function as a linear combination of contracted values
of tensor networks in the sense of Definition~\ref{def:TN0},
with labels given either by a tensor in $\cT$ or by
the random matrix $\bW$, as summarized in the following lemma.

\begin{lemma}\label{lem:tenUnroll}
Fix any constants $T,D,C_0>0$. Suppose that $f_0,f_1,\ldots,f_{T-1}$ and
$\phi_1,\phi_2$ defining $\phi$ in (\ref{eq:testfunction})
are polynomial functions that admit a representation (\ref{eq:tensorpolyrepr}) 
via a set of tensors $\cT=\bigsqcup_{k=1}^{D+1} \cT_k$. Suppose also that
$\{b_{ts}\}$ in (\ref{eq:AMP}) satisfy
$|b_{ts}|<C_0$ for all $1 \leq s<t \leq T$.

Then there exist constants $C,M>0$, a list of connected ordered
multigraphs $G_1,\ldots,G_M$ depending only on $T,D,C_0$ and independent of
$n$, and a list of $\{\cT \cup \bW\}$-labelings $\cL_1,\ldots,\cL_M$ of
$G_1,\ldots,G_M$ and coefficients $a_1,\ldots,a_M \in \RR$ with $|a_m|<C$,
such that
\[\phi(\bz_1,\ldots,\bz_T)=\sum_{m=1}^M \frac{a_m\val_{G_m}(\cL_m)}{n}.\]
\end{lemma}

Lemma \ref{lem:tenUnroll} follows from an elementary unrolling of the AMP
iterates that is similar to previous analyses
of \cite{Bayati2015,Wang2024,jones2024fourier}, and we provide its proof in
Appendix \ref{sec:GenAlg}. The primary difference in our
setting is that, since the polynomial functions $f_t,\phi_1,\phi_2$
are non-separable, the resulting tensors $\bT_v$ which represent these
polynomials are non-diagonal. This leads to a more involved moment-method
analysis, in which the BCP condition for $\cT$ is used crucially to bound the
moments of $\val_G(\cL)$. Universality of the first moment of $\val_G(\cL)$
is summarized in the following lemma, which underlies the universality of
Theorem \ref{thm:universality_poly_amp}.

\begin{lemma}\label{lem:ExpVal} 
Let $\cT=\bigsqcup_{k=1}^K \cT_k$ be a set of 
tensors satisfying the BCP, and let $\bW,\bW'$ be two
Wigner matrices satisfying Assumption \ref{assump:Wigner}. Fix any
connected ordered multigraph $G$ independent of $n$,
let $\cL$ be a $\{\cT \cup \bW\}$-labeling of $G$, and let
$\cL'$ be the $\{\cT \cup \bW'\}$-labeling that replaces $\bW$ by $\bW'$.
Then there is a constant $C>0$ independent of $n$ for which
    \[\bigg|\EE\bigg[\frac{1}{n}\val_G(\cL)\bigg]
-\EE\bigg[\frac{1}{n}\val_G(\cL')\bigg]\bigg| \leq \frac{C}{\sqrt{n}}.\]
\end{lemma}

In Appendix \ref{sec:Univ}, we prove Lemma \ref{lem:ExpVal}, and then
strengthen this to a statement of almost-sure convergence by bounding also
the fourth central moment $\EE(\val_G(\cL)-\EE \val_G(\cL))^4$.
Combining with Lemma \ref{lem:tenUnroll}, this will conclude
the proof of Theorem \ref{thm:universality_poly_amp} for general Wigner matrices
$\bW$.

\appendix

\section{Elementary properties of the BCP}\label{sec:BCP}

We begin by giving an alternative definition of the Bounded Composition
Property (BCP), which will be useful to formalize some of the following proofs.

\begin{definition}[Alternative definition of BCP]\label{def:BCPalt}
A set of deterministic tensors $\cT=\bigsqcup_{k=1}^K \cT_k$, where
$\cT_k \subseteq (\RR^n)^{\otimes k}$ for each $k=1,\ldots,K$,
satisfies the {\bf Bounded Composition Property (BCP)} if
the following holds:
Fix any integers $m,\ell \geq 1$ and 
$k_1,\ldots,k_m \in \{1,\ldots,K\}$ independent of $n$, and define $k_0^+=0$ and
$k_a^+=k_1+k_2+\ldots+k_a$.
Fix any surjective map $\pi:[k_m^+] \to [\ell]$ such that
\begin{itemize}
\item[(1)] For each $j \in [\ell]$, the set of indices $\{k \in
[k_m^+]:\pi(k)=j\}$ has even cardinality.
\item[(2)] There does not exist a partition of $\{1,\ldots,m\}$ into two
disjoint sets $A,A'$ for which the indices
$\pi(\bigcup_{a \in A} \{k_{a-1}^++1,\ldots,k_a^+\})$ are disjoint from
$\pi(\bigcup_{a \in A'}\{k_{a-1}^++1,\ldots,k_a^+\})$.
(Here, $\pi(\bigcup_{a \in A} \{k_{a-1}^++1,\ldots,k_a^+\})$ denotes the image of the set $\bigcup_{a \in A} \{k_{a-1}^++1,\ldots,k_a^+\} \subseteq [k_m^+]$
under the map $\pi$.)
\end{itemize}
Then there exists a constant $C>0$ depending only on
$m,\ell,k_1,\ldots,k_m,\pi$ and independent of $n$ such that
for all large $n$,
\begin{equation}\label{eq:BCPcondition}
\sup_{\bT_1 \in \cT_{k_1},\ldots,\bT_m \in \cT_{k_m}}
\frac{1}{n} \left|\sum_{i_1,\ldots,i_\ell=1}^n \prod_{a=1}^m
\bT_a[i_{\pi(k_{a-1}^++1)},\ldots,i_{\pi(k_a^+)}]\right| \leq C.
\end{equation}
\end{definition}

The equivalence between Definition \ref{def:BCPalt} and Definition \ref{def:BCP}
of the main text is as follows:
\begin{itemize}
    \item The vertex set $\cV_T$ has $|\cV_T|=m$, and each $i^\text{th}$ vertex
$v_i \in \cV_T$ has $\deg(v_i)=k_i$.
    \item The vertex set $\cV_{\Id}$ has $|\cV_{\Id}|=\ell$, and each $j^\text{th}$
vertex $u_j \in \cV_{\Id}$ has $\deg(u_j) = |\{k \in [k_m^+]: \pi(k) = j\}|$.
    \item The edge set $\cE$ has $|\cE| = k_m^+ = \sum_{i=1}^m k_i$.
For each $e \in [k_m^+]$, if $e \in \{k_{i-1}^++1,\ldots,k_i^+\}$
and $\pi(e)=j$, then $e$ corresponds to an edge of $\cE$ connecting $v_i$ with
$u_j$.
\item The value of $\sup_{\cL} \frac{1}{n}|\val_G(\cL)|$ over all
$(\Id,\cT)$-labelings $\cL$ of $G$ is precisely the left side of
\eqref{eq:BCPcondition}.
\end{itemize}
Condition~$(1)$ of Definition \ref{def:BCPalt} coincides with the condition
in Definition \ref{def:BCP} that each vertex $u \in \cV_\Id$ has even degree.
Condition~$(2)$ of Definition \ref{def:BCPalt} is the statement that
$G=(\cV_T \sqcup \cV_\Id,\cE)$ is connected, as the existence of such sets
$A,A'$ is equivalent to a partition of $\cV_T$ that induces two disconnected
subgraphs of $G$. Thus, Definition \ref{def:BCPalt} is equivalent to
Definition \ref{def:BCP}.

With this equivalence established, we collect in the remainder of this appendix
several closure properties for sets of tensors $\cT$ that satisfy the BCP.

\begin{lemma}\label{lemma:BCPcontraction}
Suppose $\cT=\bigsqcup_{k=1}^K \cT_k$ satisfies the BCP, where $\cT_k \subseteq
(\RR^n)^{\otimes k}$.
\begin{enumerate}[(a)]
\item If $\bT \in \cT$ and $|a_n|<C$ for a constant $C>0$,
then $\cT \cup \{a_n \bT\}$ satisfies the BCP.
\item If $\bT \in \cT$ and $\tilde \bT$ is any transposition of $\bT$
(e.g.\ $\tilde\bT[i_1,i_2,i_3]=\bT[i_3,i_1,i_2]$ for all $i_1,i_2,i_3 \in [n]$)
then $\cT \cup \{\tilde \bT\}$ satisfies the BCP.
\item If $\bT_1 \in \cT_{k_1}$, $\bT_2 \in \cT_{k_2}$, and $\bT$ is a
contraction of $\bT_1,\bT_2$, i.e.\ there exist transpositions $\tilde
\bT_1,\tilde \bT_2$ of $\bT_1,\bT_2$ and an index $k \leq \min(k_1,k_2)$ for
which $\bT \in (\RR^n)^{\otimes (k_1+k_2-2k)}$ is given by
\[\bT[j_1,\ldots,j_{k_1-k},\ell_1,\ldots,\ell_{k_2-k}]
=\sum_{i_1,\ldots,i_k=1}^n \tilde \bT_1[i_1,\ldots,i_k,j_1,\ldots,j_{k_1-k}]
\tilde \bT_2[i_1,\ldots,i_k,\ell_1,\ldots,\ell_{k_2-k}],\]
then $\cT \cup \{\bT\}$ satisfies the BCP.
\end{enumerate}
\end{lemma}
\begin{proof}
Statements (a) and (b) are immediate from Definition \ref{def:BCPalt}. For
statement (c), note that any expression inside the supremum of
(\ref{eq:BCPcondition}) that has $\ell$ indices $i_1,\ldots,i_\ell$
and $m' \in \{1,\ldots,m\}$ copies of $\bT$ may be expanded into an expression
using $\bT_1,\bT_2$ with $\ell+km'$ indices, where each additional index
$i_{\ell+1},\ldots,i_{\ell+km'}$ appears twice. Then the BCP for
$\cT \cup \{\bT\}$ follows from the BCP for $\cT$.
\end{proof}

\begin{lemma}\label{lemma:BCPId}
Let $\Id \in (\RR^n)^{\otimes 2}$ denote the identity matrix, viewed as a
tensor of order 2. If $\cT$ satisfies the BCP, then so does $\cT \cup \{\Id\}$.
\end{lemma}
\begin{proof}
Consider any expression inside the supremum of
(\ref{eq:BCPcondition}) where the first
$m'$ tensors are given by $\Id$ and the last $m-m'$ are tensors in $\cT$.
Such an expression is equal to $n^{-1}|\val|$ for a value of the form
\[\val=\sum_{i_1,\ldots,i_\ell=1}^n
\prod_{a=1}^{m'} \Id[i_{\pi(2a-1)},i_{\pi(2a)}]
\prod_{a=m'+1}^m \bT_a[i_{\pi(k_{a-1}^++1)},\ldots,i_{\pi(k_a^+)}].\]
For each $a \in \{1,\ldots,m'\}$, if $\pi(2a-1)=\pi(2a)$, then $\val$ is
unchanged upon removing the factor $\Id[i_{\pi(2a-1)},i_{\pi(2a)}]$.
If $\pi(2a-1) \neq \pi(2a)$, then $\val$ is unchanged upon
removing the factor $\Id[i_{\pi(2a-1)},i_{\pi(2a)}]$ and
identifying $i_{\pi(2a)}$ with $i_{\pi(2a-1)}$ (i.e.\ replacing all instances of
$i_{\pi(2a)}$ by $i_{\pi(2a-1)}$ and then removing $i_{\pi(2a)}$ from the
summation). Iterating this procedure for $a=1,\ldots,m'$, we reduce either
to a form $\sum_{i=1}^n \Id[i,i]$ with a single identity tensor, or to
a form where $m'=0$ and all remaining tensors belong to $\cT$. In the
former case we have $n^{-1}\val=1$, while in the latter case we have
$n^{-1}|\val| \leq C$ for all large $n$ uniformly over all
$\bT_{m'+1},\ldots,\bT_m \in \cT$ by the BCP for $\cT$.
Thus the BCP holds for $\cT \cup \{\Id\}$.
\end{proof}

The next lemma considers expressions of the form (\ref{eq:BCPcondition}) in the
definition of the BCP, when a subset of the tensors have order 1 and are given
by standard Gaussian vectors $\bxi_1,\ldots,\bxi_t \in \RR^n$. The lemma bounds
the mean and variance of the resulting expression over $\bxi_1,\ldots,\bxi_t$.

\begin{lemma}\label{lem:GauPolyVar}
Fix any integers $m \geq m' \geq 1$, $k_1=\ldots=k_{m'}=1$, and
$k_{m'+1},\ldots,k_m \in \{1,\ldots,K\}$, and define $k_0^+=0$ and
$k_a^+=k_1+k_2+\ldots+k_a$. Fix $\ell \geq 1$ and a surjective map
$\pi:[k_m^+] \to [\ell]$ satisfying the two conditions of Definition
\ref{def:BCPalt}. Fix also $t \geq 1$ and a coordinate map $\sigma:[m'] \to [t]$.

Suppose $\cT$ is a set of tensors satisfying the BCP,
and $\bxi_1,\ldots,\bxi_t \in \RR^n$ are independent vectors with
i.i.d.\ $\cN(0,1)$ entries. Then there is a
constant $C>0$ such that for any $\bT_{m'+1},\ldots,\bT_m \in \cT$ of the
appropriate orders $k_{m'+1},\ldots,k_m$, the function
\begin{equation}\label{eq:GauPoly}
\val(\bxi_{1:t})=\sum_{i_1,\ldots,i_\ell=1}^n
\left(\prod_{a=1}^{m'} \bxi_{\sigma(a)}[i_{\pi(a)}]\right)
\left(\prod_{a=m'+1}^m \bT_a[i_{\pi(k_{a-1}^++1)},\ldots,i_{\pi(k_a^+)}]\right)
\end{equation}
satisfies
\[|\EE \val(\bxi_{1:t})| \leq Cn, \qquad \Var[\val(\bxi_{1:t})] \leq Cn.\]
\end{lemma}
\begin{proof}
For the expectation, let $\sP$ be the set of all pairings $\tau$ of $[m']$
for which every pair $\{a,b\} \in \tau$ satisfies $\sigma(a)=\sigma(b)$.
Then applying Wick's rule (Lemma \ref{lem:wick}),
\[\EE \val(\bxi_{1:t})=\sum_{\tau \in \sP}
\underbrace{\sum_{\bi \in [n]^\ell}
\prod_{\{a,b\} \in \tau} \Id[i_{\pi(a)},i_{\pi(b)}]
\prod_{a=m'+1}^m \bT_a[i_{\pi(k_{a-1}^++1)},\ldots,i_{\pi(k_a^+)}]}_{:=T(\tau)}.\]
For each $\tau \in \sP$, this summand $T(\tau)$ is of the form
(\ref{eq:BCPcondition}) with tensors belonging to $\cT \cup \{\Id\}$, and
continues to satisfy both conditions of Definition \ref{def:BCPalt}. Then by the
BCP for $\cT \cup \{\Id\}$ given in Lemma \ref{lemma:BCPId},
we have $|T(\tau)| \leq Cn$ and hence also
$|\EE \val(\bxi_{1:t})| \leq C'n$ for some constants $C,C'>0$.

For the variance, let us write $i_{\ell+1},\ldots,i_{2\ell}$ for a duplication
of the indices $i_1,\ldots,i_\ell$. We duplicate also the set of tensors,
setting $\bxi_{\sigma(m+a)}=\bxi_{\sigma(a)}$ for $a=1,\ldots,m'$ and
$\bT_{m+a}=\bT_a$ for $a=m'+1,\ldots,m$, having orders $k_{m+a}=k_a$ for all
$a=1,\ldots,m$. Then, defining $k_a^+=k_1+\ldots+k_a$ for each $a \in [2m]$
and extending $\pi$ to a map $\pi:[2k_m^+] \to [2\ell]$ by
$\pi(k_m^++k)=\pi(k)+\ell$ for all $k \in [k_m^+]$, we have
\begin{align*}
\EE[\val(\bxi_{1:t})^2] &= \sum_{\bi \in [n]^{2\ell}}
\EE \Bigg[\prod_{a=1}^{m'} \bxi_{\sigma(a)}[i_{\pi(a)}]
\prod_{a=m+1}^{m+m'} \bxi_{\sigma(a)}[i_{\pi(a)}]\Bigg]\\
&\hspace{1in}\cdot
\prod_{a=m'+1}^m \bT_a[i_{\pi(k_{a-1}^++1)},\ldots,i_{\pi(k_a^+)}]
\prod_{a=m+m'+1}^{2m} \bT_a[i_{\pi(k_{a-1}^++1)},\ldots,i_{\pi(k_a^+)}].
\end{align*}
Let $\sP$ be the set of all pairings $\tau$ of $\{1,\ldots,m'\} \cup
\{m+1,\ldots,m+m'\}$ for which every pair $\{a,b\} \in \tau$ satisfies
$\sigma(a)=\sigma(b)$. Then again by Wick's rule,
\[\EE[\val(\bxi_{1:t})^2]=\sum_{\tau \in \sP} T(\tau)\]
where
\begin{equation}\label{eq:Ttau}
T(\tau)=\sum_{\bi \in [n]^{2\ell}}
\prod_{\{a,b\} \in \tau} \Id[i_{\pi(a)},i_{\pi(b)}]
\prod_{a=m'+1}^m \bT_a[i_{\pi(k_{a-1}^++1)},\ldots,i_{\pi(k_a^+)}]
\prod_{a=m+m'+1}^{2m} \bT_a[i_{\pi(k_{a-1}^++1)},\ldots,i_{\pi(k_a^+)}].
\end{equation}
Now let $\sP' \subset \sP$ be those pairings for which
each pair $\{a,b\}$ has both elements in $\{1,\ldots,m'\}$ or both elements in
$\{m+1,\ldots,m+m'\}$, and observe similarly by Wick's rule that
\begin{align*}
\big(\EE \val(\bxi_{1:t})\big)^2
&=\sum_{\bi \in [n]^{2\ell}}
\Bigg(\EE \prod_{a=1}^{m'} \bxi_{\sigma(a)}[i_{\pi(a)}]
\cdot \EE \prod_{a=m+1}^{m+m'} \bxi_{\sigma(a)}[i_{\pi(a)}]\Bigg)\\
&\hspace{1in}\cdot
\prod_{a=m'+1}^m \bT_a[i_{\pi(k_{a-1}^++1)},\ldots,i_{\pi(k_a^+)}]
\prod_{a=m+m'+1}^{2m} \bT_a[i_{\pi(k_{a-1}^++1)},\ldots,i_{\pi(k_a^+)}]\\
&=\sum_{\tau \in \sP'} T(\tau).
\end{align*}
Thus
\[\Var[\val(\bxi_{1:t})]=\sum_{\tau \in \sP \setminus \sP'} T(\tau).\]
For each $\tau \in \sP \setminus \sP'$, this summand $T(\tau)$ in (\ref{eq:Ttau}) is of the form
(\ref{eq:BCPcondition}) with tensors belonging to $\cT \cup \{\Id\}$.
The first even cardinality condition of Definition \ref{def:BCPalt} holds for
$T(\tau)$, because it holds for the original expression (\ref{eq:GauPoly}).
The second connectedness condition of Definition \ref{def:BCPalt} also holds for
$T(\tau)$: This is because, by the given condition that $\pi$ defining
(\ref{eq:GauPoly}) satisfies Definition \ref{def:BCPalt},
there is no partition of $i_1,\ldots,i_{\ell}$ or of
$i_{\ell+1},\ldots,i_{2\ell}$ into two index sets that appear on disjoint
sets of tensors, and furthermore since $\tau \notin \sP'$,
there is also at least one
pair $\{a,b\} \in \tau$ for which $i_{\pi(a)}$ is one of $i_1,\ldots,i_{\ell}$
and $i_{\pi(b)}$ is one of $i_{\ell+1},\ldots,i_{2\ell}$.
Then by the BCP property for $\cT \cup \{\Id\}$ given in
Lemma \ref{lemma:BCPId}, we have $|T(\tau)| \leq Cn$ and hence also
$\Var[\val(\bxi_{1:t})] \leq C'n$ for some constants $C,C'>0$.
\end{proof}

\begin{corollary}\label{cor:BCPgaussian}
Suppose $\cT$ satisfies the BCP and has cardinality $|\cT| \leq C$ for a
constant $C>0$ independent of $n$. Let $\bxi_s \equiv \bxi_s(n) \in \RR^n$ for
$s=1,\ldots,t$ be independent vectors with i.i.d.\ $\cN(0,1)$ entries,
viewed as tensors of order 1, where $t$ is also independent of $n$.
Then $\cT \cup \{\bxi_1,\ldots,\bxi_t\}$ satisfies the BCP almost surely with
respect to $\{\bxi_1(n),\ldots,\bxi_t(n)\}_{n=1}^\infty$.
\end{corollary}
\begin{proof}
Consider any expression inside the supremum of (\ref{eq:BCPcondition}), where the
first $m'$ tensors belong to $\{\bxi_1,\ldots,\bxi_t\}$ and the last $m-m'$
belong to $\cT$. Such an expression is given by $n^{-1}|\val(\bxi_{1:t})|$ where
$\val(\bxi_{1:t})$ is a value of the form (\ref{eq:GauPoly}).
Lemma \ref{lem:GauPolyVar} implies $\Var[n^{-1}\val(\bxi_{1:t})] \leq Cn^{-1}$
for some constant $C>0$.
As $n^{-1}\val(\bxi_{1:t})$ is a polynomial of degree $m'$ in
the standard Gaussian variables $\bxi_{1:t}$, it follows from Gaussian
hypercontractivity (Lemma \ref{lem:hypercontractivity}) that
there exist constants $C',c>0$ for which, for any $\epsilon>0$, 
\[\PP[|n^{-1}\val(\bxi_{1:t})-n^{-1}\EE\val(\bxi_{1:t})|>\epsilon] \leq
C'e^{-(c\epsilon^2 n)^{1/m'}}.\]
Applying this and the bound $|n^{-1}\EE \val(\bxi_{1:t})| \leq C$ from
Lemma \ref{lem:GauPolyVar}, we obtain for some constants $C,C',c>0$ that
\begin{equation}\label{eq:valtailbound}
\PP[|n^{-1}\val(\bxi_{1:t})|>C] \leq C'e^{-(cn)^{1/m'}}.
\end{equation}
As $|\cT|$ is bounded independently of $n$, the number of choices for
$\bT_1,\ldots,\bT_m \in \cT \cup \{\bxi_1,\ldots,\bxi_t\}$
in (\ref{eq:BCPcondition})
is also bounded independently of $n$. Taking the union bound of
(\ref{eq:valtailbound}) over all such choices and applying the Borel-Cantelli
lemma, we obtain that (\ref{eq:BCPcondition}) holds almost surely, and thus
$\cT \cup \{\bxi_1,\ldots,\bxi_t\}$ almost surely satisfies the BCP.
\end{proof}

\section{State evolution for Gaussian matrices}\label{sec:StrongSE}

In this appendix, we prove Proposition \ref{prop:simpleGOESE} and
Theorem \ref{thm:StrongSE0} on the state evolution
for AMP algorithms defined by stable functions $f_0,\ldots,f_{T-1}$
when $\bW \sim \GOE(n)$.
We then show Theorem \ref{thm:universality_poly_amp} in this Gaussian setting.

Recall the notation $X \prec n^{-a}$ from (\ref{eq:precnotation}). We will
use throughout the basic properties that
$X,Y \prec n^{-a} \Rightarrow X+Y \prec n^{-a}$
and $X \prec n^{-a},Y \prec n^{-b} \Rightarrow XY \prec n^{-(a+b)}$.

\begin{proof}[Proof of Proposition \ref{prop:simpleGOESE}]
    It suffices to prove that both 
    \begin{align}
        \frac{1}{n}\Big|f_n(\bZ+\bE)^\top g_n(\bZ+\bE)
    -f_n(\bZ)^\top g_n(\bZ)\Big| 
    &\prec \frac{1}{\sqrt{n}},\label{eq:PercPert}\\
    \frac{1}{n}\Big|f_n(\bZ)^\top g_n(\bZ)
    -\EE\big[f_n(\bZ)^\top g_n(\bZ)\big]\Big| 
    &\prec \frac{1}{\sqrt{n}}\label{eq:PercExp}
    \end{align}
    hold. Here \eqref{eq:PercExp} is the statement of condition (4.)
of the proposition, so it remains to demonstrate \eqref{eq:PercPert}. Given any
$D>0$, take $C_1 \equiv C_1(D)>0$ and $C_2 \equiv C_2(D)>0$ sufficiently
large and denote the events
\begin{align*}
\cE_1&=\{\bZ \in \cA_D, \bZ+\bE \in \cA_D, \|\bE\|_\Fro \leq
(\log n)^{C_1}\}\\
\cE_2&=\left\{\frac{1}{n}\Big|f_n(\bZ+\bE)^\top g_n(\bZ+\bE)
    -f_n(\bZ)^\top g_n(\bZ)\Big| 
    < \frac{(\log n)^{C_2}}{\sqrt{n}}\right\}
\end{align*}
By the condition $\|\bE\|_\Fro \prec 1$ and condition (1.) of the proposition, $C_1>0$
may be chosen large enough so that $\PP[\cE_1] > 1-n^{-D}$. On this event
$\cE_1$, conditions (2.--3.) imply that
    \begin{align*}
&\frac{1}{n}\Big|f_n(\bZ+\bE)^\top g_n(\bZ+\bE)-f_n(\bZ)^\top g_n(\bZ) \Big|\\
&\leq \frac{1}{n}\Big\|f_n(\bZ+\bE)-f_n(\bZ)\Big\|_2 \cdot
\Big\|g_n(\bZ+\bE)\Big\|_2
+\frac{1}{n}\Big\|g_n(\bZ+\bE)-g_n(\bZ)\Big\|_2 \cdot
\Big\|f_n(\bZ)\Big\|_2 <\frac{(\log n)^{C_2}}{\sqrt{n}}
\end{align*}
for large enough $C_2>0$, so $\cE_1 \subseteq \cE_2$. Thus also
$\PP[\cE_2]>1-n^{-D}$, showing \eqref{eq:PercPert}.
\end{proof}

\begin{proof}[Proof of Theorem \ref{thm:StrongSE0}]\label{pf:StrongSE}
Consider the following statements,
where the constant $C \equiv C(D)>0$ underlying $\prec$ may depend on $t$.
    \begin{itemize}
        \item[(\rom{1}$_t$)] There exist random vectors $\bZ_{1:t},\bE_{1:t}
\in\RR^{n\times t}$ in the probability space of $\bW$,
with $\bZ_{1:t} \sim \cN(0,\bSigma_t \otimes \Id)$
and $\|\bE_{1:t}\|_\Fro \prec 1$, such that
        \[
            \bz_{1:t}=\bZ_{1:t}+\bE_{1:t}.
        \]
Here $\bZ_{1:t}$ is $\cF_t$-measurable for some $\sigma$-algebra $\cF_t$
generated by $\bW\bu_{1:t} \in \RR^{n \times t}$
and auxiliary random variables independent of $\bW$.
        \item[(\rom{2}$_t$)] For all $s,\tau \in \{1,\ldots,t\}$, 
        \[
            \frac{1}{n}\langle \bz_s, \bz_\tau\rangle-
            \frac{1}{n}\langle \bu_s, \bu_\tau\rangle \prec
\frac{1}{\sqrt{n}}.
        \]
        \item[(\rom{3}$_t$)] For all $s,\tau \in \{1,\ldots,t\}$,
        \[
            \frac{1}{n}\langle \bz_s, \bu_{\tau+1}\rangle
            -\sum_{r=1}^\tau b_{\tau+1,r} \cdot \frac{1}{n}\langle
            \bz_s,\bz_r\rangle \prec \frac{1}{\sqrt{n}}.
        \]
    \end{itemize}
    We will show inductively that (\rom{1}$_t$) holds for $t=1$, and that
for each $t=1,\ldots,T-1$, (\rom{1}$_t$) implies (\rom{2}$_t$,\rom{3}$_t$),
and (\rom{1}$_t$,\rom{2}$_t$,\rom{3}$_t$) imply (\rom{1}$_{t+1}$).
The theorem then follows from (\rom{1}$_t$) for $t=T$.\\

    \textbf{Base case (\rom{1}$_t$ for $t=1$):}
Recall from \eqref{eq:AMP} that $\bz_1 = \bW\bu_1$ and $\bu_2=f_1(\bz_1)$.
    As each entry of $\bW$ is a mean-zero Gaussian random variable, so is
    each entry of $\bz_1$.
A direct calculation using the law of $\bW \sim \GOE(n)$
shows that the covariance of $\bz_1=\bW\bu_1$ is given by
    \begin{equation}\label{eq:base_covariance}
        \EE[\bz_1 \bz_1^\top]
=\EE[\bW\bu_1\bu_1^\top\bW]= \frac{1}{n}\|\bu_1\|_2^2 \cdot \Id +
        \frac{1}{n} \bu_1 \bu_1^\top.
\end{equation}
Note that $n^{-1}\|\bu_1\|_2^2=\bSigma_1$ is strictly positive by
assumption. Let $\bP_{\bu_1}=\bu_1\bu_1^\top/\|\bu_1\|_2^2$ be the projection
onto the span of $\bu_1$, and $\bP_{\bu_1}^\perp=\Id-\bP_{\bu_1}$. Let
$\bxi_1 \sim \cN(0,n^{-1}\|\bu_1\|_2^2 \cdot \bP_{\bu_1}) \in \RR^n$ be a
Gaussian vector independent of $\bW$, and set
\[\bZ_1=\bP_{\bu_1}^\perp\bz_1+\bxi_1, \qquad \bE_1=\bP_{\bu_1}\bz_1-\bxi_1.\]
Then $\bz_1=\bZ_1+\bE_1$, where $\bZ_1\sim\cN(0,n^{-1}\|\bu_1\|_2^2 \cdot
\Id)=\cN(0,\bSigma_1 \otimes \Id)$
and $\bE_1 \sim \cN(0,3n^{-1}\bu_1 \bu_1^\top)$.
Letting $\cF_1$ be
the $\sigma$-algebra generated by $(\bW\bu_1,\bxi_1)$, we note that $\bZ_1$ is
$\cF_1$-measurable. Also $\|\bE_1\|_2$ is equal in law to
$(3/n)^{1/2}\|\bu_1\|_2 \cdot |\xi|$ where $\xi \sim \cN(0,1)$, so
$\|\bE_1\|_2 \prec 1$ by the assumption
$n^{-1}\|\bu_1\|_2^2=\|\bSigma_1\|_\op<C$ and a Gaussian tail bound.
This establishes (\rom{1}$_1$).\\

    \textbf{Induction step: (\rom{1}$_t$) $\Rightarrow$ (\rom{2}$_t$,
\rom{3}$_t$)} Suppose (\rom{1}$_t$) holds for some $t \leq T-1$.
For any $s,\tau \in \{1,\ldots,t\}$, note that
$n^{-1}\langle \bZ_s,\bZ_\tau \rangle=n^{-1}\EE\langle \bZ_s,\bZ_\tau
\rangle+\Oprec(n^{-1/2})$
by a standard concentration argument for Gaussian vectors.  Here
$n^{-1}\EE\langle \bZ_s,\bZ_\tau \rangle=\bSigma_t[s,\tau] \leq C$. Then by
(\rom{1}$_t$), the bounds $\|\bE_s\|_2,\|\bE_\tau\|_2 \prec 1$, and
Cauchy-Schwarz,
    \[
        \frac{1}{n}\langle \bz_s, \bz_\tau \rangle=\frac{1}{n}
        \langle \bZ_s + \bE_s, \bZ_\tau
        + \bE_\tau \rangle=\bSigma_t[s,\tau] + \Oprec(n^{-1/2}).
    \]
Recall that $\bu_s=f_{s-1}(\bz_{1:(s-1)})$.
Then by (\rom{1}$_t$), also
\[\frac{1}{n}\langle \bu_s,\bu_\tau \rangle
=\frac{1}{n}\langle f_{s-1}(\bZ_{1:(s-1)}+\bE_{1:(s-1)}),
f_{\tau-1}(\bZ_{1:(\tau-1)}+\bE_{1:(\tau-1)}) \rangle
=\bSigma_t[s,\tau]+\Oprec(n^{-1/2})\]
where the last equality applies condition \eqref{eq:makeStateEv0} of Assumption
\ref{assump:GOESE} and 
Definition \ref{def:non_asymp_se} for $\bSigma_t$. Combining these two
statements shows (\rom{2}$_t$).
To show (\rom{3}$_t$), using (\rom{1}$_t$) and condition \eqref{eq:Stein0} of
Assumption \ref{assump:GOESE}, we have for
any $s,\tau \in \{1,\ldots,t\}$ that
    \begin{equation}
        \frac{1}{n}\langle \bz_s,\bu_{\tau + 1} \rangle
        =\frac{1}{n}\langle \bZ_s + \bE_s,\,
        f_\tau(\bZ_{1:\tau} + \bE_{1:\tau}) \rangle
        =\frac{1}{n}\EE\langle \bZ_s,f_\tau(\bZ_{1:\tau})\rangle
+\Oprec(n^{-1/2}).\label{eq:beforeStein}
    \end{equation}
   Stein's lemma (c.f.\ Lemma \ref{lem:stein}) gives, for each
coordinate $i=1,\ldots,n$, the equality
$\EE[\bZ_s[i]f_\tau(\bZ_{1:\tau})[i]]
=\sum_{r=1}^\tau \EE[\partial_{\bZ_r[i]}f_\tau(\bZ_{1:\tau})[i]]
\cdot \bSigma_t[s,r]$. Then
        \[\frac{1}{n}\EE\langle \bZ_s, f_\tau(\bZ_{1:\tau})
        \rangle = \sum_{r=1}^\tau \frac{1}{n}\EE[\div\nolimits_r f_\tau(\bZ_{1:\tau})]
\cdot \bSigma_t[s,r] =\sum_{r=1}^\tau b_{\tau+1,r}\bSigma_t[s,r]\]
    where $b_{\tau+1,r}$ is defined in (\ref{eq:def_b_ts}).
    Combining this with $n^{-1}\langle \bz_s,\bz_r
\rangle=\bSigma_t[s,r]+\Oprec(n^{-1/2})$ as shown above and the assumption
$|b_{\tau+1,r}| \leq C$, this shows (\rom{3}$_t$).\\

\textbf{Induction step: (\rom{1}$_t$, \rom{2}$_t$, \rom{3}$_t$)
$\Rightarrow$ (\rom {1}$_{t+1}$)}
Suppose (\rom{1}$_t$,\rom{2}$_t$,\rom{3}$_t$) hold for some $t \leq T-1$.
Recall that $\bu_{s}=f_{s-1}(\bz_{1:(s-1)})$. Then
by the induction hypothesis (\rom{1}$_t$),
the condition \eqref{eq:makeStateEv0} of Assumption \ref{assump:GOESE}, and
Definition \ref{def:non_asymp_se} for $\bSigma_{t+1}$,
\begin{equation}\label{eq:uuapprox}
n^{-1}\bu_s^\top\bu_\tau=\bSigma_{t+1}[s,\tau]+\Oprec(n^{-1/2})
\text{ for any } s,\tau \in \{1,\ldots,t+1\}.
\end{equation}
Define the event
\[\cE=\{n^{-1}\bu_{1:(t+1)}^\top \bu_{1:(t+1)} \in \RR^{(t+1) \times (t+1)}
\text{ is invertible}\}.\]
The bound (\ref{eq:uuapprox}) and assumption
$\lambda_{\min}(\bSigma_{t+1})>c$ imply that
\begin{equation}\label{eq:goodeventbound}
\PP[\cE]>1-n^{-D}
\end{equation}
for any fixed $D>0$ and all large $n$.
To ease notation, let us denote
\[\bu=\bu_{1:t}=(\bu_1,\ldots,\bu_t) \in \RR^{n \times t},
\quad \bz=\bz_{1:t}=(\bz_1,\ldots,\bz_t) \in \RR^{n \times t},\]
and also introduce
\[\bb=(b_{t+1,1},\ldots,b_{t+1,t})^\top\in\RR^t,
\quad \by=\bigg(\bz_1,\,\bz_2+b_{2,1}\bu_1,\,\ldots,\,
        \bz_t+\sum_{s=1}^{t-1} b_{t,s}\bu_s\bigg)\in \RR^{n\times t}.\]
On the event $\cE$,
let $\bP_\bu=\bu(\bu^\top \bu)^{-1}\bu^\top$ be the projection onto the column
span of $\bu$, and $\bP_\bu^\perp=\Id-\bP_\bu$. 
By definition of the AMP algorithm (\ref{eq:AMP}) and the above quantities,
we have $\by=\bW\bu$. This implies that on $\cE$,
\begin{align}
\bW&=\bW \bP_\bu+\bP_\bu \bW\bP_\bu^\perp
+\bP_\bu^\perp \bW\bP_\bu^\perp\notag\\
&=\by(\bu^\top \bu)^{-1}\bu^\top
+\bu(\bu^\top\bu)^{-1}\by^\top \bP_\bu^\perp
+\bP_\bu^\perp \bW \bP_\bu^\perp.\label{eq:Wdecomp}
\end{align}
Let $\bu_{t+1,\parallel} = \bP_\bu \bu_{t+1}$ and 
    $\bu_{t+1,\perp} = \bP_\bu^\perp \bu_{t+1}=\bu_{t+1}
    -\bu_{t+1,\parallel}$. 
Then applying the definition of $\bz_{t+1}$ in the AMP algorithm (\ref{eq:AMP})
and (\ref{eq:Wdecomp}) gives
    \begin{equation}\label{eq:useYUSwap}
        \bz_{t+1}
        =\bW\bu_{t+1}-\bu \bb
        = \by (\bu^\top\bu)^{-1} \bu^\top \bu_{t+1,\parallel} 
        + \bu(\bu^\top\bu)^{-1} \by^\top \bu_{t+1,\perp}+ 
        \bP_\bu^\perp \bW \bu_{t+1,\perp} - \bu\bb.
    \end{equation}
    Using $\bu_{t+1,\perp}^\top \bu_s=0$ for all $s \leq t$
and the definition of $\by$,
    we have $\by^\top \bu_{t+1,\perp} = \bz^\top 
    \bu_{t+1,\perp}$, so on $\cE$,
    \begin{align}
        \bz_{t+1}
        &=\by
        (\bu^\top\bu)^{-1} \bu^\top \bu_{t+1,\parallel} + \bu
        (\bu^\top\bu)^{-1} \bz^\top \bu_{t+1,\perp} +
        \bP_{\bu}^\perp \bW \bu_{t+1,\perp} -
        \bu\bb\notag\\
        &= \underbrace{(\by-\bz)(\bu^\top\bu)^{-1}
        \bu^\top\bu_{t+1,\parallel} + \bu(\bu^\top\bu)^{-1}
        \bz^\top \bu_{t+1,\perp} - \bu\bb}_{:=\bv_1}\notag\\
        &\qquad +
        \underbrace{\bz (\bu^\top\bu)^{-1} \bu^\top
        \bu_{t+1,\parallel} + \bP_{\bu}^\perp \bW \bu_{t+1,\perp}}_{:=\bv_2}.
\label{eq:ztplus1decomp}
    \end{align}

    We first establish that
\begin{equation}\label{eq:v1bound}
\1\{\cE\} \cdot \|\bv_1\|_2 \prec 1.
\end{equation}
Restricting to the event $\cE$, since
$\by-\bz$ belongs to the column span of $\bu$, 
we have $\bv_1=\sum_{s=1}^t \alpha_s\bu_s$ for some
    coefficients $\alpha_s \in \RR$. Let us calculate these
    coefficients. The $\tau$-th column of $\by-\bz$ contains $\bu_s$ only when
    $\tau>s$, with the corresponding coefficient being $b_{\tau,s}$. Therefore,
    \begin{equation}\label{eq:alpha_s}
\alpha_s = \sum_{\tau=s+1}^t b_{\tau,s} 
        \Big((\bu^\top\bu)^{-1} \bu^\top 
        \bu_{t+1,\parallel}\Big)\ss{\tau} + \Big((\bu^\top\bu)^{-1}
        \bz^\top\bu_{t+1,\perp}\Big)\ss{s} - b_{t+1,s}.
\end{equation}
    Defining $\beta_\tau=((\bu^\top\bu)^{-1}\bu^\top
\bu_{t+1,\parallel})\ss{\tau}$ for each $\tau=1,\ldots,t$, we have
    $\bu_{t+1,\parallel}=\bu(\bu^\top\bu)^{-1} \bu^\top
    \bu_{t+1,\parallel}=\sum_{\tau=1}^t\beta_\tau\bu_\tau$, and correspondingly
    $\bu_{t+1,\perp} = \bu_{t+1}-\sum_{\tau=1}^t\beta_\tau\bu_\tau$.
    This allows us to expand the second term on the right side of
     \eqref{eq:alpha_s} as
    \begin{equation}\label{eq:alpha_s_second_term_1}
\begin{split}
        \Big((\bu^\top\bu)^{-1}\bz^\top\bu_{t+1,\perp}\Big)
        \ss{s}&=\sum_{\tau=1}^t \left((\bu^\top 
        \bu)^{-1}\ss{s,\tau}\right) \bz_\tau^\top\bu_{t+1,\perp}\\
        &= \sum_{\tau=1}^t \left((\bu^\top
        \bu)^{-1}\ss{s,\tau}\right) \bigg(\bz_\tau^\top \bu_{t+1} 
        - \sum_{r=1}^t \beta_r\bz_\tau^\top\bu_r\bigg).
    \end{split}
\end{equation}
    Using the induction hypotheses (\rom{2}$_t$, \rom{3}$_t$), we have for any
$\tau,r \leq t$ that
    \begin{equation}\label{eq:alpha_s_second_term_2}
\frac{1}{n} \bz_\tau^\top \bu_{r+1}
=\sum_{q=1}^r b_{r+1,q} \cdot \frac{1}{n} \bz_\tau^\top \bz_q
+\Oprec(n^{-1/2})
=\sum_{q=1}^r b_{r+1,q} \cdot \frac{1}{n} \bu_\tau^\top \bu_q
+\Oprec(n^{-1/2}).
\end{equation}
Using (\rom{1}$_t$), the bound $n^{-1}\|\bu_1\|_2^2=\|\bSigma_1\|_\op \leq C$,
and a Gaussian tail bound, we have also
\begin{equation}\label{eq:alpha_s_second_term_2a}
\frac{1}{n}\bz_\tau^\top \bu_1
=\frac{1}{n}\bZ_\tau^\top \bu_1+
\frac{1}{n}\bE_\tau^\top \bu_1 \prec n^{-1/2}.
\end{equation}
    Combining \eqref{eq:alpha_s_second_term_1},
    \eqref{eq:alpha_s_second_term_2}, \eqref{eq:alpha_s_second_term_2a},
    and the bound $\1\{\cE\}\|(n^{-1}\bu^\top
\bu)^{-1}\|_\op \prec 1$ which follows from (\ref{eq:uuapprox})
and $\lambda_{\min}(\bSigma_t)>c$, we have
    \begin{align}
        &\1\{\cE\}\Big((\bu^\top\bu)^{-1}\bz^\top\bu_{t+1,\perp}\Big)\ss{s}\notag\\
        &=\1\{\cE\} \sum_{\tau=1}^t \bigg(\frac{1}{n}\bu^\top
        \bu\bigg)^{-1}\ss{s,\tau}\bigg(\sum_{r=1}^t b_{t+1,r} 
        \cdot \frac{1}{n}\bu_\tau^\top \bu_r - \sum_{r=2}^t \beta_r 
        \sum_{q=1}^{r-1} b_{r,q} \cdot \frac{1}{n} \bu_\tau^\top\bu_{q}\bigg)
+\Oprec(n^{-1/2}).\label{eq:alpha_s_second_term_3}
    \end{align}
    Plugging \eqref{eq:alpha_s_second_term_3} back into \eqref{eq:alpha_s}, and
    rearranging terms, we get
    \begin{align*}
        \1\{\cE\} \cdot \alpha_s
        &= \1\{\cE\}\Bigg(\sum_{\tau=s+1}^t \beta_\tau b_{\tau,s} + \sum_{\tau=1}^t 
        \bigg(\frac{1}{n}\bu^\top \bu\bigg)^{-1}\ss{s,\tau}
        \bigg(\sum_{r=1}^t b_{t+1,r} \cdot \frac{1}{n}\bu_\tau^\top \bu_r 
        - \sum_{r=1}^t \beta_r \sum_{q=1}^{r-1}b_{r,q} \cdot \frac{1}{n} 
        \bu_\tau^\top\bu_{q}\bigg)\\
&\hspace{1in}-b_{t+1,s}\Bigg)+\Oprec(n^{-1/2})\\
        &= \1\{\cE\}\Bigg( \sum_{\tau=s+1}^t \beta_\tau b_{\tau,s} - \sum_{r=1}^t \beta_r
        \sum_{q=1}^{r-1} b_{r,q} \sum_{\tau=1}^t \bigg(\frac{1}{n}\bu^\top
        \bu\bigg)^{-1}\ss{s,\tau}\cdot \frac{1}{n} \bu_\tau^\top\bu_{q}
        \\
        &\hspace{1in} + \sum_{r=1}^t b_{t+1,r} \sum_{\tau=1}^t \bigg(\frac{1}{n}
        \bu^\top \bu\bigg)^{-1}\ss{s,\tau} \cdot \frac{1}{n}
        \bu_\tau^\top \bu_r - b_{t+1,s}\Bigg)
        + \Oprec(n^{-1/2})\\
        &=\1\{\cE\}\Bigg(\sum_{\tau=s+1}^t \beta_\tau b_{\tau,s} - \sum_{r=1}^t \beta_r 
        \sum_{q=1}^{r-1} b_{r,q} \1\{s=q\} + \sum_{r=1}^t b_{t+1,r}\cdot 
        \1\{s=r\} - b_{t+1,s}\Bigg) + \Oprec(n^{-1/2})\\
        &= \Oprec(n^{-1/2}).
    \end{align*}
Hence $\1\{\cE\} \cdot \alpha_s \prec n^{-1/2}$ for all $s=1,\ldots,t$.
    Moreover, we have $n^{-1/2}\|\bu_s\|_2 \prec 1$ by (\ref{eq:uuapprox}),
and thus (\ref{eq:v1bound}) holds.
    
    Next, let us define the Gaussian vector $\bZ_{t+1}$ and $\sigma$-algebra
$\cF_{t+1}$. Note that by definition of the AMP algorithm (\ref{eq:AMP}),
$\bu_2=f_1(\bz_1)$ is a function of $\bW\bu_1$,
$\bu_3=f_2(\bz_1,\bz_2)$ is then a function of $\bW\bu_{1:2}$, etc., and
$\bu_{t+1}=f_t(\bz_1,\ldots,\bz_t)$ is then a function of $\bW\bu_{1:t}$.
Thus by the assumption for $\cF_t$ in the induction hypothesis (\rom{1}$_t$),
$\bu_{1:(t+1)}$ and the above event $\cE$ are $\cF_t$-measurable. On this
event $\cE$, we construct a vector $\tilde \bZ_{t+1}$ as follows:
Let $\bP_{\bu_{1:(t+1)}}$ be the projection onto the column span of
$\bu_{1:(t+1)}$, and let $\bP_{\bu_{1:(t+1)}}^\perp=\Id-\bP_{\bu_{1:(t+1)}}$.
Let $\bxi_{t+1} \in \RR^n$ be a function of $\bu_{1:(t+1)}$ and some auxiliary
randomness independent of $\bW$ and $\cF_t$, such that conditional on $\cF_t$
and on the event $\cE$,
we have that $\bxi_{t+1}$ and $\bW$ are independent with
$\bxi_{t+1} \sim \cN(0,\bP_{\bu_{1:(t+1)}})$. Define
\begin{equation}\label{eq:tildeZtplus1}
\tilde\bZ_{t+1}=\frac{\bP_{\bu_{1:(t+1)}}^\perp
\bW\bu_{t+1,\perp}}{n^{-1/2}\|\bu_{t+1,\perp}\|_2}+\bxi_{t+1}.
\end{equation}
Note that by rotational invariance of $\GOE(n)$,
the law of $\bP_\bu^\perp \bW\bP_\bu^\perp$ conditioned on $\cF_t$ 
is equal to that conditioned on $(\bu,\bW\bu)$, which is
Gaussian and equal to that of $\bP_\bu^\perp \widetilde \bW\bP_\bu^\perp$ where
$\widetilde\bW \sim \GOE(n)$ is independent of $\cF_t$. Then conditional on
$\cF_t$, the law of $\bP_{\bu_{1:(t+1)}}^\perp \bW\bu_{t+1,\perp}
=\bP_{\bu_{1:(t+1)}}^\perp \bP_\bu^\perp \bW\bP_\bu^\perp \bu_{t+1}$
is that of a mean-zero Gaussian vector with covariance given by
\begin{align}
&\EE\Big[\big(\bP_{\bu_{1:(t+1)}}^\perp \bW\bu_{t+1,\perp}\big)
\big(\bP_{\bu_{1:(t+1)}}^\perp \bW\bu_{t+1,\perp}\big)^\top\Big|\cF_t\Big]\notag\\
&\qquad=\EE\Big[\bP_{\bu_{1:(t+1)}}^\perp \big(\widetilde
\bW\bu_{t+1,\perp}\big)\big(\widetilde
\bW\bu_{t+1,\perp}\big)^\top\bP_{\bu_{1:(t+1)}}^\perp \Big|\cF_t\Big]\notag\\
&\qquad=\bP_{\bu_{1:(t+1)}}^\perp\bigg(
        \frac{1}{n}\|\bu_{t+1,\perp}\|_2^2 \cdot \mathrm{Id} +
\frac{1}{n}\bu_{t+1, \perp} \bu_{t+1,
\perp}^\top\bigg)\bP_{\bu_{1:(t+1)}}^\perp
=\frac{1}{n}\|\bu_{t+1,\perp}\|_2^2 \cdot 
\bP_{\bu_{1:(t+1)}}^\perp,\label{eq:GtildeCov}
\end{align}
the second equality applying a calculation for the expectation over $\widetilde
\bW$ that is similar to \eqref{eq:base_covariance}. Then, applying this
and the definition of $\bxi_{t+1}$, conditional on $\cF_t$ and on the event
$\cE$,
\[\tilde\bZ_{t+1}=\frac{\bP_{\bu_{1:(t+1)}}^\perp
\bW\bu_{t+1,\perp}}{n^{-1/2}\|\bu_{t+1,\perp}\|_2}+\bxi_{t+1} \sim
\cN(0,\Id).\]
On the complementary event $\cE^c$, let us simply set $\tilde \bZ_{t+1}$ to be
equal to an auxiliary $\cN(0,\Id)$ random vector that is independent of
$\cF_t$ and $\bW$. Then, since the law of $\tilde \bZ_{t+1}$ conditional on
$\cF_t$ does not depend on $\cF_t$, we have that
$\tilde \bZ_{t+1}$ is independent of $\cF_t$ and 
$\tilde \bZ_{t+1} \sim \cN(0,\Id)$ unconditionally. Now 
let $\bSigma_{1:t,t+1}$, $\bSigma_{t+1,1:t}$, and $\bSigma_{t+1,t+1}$ denote
the entries in the last row/column of $\bSigma_{t+1}$, and set
\begin{equation}
\bZ_{t+1}=\bZ_{1:t}\bSigma_t^{-1}\bSigma_{1:t,t+1}
+\Big(\bSigma_{t+1,t+1}
-\bSigma_{t+1,1:t}\bSigma_t^{-1}\bSigma_{1:t,t+1}\Big)^{1/2}
\tilde\bZ_{t+1}.
\label{eq:Ztplus1}\end{equation}
Then by the induction hypothesis (\rom{1}$_{t}$) that $\bZ_{1:t}$ is
$\cF_t$-measurable with $\bZ_{1:t} \sim \cN(0,\bSigma_t \otimes \Id)$,
we may check that $\bZ_{1:(t+1)} \sim \cN(0,\bSigma_{t+1} \otimes \Id)$.
Furthermore, letting $\cF_{t+1}$ be the $\sigma$-algebra generated by
$\cF_t$, $\bW\bu_{t+1}$, and the auxiliary randomness defining
$\tilde \bZ_{t+1}$ above, we have that $\bZ_{t+1}$ is $\cF_{t+1}$-measurable.

To conclude the proof of (\rom{1}$_{t+1}$), it remains to show for $\bv_2$
in (\ref{eq:ztplus1decomp}) that
\begin{equation}\label{eq:v2bound}
\1\{\cE\} \|\bv_2-\bZ_{t+1}\|_2 \prec 1.
\end{equation}
On the event $\cE$, recall that $\bv_2=\bz (\bu^\top\bu)^{-1} \bu^\top
    \bu_{t+1,\parallel} + \bP_{\bu}^\perp \bW \bu_{t+1,\perp}$.
For the first term, note by (\ref{eq:uuapprox}) that $n^{-1}\bu^\top
\bu_{t+1,\parallel}=n^{-1}\bu^\top \bu_{t+1}=\bSigma_{1:t,t+1}+\Oprec(n^{-1/2})$
and $\1\{\cE\} \|(n^{-1}\bu^\top \bu)^{-1}-\bSigma_t^{-1}\|_\op \prec n^{-1/2}$.
Combining these bounds with
$\bz=\bZ_{1:t}+\bE_{1:t}$ by (\rom{1}$_t$) where
$\|\bZ_s\|_2 \prec n^{1/2}$ and $\|\bE_s\|_2 \prec 1$
for each $s=1,\ldots,t$, we see that
\begin{equation}\label{eq:I2approx1}
\1\{\cE\} \big\|\bz (\bu^\top\bu)^{-1} \bu^\top
    \bu_{t+1,\parallel}-\bZ_{1:t}\bSigma_t^{-1}\bSigma_{1:t,t+1}\big\|_2 \prec
1.
\end{equation}
For the second term, recall the definition of $\tilde \bZ_{t+1}$ from
(\ref{eq:tildeZtplus1}). Let us approximate the denominator
$n^{-1/2}\|\bu_{t+1,\perp}\|_2$: By (\ref{eq:uuapprox}),
$n^{-1}\|\bu_{t+1}\|_2^2=\bSigma_{t+1,t+1}+\Oprec(n^{-1/2})$
and $\1\{\cE\} \cdot n^{-1}\|\bu_{t+1,\parallel}\|^2
        =\1\{\cE\} \cdot (n^{-1}\bu_{t+1}^\top \bu)(n^{-1}\bu^\top
        \bu)^{-1}(n^{-1}\bu^\top \bu_{t+1})=\1\{\cE\} \cdot
\bSigma_{t+1,1:t}\bSigma_t^{-1}\bSigma_{1:t,t+1}+\Oprec(n^{-1/2})$. Then
\begin{align*}
\1\{\cE\} \cdot n^{-1}\|\bu_{t+1,\perp}\|_2^2&=
\1\{\cE\} \Big(n^{-1}\|\bu_{t+1}\|_2^2
-n^{-1}\|\bu_{t+1,\parallel}\|^2\Big)\\
&=\1\{\cE\} \Big(\bSigma_{t+1,t+1}-
\bSigma_{t+1,1:t}\bSigma_t^{-1}\bSigma_{1:t,t+1}\Big)+\Oprec(n^{-1/2}).
\end{align*}
We note that $\bSigma_{t+1,t+1}-
\bSigma_{t+1,1:t}\bSigma_t^{-1}\bSigma_{1:t,t+1}$ is the reciprocal of the
lower-right entry of $\bSigma_{t+1}^{-1}$, and hence is bounded below by
$\lambda_{\min}(\bSigma_{t+1})>c$. So the above implies also
\[\1\{\cE\} \cdot \frac{1}{n^{-1/2}\|\bu_{t+1,\perp}\|_2}=
\1\{\cE\} \Big(\bSigma_{t+1,t+1}-
\bSigma_{t+1,1:t}\bSigma_t^{-1}\bSigma_{1:t,t+1}\Big)^{-1/2}+\Oprec(n^{-1/2}).\]
In the definition of $\tilde \bZ_{t+1}$ in \eqref{eq:tildeZtplus1},
we have $\|\bW\bu_{t+1,\perp}\|_2 \leq \|\bW\|_\op\|\bu_{t+1}\|_2 \prec
n^{1/2}$, and $\|\bxi_{t+1}\|_2^2 \sim \chi^2_{t+1}$ conditional on $\cF_t$,
hence $\|\bxi_{t+1}\|_2 \prec 1$. Applying these statements to 
\eqref{eq:tildeZtplus1} shows
\begin{equation}
\label{eq:I2approx2}
\1\{\cE\} \Big\|\bP_{\bu}^\perp \bW \bu_{t+1,\perp}-
\Big(\bSigma_{t+1,t+1}-
\bSigma_{t+1,1:t}\bSigma_t^{-1}\bSigma_{1:t,t+1}\Big)^{1/2} \tilde
\bZ_{t+1}\Big\|_2 \prec 1.
\end{equation}
Then combining \eqref{eq:I2approx1} and
\eqref{eq:I2approx2} shows (\ref{eq:v2bound}) as claimed.

Applying (\ref{eq:v1bound}) and (\ref{eq:v2bound}) to (\ref{eq:ztplus1decomp})
gives $\1\{\cE\} \cdot \|\bz_{t+1}-\bZ_{t+1}\|_2 \prec 1$. Then, defining
$\bE_{t+1}=\bz_{t+1}-\bZ_{t+1}$ and applying also the probability bound
(\ref{eq:goodeventbound}) for $\cE^c$, we have $\|\bE_{t+1}\|_2 \prec 1$,
establishing (\rom{1}$_{t+1}$) and completing the induction.
\end{proof}

We now show Theorem \ref{thm:universality_poly_amp} in the Gaussian case,
by checking the conditions of Theorem \ref{thm:StrongSE0}.

\begin{lemma}\label{lem:BCP_SEbound}
In the AMP algorithm (\ref{eq:AMP}),
suppose $\cP=\{f_0,f_1,\ldots,f_{T-1}\}$ is a BCP-representable set of
polynomial functions, where $f_0(\cdot) \equiv \bu_1$. Then there is a constant
$C>0$ such that $\|\bSigma_t\|_\op<C$ and $|b_{ts}|<C$ for all $1 \leq s<t \leq
T$.
\end{lemma}

\begin{proof}
Let $\cT=\bigsqcup_{k=1}^K \cT_k$ be the set of
tensors satisfying the BCP which represent $\cP$. We induct on $t$. 

{\bf Base case ($t=1$):}
$\bu_1 \in \cT_1$ by assumption, as the constant function $f_0(\cdot) \equiv
\bu_1$ belongs to $\cP$. Then
\[\|\bSigma_1\|_\op=\frac{1}{n}\|\bu_1\|_2^2
=\frac{1}{n}\sum_{i=1}^n \bu_1[i]\bu_1[i] \leq C\]
for a constant $C>0$, by the definition of the BCP for $\cT$.
The bound for $b_{ts}$ is vacuous when $t=1$.

{\bf Induction step, bound for $\bSigma_{t+1}$:} Assume the lemma holds up to
some iteration $t \leq T-1$.
Fixing any tensors $\bT,\bT' \in \cT$ of some orders $d+1,d'+1$
and a coordinate map $\sigma:[d+d'] \to [t]$, consider first
the expression
\begin{equation}\label{eq:Sigmarepr}
n^{-1}\EE \val(\bxi_{1:t})=n^{-1}\EE \Big\<\bT[\bxi_{\sigma(1)},\ldots,\bxi_{\sigma(d)},\cdot],
\bT'[\bxi_{\sigma(d+1)},\ldots,\bxi_{\sigma(d+d')},\cdot]\Big\>
\end{equation}
where $\bxi_1,\ldots,\bxi_t \overset{iid}{\sim} \cN(0,\Id)$.
By the BCP for $\cT$ and Lemma \ref{lem:GauPolyVar},
$|n^{-1}\EE \val(\bxi_{1:t})| \leq C$ for a constant $C>0$.
Observe that each entry of
$\bSigma_{t+1}$ takes a form $n^{-1}\EE[f_r(\bZ_{1:r})^\top
f_s(\bZ_{1:s})]$ for some $r,s \in \{0,\ldots,t\}$. Applying the
representation \eqref{eq:tensorpolyrepr} of $f_r$ and
$f_s$, this is a linear combination of terms of the form
\[n^{-1}\EE \Big\<\bT[\bZ_{\sigma(1)},\ldots,\bZ_{\sigma(d)},\cdot],
\bT'[\bZ_{\sigma(d+1)},\ldots,\bZ_{\sigma(d+d')},\cdot]\Big\>\]
over tensors $\bT,\bT' \in \cT$ (of some orders $d+1,d'+1$) and coordinate maps
$\sigma:[d+d'] \to [t]$. Writing
$[\bZ_1,\ldots,\bZ_t]=[\bxi_1,\ldots,\bxi_t]\bSigma_t^{1/2}$,
this is further a linear combination of terms of the
form (\ref{eq:Sigmarepr}), with coefficients given
by products of entries of $\bSigma_t^{1/2}$. The inductive hypothesis 
implies that $\bSigma_t^{1/2}$ is bounded independently of $n$,
so this and the boundedness of (\ref{eq:Sigmarepr}) argued above show that
$\|\bSigma_{t+1}\|_\op \leq C$ for some constant $C>0$ independent of $n$.

{\bf Induction step, bound for $b_{t+1,1},\ldots,b_{t+1,t}$:} 
Fix any $\bT \in \cT$ of some order $d+1$, a coordinate map $\sigma:[d] \to
[t]$, and an index $k \in [d]$, and consider the expression
\begin{equation}\label{eq:ThisIsBounded}
n^{-1}\EE\val(\bxi_{1:t})=n^{-1}\sum_{j=1}^n
\EE\,\bT[\bxi_{\sigma(1)},\ldots,\bxi_{\sigma(k-1)},\be_j,
\bxi_{\sigma(k+1)},\ldots,\bxi_{\sigma(d)},\be_j]
\end{equation}
with the standard basis vector $\be_j \in \RR^n$ in positions $k$ and $d+1$.
Then again by Lemma \ref{lem:GauPolyVar}, $|n^{-1}\EE \val(\bxi_{1:t})| \leq C$
for a constant $C>0$. For any such $\bT$ and $\sigma$, note that
the function $(\bZ_1,\ldots,\bZ_t) \mapsto
\bT[\bZ_{\sigma(1)},\ldots,\bZ_{\sigma(d)},\cdot]$
has divergence with respect to $\bZ_s$ given by
\[\div\nolimits_s \bT[\bZ_{\sigma(1)},\ldots,\bZ_{\sigma(d)},\cdot]
=\sum_{k \in \sigma^{-1}(s)} \sum_{j=1}^n
\bT[\bZ_{\sigma(1)},\ldots,\bZ_{\sigma(k-1)},\be_j,\bZ_{\sigma(k+1)},
\ldots,\bZ_{\sigma(d)},\be_j]\]
Thus, applying the representation \eqref{eq:tensorpolyrepr} of $f_t$,
observe that $b_{t+1,s}$ is a linear combination of terms of this form, scaled
by $n^{-1}$. Again using the representation
$[\bZ_1,\ldots,\bZ_t]=[\bxi_1,\ldots,\bxi_t]\bSigma_t^{1/2}$,
it follows from linearity, the inductive
hypothesis for $\bSigma_t$, and the boundedness of 
(\ref{eq:ThisIsBounded}) argued above that
$|b_{t+1,s}|<C$ for each $s=1,\ldots,t$ and some constant $C>0$ independent of
$n$. This completes the induction.
\end{proof}

\begin{lemma}\label{lem:BCPimpliesstability}
In the AMP algorithm (\ref{eq:AMP}),
suppose $\cP=\{f_0,f_1,\ldots,f_{T-1}\}$ is a BCP-representable set of
polynomial functions, where $f_0(\cdot) \equiv \bu_1$. 
Then Assumption \ref{assump:GOESE} holds.
\end{lemma}
\begin{proof}
For any two tensors $\bT,\bT' \in \cT$ of orders $d+1,d'+1$ and
any coordinate map $\sigma:[d+d'] \to [T]$, define a function
$f_{\bT,\bT',\sigma}:\RR^{n \times T} \to \RR$ by
\begin{equation}\label{eq:fTTsigma}
f_{\bT,\bT',\sigma}(\bx_{1:T})
=\frac{1}{n}\Big\<\bT[\bx_{\sigma(1)},\ldots,\bx_{\sigma(d)},\cdot],\bT'[\bx_{\sigma(d+1)},\ldots,\bx_{\sigma(d+d')},\cdot]\Big\>.
\end{equation}
Letting $\bxi_{1:T} \sim \cN(0,\Id_T \otimes \Id)$,
Lemma \ref{lem:GauPolyVar}
implies $\Var[f_{\bT,\bT',\sigma}(\bxi_{1:T})] \leq C/n$ for some constant
$C>0$. As $f_{\bT,\bT',\sigma}(\bxi_{1:T})$ is a polynomial of degree $d+d'$ in
the standard Gaussian variables $\bxi_{1:T}$, it follows from Gaussian
hypercontractivity (Lemma \ref{lem:hypercontractivity}) that
there exist constants $C',c>0$ such that, for any $\epsilon>0$, 
\[\PP[|f_{\bT,\bT',\sigma}(\bxi_{1:T})-\EE
f_{\bT,\bT',\sigma}(\bxi_{1:T})|>\epsilon] \leq
C'e^{-(c\epsilon^2 n)^{\frac{1}{d+d'}}}.\]
Applying this with $\epsilon=(\log n)^C/\sqrt{n}$ for sufficiently large
$C>0$ shows
\begin{equation}\label{eq:fxicontra}
f_{\bT,\bT',\sigma}(\bxi_{1:T})-\EE f_{\bT,\bT',\sigma}(\bxi_{1:T}) \prec
n^{-1/2}.
\end{equation}
Recall that $[\bZ_1,\ldots,\bZ_T]=[\bxi_1,\ldots,\bxi_T]\bSigma_T^{1/2}$
where 
$\|\bSigma_T\|_\op<C$ for a constant $C>0$ by Lemma \ref{lem:BCP_SEbound}.
Then by linearity, if \eqref{eq:fxicontra} holds for every $\sigma:[d+d']
\to [T]$, then also for every $\sigma:[d+d'] \to [T]$ we have
\begin{equation}\label{eq:fTTconcentration}
f_{\bT,\bT',\sigma}(\bZ_{1:T})-\EE
f_{\bT,\bT',\sigma}(\bZ_{1:T}) \prec n^{-1/2}.
\end{equation}
Defining similarly
\begin{equation}\label{eq:fTsigma}
f_{\bT,\sigma}(\bx_{1:T})
=\frac{1}{n}\Big\<\bT[\bx_{\sigma(1)},\ldots,\bx_{\sigma(d)},\cdot]
,\bx_{\sigma(d+1)}\Big\>
=\frac{1}{n}\bT[\bx_{\sigma(1)},\ldots,\bx_{\sigma(d)},\bx_{\sigma(d+1)}],
\end{equation}
we have by Lemma \ref{lem:GauPolyVar} that
$\Var[f_{\bT,\sigma}(\bxi_{1:T})] \leq C/n$.
Then by a similar application of Gaussian hypercontractivity and linearity,
for any $\bT \in \cT$ and $\sigma:[d+1] \to [T]$,
\begin{equation}\label{eq:fTconcentration}
f_{\bT,\sigma}(\bZ_{1:T})-\EE f_{\bT,\sigma}(\bZ_{1:T}) \prec n^{-1/2}.
\end{equation}

Now consider the error
\[f_{\bT,\bT',\sigma}(\bZ_{1:T}+\bE_{1:T})-f_{\bT,\bT',\sigma}(\bZ_{1:T})\]
where $\bE_{1:T}$ is any random matrix satisfying the assumption
$\|\bE_{1:T}\|_\Fro \prec 1$ in (\ref{eq:Econd}).
Using multi-linearity and the form of $f_{\bT,\bT',\sigma}$ from
\eqref{eq:fTTsigma}, we can expand
\begin{align}
&f_{\bT,\bT',\sigma}(\bZ_{1:T}+\bE_{1:T})-f_{\bT,\bT',\sigma}(\bZ_{1:T})\notag\\
&=\mathop{\sum_{S \subseteq [d+d']}}_{S \neq \varnothing}
\underbrace{\frac{1}{n}\sum_{\bi \in [n]^{d+d'}} \sum_{j \in [n]}
\bT[i_1,\ldots,i_d,j]\bT'[i_{d+1},\ldots,i_{d+d'},j]
\prod_{a \in S} \bE_{\sigma(a)}[i_a]
\prod_{a \in [d+d'] \setminus  S} \bZ_{\sigma(a)}[i_a]}_{:=A(S)}.\label{eq:SnotEmpty}
\end{align}
Here, the removal of the summand for $S=\varnothing$ corresponds to the
subtraction of $f_{\bT,\bT',\sigma}(\bZ_{1:T})$. For each summand $A(S)$,
we apply Cauchy-Schwarz over indices $\bi \in [n]^S$ to give
\begin{align*}
|A(S)|
&\leq \Bigg(\underbrace{\frac{1}{n}\sum_{\bi \in [n]^S}\prod_{a \in S}
\bE_{\sigma(a)}[i_a]^2}_{:=A_1(S)}\Bigg)^{1/2} \times\\
\hspace{1in}
&\Bigg(\underbrace{\frac{1}{n}\sum_{\bi \in [n]^S}
\Bigg(\sum_{\bi \in [n]^{[d+d'] \setminus S}} \sum_{j \in [n]}
\bT[i_1,\ldots,i_d,j]\bT'[i_{d+1},\ldots,i_{d+d'},j]
\prod_{a \in [d+d'] \setminus  S} \bZ_{\sigma(a)}[i_a]\Bigg)^2}_{:=A_2(S)}\Bigg)^{1/2}
\end{align*}
Here, $A_1(S)=n^{-1}\prod_{a \in S} \|\bE_{\sigma(a)}\|_2^2 \prec n^{-1}$
by the given condition (\ref{eq:Econd}) for $\bE_{1:T}$. For $A_2(S)$, we write
$[\bZ_1,\ldots,\bZ_T]=[\bxi_1,\ldots,\bxi_T]\bSigma_T^{1/2}$. Then
$A_2(S)$ is a linear combination of terms of the form
\begin{align}
&\frac{1}{n}\sum_{\substack{\bi,\bi' \in [n]^{[d+d']}\\ j,j' \in [n]}}
\bT[i_1,\ldots,i_d,j]\bT'[i_{d+1},\ldots,i_{d+d'},j]
\bT[i_1',\ldots,i_d',j'] \bT'[i_{d+1}',\ldots,i_{d+d'}',j']\notag\\
&\hspace{1in} \times \prod_{a \in S} \Id[i_a,i_a']
\prod_{a \in [d + d'] \setminus S}
\bxi_{\sigma(a)}[i_a]\bxi_{\sigma'(a)}[i_a']\label{eq:A2Sterm}
\end{align}
for some $\sigma,\sigma':[d+d'] \to [T]$,
with coefficients given by products of entries of $\bSigma_T^{1/2}$.
For each such term (\ref{eq:A2Sterm}),
both conditions of Definition \ref{def:BCPalt} hold, where the
second condition holds
because the first two tensors $\bT,\bT'$ have a shared index $j$,
the last two tensors $\bT,\bT'$ have a shared index $j'$,
and either the first and third tensors $\bT,\bT$ or the second and fourth
tensors $\bT',\bT'$ have indices $(i_a,i_a')$
for some $a \in S$ since $S$ is non-empty. Thus, by Lemma \ref{lem:GauPolyVar},
$|\EE A_2(S)| \leq C$ and $\Var A_2(S) \leq C/n$ for a constant $C>0$. Then
Gaussian hypercontractivity implies as above that $A_2(S) \prec 1$.

Combining these bounds $A_1(S) \prec n^{-1}$ and $A_2(S) \prec 1$ gives
$A(S) \prec n^{-1/2}$, so also
\begin{equation}\label{eq:fTTapproximation}
f_{\bT,\bT',\sigma}(\bZ_{1:T}+\bE_{1:T})-f_{\bT,\bT',\sigma}(\bZ_{1:T}) \prec
n^{-1/2}.
\end{equation}
A similar argument applied to the functions $f_{\bT,\sigma}$ of
(\ref{eq:fTsigma}) shows
\begin{equation}\label{eq:fTapproximation}
f_{\bT,\sigma}(\bZ_{1:T}+\bE_{1:T})-f_{\bT,\sigma}(\bZ_{1:T})
\prec n^{-1/2}.
\end{equation}
Combining (\ref{eq:fTTconcentration}) and (\ref{eq:fTTapproximation}),
\[f_{\bT,\bT',\sigma}(\bZ_{1:T}+\bE_{1:T})-\EE f_{\bT,\bT',\sigma}(\bZ_{1:T})
\prec n^{-1/2}.\]
Applying the tensor representations \eqref{eq:tensorpolyrepr},
the left side of \eqref{eq:makeStateEv0} for any $s,t \leq T-1$
is a sum of such quantities
over a number of tuples $(\bT,\bT',\sigma)$ independent of $n$.
Hence \eqref{eq:makeStateEv0} follows from this bound and linearity.
Similarly, combining (\ref{eq:fTconcentration}) with
(\ref{eq:fTapproximation}),
\[f_{\bT,\sigma}(\bZ_{1:T}+\bE_{1:T})-\EE f_{\bT,\sigma}(\bZ_{1:T}) \prec
n^{-1/2}.\]
The left side of \eqref{eq:Stein0} for $t \leq T$ and $s \leq T-1$
is a sum of such
quantities over a number of tuples $(\bT,\sigma)$ also independent of $n$,
showing \eqref{eq:Stein0}.
\end{proof}

\begin{proof}[Proof of Theorem \ref{thm:universality_poly_amp} when
$\bW \sim \GOE(n)$]

The given conditions of Theorem \ref{thm:universality_poly_amp} together with
Lemmas \ref{lem:BCP_SEbound} and \ref{lem:BCPimpliesstability} verify the
assumptions of Theorem \ref{thm:StrongSE0}. Thus
Theorem \ref{thm:StrongSE0} shows a decomposition
\[\bz_{1:T}=\bZ_{1:T}+\bE_{1:T}\]
where $\bZ_{1:T} \sim \cN(0,\bSigma_T \otimes \Id)$ and
$\|\bE_{1:T}\|_\Fro \prec 1$. For the functions
$\phi_1,\phi_2$ of Theorem \ref{thm:universality_poly_amp} that also belong to
the BCP-representable set $\cP$,
the same argument as in Lemma \ref{lem:BCPimpliesstability} shows
that \eqref{eq:makeStateEv0} holds for $\phi_1,\phi_2$, i.e.
\begin{align*}
\phi(\bz_{1:T})
&=\frac{1}{n}\phi_1(\bZ_{1:T}+\bE_{1:T})^\top \phi_2(\bZ_{1:T}+\bE_{1:T})\\
&=\frac{1}{n}\EE[\phi_1(\bZ_{1:T})^\top \phi_2(\bZ_{1:T})]
+\Oprec(n^{-1/2})=\EE[\phi(\bZ_{1:T})]+\Oprec(n^{-1/2}).
\end{align*}
Then in particular
$\lim_{n \to \infty} \phi(\bz_{1:T})-\EE[\phi(\bZ_{1:T})]=0$ a.s.\ by the
Borel-Cantelli lemma.
\end{proof}

\section{Moment-method analysis of tensor networks}\label{sec:Univ}

In this appendix, we now carry out the moment method analyses that prove Theorem
\ref{thm:universality_poly_amp} in the setting of a general Wigner matrix $\bW$.
Appendix \ref{sec:universal_expectation} proves Lemma \ref{lem:ExpVal} on the
first moment of the tensor network value $\val_G(\cL)$,
Appendix \ref{sec:asconvergence} bounds $\EE[(\val_G(\cL)-\EE\val_G(\cL))^4]$, and
Appendix \ref{sec:proofconclusion} concludes the proof of
Theorem \ref{thm:universality_poly_amp}.

\subsection{Universality in expectation}\label{sec:universal_expectation}

\begin{proof}[Proof of Lemma \ref{lem:ExpVal}]
Throughout the proof, we fix the ordered multigraph $G=(\cV,\cE)$
and a decomposition of its vertex set $\cV=\cV_W \sqcup \cV_T$, where vertices
of $\cV_W$ have degree 2. It suffices to
prove the result for $\{\cT \cup \bW\}$-labelings $\cL$ that
assign label $\bW$ to $\cV_W$ and labels in $\cT$ to $\cV_T$, for each fixed
decomposition $\cV=\cV_W \sqcup \cV_T$.
By Lemma \ref{lemma:BCPId}, $\cT \cup \{\Id\}$
augmented with the identity matrix $\Id \in \RR^{n \times n}$ also satisfies the
BCP. Thus,
by inserting an additional degree-2 vertex with label $\Id$ between each
pair of adjacent vertices of $\cV_W$, we will assume without loss of
generality that no two vertices of $\cV_W$ are adjacent in $G$.

For any such decomposition $\cV=\cV_W \sqcup \cV_T$ and
labeling $\cL$, taking the expectation over $\bW$ in the definition of
the value (\ref{eq:val}),
\[\EE\bigg[\frac{1}{n}\val_G(\cL)\bigg]=
\frac{1}{n^{1+|\cV_W|/2}}\sum_{\bi\in[n]^\cE} \EE\Bigg[\prod_{v\in\cV_W} 
n^{1/2}\bW[i_e:e \sim v]\Bigg] \prod_{v\in\cV_T} \bT_v[i_e: e \sim v].\]
Let $\sP(\cE)$ be the set of all partitions of the edge set $\cE$.
Let $\pi(\bi) \in \sP(\cE)$ denote the partition
that is induced by the index tuple $\bi \in [n]^\cE$: edges $e,e' \in \cE$
belong to the same block of $\pi(\bi)$ if and only if $i_e=i_{e'}$. We write
$[e]$ for the block of $\pi$ that contains edge $e$. Then the
above summation may be decomposed as
\begin{equation}\label{eq:Eval}
\EE\bigg[\frac{1}{n}\val_G(\cL)\bigg]=
\sum_{\pi \in \sP(\cE)} \frac{1}{n^{1+|\cV_W|/2}}
\sum_{\bi \in [n]^{\pi}}^* \EE\Bigg[\prod_{v\in\cV_W} 
n^{1/2}\bW[i_{[e]}:e \sim v]\Bigg] \prod_{v\in\cV_T} \bT_v[i_{[e]}: e \sim v].
\end{equation}
Here, the first summation is over all possible edge partitions $\pi=\pi(\bi)$,
and the second summation $\sum_{\bi \in [n]^{\pi}}^*$
is over a distinct index $i_{[e]} \in [n]$ for each distinct block
$[e] \in \pi$, where $*$ denotes that indices $i_{[e]},i_{[e']}$ must be
distinct for different blocks $[e] \neq [e'] \in \pi$.

Let $\sP(\cV_W)$ be the set of all partitions of the vertex subset $\cV_W$.
Given a partition $\pi \in \sP(\cE)$, we associate to it a
partition $\pi_W(\pi) \in \sP(\cV_W)$ where $v,u \in \cV_W$ belong to the same
block of $\pi_W(\pi)$ if their incident edges belong to the same two blocks of
$\pi$. More precisely:

\begin{definition}\label{def:piW}
For any $v,u \in \cV_W$, let $e,e'$ be the two edges incident to $v$, and
$f,f'$ the two edges incident to $u$. The partition $\pi_W(\pi) \in \sP(\cV_W)$
{\bf associated to} $\pi$ is such that $v,u$ belong to the
same block of $\pi_W(\pi)$ if and only if
\[\{[e],[e']\}=\{[f],[f']\}\]
(as equality of unordered sets, where possibly $[e]=[e']$ and $[f]=[f']$).

Writing $[v] \in \pi_W(\pi)$ for the block of $\pi_W(\pi)$ containing $v$,
we say that these blocks $[e],[e'] \in \pi$ are {\bf incident to} the block
$[v] \in \pi_W(\pi)$ and denote this by $[e] \sim [v]$.
\end{definition}

This definition is such that for any $\bi \in [n]^\pi$ 
of the summation $\sum_{\bi \in [n]^{\pi}}^*$,
the entries $\bW[i_{[e]}:e \sim v]$ and $\bW[i_{[e]}:e \sim u]$ of
$\bW$ are equal if $v,u$ belong to the same block of $\pi_W(\pi)$, and are
independent otherwise. Thus each block $[v] \in \pi_W(\pi)$ corresponds to a
different independent entry of $\bW$.
For each $k \geq 1$, define $\bM_k \in \RR^{n \times n}$ as the
matrix with entries
\begin{equation}\label{eq:momentmatrix}
\bM_k[i,j]=\EE[n^{k/2}\bW[i,j]^k],
\end{equation}
where Assumption \ref{assump:Wigner} guarantees that $\bM_k$ is symmetric and 
$|\bM_k[i,j]|<C_k$ for a constant $C_k>0$.
Then evaluating the expectation over $\bW$ in \eqref{eq:Eval} gives
\[\EE\bigg[\frac{1}{n}\val_G(\cL)\bigg]=
\sum_{\pi \in \sP(\cE)}\frac{1}{n^{1+|\cV_W|/2}}
\sum_{\bi \in [n]^\pi}^* \prod_{[v] \in \pi_W(\pi)}
\bM_{k[v]}[i_{[e]}:[e] \sim [v]] \prod_{v \in \cV_T} \bT_v[i_{[e]}:e \sim v].\]
Here, the first product is over all blocks $[v] \in \pi_W(\pi)$, $k[v]$ denotes
the number of vertices of $\cV_W$ in the block $[v]$, and
$[i_{[e]}:[e] \sim [v]]$ is the index pair $[i_{[e]},i_{[e']}]$ for the blocks
$[e],[e']$ incident to $[v]$.

\begin{definition}
$\pi \in \sP(\cE)$ is {\bf single} if some block
$[v] \in \pi_W(\pi)$ has a single vertex, i.e.\ $k[v]=1$.
A block $[v] \in \pi_W(\pi)$ is {\bf paired} if $k[v]=2$ and if 
its incident blocks $[e],[e'] \in \pi$ are such that $[e] \neq [e']$.

(Thus if $\pi$ is not single and $[v] \in \pi_W(\pi)$ is not paired, then
either $k[v] \geq 3$ or $k[v]=2$ and $[e]=[e']$.)
\end{definition}
By the vanishing of first moments of $\bW[i,j]$ in Assumption \ref{assump:Wigner},
if $\pi$ is single then there is some $[v] \in \pi_W(\pi)$ for
which $k[v]=1$ and hence $\bM_{k[v]}=0$. By the assumption for second moments
of off-diagonal entries $\bW[i,j]$,
if $[v] \in \pi_W(\pi)$ is paired then $k[v]=2$ and
$\bM_{k[v]}[i_{[e]}:[e] \sim [v]]=1$.
Applying these observations above,
\begin{equation}\label{eq:EvalM}
\EE\bigg[\frac{1}{n}\val_G(\cL)\bigg]=
\mathop{\sum_{\pi \in \sP(\cE)}}_{\text{not single}} \frac{1}{n^{1+|\cV_W|/2}}
\sum_{\bi \in [n]^\pi}^* \mathop{\prod_{[v] \in \pi_W(\pi)}}_{\text{not paired}}
\bM_{k[v]}[i_{[e]}:[e] \sim [v]] \prod_{v \in \cV_T} \bT_v[i_{[e]}:e \sim v].
\end{equation}

Next, we apply an inclusion-exclusion argument followed by
Cauchy-Schwarz to bound the difference of (\ref{eq:EvalM}) between $\cL$ and
$\cL'$. Endow $\sP(\cE)$ with the partial ordering $\tau \geq \pi$ if 
$\pi$ refines $\tau$ (i.e.\ each block of $\tau$ is a union of one or more
blocks of $\pi$). We will use
$\<e\> \in \tau$ to denote the block of $\tau$ containing edge $e$, to avoid
notational confusion with the block $[e] \in \pi$. Note that if $v,u \in \cV_W$
belong to the same block of $\pi_W(\pi)$, then the two edges incident to
$v$ and those incident to $u$ belong to the same blocks $[e],[e'] \in \pi$, and
hence also the same blocks $\<e\>,\<e'\> \in \tau$ since $\tau \geq \pi$.
Analogous to Definition \ref{def:piW}, we continue to say that
$\<e\>,\<e'\> \in \tau$ are the blocks {\bf incident to} $[v] \in \pi_W(\pi)$
and denote this by $\<e\> \sim [v]$.

Let $\mu(\pi,\tau)$ be the inclusion-exclusion (i.e.\ M\"obius inversion)
coefficients such that, for any fixed $\pi \in \sP(\cE)$ whose blocks we denote
momentarily by $[e_1],\ldots,[e_m]$ (where $e_1,\ldots,e_m$ are any choices of
a representative edge in each block), and for any function $f:[n]^\pi \to \RR$,
\[\sum_{\bi \in [n]^\pi}^* f(i_{[e_1]},\ldots,i_{[e_m]})
=\sum_{\tau \in \sP(\cE):\tau \geq \pi} \mu(\pi,\tau) \sum_{\bi \in [n]^\tau}
f(i_{\<e_1\>},\ldots,i_{\<e_m\>}).\]
The sum $\sum_{\bi \in [n]^\tau}$ on the right side is over one index
$i_{\<e\>} \in [n]$ for each block $\<e\> \in \tau$, and no longer restricts
indices for different blocks $\<e\> \in \tau$ to be distinct.
Applying this inclusion-exclusion relation to (\ref{eq:EvalM}),
\begin{equation}\label{eq:Evalfinal}
\EE\bigg[\frac{1}{n}\val_G(\cL)\bigg]
=\mathop{\sum_{\pi \in \sP(\cE)}}_{\text{not single}}
\sum_{\tau \in \sP(\cE):\tau \geq \pi}
\frac{\mu(\pi,\tau)}{n^{1+|\cV_W|/2}}
\underbrace{\sum_{\bi \in [n]^\tau}
\mathop{\prod_{[v] \in \pi_W(\pi)}}_{\text{not paired}}
\bM_{k[v]}[i_{\<e\>}:\<e\> \sim [v]] \prod_{v \in \cV_T}
\bT_v[i_{\<e\>}:e \sim v]}_{:=\val_{\check G}(\check \cL)}.
\end{equation}
We clarify that here, $\pi_W(\pi)$ in the first product of
$\val_{\check G}(\check \cL)$ continues to be defined by the partition
$\pi$ (not by $\tau$), and $[i_{\<e\>}:\<e\> \sim [v]]$ is the index
tuple $[i_{\<e\>},i_{\<e'\>}]$ for the blocks $\<e\>,\<e'\> \in \tau$ that are
incident to $[v] \in \pi_W(\pi)$. For later reference in the proof, it is
helpful to interpret $\val_{\check G}(\check \cL)$ in (\ref{eq:Evalfinal}) as
the value of a $(\pi,\tau)$-dependent
tensor network $(\check G,\check \cL)$ constructed as follows:
\begin{itemize}
\item $\check G=(\check \cV,\check \cE)$ has three disjoint sets of vertices
$\check \cV=\check \cV_W \sqcup \check \cV_\Id \sqcup \check \cV_T$, and each
edge $e \in \check \cE$ connects a vertex of
$\check \cV_\Id$ with a vertex of either $\check \cV_W$ or $\check \cV_T$.
\item The vertices of $\check \cV_\Id$ are the blocks of $\tau$.
Each vertex $\<e\> \in \check \cV_\Id \equiv \tau$ is labeled by the identity
tensor $\Id^k$ of the appropriate order, and the
ordering of its edges is arbitrary (as the tensor $\Id^k$ is symmetric).
\item The vertices of $\check \cV_W$ are the blocks of $\pi_W(\pi)$.
Each vertex $[v] \in \check \cV_W \equiv \pi_W(\pi)$ is labeled by
$\bone\bone^\top \in \RR^{n \times n}$ (the all-1's matrix)
if $[v]$ is paired or by $\bM_{k[v]}$ if $[v]$ is not paired,
and this vertex has two edges (ordered arbitrarily) connecting to the blocks
$\<e\>,\<e'\> \in \check \cV_\Id \equiv \tau$ that are incident to $[v]$.
\item $\check \cV_T$ is the same as the vertex set $\cV_T$ of $G$, with the same
tensor labels.
For each vertex $v \in \cV_T$ with ordered edges $e_1,\ldots,e_m$ in $G$,
the vertex $v \in \check \cV_T \equiv \cV_T$ has ordered edges connecting to
$\<e_1\>,\ldots,\<e_m\> \in \check \cV_\Id \equiv \tau$.
\end{itemize}
An example of this construction of $(\check G,\check \cL)$ from
$(G,\cL,\pi,\tau)$ is depicted in Figure \ref{fig:GpiEx}. It is direct to check
that the quantity $\val_{\check G}(\check \cL)$ defined in (\ref{eq:Evalfinal})
indeed equals the value of this tensor network $(\check G,\check \cL)$,
where the label $\Id^k$ on each vertex $\<e\> \in \check \cV_\Id \equiv \tau$ ensures that
only summands which have the same index value $i_{\<e\>} \in [n]$ for all edges
incident to $\<e\>$ contribute to the tensor network value in (\ref{eq:val}).

Then, defining $\bM_k'$ and $\val_{\check G}(\check \cL')$ as in
(\ref{eq:momentmatrix}) and (\ref{eq:Evalfinal})
with $\bW'$ in place of $\bW$, we have
\begin{align}
&\bigg|\EE\bigg[\frac{1}{n}\val_G(\cL)\bigg]-
\EE\bigg[\frac{1}{n}\val_G(\cL')\bigg]\bigg|
\leq \mathop{\sum_{\pi \in \sP(\cE)}}_{\text{not single}}
\sum_{\tau \in \sP(\cE):\tau \geq \pi}
\frac{|\mu(\pi,\tau)|}{n^{1+|\cV_W|/2}} \times \notag\\
&\Bigg|\underbrace{\sum_{\bi \in [n]^\tau}
\bigg(\mathop{\prod_{[v] \in \pi_W(\pi)}}_{\text{not paired}}
\bM_{k[v]}[i_{\<e\>}:\<e\> \sim [v]] 
-\mathop{\prod_{[v] \in \pi_W(\pi)}}_{\text{not paired}}
\bM_{k[v]}'[i_{\<e\>}:\<e\> \sim [v]]\bigg)
\prod_{v \in \cV_T} \bT_v[i_{\<e\>}:e \sim v]}_{=
\val_{\check G}(\check \cL)-\val_{\check G}(\check \cL')}\Bigg|.
\label{eq:valdiffbound}
\end{align}

\begin{figure}[t]
    \centering

    \hspace{-1.4cm}\begin{tikzpicture}[scale = .9]
        \begin{scope}[every node/.style={circle,thick,draw}]
            \node (W1) at (0,-1) {$\bW$};
            \node (W2) at (2,-1) {$\bW$};
            \node (W3) at (4,-1) {$\bW$};
            \node (W4) at (6,-1) {$\bW$};
            \node (W5) at (8,-1) {$\bW$};

            \node (T1) at (-1,2) {$\bT_1$};
            \node (T2) at (1,2) {$\bT_2$};
            \node (T3) at (3,2) {$\bT_3$};
            \node (T4) at (5,2) {$\bT_4$};
            \node (T5) at (7,2) {$\bT_5$};
        \end{scope}
        
        \begin{scope}[>={Stealth[black]},
                      every node/.style={fill=white,circle},
                      every edge/.style={draw=black,very thick},
                      line width = 1pt]
            \draw[-] (W1) to node{$[1]$} (T1);
            \draw[-] (W1) to node{$[2]$} (T2);
            \draw[-] (W2) to node{$[3]$} (T2);
            \draw[-] (W3) to node{$[1]$} (T2);
            \draw[-] (W2) to node{$[1]$} (T3);
            \draw[-] (W3) to node{$[3]$} (T3);
            \draw[-] (W4) to node{$[2]$} (T3);
            \draw[-] (W4) to node{$[1]$} (T4);
            \draw[-] (W5) to node{$[3]$} (T4);
            \draw[-] (W5) to node{$[1]$} (T5);
        \end{scope}

        \node at (9, 1) {$\rightarrow$};

        \begin{scope}[every node/.style={circle,thick,draw}]
            \node (P1) at (12,-1) {$\bone\bone^\top$};
            \node (P2) at (14,-1) {$\bM_3$};

            \node (E1) at (13,1) {$\Id$};
            \node (E2) at (11,1) {$\Id$};
            \node (E3) at (15,1) {$\Id$};

            \node (T1A) at (9,3) {$\bT_1$};
            \node (T2A) at (11,3) {$\bT_2$};
            \node (T3A) at (13,3) {$\bT_3$};
            \node (T4A) at (15,3) {$\bT_4$};
            \node (T5A) at (17,3) {$\bT_5$};
        \end{scope}
        \node at (13.8,.75) {$\langle 1 \rangle$};
        \node at (11.8,.75) {$\langle 2 \rangle$};
        \node at (15.8,.75) {$\langle 3 \rangle$};
        \node at (18, 1) {$\check \cV_\Id$};
        \node at (18,-1) {$\check \cV_W$};
        \node at (18, 3) {$\check \cV_T$};

        \begin{scope}[>={Stealth[black]},
            every node/.style={fill=white,circle},
            every edge/.style={draw=black,very thick},
            line width = 1pt]
                \draw[-] (T1A) to (E1);
                \draw[-] (T2A) to (E2);
                \draw[-] (T2A) to (E3);
                \draw[-] (T2A) to (E1);
                \draw[-] (T3A) to (E1);
                \draw[-] (T3A) to (E2);
                \draw[-] (T3A) to (E3);
                \draw[-] (T4A) to (E1);
                \draw[-] (T4A) to (E3);
                \draw[-] (T5A) to (E1);

                \draw[-] (E2) to (P1);
                \draw[-] (E1) to (P1);
                \draw[-] (E1) to (P2);
                \draw[-] (E3) to (P2);

        \end{scope}

        \begin{scope}[every node/.style={circle,thick,draw}]
            \node (G21) at (6,-5) {$\Id$};
            \node (G22) at (8,-5) {$\Id$};
            \node (B1) at (10,-5) {$\Id$};
            \node (B3) at (12,-5) {$\Id$};

            \node (T1B) at (5,-3) {$\bT_1$};
            \node (T2B) at (7,-3) {$\bT_2$};
            \node (T3B) at (9,-3) {$\bT_3$};
            \node (T4B) at (11,-3) {$\bT_4$};
            \node (T5B) at (13,-3) {$\bT_5$};

            \node (T1C) at (5,-7) {$\bT_1$};
            \node (T2C) at (7,-7) {$\bT_2$};
            \node (T3C) at (9,-7) {$\bT_3$};
            \node (T4C) at (11,-7) {$\bT_4$};
            \node (T5C) at (13,-7) {$\bT_5$};
        \end{scope}
        
        \begin{scope}[>={Stealth[black]},
                      every node/.style={fill=white,circle},
                      every edge/.style={draw=black,very thick},
                      line width = 1pt]
            \draw[-] (T1B) to (B1);
            \draw[-] (T2B) to (B1);
            \draw[-] (T2B) to (G21);
            \draw[-] (T2B) to (B3);
            \draw[-] (T3B) to (B1);
            \draw[-] (T3B) to (G21);
            \draw[-] (T3B) to (B3);
            \draw[-] (T4B) to (B1);
            \draw[-] (T4B) to (B3);
            \draw[-] (T5B) to (B1);

            \draw[-] (T1C) to (B1);
            \draw[-] (T2C) to (B1);
            \draw[-] (T2C) to (G22);
            \draw[-] (T2C) to (B3);
            \draw[-] (T3C) to (B1);
            \draw[-] (T3C) to (G22);
            \draw[-] (T3C) to (B3);
            \draw[-] (T4C) to (B1);
            \draw[-] (T4C) to (B3);
            \draw[-] (T5C) to (B1);
        \end{scope}

        \node at (5.1,-5) {$\langle 2 \rangle^1$};
        \node at (7.1,-5) {$\langle 2 \rangle^2$};    
        \node at (10.8,-5) {$\langle 1 \rangle$};
        \node at (12.8,-5) {$\langle 3 \rangle$};
        
        \node at (14, -3) {$ \cV_T^1$};
        \node at (14, -7) {$\cV_T^2$};
        \node at (14, -5) {$\tilde{\cV}_{\Id}$};

        \node at (3,3) {$(G, \cL)$};
        \node at (13,4) {$(\check G, \check \cL)$};
        \node at (9, -2) {$(\tilde G, \tilde \cL)$};
        
    \end{tikzpicture}
    \caption{An example conversion from $(G, \cL) \to (\check G, \check \cL) \to
    (\tilde G, \tilde \cL)$. (Top left) The initial graph $G$ with labels $\cL$
in $\bT_1,\ldots,\bT_5,\bW$, and an edge partition $\pi \in \sP(\cE)$ consisting
of three blocks $[1],[2],[3]$. This induces two blocks $[v] \in \pi_W(\pi)$, one
which is paired and has incident blocks $[1],[2] \in \pi$, and a second with
$k[v]=3$ and incident blocks $[1],[3] \in \pi$.
(Top right) The graph $(\check G, \check \cL)$ representing
(\ref{eq:Evalfinal}) in the case $\tau=\pi$ and $\langle e \rangle=[e]$ for each
$e=1,2,3$. The vertices of $\check G$ are
partitioned as $\check \cV_W \sqcup \check \cV_\Id \sqcup \check \cV_T$.
Two vertices in $\check \cV_W$ correspond to the blocks of
$\pi_W(\pi)$, one paired and labeled with $\bone\bone^\top$ and the second unpaired and
labeled with $\bM_3$.
One vertex of $\check \cV_\Id$ corresponds to each block of $\tau$. (Bottom)
The graph $(\tilde G, \tilde \cL)$ representing (\ref{eq:Fpitaubound}).
Here $\langle 2 \rangle \in \check \cV_\Id$ is good and thus corresponds to
two vertices in $\tilde \cV_\Id$, while
$\langle 1 \rangle,\langle 3 \rangle \in \check \cV_\Id$ are bad and each
correspond to a single vertex in $\tilde \cV_\Id$.}\label{fig:GpiEx}
\end{figure}

\begin{definition}\label{def:goodbad}
Given partitions $\pi,\tau \in \sP(\cE)$ with $\tau \geq \pi$, a
block $\<e\> \in \tau$ is {\bf bad} if there exists at least one block
$[v] \in \pi_W(\pi)$ that is not paired and that is incident to $\<e\>$,
and {\bf good} otherwise. We write $\tau=\tau^b \sqcup \tau^g$ where $\tau^b$ and
$\tau^g$ are the sets of bad and good blocks, respectively.
\end{definition}

Note that if $|\tau^b|=0$, i.e.\ all blocks of $\tau$ are good, then every
block $[v] \in \pi_W(\pi)$ must be paired, so the products
$\prod_{[v] \in \pi_W(\pi):\text{not paired}}$ defining
$\val_{\check G}(\check \cL),\val_{\check G}(\check \cL')$ are both trivial
and equal to 1, and $\val_{\check G}(\check \cL)-\val_{\check G}(\check
\cL')=0$. When $|\tau^b| \neq 0$, these products
involve only indices corresponding to $\<e\> \in \tau^b$ and not
$\<e\> \in \tau^g$. Thus
\begin{align*}
\val_{\check G}(\check \cL)-\val_{\check G}(\check \cL')
&=\sum_{\bi \in [n]^{\tau^b}}
\Bigg[\bigg(\mathop{\prod_{[v] \in \pi_W(\pi)}}_{\text{not paired}}
\bM_{k[v]}[i_{\<e\>}:\<e\> \sim [v]] 
-\mathop{\prod_{[v] \in \pi_W(\pi)}}_{\text{not paired}}
\bM_{k[v]}'[i_{\<e\>}:\<e\> \sim [v]]\bigg) \times \\
&\hspace{2in}\sum_{\bi \in [n]^{\tau^g}}
\prod_{v \in \cV_T} \bT_v[i_{\<e\>}:e \sim v]\Bigg]\1\{|\tau^b| \neq 0\}.
\end{align*}

Applying Cauchy-Schwarz over the outer summation $\sum_{\bi \in [n]^{\tau^b}}$,
\begin{align*}
|\val_{\check G}(\check \cL)-\val_{\check G}(\check \cL')| &\leq 
\Bigg[\sum_{\bi \in [n]^{\tau^b}}
\bigg(\mathop{\prod_{[v] \in \pi_W(\pi)}}_{\text{not paired}}
\bM_{k[v]}[i_{\<e\>}:\<e\> \sim [v]] 
-\mathop{\prod_{[v] \in \pi_W(\pi)}}_{\text{not paired}}
\bM_{k[v]}'[i_{\<e\>}:\<e\> \sim [v]]\bigg)^2\Bigg]^{1/2} \times \\
&\hspace{1in}\Bigg[\sum_{\bi \in [n]^{\tau^b}}
\bigg(\sum_{\bi \in [n]^{\tau^g}}
\prod_{v \in \cV_T} \bT_v[i_{\<e\>}:e \sim v]\bigg)^2\Bigg]^{1/2}\1\{|\tau^b| \neq 0\}.
\end{align*}
Then applying that $|\bM_k[i,j]| \leq C_k$ for a constant $C_k>0$ and all $i,j
\in [n]$, there exists a constant $C(\pi,\tau)>0$ for
which the first factor is at most $C(\pi,\tau)n^{|\tau^b|/2}$, so
\begin{equation}
|\val_{\check G}(\check \cL)-\val_{\check G}(\check \cL')| 
\leq \1\{|\tau^b| \neq 0\} C(\pi,\tau)n^{|\tau^b|/2}
\Bigg[\underbrace{\sum_{\bi \in [n]^{\tau^b}} \bigg(\sum_{\bi \in [n]^{\tau^g}}
\prod_{v \in \cV_T} \bT_v[i_{\<e\>}:e \sim
v]\bigg)^2}_{:=\val_{\tilde G}(\tilde \cL)}\Bigg]^{1/2}.
\label{eq:Fpitaubound}\end{equation}

We interpret the quantity $\val_{\tilde G}(\tilde \cL)$ in
(\ref{eq:Fpitaubound}) as the value of a $(\pi,\tau)$-dependent
bipartite tensor network $\tilde G=(\tilde \cV_\Id \sqcup \tilde \cV_T,
\tilde \cE)$ with $(\Id,\cT)$-labeling $\tilde \cL$, constructed as follows:
\begin{itemize}
\item $\tilde \cV_\Id$ has one vertex for each block $\<e\> \in \tau^b$, which
we denote also by $\<e\> \in \tilde \cV_\Id$, and two
vertices for each block $\<e\> \in \tau^g$, which we denote by
$\<e\>^1,\<e\>^2 \in \tilde \cV_\Id$. These are labeled by $\Id$,
and the ordering of their edges is arbitrary.
\item $\tilde \cV_T=\cV_T^1 \sqcup \cV_T^2$ consists of two copies
of the original vertex set $\cV_T$ of $G$, with the same tensor labels.
For each $v \in \cV_T$, we denote
its copies by $v^1 \in \cV_T^1$ and $v^2 \in \cV_T^2$.
Suppose $v \in \cV_T$ has ordered edges $e_1,\ldots,e_m$ in the original
graph $G$. If $\<e_i\> \in \tau^b$, then the $i^\text{th}$ edge of both
$v^1 \in \cV_T^1$ and $v^2 \in \cV_T^2$ connects to
$\<e_i\> \in \tilde \cV_\Id$. If $\<e_i\> \in \tau^g$
then the $i^\text{th}$ edge of $v^1 \in \cV_T^1$ connects to
$\<e_i\>^1 \in \tilde \cV_\Id$, and the $i^\text{th}$ edge of
$v^2 \in \cV_T^2$ connects to $\<e_i\>^2 \in \tilde \cV_\Id$.
\end{itemize}
An example of this construction is also illustrated in Figure \ref{fig:GpiEx}.
Note that since each edge $e \in \cE$ of the original graph $G=(\cV,\cE)$ is
incident to at least one vertex $v \in \cV_T$ (under our starting assumption
that no two vertices of
$\cV_W$ are adjacent), each block $\<e\> \in \tau^b \sqcup \tau^g$ has also at
least one vertex $v \in \cV_T$ that is incident to an edge of that block. Then
it is direct to check that the quantity $\val_{\tilde G}(\tilde \cL)$ of
(\ref{eq:Fpitaubound}) is indeed the value of this tensor network $(\tilde
G,\tilde \cL)$ as defined in (\ref{eq:val}).

Finally, we bound $\val_{\tilde G}(\tilde \cL)$ using the given BCP property of
$\cT$ and a combinatorial argument.
Fixing any $\pi \in \sP(\cE)$ that is not single, we categorize
the possible types of blocks $[v] \in \pi_W(\pi)$ based on 
$k[v]$ (the number of vertices belonging to $[v]$) and on its
incident blocks $[e],[e'] \in \pi$:
\begin{itemize}
\item Let $N_3$ be the number of blocks $[v]$ with $k[v] \geq 3$.
\item Let $N_2$ be the number of paired blocks $[v]$, i.e.\ with $k[v]=2$ and
$[e] \neq [e']$.
\item Let $N_1$ be the number of blocks $[v]$ with $k[v]=2$ and $[e]=[e']$.
\end{itemize}
Let $\c(\tilde G)$ be the number of connected components of $\tilde G$.
We claim the following combinatorial properties:
\begin{enumerate}
\item The number of vertices of $\cV_W$ satisfies $|\cV_W| \geq
3N_3+2N_2+2N_1$.
\item The number of blocks of $\tau^b$ satisfies $|\tau^b| \leq 2N_3+N_1$.
\item The degree of each vertex of $\tilde \cV_\Id$ in $\tilde G$ is even.
\item If $|\tau^b| \neq 0$, then the number of connected components of
$\tilde G$ satisfies $\c(\tilde G) \leq 1+2N_2+N_3$.
\end{enumerate}

Let us verify each of these claims: (1) holds because each block $[v] \in
\pi_W(\pi)$ counted by $N_1$ or $N_2$ contains exactly $k[v]=2$ vertices of
$\cV_W$, and each block counted by $N_3$ contains $k[v] \geq 3$ vertices.

(2) holds because any block of $\tau^b$ must be incident to some
block $[v] \in \pi_W(\pi)$ that is not paired.
Each non-paired block $[v] \in \pi_W(\pi)$ that is counted by $N_3$ is
incident to two distinct blocks $[e],[e'] \in \pi$ --- hence at most two
blocks in $\tau^b$ because $\tau \geq \pi$ --- and each non-paired
block counted by $N_1$ is incident to one distinct block $[e] \in \pi$ --- hence
also one block in $\tau^b$.

For (3), consider first a bad block $\<e\> \in \tau^b$. By construction, the
edges of
its corresponding vertex $\<e\> \in \tilde \cV_\Id$ come in pairs, connecting to pairs of
vertices $(v^1,v^2)$. Thus $\<e\>$ has even degree. Now consider a good block
$\<e\> \in \tau^g$ and its corresponding vertices $\<e\>^1,\<e\>^2 \in \tilde \cV_\Id$. Let $e_1,\ldots,e_m$ be the edges of $G$
that belong to this block $\<e\> \in \tau^g$. If such an edge $e_i$ connects two
vertices of $\cV_T$, then there are two corresponding edges in $\tilde G$ that
connect these vertices of $\cV_T^1$ with $\<e\>^1$. Otherwise $e_i$
connects a vertex $u \in \cV_T$ with a vertex $v \in \cV_W$. (This is the case
for the block $\<2\>$ in Figure \ref{fig:GpiEx}.)
Since $\<e\> \in \tau^g$ is good, the block $[v] \in \pi_W(\pi)$
containing this vertex $v \in \cV_W$ must be paired --- thus, there is exactly
one other vertex $v' \in \cV_W$ that belongs to $[v]$. If $v$ is incident
to exactly one edge in this block $\<e\>$, then so is $v'$, and if $v$ is
incident to two edges both in $\<e\>$ (which may occur if its incident blocks
$[e] \neq [e'] \in \pi$ are merged into a single block $\<e\> \in \tau$)
then so is $v'$. This shows that the edges among $e_1,\ldots,e_m$ that connect
$\cV_T$ to $\cV_W$ come in pairs, and each pair contributes two edges of
$\tilde G$ between $\cV_T^1$ and $\<e\>^1$. So
$\<e\>^1$ has even degree. Similarly $\<e\>^2$ has even degree, which shows (3).

For (4), note that $(\tilde G,\tilde\cL)$ may be obtained from
$(\check G,\check \cL)$ by removing all vertices of $\check \cV_W$ and their
incident edges from $\check G$, duplicating the
remaining graph on the vertex set $\check \cV_\Id \cup \check \cV_T$ into two
disjoint copies on $\check \cV_\Id^1 \cup \check \cV_T^1$ and
$\check \cV_\Id^2 \cup \check \cV_T^2$, and merging the vertices of
$\check \cV_\Id^1$ representing bad blocks $\<e\> \in \tau^b$ with their
copies in $\check \cV_\Id^2$ while keeping the remaining vertices of $\check
\cV_\Id^1,\check \cV_\Id^2$ (representing good blocks $\<e\> \in \tau^g$)
distinct. We may then bound $\c(\tilde G)$ via the following observations:

\begin{itemize}
\item $\check G$ is a connected graph, because the original graph $G$ is
connected by assumption. 
\item For any connected component $K$ of $\check G$, call it \emph{good} if
all vertices of $K \cap \check \cV_\Id$ represent good blocks $\<e\> \in
\tau^g$, and \emph{bad} if at least one vertex of $K \cap \check \cV_\Id$
represents a bad block $\<e\> \in \tau^b$. We track the number $N_g$ of good
connected components and $N_b$ of bad connected components
as we sequentially remove vertices of $\check \cV_W$ from $\check G$
one at a time:

Supposing that $|\tau^b| \neq 0$ as assumed in claim (4), the starting
connected graph $\check G$ is bad, so $N_g=0$ and $N_b=1$. Each vertex
$[v] \in \check \cV_W$ counted by $N_1$ can be connected to only one vertex of
$\check \cV_\Id$, so its removal does not change $(N_g,N_b)$. Each
vertex $[v] \in \check \cV_W$ counted by $N_3$ is connected to at most 2
vertices of $\check \cV_\Id$, both of which are bad by definition, so its
removal does not change $N_g$ and increases $N_b$ by at most 1. Each vertex
$[v] \in \check \cV_W$ counted by $N_2$ is connected to at most 2 vertices of $\check \cV_\Id$
which may be either good or bad, so its removal increases the total number of
connected components $N_b+N_g$ by at most
1. Thus, after removing all vertices of $\check \cV_W$ from $\check G$, we have
\[N_b+N_g \leq 1+N_2+N_3, \qquad N_g \leq N_2.\]

\item After removing all vertices of $\check \cV_W$ and
applying the above duplication process to obtain $\tilde G$,
each component counted by $N_b$ results in one
connected component of $\tilde G$, while each component counted by $N_g$ results
in two connected components of $\tilde G$. Thus
\[\c(\tilde G)=N_b+2N_g,\]
and applying the above bounds gives $\c(\tilde G) \leq 1+2N_2+N_3$ which is
claim (4).
\end{itemize}

We apply these combinatorial claims and the BCP property to conclude the proof:
Suppose $\pi,\tau \in \sP(\cE)$ are such that $\pi$ is not single,
$\tau \geq \pi$, and $|\tau^b| \neq 0$.
Recalling that $\val_{\tilde G}(\tilde \cL)$ factorizes as the product of the
values across connected components of $\tilde G$, and applying claims (3--4) and
the BCP for $\cT$ in the form of Definition \ref{def:BCP} to each
connected component of $\tilde G$, we have
\begin{equation}\label{eq:tildeGval}
\val_{\tilde G}(\tilde \cL) \leq C(\tilde G)n^{\c(\tilde G)}
\leq C(\tilde G)n^{1+2N_2+N_3}
\end{equation}
for a constant $C(\tilde G)>0$. Since $\tilde G$ is
determined by $\pi$ and $\tau$,
applying (\ref{eq:tildeGval}) and claim (2) to (\ref{eq:Fpitaubound})
gives, for some different constant $C(\pi,\tau)>0$,
\[|\val_{\check G}(\check \cL)-\val_{\check G}(\check \cL')| 
\leq C(\pi,\tau)\cdot
n^{\frac{2N_3+N_1}{2}} \cdot n^{\frac{1+2N_2+N_3}{2}}.\]
Applying this and claim (1) back to (\ref{eq:valdiffbound}),
and noting that the number of such partitions $\pi,\tau \in \sP(\cE)$ is a
constant independent of $n$, we obtain as desired
\[\bigg|\EE\bigg[\frac{1}{n}\val_G(\cL)\bigg]-
\EE\bigg[\frac{1}{n}\val_G(\cL')\bigg]\bigg|
\leq C \cdot \frac{1}{n^{1+\frac{3N_3+2N_2+2N_1}{2}}}
\cdot n^{\frac{2N_3+N_1}{2}} \cdot n^{\frac{1+2N_2+N_3}{2}}
\leq Cn^{-1/2}.\]
\end{proof}

\subsection{Almost-sure convergence}\label{sec:asconvergence}

We now strengthen Lemma \ref{lem:ExpVal} to an almost-sure convergence
statement.

\begin{lemma}\label{lem:ConcValBCP}
Let $\cT$, $\bW,\bW'$, and $\cL,\cL'$ be as in 
Lemma \ref{lem:ExpVal}. Then almost surely
\[\lim_{n \to \infty} \frac{1}{n}\,\val_G(\cL)
-\frac{1}{n}\,\val_G(\cL')=0.\]
\end{lemma}

\begin{proof}
We will show that for a constant $C>0$,
\begin{equation}\label{eq:4thmomentbound}
\EE\bigg[\bigg(\frac{1}{n}\val_{G}(\cL)-\frac{1}{n}\EE\val_{G}(\cL)\bigg)^4\bigg] \leq \frac{C}{n^2}.
\end{equation}
We again fix the ordered multigraph $G=(\cV,\cE)$
and a decomposition $\cV=\cV_W \sqcup \cV_T$ of its vertices,
and consider a labeling $\cL$ that assigns $\bW$ to $\cV_W$ and elements of
$\cT$ to $\cV_T$. We again assume without loss of generality that no two
vertices of $\cV_W$ are adjacent.

Let $G^{\sqcup 4}=(\cV^{\sqcup 4},\cE^{\sqcup 4})$
be the ordered multigraph consisting of four disjoint copies of $G$, where
$\cV^{\sqcup 4}=\cV^1 \sqcup \cV^2 \sqcup \cV^3 \sqcup \cV^4$ are the four
copies of $\cV$ decomposed as $\cV^j=\cV_W^j \sqcup \cV_T^j$ for $j=1,2,3,4$,
and $\cE^{\sqcup 4}=\cE^1 \sqcup \cE^2 \sqcup \cE^3 \sqcup
\cE^4$ are the four copies of $\cE$. Let
$\bW^1,\ldots,\bW^4$ be four independent copies of the Wigner matrix $\bW$.
For any word $a=a_1a_2a_3a_4$ with letters $a_1,a_2,a_3,a_4 \in \{1,2,3,4\}$,
define $\cL_a$ as the tensor labeling of $G^{\sqcup 4}$ such
that for each $j=1,2,3,4$, vertices of $\cV_W^j$ are labeled by the matrix
$\bW^{a_j}$, and vertices of $\cV_T^j$ have the same labels as $\cV_T$
under $\cL$. Then
    \begin{align}
            &\EE[(\val_{G}(\cL)
            - \EE \val_{G}(\cL))^4]\notag\\
                &= \EE[\val_{G}(\cL)^4] -
            4\EE[\val_{G}(\cL)^3]\EE[\val_{G}(\cL)]
            +6\EE[\val_{G}(\cL)^2]\EE[\val_{G}(\cL)]^2 -
            3\EE[\val_{G}(\cL)]^4\notag\\
                &=\EE[\val_{G^{\sqcup 4}}(\cL_{1111}) - 4\val_{G^{\sqcup
4}}(\cL_{1112}) 
        + 6\val_{G^{\sqcup 4}}(\cL_{1123}) - 3\val_{G^{\sqcup 4}}(\cL_{1234})]
        \label{eq:partSum}
        \end{align}
    where the expectation on the last line is over the independent Wigner
matrices $\bW^1,\ldots,\bW^4$.

Let $\sP(\cE^{\sqcup 4})$ be the set of all partitions of the combined
edge set $\cE^{\sqcup 4}$. For any $a=a_1a_2a_3a_4$, we have analogously
to (\ref{eq:Eval})
\begin{equation}\label{eq:Vapi}
\EE\bigg[\frac{1}{n^4}\val_{G^{\sqcup 4}}(\cL_a)\bigg]
=\sum_{\pi \in \sP(\cE^{\sqcup 4})} \underbrace{\frac{1}{n^{4+2|\cV_W|}}
\sum_{\bi \in [n]^\pi}^* \EE\bigg[\prod_{j=1}^4 \prod_{v \in \cV_W^j}
n^{1/2}\bW^{a_j}[i_{[e]}:e \sim v]\bigg]
\prod_{j=1}^4 \prod_{v \in \cV_T^j} \bT_v[i_{[e]}:e \sim
v]}_{:=V_a(\pi)}.
\end{equation}
Let us split $\sP(\cE^{\sqcup 4})$ into three disjoint sets:
    \begin{itemize}
        \item $\cA$: Partitions $\pi$ such that every block $[e] \in
\pi$ satisfies $[e] \subseteq \cE^j$ for a single copy $j=1,2,3,4$.
        \item $\cB$: Partitions $\pi$ for which there is a
decomposition $\{1,2,3,4\}=\{j_1,j_2\} \sqcup \{k_1,k_2\}$ such that every block
$[e] \in \pi$ satisfies either $[e] \subseteq \cE^{j_1}$,
$[e] \subseteq \cE^{j_2}$, or
$[e] \subseteq \cE^{k_1} \cup \cE^{k_2}$, and at least one block
$[e] \in \pi$ has a nonempty intersection with both $\cE^{k_1}$ and $\cE^{k_2}$.
        \item $\cC$: All remaining partitions of $\sP(\cE^{\sqcup
4})$.
    \end{itemize}
We write correspondingly
\[V_a(\cA)=\sum_{\pi \in \cA} V_a(\pi),
\qquad V_a(\cB)=\sum_{\pi \in \cB} V_a(\pi),
\qquad V_a(\cC)=\sum_{\pi \in \cC} V_a(\pi)\]
so that
$\EE[n^{-4}\val_{G^{\sqcup 4}}(\cL_a)]=V_a(\cA)+V_a(\cB)+V_a(\cC)$. Then
\begin{equation}
\EE\bigg[\bigg(\frac{1}{n}\val_{G}(\cL) -\frac{1}{n} \EE
\val_{G}(\cL)\bigg)^4\bigg]=\sum_{\cS \in \{\cA,\cB,\cC\}}
V_{1111}(\cS)-4V_{1112}(\cS)+6V_{1123}(\cS)-3V_{1234}(\cS).
\label{eq:Vpartition}\end{equation}

We now analyze separately the terms of (\ref{eq:Vpartition}) for
$\cS=\cA,\cB,\cC$: For $\cA$, observe that for any $\pi \in \cA$, since the edge
sets $\cE^1,\cE^2,\cE^3,\cE^4$ are unions of disjoint blocks of $\pi$, the
indices of each of the matrices $\bW^1,\bW^2,\bW^3,\bW^4$ are distinct in
(\ref{eq:Vapi}). Then $V_a(\pi)$ has the same value for all words
$a=a_1a_2a_3a_4$, so $V_{1111}(\pi)=V_{1112}(\pi)=V_{1123}(\pi)=V_{1234}(\pi)$,
and hence
\begin{equation}\label{eq:VAbound}
V_{1111}(\cA)-4V_{1112}(\cA)+6V_{1123}(\cA)-3V_{1234}(\cA)=0.
\end{equation}

For $\cB$, recall that each $\pi \in \cB$ corresponds to a (unique) associated
decomposition $\{1,2,3,4\}=\{j_1,j_2\} \sqcup \{k_1,k_2\}$ where each block $[e]
\in \pi$ belongs to $\cE^{j_1}$, $\cE^{j_2}$, or $\cE^{k_1} \cup \cE^{k_2}$.
We further decompose
\[V_{a_1a_2a_3a_4}(\cB)=V_{\u{a_1}\u{a_2}a_3a_4}
+V_{\u{a_1}a_2\u{a_3}a_4}+V_{\u{a_1}a_2a_3\u{a_4}}
+V_{a_1\u{a_2}\u{a_3}a_4}+V_{a_1\u{a_2}a_3\u{a_4}}+V_{a_1a_2\u{a_3}\u{a_4}}\]
where each term is a summation over those $\pi \in \cB$ corresponding to a
single such decomposition $\{1,2,3,4\}=\{j_1,j_2\} \sqcup \{k_1,k_2\}$,
and the underlined positions indicate the indices $\{k_1,k_2\}$ while the
non-underlined positions indicate the indices $\{j_1,j_2\}$. So for instance,
$V_{\u{a_1}a_2\u{a_3}a_4}$ is the summation of $V_{a_1a_2a_3a_4}(\pi)$ over
those $\pi \in \cB$ for which each block $[e] \in \pi$ belongs to either
$\cE^1 \cup \cE^3$, $\cE^2$, or $\cE^4$. Note that for any such $\pi$,
the indices of $\bW^2$ and $\bW^4$ in (\ref{eq:Vapi}) are distinct
from those of $\{\bW^1,\bW^3\}$, and hence for any $a_1,a_3 \in \{1,2,3,4\}$,
the value $V_{\u{a_1}a_2\u{a_3}a_4}$ is the same for all choices of $a_2,a_4$.
This type of observation, together with symmetry of
$V_{\u{a_1}a_2\u{a_3}a_4}$ under permutations of the four indices and
relabelings of the copies $\{1,2,3,4\}$, yields the identities
\begin{align*}
V_{1111}(\cB)&=6V_{\u{1}\u{1}11}=6V_{\u{1}\u{1}23}\\
V_{1112}(\cB)&=3V_{\u{1}\u{1}12}+3V_{\u{1}11\u{2}}
=3V_{\u{1}\u{1}23}+3V_{\u{1}\u{2}34}\\
V_{1123}(\cB)&=V_{\u{1}\u{1}23}+2V_{\u{1}1\u{2}3}
+2V_{\u{1}12\u{3}}+V_{11\u{2}\u{3}}
=V_{\u{1}\u{1}23}+5V_{\u{1}\u{2}34}\\
V_{1234}(\cB)&=6V_{\u{1}\u{2}34}.
\end{align*}
Applying these identities shows
\begin{equation}\label{eq:VBbound}
V_{1111}(\cB)-4V_{1112}(\cB)+6V_{1123}(\cB)-3V_{1234}(\cB)=0.
\end{equation}

Finally, for $\cC$, we claim that there is a constant $C>0$ such that
for any $a=a_1a_2a_3a_4$, we have
\[|V_a(\cC)| \leq Cn^{-2}.\] The proof is similar to the analysis in Lemma
\ref{lem:ExpVal}: Fix any $a=a_1a_2a_3a_4$. Associated to any edge partition
$\pi \in \cC$, consider the vertex partition $\pi_W(\pi) \in \sP(\cV_W^1 \sqcup
\cV_W^2 \sqcup \cV_W^3 \sqcup \cV_W^4)$ such that $v,u$ belong to the same block
of $\pi_W(\pi)$ if and only if their incident edges belong to the same two
incident blocks of $\pi$ and, in addition, $v \in \cV_W^j$ and $u \in \cV_W^k$
for two indices $j,k \in \{1,2,3,4\}$ such that $a_j=a_k$ (i.e.\ $v,u$
correspond to the same Wigner matrix $\bW^{a_j}=\bW^{a_k}$). Let $k[v]$ be the
number of vertices in the block $[v] \in \pi_W(\pi)$, call $\pi$ single if some
block $[v] \in \pi_W(\pi)$ has $k[v]=1$, and call $[v] \in \pi_W(\pi)$ paired if
$k[v]=2$ and its incident blocks $[e],[e'] \in \pi$ satisfy $[e] \neq [e']$.
Then evaluating the expectation over $\bW^1,\ldots,\bW^4$ in (\ref{eq:Vapi}), we
get analogously to (\ref{eq:EvalM}) and (\ref{eq:Evalfinal})
\begin{align}
V_a(\cC)
&=\mathop{\sum_{\pi \in \cC}}_{\text{not single}} \frac{1}{n^{4+2|\cV_W|}}
\sum_{\bi \in [n]^\pi}^* \mathop{\prod_{[v] \in \pi_W(\pi)}}_{\text{not paired}}
\bM_{k[v]}[i_{[e]}:[e] \sim [v]] \prod_{j=1}^4 \prod_{v \in \cV_T^j}
\bT_v[i_{[e]}:e \sim v]\notag\\
&=\mathop{\sum_{\pi \in \cC}}_{\text{not single}} 
\sum_{\tau \in \sP(\cE^{\sqcup 4}):\tau \geq \pi}
\frac{\mu(\pi,\tau)}{n^{4+2|\cV_W|}}
\underbrace{\sum_{\bi \in [n]^\tau} \mathop{\prod_{[v] \in \pi_W(\pi)}}_{\text{not paired}}
\bM_{k[v]}[i_{\<e\>}:\<e\> \sim [v]] \prod_{j=1}^4 \prod_{v \in \cV_T^j}
\bT_v[i_{\<e\>}:e \sim v]}_{\val_{\check G}(\check \cL)}\label{eq:valC}.
\end{align}
Let $\tau^b,\tau^g$ denote the sets of bad and good blocks of $\tau$ defined in
the same way as Definition \ref{def:goodbad}. Then applying Cauchy-Schwarz over
$\sum_{\bi \in [n]^{\tau^b}}$, we obtain analogously to (\ref{eq:Fpitaubound})
\begin{equation}
|\val_{\check G}(\check \cL)|
\leq C(\pi,\tau)n^{|\tau^b|/2}\bigg[\underbrace{\sum_{\bi \in [n]^{\tau^b}}
\bigg(\sum_{\bi \in [n]^{\tau^g}}
\prod_{j=1}^4 \prod_{v \in \cV_T^j} \bT_v[i_{\<e\>}:e \sim
v]\bigg)^2}_{:=\val_{\tilde G}(\tilde \cL)}\bigg]^{1/2}.
\label{eq:valCbound}\end{equation}
Now let $N_3$, $N_2$, and $N_1$ be the numbers of blocks $[v] \in \pi_W(\pi)$
with $k[v] \geq 3$, with $k[v]=2$ and incident blocks $[e] \neq [e'] \in \pi$,
and with $k[v]=2$ and incident blocks $[e]=[e'] \in \pi$, respectively. Then
the same arguments as in Lemma \ref{lem:ExpVal} show that
\begin{enumerate}
\item $4|\cV_W| \geq 3N_3+2N_2+2N_1$.
\item $|\tau^b| \leq 2N_3+N_1$.
\item The degree of each vertex of $\tilde \cV_\Id$ in $\tilde G$ is even.
\end{enumerate}
Furthermore we may count the number of connected components $\c(\tilde G)$ of
$\tilde G$
by the following extension of the argument in Lemma \ref{lem:ExpVal}:
Analogous to Lemma \ref{lem:ExpVal}, $\check{G}$ above
is an ordered multigraph with three disjoint sets of
vertices $\check \cV_W \equiv \pi_W(\pi)$, $\check \cV_\Id \equiv \tau$,
and $\check \cV_T \equiv \cV_T^1 \sqcup \cV_T^2 \sqcup \cV_T^3 \sqcup \cV_T^4$,
and $\tilde{G}$ is again obtained from $\check{G}$ by removing all vertices of
$\check \cV_W$, duplicating the resulting graph on $\check \cV_\Id \cup \check
\cV_T$, and merging the two copies of vertices in $\check \cV_\Id$ that
correspond to bad blocks $\<e\> \in \tau^b$. Observe that:
\begin{itemize}
\item By definition, $G^{\sqcup 4}$ consists of 4 connected components. For
any $\pi \in \cC$, there are at least two different pairs of indices
$1 \leq j<k \leq 4$ for which a block of $\pi$ has non-empty intersection with
both $\cE^j$ and $\cE^k$. (Otherwise, we would have $\pi \in \cA$ or
$\pi \in \cB$.) Then $\check G$ has at most 2 connected components.
\item Call a connected component $K$ of $\check G$ good if all vertices
$K \cap \check \cV_\Id$ represent good blocks $\<e\> \in \tau^g$, and bad
otherwise. We again track the numbers $N_g$ and $N_b$ of good and bad connected
components of $\check G$ as we sequentially remove vertices of $\check \cV_W$.
The 1 or 2 connected components of the starting graph $\check G$ can be either
good or bad. Removing a vertex $[v] \in \check \cV_W$ counted by $N_1$ does not
change $(N_g,N_b)$, removing a vertex $[v] \in \check \cV_W$ counted by $N_3$
does not change $N_g$ and increases $N_b$ by at most 1, and removing a vertex
counted by $N_2$ increases $N_b+N_g$ by at most 1. Hence, after removing all
vertices of $\check \cV_W$ from $\check G$, we have
\[N_b+N_g \leq 2+N_2+N_3, \qquad N_g \leq 2+N_2.\]
\item After removing all vertices of $\check \cV_W$ and applying the duplication
procedure to obtain $\tilde G$, we have $\c(\tilde G)=N_b+2N_g$.
\end{itemize}
Thus we have also
\begin{enumerate}
\item[(4)] $\c(\tilde G) \leq 4+2N_2+N_3$.
\end{enumerate}
Applying these properties (1--4) and the BCP condition to (\ref{eq:valC}) and
(\ref{eq:valCbound}),
\[|V_a(\cC)| \leq C \cdot \frac{1}{n^{4+\frac{3N_3+2N_2+2N_1}{2}}}
\cdot n^{\frac{2N_3+N_1}{2}} \cdot n^{\frac{4+2N_2+N_3}{2}}
\leq Cn^{-2}\]
as claimed. Thus
\begin{equation}\label{eq:VCbound}
|V_{1111}(\cC)-4V_{1112}(\cC)+6V_{1123}(\cC)-3V_{1234}(\cC)|
\leq C'n^{-2}.
\end{equation}
Applying (\ref{eq:VAbound}), (\ref{eq:VBbound}) and (\ref{eq:VCbound}) to
(\ref{eq:Vpartition}) proves the fourth moment bound (\ref{eq:4thmomentbound}).

Then by Markov's inequality, for any $\epsilon>0$,
    \[
        \PP\bigg(\bigg|\frac{1}{n}\val_G(\cL) - \frac{1}{n}\EE
\val_G(\cL)\bigg| > \epsilon\bigg)
        \leq \frac{C}{\epsilon^4 n^2}.
    \]
   This bound is summable in $n$, so by the Borel-Cantelli lemma, almost surely
    \[\lim_{n \to \infty} \frac{1}{n}\val_G(\cL) - 
    \EE\left[\frac{1}{n}\val_G(\cL)\right] = 0.\]
    The same statement holds for $\cL'$, and combining this with
    Lemma \ref{lem:ExpVal} concludes the proof.
\end{proof}

\subsection{Concluding the proof}\label{sec:proofconclusion}

We now conclude the proof of Theorem \ref{thm:universality_poly_amp} on the
universality of polynomial AMP for general Wigner matrices $\bW$.

\begin{proof}[Proof of Theorem \ref{thm:universality_poly_amp}]
Let $\bW$ be the given Wigner matrix, and let $\bW' \sim \GOE(n)$.
Let $\bz_{1:T}$ and $\bz_{1:T}'$ denote the iterates of the AMP
algorithm (\ref{eq:AMP}) applied with $\bW$ and $\bW'$.

By assumption, $\cP=\{f_0,f_1,\ldots,f_{T-1},\phi_1,\phi_2\}$ admit
representations (\ref{eq:tensorpolyrepr}) by a set of tensors $\cT$ satisfying
the BCP. Lemma \ref{lem:BCP_SEbound} then ensures that $|b_{ts}|$ are uniformly
bounded for all $1 \leq s<t \leq T$, so Lemma \ref{lem:tenUnroll} yields
representations of the test function values
\[\phi(\bz_{1:T})=\sum_{m=1}^M \frac{a_m}{n}\,\val_{G_m}(\cL_m),
\qquad \phi(\bz_{1:T}')=\sum_{m=1}^M \frac{a_m}{n}\,\val_{G_m}(\cL_m')\]
where $|a_m|<C$ for each $m=1,\ldots,M$, and $C,M>0$ are constants independent
of $n$. By Lemma \ref{lem:ConcValBCP}, for each fixed $m=1,\ldots,M$,
almost surely
\[\lim_{n \to \infty}
\frac{1}{n}\,\val_{G_m}(\cL_m)-\frac{1}{n}\,\val_{G_m}(\cL_m')=0.\]
Thus, almost surely
$\lim_{n \to \infty} \phi(\bz_{1:T})-\phi(\bz_{1:T}')=0$.
The theorem follows from this and the statement
$\lim_{n \to \infty} \phi(\bz_{1:T}')-\EE \phi(\bZ_{1:T})=0$
for the iterates driven by $\bW' \sim \GOE(n)$,
as already shown in Appendix~\ref{sec:StrongSE}.
\end{proof}

In settings where the condition $\lambda_{\min}(\bSigma_t)>c$ of Theorem
\ref{thm:universality_poly_amp} may not hold, let us establish the following
corollary showing that the theorem holds for a random Gaussian
perturbation of the functions $f_0,f_1,\ldots,f_{T-1}$.

\begin{corollary}\label{cor:universality_perturbed}
Fix any $T \geq 1$, and let $\cP=\{f_0,f_1,\ldots,f_{T-1},\phi_1,\phi_2\}$
and $\bW$ satisfy all assumptions of Theorem \ref{thm:universality_poly_amp}
except possibly the condition $\lambda_{\min}(\bSigma_t)>c$ for each
$t=1,\ldots,T$.

Let $\bxi_1,\ldots,\bxi_T \in \RR^n$
be random vectors with i.i.d.\ $\cN(0,1)$ entries,
independent of each other and of $\bW$. Fix any $\delta>0$, and
consider the perturbed algorithm
\begin{equation}\label{eq:AMPperturbed}
\bz_t^\delta=\bW\bu_t^\delta-\sum_{s=1}^{t-1} b_{ts}^\delta \bu_s^\delta,
\qquad \bu_{t+1}^\delta=
f_t^\delta(\bz_1^\delta,\ldots,\bz_t^\delta) 
\equiv f_t(\bz_1^\delta,\ldots,\bz_t^\delta)+\delta \bxi_{t+1}
\end{equation}
with initialization
\[f_0^\delta(\cdot) \equiv \bu_1^\delta=\bu_1+\delta \bxi_1.\]
Here, we define $b_{ts}^\delta$, $\bSigma_t^\delta$, and $\bZ_t^\delta$ as in
Definition \ref{def:non_asymp_se} for the functions
$f_0^\delta,\ldots,f_{T-1}^\delta$,
with all expectations taken conditional on the realization of $\bxi_{1:T}$.
Then for the test function $\phi=n^{-1}\phi_1^\top \phi_2$, almost surely
\begin{equation}\label{eq:perturbedpolySE}
\lim_{n \to \infty} \phi(\bz_{1:T}^\delta)-\EE[\phi(\bZ_{1:T}^\delta)
\mid \bxi_{1:T}]=0.
\end{equation}
\end{corollary}
\begin{proof}
The corollary follows directly from
Theorem \ref{thm:universality_poly_amp} upon checking that the
perturbed functions $\{f_0^\delta,\ldots,f_{T-1}^\delta,\phi_1,\phi_2\}$
are BCP-representable almost surely, and that
$\lambda_{\min}(\bSigma_t^\delta)>c$ for a constant $c>0$ and each
$t=1,\ldots,T$ almost surely for all large $n$.

For BCP-representability, note that $\{f_0,\ldots,f_{T-1},\phi_1,\phi_2\}$ 
must admit the representations (\ref{eq:tensorpolyrepr}) for a set of tensors
$\cT$ satisfying the BCP that has finite cardinality independent of $n$.
Then $\{f_0^\delta,\ldots,f_{T-1}^\delta,\phi_1,\phi_2\}$
admit the representations (\ref{eq:tensorpolyrepr}) for the set of tensors
$\cT \cup \{\delta\bxi_1,\ldots,\delta\bxi_T\}$. By Lemma
\ref{lemma:BCPcontraction} and Corollary
\ref{cor:BCPgaussian}, this set satisfies the BCP
almost surely with respect to $\bxi_1,\ldots,\bxi_T$. Thus
$\{f_0^\delta,\ldots,f_{T-1}^\delta,\phi_1,\phi_2\}$ is almost surely BCP-representable.

To check that
$\lambda_{\min}(\bSigma_t^\delta)>c$ for each $t=1,\ldots,T$, we induct on $t$.
The state evolution covariances $\{\bSigma_t^\delta\}_{t=1}^T$
are defined conditionally on $\bxi_{1:T}$ by
\begin{align*}
\bSigma_1^\delta&=\frac{1}{n}\|\bu_1^\delta\|_2^2
=\frac{1}{n}\|\bu_1+\delta \bxi_1\|_2^2,\\
\bSigma_{t+1}^\delta[r+1,s+1]&=\frac{1}{n}\,
\EE[f_r^\delta(\bZ_{1:r}^\delta)^\top f_s^\delta(\bZ_{1:s}^\delta)
\mid \bxi_{1:(t+1)}]\\
&=\frac{1}{n}\,\EE[(f_r(\bZ_{1:r}^\delta)+\delta\bxi_{r+1})^\top
(f_s(\bZ_{1:s}^\delta)+\delta \bxi_{s+1}) \mid \bxi_{1:(t+1)}]
 \text{ for } r,s=0,\ldots,t,
\end{align*}
where $\bZ_{1:t}^\delta$ has i.i.d.\ rows with law $\cN(0,\bSigma_t^\delta)$
and $\bSigma_t^\delta$ depends on $\bxi_{1:t}$.
For the base case of $t=1$, writing
$\bar \bSigma_1^\delta=\EE \bSigma_1^\delta=n^{-1}\|\bu_1\|_2^2+\delta^2$,
we have
\[\|\bSigma_1^\delta-\bar\bSigma_1^\delta\|_\op
\leq \frac{2\delta}{n}\,|\bu_1^\top \bxi_1-\EE \bu_1^\top \bxi_1|
+\frac{\delta^2}{n}\,\big|\|\bxi_1\|_2^2-\EE\|\bxi_1\|_2^2\big|.\]
Since $n^{-1}\|\bu_1\|_2^2=\bSigma_1<C$ for all large $n$, this implies
$\lim_{n \to \infty} \|\bSigma_1^\delta-\bar\bSigma_1^\delta\|_\op=0$ a.s.\ by
a standard tail bound for $\bxi_1$. Then since
$\lambda_{\min}(\bar\bSigma_1^\delta) \geq \delta^2$, 
we have $\lambda_{\min}(\bSigma_1^\delta)>\delta^2/2$ a.s.\ for all large $n$.

Now suppose inductively that $\lambda_{\min}(\bSigma_t^\delta)>c$ 
for some $t \leq T-1$ a.s.\ for all large $n$. Define
$\bar \bSigma_{t+1}^\delta=\EE[\bSigma_{t+1}^\delta \mid \bxi_{1:t}]$ with
expectation over only $\bxi_{t+1}$. Then observe that
\begin{align*}
&\bSigma_{t+1}^\delta[r+1,s+1]-\bar \bSigma_{t+1}^\delta[r+1,s+1]\\
&=
\begin{cases} 0 & \text{ if } r,s \leq t-1 \\
\frac{\delta}{n}\bxi_{t+1}^\top \EE[f_s^\delta(\bZ_{1:s}^\delta) \mid
\bxi_{1:(s+1)}] &
\text{ if } r=t \text{ and } s \leq t-1\\
\frac{\delta}{n} \EE[f_r^\delta(\bZ_{1:r}^\delta) \mid \bxi_{1:(r+1)}]^\top \bxi_{t+1}
 & \text{ if } r \leq t-1 \text{ and } s=t\\
\frac{2\delta}{n}\EE[f_t(\bZ_{1:t}^\delta) \mid \bxi_{1:t}]^\top
\bxi_{t+1}+\delta^2(\frac{1}{n}\|\bxi_{t+1}\|_2^2-1)
& \text{ if } r=s=t.
\end{cases}
\end{align*}
Since $\{f_0^\delta,\ldots,f_{t-1}^\delta,f_t\}$ is
BCP-representable a.s.\ for all large $n$,
we have by Lemma \ref{lem:BCP_SEbound} that
for a constant $C>0$, a.s.\ for all large $n$,
$n^{-1}\EE[\|f_s^\delta(\bZ_{1:s}^\delta)\|_2^2 \mid \bxi_{1:(s+1)}]<C$
for each $s=0,\ldots,t-1$ and
$n^{-1}\EE[\|f_t(\bZ_{1:t}^\delta)\|_2^2 \mid \bxi_{1:t}]<C$.
Then a standard tail bound for $\bxi_{t+1}$ implies again that
\begin{equation}\label{eq:barSigmaconc}
\lim_{n \to \infty}
\|\bSigma_{t+1}^\delta-\bar \bSigma_{t+1}^\delta\|_\op=0 \text{ a.s.}
\end{equation}
To analyze $\bar\bSigma_{t+1}^\delta$, observe that
\[\bar \bSigma_{t+1}^\delta=\underbrace{\begin{pmatrix}
\bSigma_t^\delta & \bv_t \\
\bv_t^\top & \sigma_t^2
\end{pmatrix}}_{:=\bA_{t+1}} + \begin{pmatrix} 0 & 0 \\
0 & \delta^2 \end{pmatrix}\]
where
\[\bv_t=\Big(n^{-1}\EE[f_s^\delta(\bZ_{1:s}^\delta)^\top
f_t(\bZ_{1:t}^\delta) \mid \bxi_{1:t}]\Big)_{s=0}^{t-1},
\quad \sigma_t^2=n^{-1}\EE[\|f_t(\bZ_{1:t}^\delta)\|_2^2 \mid \bxi_{1:t}].\]
Applying again the above bounds
$n^{-1}\EE[\|f_s^\delta(\bZ_{1:s}^\delta)\|_2^2 \mid \bxi_{1:(s+1)}]<C$ 
and $n^{-1}\EE[\|f_t(\bZ_{1:t}^\delta)\|_2^2 \mid \bxi_{1:t}]<C$ 
a.s.\ for all large $n$, we have for a constant $C_t>0$ that
\[\|\bv_t\|_2<C_t.\]
Observe that $\bA_{t+1}$ is the conditional covariance
of $(f_0^\delta,\ldots,f_{t-1}^\delta,f_t)$, and hence is positive semidefinite.
Furthermore, by the inductive hypothesis, there is a constant $c_t>0$ for which
$\lambda_{\min}(\bSigma_t^\delta)>c_t$ a.s.\ for all large $n$.
Consider any unit vector $\bw_{t+1}=(\bw_t,w) \in \RR^{t+1}$.
If $|w|>\min(c_t/(8C_t),1/2)$ then let us
lower-bound $\bw_{t+1}^\top \bar\bSigma_{t+1}^\delta \bw_{t+1} \geq
\delta^2w^2$. If $|w| \leq \min(c_t/(8C_t),1/2)$, then let us bound
\begin{align*}
\bw_{t+1}^\top \bar\bSigma_{t+1}^\delta \bw_{t+1}
&\geq \bw_{t+1}^\top \bA_{t+1} \bw_{t+1}
\geq \bw_t^\top \bSigma_t^\delta
\bw_t-2|w \cdot \bv_t^\top \bw_t|\\
&\geq c_t(1-w^2)-2C_t|w|\sqrt{1-w^2}
\geq 3c_t/4-2C_t|w| \geq c_t/2.
\end{align*}
Combining these cases, $\bw_{t+1}^\top \bar\bSigma_{t+1}^\delta \bw_{t+1}
\geq c'$ for all unit vectors $\bw_{t+1}$ and some constant $c'>0$. Thus
$\lambda_{\min}(\bar\bSigma_{t+1}^\delta)>c'$ a.s.\ for all large $n$,
completing the induction and the proof.
\end{proof}

\section{Polynomial approximation}\label{sec:univ_poly_approx}

In this appendix, we prove Theorem \ref{thm:main_universality} and
Corollary \ref{cor:equivalent_SE} on the universality of
AMP algorithms with BCP-approximable Lipschitz functions, using a polynomial
approximation argument.

Under the condition (\ref{eq:Lipschitz}) for $f_0,\ldots,f_{T-1}$ and Definition
\ref{def:non_asymp_se} for $\bSigma_t$, there exists a constant $C_0>0$
(depending on $T$ and $L$) for which
\begin{equation}\label{eq:C0def}
\|\bSigma_t\|_\op+1<C_0
\end{equation}
for all $t=1,\ldots,T$. Fixing this $C_0>0$ and
any small constant $\epsilon>0$, let $\cP=\bigsqcup_{t=0}^T \cP_t$ 
and $\cQ=\bigsqcup_{t=0}^T \cQ_t$ be the sets of polynomial functions given in
Definition \ref{def:BCP_approx} for BCP-approximability.
We introduce random vectors
$\bxi_1,\ldots,\bxi_T \in \RR^n$ having i.i.d. $\cN(0,1)$ entries independent of
each other and of $\bW$, and define an auxiliary AMP algorithm
\begin{equation}\label{eq:auxiliary_amp}
\begin{aligned}
    \tilde\bz_t &= \bW \tilde\bu_t - \sum_{s=1}^{t-1} \tilde b_{ts} \tilde\bu_s,
\qquad 
    \tilde\bu_{t+1}=p_t^\epsilon(\tilde\bz_1,\ldots,\tilde\bz_t)
\equiv p_t(\tilde\bz_1,\ldots,\tilde\bz_t)+\epsilon \bxi_{t+1}
\end{aligned}
\end{equation}
with initialization
\[\tilde \bu_1=p_0^\epsilon(\cdot) \equiv p_0(\cdot)+\epsilon \bxi_1,
\qquad \widetilde \bSigma_1=n^{-1}\|\tilde \bu_1\|_2^2.\]
Throughout this section, we will condition on a realization of
$\bxi_{1:T} \equiv \bxi_{1:T}(n)$ and establish statements which hold
almost surely over $\{\bxi_{1:T}(n)\}_{n=1}^\infty$.
The above coefficients $\tilde b_{ts}$ and polynomial
functions $p_t \in \cP_t$ are defined as follows:
\begin{enumerate}
\item Given $\widetilde\bSigma_t$ (defined conditionally on $\bxi_{1:T}$), let
$\widetilde\bZ_{1:t} \sim \cN(0,\widetilde\bSigma_t \otimes \Id_n)$,
and let $p_t \in \cP_t$ be a polynomial function such that 
\begin{equation}\label{eq:poly_approx_assumption}
\frac{1}{n}\,\EE[\|f_t(\widetilde\bZ)-p_t(\widetilde \bZ)\|_2^2 \mid
\bxi_{1:T}]<\epsilon \text{ a.s.\ for all large } n.
\end{equation}
(For $t=0$, this is a constant vector $p_0 \in \cP_0$ for which
$n^{-1}\|f_0-p_0\|_2^2<\epsilon$.)
For sufficiently small $\epsilon>0$,
Lemma \ref{lem:auxiliary_amp} below implies inductively that $\|\widetilde
\bSigma_t\|_\op<\|\bSigma_t\|_\op+\iota(\epsilon)<C_0$
a.s.\ for all large $n$, so such a polynomial
$p_t \in \cP_t$ exists a.s.\ for all large $n$ by Definition \ref{def:BCP_approx}. 
If $f_t(\bz_{1:t})$ depends only on the preceding iterates $\{\bz_s:s \in S_t\}$
for a subset $S_t \subset \{1,\ldots,t\}$, then Definition \ref{def:BCP_approx} 
guarantees that so does $p_t(\bz_{1:t})$. We set
\[p_t^\epsilon(\cdot)=p_t(\cdot)+\epsilon \bxi_{t+1}.\]
Note that since $\lim_{n \to \infty} n^{-1}\|\bxi_{t+1}\|_2^2=1$,
(\ref{eq:poly_approx_assumption}) implies also
\begin{equation}\label{eq:poly_approx_perturbed}
\frac{1}{n}\,\EE[\|f_t(\widetilde\bZ)-p_t^\epsilon(\widetilde\bZ)\|_2^2
\mid \bxi_{1:T}]<2(\epsilon+\epsilon^2) \text{ a.s.\ for all large } n.
\end{equation}
\item Then given $\widetilde\bSigma_t$ and
$p_0^\epsilon,\ldots,p_t^\epsilon$, define
$\{\tilde b_{t+1,s}\}_{s \leq t}$ in (\ref{eq:auxiliary_amp})
and $\widetilde\bSigma_{t+1} \in \RR^{(t+1) \times (t+1)}$
as in Definition \ref{def:non_asymp_se} by
\[\tilde b_{t+1,s}=\frac{1}{n}\EE[\div\nolimits_s p_t^\epsilon(\widetilde\bZ)
\mid \bxi_{1:T}], \quad \widetilde\bSigma_{t+1}[r+1,s+1]
=\frac{1}{n}\EE[p_r^\epsilon(\widetilde\bZ)^\top
p_s^\epsilon(\widetilde\bZ) \mid \bxi_{1:T}].\]
\end{enumerate}

The following lemma shows that the iterates of this auxiliary AMP algorithm are 
well-defined and close to the original iterates.

\begin{lemma}\label{lem:auxiliary_amp}
Suppose the conditions of Theorem \ref{thm:main_universality} hold.
Then there are constants $C>0$ and $\iota(\epsilon)>0$ satisfying
$\iota(\epsilon) \to 0$ as $\epsilon\to 0$ such that
for the auxiliary AMP algorithm \eqref{eq:auxiliary_amp} defined with any
$\epsilon>0$ sufficiently small, for each $t=1,\ldots,T$,
almost surely for all large $n$,
\[\|\bSigma_t-\widetilde\bSigma_t\|_{\op} < \iota(\epsilon),
\qquad
\frac{1}{\sqrt{n}} \|\bz_t-\tilde\bz_t\|_2 
< \iota(\epsilon), \qquad \frac{1}{\sqrt{n}} \|\bz_t\|_2 < C.\]
\end{lemma}
\begin{proof}
We prove by induction on $t$ the following claims, for constants $C>0$ and
$\iota(\epsilon)>0$ 
satisfying $\iota(\epsilon)\to0$ as $\epsilon\to0$:
\begin{enumerate}
\item $\frac{1}{\sqrt{n}}\|\bu_t-\tilde\bu_t\|_2<\iota(\epsilon)$ and
$\frac{1}{\sqrt{n}}\|\bu_t\|_2<C$ almost surely for all large
$n$;
\item $\max_{s=1}^{t-1} |b_{ts}-\tilde b_{ts}|<\iota(\epsilon)$;
\item $\frac{1}{\sqrt{n}}\|\bz_t-\tilde\bz_t\|_2<\iota(\epsilon)$
and $\frac{1}{\sqrt{n}} \|\bz_t\|_2 < C$ almost surely for all large $n$;
\item $\|\bSigma_t-\widetilde\bSigma_t\|_\op<\iota(\epsilon)$.
\end{enumerate}
For the base case $t=1$, (1) holds by the bounds
$n^{-1}\|\bu_1\|_2^2=\|\bSigma_1\|_\op<C_0$,
$n^{-1}\|\bu_1-\tilde\bu_1\|_2^2
\leq 2n^{-1}\|p_0-f_0\|_2^2+2\epsilon^2n^{-1}\|\bxi_1\|_2^2$,
and a standard chi-squared tail bound for $\|\bxi_1\|_2^2$.
(2) is vacuous. Since $\bz_1=\bW\bu_1$ and $\tilde \bz_1=\bW\tilde\bu_1$,
(3) holds by (1) and the operator norm bound $\|\bW\|_\op<3$
a.s.\ for all large $n$. (4) holds by (1) and the definitions 
$\bSigma_1=n^{-1}\|\bu_1\|_2^2$ and $\widetilde\bSigma_1=n^{-1}\|\tilde
\bu_1\|_2^2$.

Now suppose inductively that statements (1--4) all hold for $1,\ldots,t$, where
$t \leq T-1$. We
write $C>0$ and $\iota(\epsilon)>0$ for constants changing from instance to
instance, where $\iota(\epsilon) \to 0$ as $\epsilon \to 0$.
To check (1) for iteration $t+1$, observe from the 
definition of $\bu_{t+1}$ and $\tilde\bu_{t+1}$ that
\begin{equation}\label{eq:utildeucompare}
    \frac{1}{\sqrt{n}}\|\bu_{t+1}-\tilde\bu_{t+1}\|_2 
    \leq \frac{1}{\sqrt{n}} \|f_t(\bz_{1:t}) - f_t(\tilde\bz_{1:t})\|_2
    + \frac{1}{\sqrt{n}} \|f_t(\tilde\bz_{1:t}) - p_t^\epsilon(\tilde\bz_{1:t})\|_2.
\end{equation}
The first term of (\ref{eq:utildeucompare})
is at most $\iota(\epsilon)$ a.s.\ for all large $n$
by the Lipschitz condition (\ref{eq:Lipschitz}) and the induction hypothesis.
For the second term, note that for any $q_1,q_2 \in \cQ_t$ with degrees bounded
independently of $n$, Definition \ref{def:BCP_approx} ensures that
$\{p_0,\ldots,p_t,q_1,q_2\}$ is BCP-representable.
Then by Corollary \ref{cor:universality_perturbed},
\[\lim_{n \to \infty}
\frac{1}{n}q_1(\tilde\bz_{1:t})^\top q_2(\tilde\bz_{1:t})
-\frac{1}{n}\EE[q_1(\widetilde\bZ_{1:t})^\top q_2(\widetilde\bZ_{1:t}) \mid
\bxi_{1:T}]=0 \text{ a.s.}\]
Then condition (2) of Definition \ref{def:BCP_approx} further ensures that
\[\limsup_{n \to \infty}
\frac{1}{n}\|f_t(\tilde\bz_{1:t})-p_t(\tilde\bz_{1:t})\|_2^2
<\epsilon\text{ a.s.,}\]
so (\ref{eq:poly_approx_assumption}) and the statement
$n^{-1}\|p_t^\epsilon(\tilde\bz_{1:t})-p_t(\tilde\bz_{1:t})\|_2^2<\iota(\epsilon)$
a.s.\ for all large $n$ together imply that the second
term of (\ref{eq:utildeucompare}) is at most $\iota(\epsilon)$. Thus 
$\frac{1}{\sqrt{n}}\|\bu_{t+1}-\tilde\bu_{t+1}\|_2<\iota(\epsilon)$ a.s.\ for
all large $n$. The bound
$\frac{1}{\sqrt{n}}\|\bu_{t+1}\|_2<C$ follows directly from
the Lipschitz condition (\ref{eq:Lipschitz}) and the induction hypothesis.

For (2), let $S_t \subseteq \{1,\ldots,t\}$ be the subset for which
$f_t(\bz_{1:t}) \equiv f_t(\bz_{S_t})$
and $p_t(\bz_{1:t}) \equiv p_t(\bz_{S_t})$ depend only on
$\bz_{S_t}=\{\bz_s:s \in S_t\}$. Note that for each $s \notin S_t$, we have
$b_{t+1,s}=\tilde b_{t+1,s}=0$. For $s \in S_t$, by definition we have 
\begin{align*}
    (b_{t+1,s}-\tilde b_{t+1,s})_{s \in S_t}&=\bigg(\frac{1}{n} \EE[\div\nolimits_s
f_t(\bZ_{S_t})]
    - \frac{1}{n}\EE[\div\nolimits_s p_t^\epsilon(\widetilde\bZ_{S_t}) \mid
\bxi_{1:T}]\bigg)_{s \in S_t}\\
&=\frac{1}{n}\sum_{i=1}^n \bigg(\EE[\partial_{\bZ_s[i]}f_t(\bZ_{S_t})[i]]
    - \EE[\partial_{\widetilde \bZ_s[i]} p_t^\epsilon(\widetilde\bZ_{S_t})[i]
\mid \bxi_{1:T}]\bigg)_{s \in S_t}.
\end{align*}
For $\epsilon>0$ sufficiently small, the induction hypothesis and given
condition $\lambda_{\min}(\bSigma_t[S_t,S_t])>c$ imply that both
$\bSigma_t[S_t,S_t]$ and $\widetilde\bSigma_t[S_t,S_t]$ are non-singular
a.s.\ for all large $n$.
Then, applying Stein's lemma (Lemma \ref{lem:stein}) to each function
$f_t(\cdot)[i]$ and $p_t^\epsilon(\cdot)[i]$, we have
\begin{align*}
(b_{t+1,s}-\tilde b_{t+1,s})_{s \in S_t}
    &= \frac{1}{n} \sum_{i=1}^n \bigg(\bSigma_t[S_t,S_t]^{-1}
\EE[\bZ_{S_t}[i]f_t(\bZ_{S_t})[i]]-\widetilde\bSigma_t[S_t,S_t]^{-1}
    \EE[\widetilde\bZ_{S_t}[i]p_t^\epsilon(\widetilde\bZ_{S_t})[i] \mid
\bxi_{1:T}]\bigg).
\end{align*}
For $\epsilon>0$ sufficiently small, the induction hypothesis and condition
$\lambda_{\min}(\bSigma_t[S_t,S_t])>c$ imply also
$\|\bSigma_t[S_t,S_t]^{-1}-\widetilde\bSigma_t[S_t,S_t]^{-1}\|_\op<\iota(\epsilon)$,
and there exists a coupling of $\bZ_{1:t}$ (independent of $\bxi_{1:T}$)
and $\widetilde\bZ_{1:t}$ such that
$n^{-1}\EE[\|\bZ_{1:t}-\widetilde\bZ_{1:t}\|_{\Fro}^2 \mid \bxi_{1:T}]<\iota(\epsilon)$ a.s.\ for
all large $n$. Then, together with the Lipschitz
condition (\ref{eq:Lipschitz}) for $f_t$, the approximation bound
(\ref{eq:poly_approx_perturbed}), and Cauchy-Schwarz, this implies
\begin{align*}
&\|(b_{t+1,s}-\tilde b_{t+1,s})_{s \in S_t}\|_2\\
&\leq
\|\bSigma_t[S_t,S_t]^{-1}-\widetilde\bSigma_t[S_t,S_t]^{-1}\|_\op
\left\|\frac{1}{n} \sum_{i=1}^n 
\EE[\bZ_{S_t}[i]f_t(\bZ_{S_t})[i]]\right\|_2\\
&\hspace{0.3in}+\|\widetilde\bSigma_t[S_t,S_t]^{-1}\|_\op
\left\|\frac{1}{n} \sum_{i=1}^n \bigg(
\EE[\bZ_{S_t}[i]f_t(\bZ_{S_t})[i]]-\EE[\widetilde\bZ_{S_t}[i]f_t(\widetilde\bZ_{S_t})[i]
\mid \bxi_{1:T}]\bigg)\right\|_2\\
&\hspace{0.3in}+\|\widetilde\bSigma_t[S_t,S_t]^{-1}\|_\op
\left\|\frac{1}{n} \sum_{i=1}^n \bigg(
\EE[\widetilde\bZ_{S_t}[i]f_t(\widetilde\bZ_{S_t})[i] \mid
\bxi_{1:T}]-\EE[\widetilde\bZ_{S_t}[i]p_t^\epsilon(\widetilde\bZ_{S_t})[i] \mid
\bxi_{1:T}]\bigg)\right\|_2<\iota(\epsilon)
\end{align*}
for some $\iota(\epsilon)>0$ a.s.\ for all large $n$, establishing (2).

For (3), from the definition of $\bz_{t+1}$ and $\tilde\bz_{t+1}$,
\begin{align*}
    &\frac{1}{\sqrt{n}}\|\bz_{t+1}-\tilde\bz_{t+1}\|_2\\ &\leq \frac{1}{\sqrt{n}}
    \|\bW(\bu_{t+1}-\tilde\bu_{t+1})\|_2
    + \sum_{s=1}^t \Big(|b_{t+1,s}|\cdot \frac{1}{\sqrt{n}}\|\bu_s-\tilde\bu_s\|_2
    +|b_{t+1,s}-\tilde b_{t+1,s}|\cdot\frac{1}{\sqrt{n}}\|\tilde
\bu_s\|_2\Big),
\end{align*}
so (3) follows from the bound $\|\bW\|_\op<3$ a.s.\ for all large $n$
and (1) and (2) already shown.

For (4), the entries of $\bSigma_{t+1}$ are given by
$n^{-1}\EE[f_s(\bZ_{1:s})^\top f_r(\bZ_{1:r})]$, while those of $\widetilde
\bSigma_{t+1}$ are given by
$n^{-1}\EE[p_s^\epsilon(\widetilde \bZ_{1:s})^\top p_r^\epsilon(\widetilde
\bZ_{1:r}) \mid \bxi_{1:T}]$.
The induction hypothesis implies that there exists a coupling of
$\bZ_{1:t}$ (independent of $\bxi_{1:T}$) with $\widetilde \bZ_{1:t}$ for which
$n^{-1}\EE[\|\bZ_{1:t}-\widetilde\bZ_{1:t}\|_\Fro^2 \mid \bxi_{1:T}]<\iota(\epsilon)$.
Then (4) follows from this coupling,
the Lipschitz condition (\ref{eq:Lipschitz}) for $f_t$, the approximation bound
(\ref{eq:poly_approx_perturbed}), and Cauchy-Schwarz, analogous to the above
argument for (2). This completes the induction.
\end{proof}

We now prove Theorem \ref{thm:main_universality} and Corollary
\ref{cor:equivalent_SE}.

\begin{proof}[Proof of Theorem \ref{thm:main_universality}]\label{pf:main_universality}
Let $\bz_1,\ldots,\bz_T$ denote the iterates of the given AMP algorithm.
Fixing the constant $C_0>0$ satisfying (\ref{eq:C0def})
and any $\epsilon>0$ sufficiently small,
let $\tilde \bz_1,\ldots,\tilde\bz_T$
denote the iterates of the auxiliary AMP algorithm (\ref{eq:auxiliary_amp}).
We write $C>0$ and $\iota(\epsilon)>0$ for constants changing from instance to
instance, where $\iota(\epsilon)\to 0$ as $\epsilon\to0$.

We may decompose
\begin{align}
    \phi(\bz_{1:T})-\EE \phi(\bZ_{1:T})
    &=[\phi(\bz_{1:T}) - \phi(\tilde\bz_{1:T})]
    +[\phi(\tilde\bz_{1:T})-\EE[\phi(\widetilde\bZ_{1:T}) \mid
\bxi_{1:T}]]\notag\\
    &\hspace{0.3in}+[\EE[\phi(\widetilde\bZ_{1:T}) \mid \bxi_{1:T}]-\EE
\phi(\bZ_{1:T})].\label{eq:phi_W_G_decomposition}
\end{align}
For the first term of (\ref{eq:phi_W_G_decomposition}),
since both $\phi_1,\phi_2$ defining $\phi$ satisfy the Lipschitz condition
(\ref{eq:Lipschitz}), we have
\begin{align*}
    &|\phi(\bz_{1:T})-\phi(\tilde\bz_{1:T})|\\
    &\leq \bigg|\frac{1}{n} \phi_1(\bz_{1:T})^\top\phi_2(\bz_{1:T})
    - \frac{1}{n} \phi_1(\tilde\bz_{1:T})^\top\phi_2(\bz_{1:T})\bigg|
    + \bigg|\frac{1}{n} \phi_1(\tilde\bz_{1:T})^\top\phi_2(\bz_{1:T})
    - \frac{1}{n} \phi_1(\tilde\bz_{1:T})^\top\phi_2(\tilde\bz_{1:T})\bigg|\\
    &\leq \frac{1}{n} \|\phi_2(\bz_{1:T})\|_2 \cdot 
    \|\phi_1(\bz_{1:T}) - \phi_1(\tilde\bz_{1:T})\|_2
    + \frac{1}{n} \|\phi_1(\tilde\bz_{1:T})\|_2 \cdot
    \|\phi_2(\bz_{1:T}) - \phi_2(\tilde\bz_{1:T})\|_2\\
    &\leq \frac{C}{n} \bigg(\sqrt{n}+\sum_{t=1}^T \|\bz_t\|_2 + \|\tilde\bz_t\|_2\bigg) 
    \bigg(\sum_{t=1}^T \|\bz_t-\tilde\bz_t\|_2\bigg)
\end{align*}
for a constant $C>0$ depending on $L$.
Then by Lemma \ref{lem:auxiliary_amp},
\begin{equation}\label{eq:phi_aux_limit}
    |\phi(\bz_{1:T})-\phi(\tilde\bz_{1:T})|<\iota(\epsilon)
\text{ a.s.\ for all large } n.
\end{equation}

For the second term of \eqref{eq:phi_W_G_decomposition},
let $\psi_1,\psi_2 \in \cP_T$ be the polynomials guaranteed by
Definition \ref{def:BCP_approx} for which
\begin{equation}\label{eq:phi_poly_approx}
    \frac{1}{n}\,\EE[\|\phi_1(\widetilde\bZ_{1:T})
    -\psi_1(\widetilde\bZ_{1:T})\|_2^2 \mid \bxi_{1:T}] < \epsilon,\quad 
    \frac{1}{n}\,\EE[\|\phi_2(\widetilde\bZ_{1:T})
    -\psi_2(\widetilde\bZ_{1:T})\|_2^2 \mid \bxi_{1:T}] < \epsilon
\end{equation}
almost surely for all large $n$.
Writing $\psi=n^{-1}\psi_1^\top \psi_2$, let us further decompose 
\begin{align}
    \phi(\tilde\bz_{1:T})-\EE[\phi(\widetilde\bZ_{1:T}) \mid \bxi_{1:T}]
    &= [\phi(\tilde\bz_{1:T}) - \psi(\tilde\bz_{1:T})]
    +[\psi(\tilde\bz_{1:T}) - \EE[\psi(\widetilde\bZ_{1:T}) \mid
\bxi_{1:T}]]\notag\\
&\hspace{0.3in}
+[\EE[\psi(\widetilde\bZ_{1:T}) \mid \bxi_{1:T}]- \EE[\phi(\widetilde\bZ_{1:T}) \mid
\bxi_{1:T}]].\label{eq:phi_diff}
\end{align}
For the first term of (\ref{eq:phi_diff}),
we apply the same decomposition as above to get
\begin{align}
    &|\phi(\tilde\bz_{1:T})-\psi(\tilde\bz_{1:T})|\notag\\
&\leq 
    \bigg|\frac{1}{n}\phi_1(\tilde\bz_{1:T})^\top\phi_2(\tilde\bz_{1:T}) 
    - \frac{1}{n} \psi_1(\tilde\bz_{1:T})^\top \phi_2(\tilde\bz_{1:T})\bigg|
    + \bigg|\frac{1}{n}\psi_1(\tilde\bz_{1:T})^\top \phi_2(\tilde\bz_{1:T}) 
    - \frac{1}{n}\psi_1(\tilde\bz_{1:T})^\top
\psi_2(\tilde\bz_{1:T})\bigg|\notag\\
    &\leq \frac{1}{n} \|\phi_2(\tilde\bz_{1:T})\|_2 \cdot 
    \|\phi_1(\tilde\bz_{1:T}) - \psi_1(\tilde\bz_{1:T})\|_2
    + \frac{1}{n} \|\psi_1(\tilde\bz_{1:T})\|_2 \cdot 
    \|\phi_2(\tilde\bz_{1:T}) - \psi_2(\tilde\bz_{1:T})\|_2. 
    \label{eq:phi_approx_bound}
\end{align}
We will apply \eqref{eq:phi_poly_approx} to further bound the right side.
To do so, note that by Definition \ref{def:BCP_approx},
$\{p_0,\ldots,p_{T-1},q_1,q_2\}$ is BCP-representable
for any $q_1,q_2 \in \cQ_T$ of degrees bounded independently of $n$.
Then by Corollary \ref{cor:universality_perturbed},
    \[\lim_{n\to\infty} \frac{1}{n} q_1(\tilde\bz_{1:T})^\top 
    q_2(\tilde\bz_{1:T})-\frac{1}{n}\EE[q_1(\widetilde\bZ_{1:T})^\top 
    q_2(\widetilde\bZ_{1:T}) \mid \bxi_{1:T}] = 0 \text{ a.s.}\]
Then condition (2) of Definition \ref{def:BCP_approx} ensures
    \[\limsup_{n\to\infty} \frac{1}{n} 
    \|\phi_1(\tilde\bz_{1:T})-\psi_1(\tilde\bz_{1:T})\|_2^2 
    <\epsilon \text{ a.s.},\]
and the same holds with $\phi_2,\psi_2$ in place of $\phi_1,\psi_1$.
It then follows from \eqref{eq:phi_poly_approx} that
almost surely for all large $n$,
\begin{equation}\label{eq:phi_poly_approx_amp}
    \max\bigg\{\frac{1}{n} \|\phi_1(\tilde\bz_{1:T})-\psi_1(\tilde\bz_{1:T})\|_2^2, 
    \frac{1}{n} \|\phi_2(\tilde\bz_{1:T})-\psi_2(\tilde\bz_{1:T})\|_2^2\bigg\} 
    <\iota(\epsilon).
\end{equation}
Moreover, $\frac{1}{\sqrt{n}}\|\phi_1(\tilde\bz_{1:T})\|_2<C$ a.s.\ for all
large $n$ by the Lipschitz property (\ref{eq:Lipschitz}) for $\phi_1$ and
Lemma \ref{lem:auxiliary_amp}, and similarly for $\phi_2$.
Combining this with \eqref{eq:phi_poly_approx_amp}, also
$\frac{1}{\sqrt{n}}\|\psi_1(\tilde\bz_{1:T})\|_2<C$ a.s.\ for all large $n$, and
similarly for $\psi_2$.
Then, applying these bounds to \eqref{eq:phi_approx_bound},
\begin{equation}\label{eq:phi_poly_approx_bound}
    |\phi(\tilde\bz_{1:T})-\psi(\tilde\bz_{1:T})|<\iota(\epsilon)
\text{  a.s.\ for all large } n.
\end{equation}
For the second term of \eqref{eq:phi_diff},
we have from Corollary \ref{cor:universality_perturbed} that
$\lim_{n\to\infty} \psi(\tilde \bz_{1:T}) - \EE[\psi(\widetilde \bZ_{1:T}) \mid \bxi_{1:T}] = 0$.
The third term of \eqref{eq:phi_diff} is bounded via
\eqref{eq:phi_poly_approx} and an argument analogous to the preceding argument
for the first term.
Combining these bounds for the three terms of \eqref{eq:phi_diff}, we obtain
for the second term of (\ref{eq:phi_W_G_decomposition}) that
\begin{equation}\label{eq:phi_diff_limit}
    |\phi(\tilde\bz_{1:T})-\EE[\phi(\widetilde\bZ_{1:T}) \mid \bxi_{1:T}]|<\iota(\epsilon)
\text{ a.s.\ for all large } n.
\end{equation}

Finally, for the third term of (\ref{eq:phi_W_G_decomposition}), we note that
the bound $\|\bSigma_T-\widetilde \bSigma_T\|_\op<\iota(\epsilon)$ of Lemma
\ref{lem:auxiliary_amp} implies there exists a coupling of
$\bZ_{1:T}$ (independent of $\bxi_{1:T}$) with $\widetilde \bZ_{1:T}$ such that
$n^{-1}\EE[\|\bZ_{1:T}-\widetilde \bZ_{1:T}\|_\Fro^2 \mid \bxi_{1:T}]<\iota(\epsilon)$. Applying
this coupling, the Lipschitz condition (\ref{eq:Lipschitz}) for $\phi_1,\phi_2$
defining $\phi=n^{-1}\phi_1^\top\phi_2$,
and Cauchy-Schwarz, we obtain that
\begin{equation}\label{eq:phi_diff_SE}
    |\EE[\phi(\widetilde\bZ_{1:T}) \mid \bxi_{1:T}]-\EE \phi(\bZ_{1:T})|<\iota(\epsilon)
\text{ a.s.\ for all large } n.
\end{equation}

Collecting \eqref{eq:phi_W_G_decomposition}, \eqref{eq:phi_aux_limit},
\eqref{eq:phi_diff_limit}, and \eqref{eq:phi_diff_SE}, we have
    \[|\phi(\bz_{1:T})-\EE \phi(\bZ_{1:T})|<\iota(\epsilon) \text{ a.s.\ for all
large } n.\]
Since $\epsilon>0$ is arbitrary,
this implies $\lim_{n \to \infty} \phi(\bz_{1:T})-\EE \phi(\bZ_{1:T})=0$
a.s.\ as desired.
\end{proof}

\begin{proof}[Proof of Corollary \ref{cor:equivalent_SE}]
Denote the AMP algorithm defined by $\{\bar b_{ts}\}$ as
\[\bar \bz_t=\bW\bar \bu_t-\sum_{s=1}^{t-1} \bar b_{ts} \bar \bu_s,
\qquad \bar \bu_{t+1}=f_t(\bar \bz_1,\ldots,\bar \bz_t),\]
with initialization $\bar\bu_1=\bu_1$. Using $\|\bW\|_\op<3$ a.s.\ for
all large $n$ and the Lipschitz condition (\ref{eq:Lipschitz}) for $f_t(\cdot)$,
a straightforward induction on $t$ (omitted for brevity) shows that for each
$t=1,\ldots,T$,
\begin{itemize}
\item $\lim_{n \to \infty} \frac{1}{\sqrt{n}}\|\bu_t-\bar\bu_t\|_2=0$ a.s.\ and
$\frac{1}{\sqrt{n}}\|\bu_t\|_2<C$ a.s.\ for all large $n$.
\item $\lim_{n \to \infty} \frac{1}{\sqrt{n}}\|\bz_t-\bar\bz_t\|_2=0$ a.s.\ and
$\frac{1}{\sqrt{n}}\|\bz_t\|_2<C$ a.s.\ for all large $n$.
\end{itemize}
Then, applying the Lipschitz condition (\ref{eq:Lipschitz}) for $\phi_1,\phi_2$
and Cauchy-Schwarz, also
$\lim_{n \to \infty} \phi(\bz_{1:T})-\phi(\bar\bz_{1:T})=0$ a.s. Letting
$\bar\bZ_{1:T}$ have i.i.d.\ rows with distribution $\cN(0,\bar\bSigma_T)$,
since $\lim_{n \to \infty} \bSigma_T-\bar\bSigma_T=0$,
there is a coupling of $\bZ_{1:T}$ with $\bar\bZ_{1:T}$ such that $\lim_{n \to
\infty} n^{-1}\EE \|\bZ_{1:T}-\bar \bZ_{1:T}\|_{\Fro}^2 = 0$ a.s.
Applying this coupling, the condition (\ref{eq:Lipschitz}) for $\phi_1,\phi_2$,
and Cauchy-Schwarz again, we have also
$\lim_{n \to \infty} \EE\phi(\bZ_{1:T})-\EE\phi(\bar\bZ_{1:T})=0$. Thus
\[\lim_{n \to \infty} \phi(\bar\bz_{1:T})-\EE\phi(\bar\bZ_{1:T})=0 \text{
a.s.}\]
\end{proof}

\section{Verification of BCP-representability and BCP-approximability}\label{sec:BCPexamples}

In this section, we verify the conditions of BCP-representability and
BCP-approximability for the three function classes of Section
\ref{sec:examples}.
We prove Proposition \ref{prop:local} in Appendix \ref{sec:BCPLocal},
Proposition \ref{prop:anisotropic} in Appendix \ref{sec:BCPAniso},
and Proposition \ref{prop:spectral} in Appendix \ref{sec:BCPSpectral}.

\subsection{Local functions}\label{sec:BCPLocal}

Recall the classes of polynomial and Lipschitz local functions from
Definitions \ref{def:polylocal} and \ref{def:local}.
We first show Proposition \ref{prop:local}(a),
that a set $\cP$ of polynomial local functions is BCP-representable,
via the following lemma.

\begin{lemma}\label{lem:summabletensor}
Suppose $\cT=\bigsqcup_{k=1}^K \cT_k$ is a class of tensors such that for a
constant $C_0>0$, every $\bT \in \cT_k$ satisfies the condition, for each fixed
position $\ell \in [k]$ and fixed index $j \in [n]$,
\begin{equation}\label{eq:summabletensor}
\sum_{i_1,\ldots,i_{\ell-1},i_{\ell+1},\ldots,i_k=1}^n
|\bT[i_1,\ldots,i_{\ell-1},j,i_{\ell+1},\ldots,i_k]|<C_0.
\end{equation}
(For $k=1$, this means $|\bT[j]|<C_0$ for each $j \in [n]$.)
Then for any connected tensor network $(G,\cL)$ with tensors
in $\cT$, there exists a constant $C>0$ depending only on $G$ and $C_0$ such
that
\[|\val_G(\cL)| \leq Cn.\]
In particular, $\cT$ satisfies the BCP.
\end{lemma}
\begin{proof}
Let $\cL$ be any tensor labeling of $G=(\cV,\cE)$ with tensors $\{\bT_v:v \in
\cV\}$ belonging to $\cT$. We apply the upper bound
\begin{equation}\label{eq:crudevalbound}
|\val_G(\cL)|
\leq \sum_{\bi \in [n]^{\cE}}\prod_{v \in \cV}|\bT_v\ss{i_e:e \sim v}|.
\end{equation}

To analyze this bound, we may reduce to the case where $G=(\cV,\cE)$ is a
connected tree:
If $\cE$ contains a cycle, pick any edge $e=(u,v) \in \cE$ of the
cycle, and replace the sum over the shared index $i_e \in [n]$ of $\bT_u$ and
$\bT_v$ in (\ref{eq:crudevalbound})
by sums over two distinct indices $i_{e'} \in [n]$ for $\bT_u$
and $i_{e''} \in [n]$ for $\bT_v$. This does not decrease the upper bound,
as the terms with $i_{e'}=i_{e''}$ correspond precisely to
(\ref{eq:crudevalbound}) and the additional terms with $i_{e'} \neq i_{e''}$ are
non-negative. The resulting bound corresponds to (\ref{eq:crudevalbound}) 
for a graph in which we add vertices $w,x$ with the all-1's label
$\bone \in \RR^n$, add edges $e'=(u,w)$ and $e''=(v,x)$, and remove the edge
$(u,v)$. Repeating this process until the resulting graph is a tree, and
replacing $\cT$ by $\cT \cup \{\bone\}$ (where $\bone$ also
satisfies the condition (\ref{eq:summabletensor}) for $k=1$), it suffices to
bound (\ref{eq:crudevalbound}) when $G$ is a connected tree.

In the case where $G$ is a connected tree, pick any leaf vertex $u$ and suppose
$u$ is connected to $v$ via the edge $e=(u,v)$. Let $\cE_v$ denote the set of
all edges incident to $v$.
Then we may remove $e$ and contract $u,v$ into a single vertex
$w$, labeled by the contracted tensor $\bT_w$ having entries
\[\bT_w[i_{e'}:e' \in \cE_v \setminus e]
=\sum_{i_e=1}^n |\bT_u[i_e]| \cdot |\bT_v[i_e,\,i_{e'}:e' \in \cE_v \setminus e]|.\]
We note that the condition (\ref{eq:summabletensor}) for $\bT_u$ implies
$|\bT_u[i]| \leq C_0$ for
all $i \in [n]$. Then the condition (\ref{eq:summabletensor})
holds with the constant $C_0^2$ for $\bT_w$, since it holds with $C_0$
for $\bT_v$.
Denoting by $G'=(\cV',\cE')$ the contracted tree graph with $u,v$ replaced by
$w$, (\ref{eq:crudevalbound}) becomes
\[|\val_G(\cL)|
\leq \sum_{\bi \in [n]^{\cE'}}\prod_{v \in \cV'}|\bT_v\ss{i_e:e \sim v}|\]
where each $\{\bT_v:v \in \cV'\}$ satisfies (\ref{eq:summabletensor}) with
constant $C_0^2$. Iterating this contraction procedure until $G'$ has only two
vertices $w,x$, we obtain
\[|\val_G(\cL)| \leq \sum_{i=1}^n |\bT_w[i]| \cdot |\bT_x[i]|\]
where $\bT_w,\bT_x \in \RR^n$ have all entries bounded by
a constant depending only on $C_0$ and $G$. This shows
$|\val_G(\cL)| \leq Cn$.

By Definition \ref{def:BCP}, $\cT$ satisfies the BCP if
$\sup_\cL |\val_G(\cL)| \leq Cn$ where the supremum is taken
over all $(\Id,\cT)$-labelings $\cL$ of certain
bipartite multigraphs $G=(\cV_\Id \sqcup \cV_T,\cE)$. The identity tensor
$\Id$ of any order trivially satisfies the condition
(\ref{eq:summabletensor}), so the BCP for $\cT$ follows from
the above bound applied to
$\cT \cup \{\Id^1,\ldots,\Id^k\}$ where $k$ is the maximum vertex degree
of $\cV_\Id$.
\end{proof}

\begin{proof}[Proof of Proposition \ref{prop:local}(a)]
Let $\cP=\bigsqcup_{t=0}^T \cP_t$ where $\cP_t$ consists of the functions
$p:\RR^{n \times t} \to \RR^n$.
Letting $D,B>0$ be the degree and coefficient bounds of
Definition \ref{def:polylocal}, any $p \in \cP_t$ admits a
representation (\ref{eq:tensorpolyrepr}) with this value of $D$, where
each entry of $\bT^{(0)},\bT^{(\sigma)}$ is a coefficient of $p$ and hence has
magnitude at most $B$. Let $\cT=\bigsqcup_{k=1}^{D+1} \cT_k$ be the set of all
tensors arising in this representation for all $p \in \cP$.
For any $\bT \in \cT_k$, the locality condition implies that for
each fixed output index $i \in [n]$, we have
\[\sum_{i_1,\ldots,i_{k-1}=1}^n |\bT[i_1,\ldots,i_{k-1},i]|
=\sum_{i_1,\ldots,i_{k-1} \in A_i} |\bT[i_1,\ldots,i_{k-1},i]|
\leq A^{k-1} \cdot B\]
where $A \geq |A_i|$ for every $i \in [n]$.
Then also fixing the first input index $j \in [n]$,
\[\sum_{i_2,\ldots,i_k=1}^n |\bT[j,i_2,\ldots,i_k]|
=\sum_{i:j \in A_i} \sum_{i_2,\ldots,i_{k-1} \in A_i}
|\bT[j,i_2,\ldots,i_{k-1},i]|
\leq A^{k-1} \cdot B\]
where also $A \geq |\{i:j \in A_i\}|$ for every $j \in [n]$. Since $A,B$ are
constants independent of $n$, $\cT$ satisfies the BCP by Lemma
\ref{lem:summabletensor}, so $\cP$ is BCP-representable.
\end{proof}

\begin{proof}[Proof of Proposition \ref{prop:local}(b)]
Let $\cF=\bigsqcup_{t=0}^T \cF_t$, where $\cF_t$ consists
of the functions $f:\RR^{n \times t} \to \RR^n$.
Given any $C_0,\epsilon>0$ in Definition \ref{def:BCP_approx},
let $\zeta,\iota>0$ be constants depending on $L,C_0,\epsilon$ to be specified
later. We will track explicitly the dependence of our bounds
on $\zeta,\iota$, and write $C,C',c>0$ for constants changing from instance
to instance that \emph{do not} depend on $\zeta,\iota$.

We first construct a set of polynomial local functions $\cP=\bigsqcup_{t=0}^T
\cP_t$ to verify condition (1) in Definition \ref{def:BCP_approx}.
For $t=0$ and each constant vector $f=(\mathring{f}_i)_{i=1}^n \in \cF_0$,
we simply include $p=f$ in $\cP_0$, where $p$ has degree 0 and bounded entries
by the condition (2) of Definition \ref{def:local}.
For $t=1,\ldots,T$ and each $f \in\cF_t$, we construct an
approximating polynomial $p$ to include in $\cP_t$ via the following two steps:
\begin{enumerate}[(a)]
\item \label{item:epsilon_net_local} 
For each $a=0,1,\ldots,A$, define
\begin{equation}\label{eq:finitelipschitzclass}
\mathring{\cF}_a=\{\mathring{f}:\RR^{a \times t}\to\RR: \mathring{f} \text{ is
$L$-Lipschitz with } |\mathring{f}(0)|\leq L\}.
\end{equation}
Let $\cN_a\subseteq\mathring{\cF}_a$ be a $\zeta$-net under the sup-norm over
the Euclidean ball of radius $1/\zeta$, i.e., for any
$\mathring{f}\in\mathring{\cF}_a$, there
exists $\mathring{g}\in\cN_a$ such that 
\begin{equation}\label{eq:lipschitz_net_approx_local}
    \sup_{\bx\in\RR^{a \times t}:\|\bx\|_\Fro^2\leq (1/\zeta)^2} 
    |\mathring{g}(\bx)-\mathring{f}(\bx)|^2 < \zeta.
\end{equation}
The definitions and cardinalities of $\cN_a$ depend only on
$L,\zeta,a,t$ and are independent of $n$.
For each $i\in [n]$, let $\mathring{g}_i\in\cN_{|A_i|}$ be the net
approximation for $\mathring{f}_i$ satisfying \eqref{eq:lipschitz_net_approx_local},
and define $g=(\mathring{g}_i)_{i=1}^n$.

\item For each $a=0,1,\ldots,A$ and each $\mathring{g}\in\cN_a$, let 
$\mathring{p}:\RR^{a \times t}\to\RR$ be a polynomial function that
approximates $\mathring{g}$ in the sense
\begin{equation}\label{eq:poly_approx_local}
    \EE_{\bZ \sim \cN(0,\bSigma \otimes \Id_a)}[|\mathring{g}(\bZ)
    -\mathring{p}(\bZ)|^2] < \iota
\end{equation}
for every $\bSigma \in \RR^{t \times t}$ satisfying $\|\bSigma\|_\op<C_0$.
We may construct this approximation as follows: First fixing any $\delta>0$,
Lemma \ref{lem:polyapprox} implies there exists a polynomial
$\mathring{p}:\RR^{a \times t} \to \RR$ which satisfies
\[\sup_{\bz \in \RR^{a \times t}} e^{-\sum_{i,j} |\bz[i,j]|^{3/2}}
|\mathring{g}(\bz)-\mathring{p}(\bz)| \leq \delta.\]
Then, for any $\bSigma$ with $\|\bSigma\|_\op<C_0$, letting
$\bZ \sim \cN(0,\bSigma \otimes \Id_a) \in \RR^{a \times t}$,
\begin{align*}
\EE[|\mathring{g}(\bZ)-\mathring{p}(\bZ)|^2]
&=\int_0^\infty \PP[|\mathring{g}(\bZ)-\mathring{p}(\bZ)|^2>x] dx\\
&\leq \int_0^\infty \PP\bigg[\sum_{i,j}
|\bZ[i,j]|^{3/2}>\log(x^{1/2}/\delta)\bigg]dx\\
&=\int_0^\infty 2\delta^2y \cdot \PP\bigg[\sum_{i,j}
|\bZ[i,j]|^{3/2}>\log y\bigg]dy<C\delta^2
\end{align*}
for a constant $C>0$ depending only on $C_0,a,t$. Then choosing
$\delta \equiv \delta(\iota)>0$ small enough
ensures (\ref{eq:poly_approx_local}). We note that
for each $\mathring{g} \in \cN_a$, the construction of this polynomial
$\mathring{p}$ depends only on $\iota,C_0,a,t$ and is again independent of $n$.

Letting $g=(\mathring{g}_i)_{i=1}^n$ be the construction of step (a), we set
$\mathring{p}_i$ to be this approximation of $\mathring{g}_i$
that satisfies \eqref{eq:poly_approx_local}, and
include $p=(\mathring{p}_i)_{i=1}^n$ in $\cP_t$.
\end{enumerate}
The components of $p:\RR^{n \times t} \to \RR^n$
constructed in this way are independent of $n$, and hence the maximum degree of
$p$ and maximum magnitude of its coefficients are also independent of $n$.
By definition, $p$ satisfies the same locality condition as $f$.
Thus $\cP$ is a set of polynomial local functions in the
sense of Definition \ref{def:polylocal}, which is BCP-representable by
Proposition \ref{prop:local}(a).

To verify condition (1) of Definition \ref{def:BCP_approx},
it remains to bound the error of the approximation of $f$ by $p$.
Let $\bSigma$ satisfy
$\|\bSigma\|_\op<C_0$, and let $\bZ \sim \cN(0,\bSigma \otimes \Id) \in \RR^{n
\times t}$. Denoting
$\bZ[A_i] \in \RR^{|A_i| \times t}$ as the rows of $\bZ$ belonging to $A_i$,
we have
\begin{align}
\frac{1}{n} \EE\big[\|f(\bZ)-p(\bZ)\|_2^2\big]
    &= \frac{1}{n} \sum_{i=1}^n
    \EE\left[\left|\mathring{f}_i(\bZ[A_i])-\mathring{p}_i(\bZ[A_i])\right|^2\right]\notag\\
    &\leq \frac{2}{n} \sum_{i=1}^n
    \EE\left[|\mathring f_i(\bZ[A_i])-\mathring g_i(\bZ[A_i])|^2\right]
    + \frac{2}{n} \sum_{i=1}^n
    \EE\left[|\mathring g_i(\bZ[A_i])-\mathring p_i(\bZ[A_i])|^2\right]\notag\\
    &\leq \frac{2}{n} \sum_{i=1}^n
    \EE\left[|\mathring f_i(\bZ[A_i])-\mathring g_i(\bZ[A_i])|^2\right]
    + 2\iota\label{eq:poly_approx_local_1}
\end{align}
where the last inequality follows from the approximation guarantee
\eqref{eq:poly_approx_local} for each $\mathring p_i$.
For the first term of \eqref{eq:poly_approx_local_1},
we split the expectation into two parts based on whether 
$\|\bZ[A_i]\|_\Fro^2\leq 1/\zeta^2$ or not, and then 
apply the guarantee in \eqref{eq:lipschitz_net_approx_local} 
and definition of the class $\mathring \cF_a$ in (\ref{eq:finitelipschitzclass})
to get
\begin{align*}
    \frac{2}{n} \sum_{i=1}^n
    \EE\left[|\mathring f_i(\bZ[A_i])-\mathring g_i(\bZ[A_i])|^2\right]
    &\leq \frac{2}{n} \sum_{i=1}^n
    \EE\left[|\mathring f_i(\bZ[A_i])-\mathring g_i(\bZ[A_i])|^2 \cdot 
    \1\{\|\bZ[A_i]\|_\Fro^2 > (1/\zeta)^2\}\right] + 2\zeta\\
    &\leq \frac{C}{n} \sum_{i=1}^n
    \EE\left[(1 + \|\bZ[A_i]\|_\Fro^2) \cdot 
    \1\{\|\bZ[A_i]\|_\Fro^2 > (1/\zeta)^2\}\right] + 2\zeta.
\end{align*}
Further applying Cauchy-Schwarz and Markov's inequality to bound the first term, we obtain
\[\frac{2}{n} \sum_{i=1}^n
    \EE\left[|\mathring f_i(\bZ[A_i])-\mathring g_i(\bZ[A_i])|^2\right]
    \leq C'\zeta.\]
Choosing $\zeta,\iota>0$ small enough depending on $\epsilon$,
the resulting bound of \eqref{eq:poly_approx_local_1} is at most $\epsilon$.
Thus $\cP$ satisfies condition (1) of Definition \ref{def:BCP_approx}.

We next verify condition (2) of Definition \ref{def:BCP_approx}. Let
$\cQ=\bigsqcup_{t=0}^T \cQ_t$ be the set
of all polynomial functions $q=(\mathring q_i)_{i=1}^n$ with
coefficients bounded in magnitude by 1 and satisfying the locality condition (1)
of Definition \ref{def:polylocal}. 
For any $q_1,q_2 \in \cQ$ with uniformly bounded degrees, note that
$\cP \cup \{q_1,q_2\}$ is a set of polynomial local functions satisfying
Definition \ref{def:polylocal}, and hence remains BCP-representable.
Consider any $\bSigma \in \RR^{t \times t}$
with $\|\bSigma\|_\op<C_0$ and any random $\bz\in\RR^{n\times t}$ satisfying,
for any $q_1,q_2 \in \cQ_t$ of uniformly bounded degrees, almost surely
\begin{equation}\label{eq:q1q2}
    \lim_{n\to\infty} \frac{1}{n} q_1(\bz)^\top q_2(\bz)
    - \frac{1}{n} \EE_{\bZ \sim \cN(0,\bSigma \otimes \Id_n)}
[q_1(\bZ)^\top q_2(\bZ)] = 0.
\end{equation}
To control $\|f(\bz)-p(\bz)\|_2^2$, we have
\begin{align}
    \frac{1}{n}\|f(\bz)-p(\bz)\|_2^2 &=
    \frac{1}{n} \sum_{i=1}^n \big(\mathring f_i(\bz[A_i]) - \mathring
p_i(\bz[A_i])\big)^2\notag\\
&\leq \frac{2}{n} \sum_{i=1}^n \big(\mathring f_i(\bz[A_i]) - \mathring g_i(\bz[A_i])\big)^2
+\frac{2}{n} \sum_{i=1}^n \big(\mathring g_i(\bz[A_i]) - \mathring
p_i(\bz[A_i])\big)^2.
    \label{eq:poly_approx_local_lipschitz_bound_3}
\end{align}
For the first term, applying a similar argument as above,
\begin{align*}
    \frac{2}{n} \sum_{i=1}^n \big(\mathring f_i(\bz[A_i]) - \mathring
g_i(\bz[A_i])\big)^2
    &\leq 2\zeta+\frac{2}{n}\sum_{i=1}^n
    \Big(\mathring f_i(\bz[A_i])-\mathring g_i(\bz[A_i])\Big)^2
\1\{\|\bz[A_i]\|_\Fro^2>(1/\zeta)^2\}\\
    &\leq 2\zeta + C\bigg(\frac{1}{n}\sum_{i=1}^n
    (1+\|\bz[A_i]\|_\Fro^2)^2\bigg)^{1/2}
    \bigg(\frac{1}{n}\sum_{i=1}^n
\1\{\|\bz[A_i]\|_\Fro^2>(1/\zeta)^2\}\bigg)^{1/2}\\
    &\leq 2\zeta + C\zeta \bigg(\frac{1}{n}\sum_{i=1}^n
    (1+\|\bz[A_i]\|_\Fro^2)^2\bigg)^{1/2}
    \bigg(\frac{1}{n}\sum_{i=1}^n \|\bz[A_i]\|_\Fro^2\bigg)^{1/2}.
\end{align*}
Applying \eqref{eq:q1q2} with 
$q_1(\bz)=q_2(\bz)=(1+\|\bz[A_i]\|_\Fro^2)_{i=1}^n$, there exists a 
constant $C>0$ such that
\begin{equation}\label{eq:zAi4bound}
\frac{1}{n}\sum_{i=1}^n (1+\|\bz[A_i]\|_\Fro^2)^2<C \text{ a.s.\ for all large }
n.
\end{equation}
Then
\begin{equation}\label{eq:poly_approx_bound_1}
\frac{2}{n} \sum_{i=1}^n \big(\mathring f_i(\bz[A_i]) - \mathring
g_i(\bz[A_i])\big)^2 \leq C'\zeta \text{ a.s.\ for all large } n.
\end{equation}
For the second term of (\ref{eq:poly_approx_local_lipschitz_bound_3}),
define for each $a=0,1,\ldots,A$ and each $\mathring g\in\cN_a$ the index set
    \[\cI_{a,\mathring g} = \{i\in[n]: |A_i|=a \text{ and } \mathring
g_i=\mathring g\}.\]
Clearly $[n]=\sqcup_{a=0}^A \sqcup_{\mathring g\in\cN_a}\cI_{a,\mathring g}$.
Note that there is a common polynomial approximation $\mathring p_i\equiv
p_{\mathring g}$ for all $i\in \cI_{a,\mathring g}$, so
\begin{align*}
    \frac{2}{n}\sum_{i=1}^n \Big(\mathring g_i(\bz[A_i])-\mathring p_i(\bz[A_i])\Big)^2
    &= \sum_{a=0}^A \sum_{\mathring g\in\cN_a}
    \underbrace{\frac{2}{n} \sum_{i\in \cI_{a,\mathring g}}
    \Big(\mathring g(\bz[A_i])-p_{\mathring g}(\bz[A_i])\Big)^2}_{E_{a,\mathring g}}.
\end{align*}
For each $a=0,1,\ldots,A$ and $\mathring g \in \cN_a$, we claim that
\begin{equation}\label{eq:Wassersteinclaim}
\limsup_{n \to \infty} E_{a,\mathring g} \leq 2\iota \text{ a.s.}
\end{equation}
Assuming momentarily this claim, we may apply it
to each pair $(a,\mathring g)$ above to show
\begin{equation}\label{eq:poly_approx_bound_2}
\frac{2}{n}\sum_{i=1}^n \Big(\mathring g_i(\bz[A_i])-\mathring
p_i(\bz[A_i])\Big)^2<C(\zeta)\iota \text{ a.s.\ for all large } n
\end{equation}
for some constant $C(\zeta)>0$ that depends on $\zeta$
via the cardinalities $|\cN_a|$ for $a=0,1,\ldots,A$.
Applying \eqref{eq:poly_approx_bound_1} 
and \eqref{eq:poly_approx_bound_2} to
\eqref{eq:poly_approx_local_lipschitz_bound_3}, and first 
choosing $\zeta>0$ sufficiently small followed by $\iota>0$ sufficiently small
depending on $\zeta$, this is also at most $\epsilon$,
verifying condition (2) of Definition \ref{def:BCP_approx}.

To complete the proof, it remains to show the claim (\ref{eq:Wassersteinclaim}).
Suppose by contradiction that there exists
a positive probability event $\Omega$ (in the infinite sequence space as $n \to
\infty$) on which $\limsup_{n \to \infty} E_{a,\mathring g}>2\iota$. Let $D$ be
the maximum degree of polynomials in $\cP$, and let us
consider an event where
\begin{equation}\label{eq:goodevent}
\frac{1}{n}\sum_{i=1}^n (1+\|\bz[A_i]\|_\Fro^{2D})^2<C \text{ for all
large } n.
\end{equation}
This event holds with probability 1 analogously to (\ref{eq:zAi4bound}),
by applying (\ref{eq:q1q2}) with
$q_1(\bz)=q_2(\bz)=(1+\|\bz[A_i]\|_\Fro^{2D})_{i=1}^n$.
Let us consider also the class of test functions
$q(\cdot)=q_1(\cdot)^\top q_2(\cdot)$ where $q_1,q_2 \in \cQ$ are given by
\begin{equation}\label{eq:q12def}
    q_1(\bz)[i]=\begin{cases}
        1 & \text{if $i\in\cI_{a,\mathring g}$}\\
        0 & \text{otherwise},
    \end{cases}
\qquad q_2(\bz)[i]=\begin{cases}
        \mathring q & \text{if $i\in\cI_{a,\mathring g}$}\\
        0 & \text{otherwise},
    \end{cases}
\end{equation}
and $\mathring q:\RR^{a \times t} \to \RR$ is a fixed monomial (of arbitrary
multivariate degree) with
coefficient 1. Then the event where \eqref{eq:q1q2} holds
for $q(\cdot)=q_1(\cdot)^\top q_2(\cdot)$ defined by each such monomial
$\mathring q:\RR^{a \times t} \to \RR$ also has probability 1, as the set of such
monomials $\mathring q$ is countable. Letting $\Omega'$ be the intersection of
$\Omega$ with these two probability-1 events, $\Omega'$ must be non-empty.

For any $\omega \in \Omega'$, let $\{n_j\}_{j=1}^\infty$ be a (random,
$\omega$-dependent) subsequence for which $E_{a,\mathring g}>2\iota$ for each
$n_j$. Since $|\cI_{a,\mathring g}|/n \in [0,1]$ and since $\bSigma$ belongs to
a fixed compact domain, passing to a further subsequence, we may assume that
along this subsequence $\{n_j\}_{j=1}^\infty$, we have
$|\cI_{a,\mathring g}|/n_j \to \alpha$ for some $\alpha \in [0,1]$ 
and $\bSigma \to \bar \bSigma$ for some $\bar\bSigma \in \RR^{t \times t}$ as
$n_j \to \infty$. If $\alpha=0$, then
using the condition (\ref{eq:finitelipschitzclass}) for
$\mathring g$ and the fact that $p_{\mathring g}$ has degree at most $D$
and coefficients of magnitude at most $B$ for some constants $D,B>0$,
for a constant $C(L,B,D)>0$ we have
\begin{align*}
E_{a,\mathring g} &\leq \frac{C(L,B,D)}{n} \sum_{i\in\cI_{a,\mathring g}}
    (1+\|\bz[A_i]\|_\Fro^{2D})\\
    &\leq C(L,B,D) \bigg(\frac{1}{n}\sum_{i=1}^n 
    (1+\|\bz[A_i]\|_\Fro^{2D})^2\bigg)^{1/2}
    \bigg(\frac{|\cI_{a,\mathring g}|}{n}\bigg)^{1/2}.
\end{align*}
Applying $\alpha=0$ and the bound (\ref{eq:goodevent}),
we have $E_{a,\mathring g} \to 0$ along the subsequence
$\{n_j\}_{j=1}^\infty$, contradicting $E_{a,\mathring g}>2\iota$ for each $n_j$.
If instead $\alpha>0$, then the statement \eqref{eq:q1q2} for each 
function $q_1(\cdot)^\top q_2(\cdot)$ in the class
(\ref{eq:q12def}), together with the convergence $\bSigma \to \bar\bSigma$,
imply
\begin{align*}
\alpha \cdot \lim_{n_j \to \infty} \frac{1}{|\cI_{a,\mathring g}|}
\sum_{i \in \cI_{a,\mathring g}} \mathring q(\bz[A_i])
&=\lim_{n_j \to \infty} \frac{1}{n_j}\,q_1(\bz)^\top q_2(\bz)\\
&=\lim_{n_j \to \infty} \frac{1}{n_j}\EE[q_1(\bZ)^\top q_2(\bZ)]
=\alpha \cdot \EE_{\bar\bZ \sim \cN(0,\bar\bSigma \otimes \Id_a)}
[\mathring q(\bar\bZ)].
\end{align*}
This holds for each fixed monomial $\mathring q:\RR^{a \times t}\to\RR$,
so the empirical distribution of
$\{\bz[A_i]\}_{i\in\cI_{a,\mathring g}}$ converges to
$\cN(0,\bar\bSigma \otimes \Id_a)$ weakly and in Wasserstein-$k$ for every
order $k \in [1,\infty)$ (c.f.\ \cite[Theorem 30.2]{billingsley2012probability}
and \cite[Definition 6.8, Theorem 6.9]{villani2008optimal}).
Since $\mathring g-p_{\mathring g}$ is a fixed continuous function of polynomial
growth, this then implies
\begin{align*}
\lim_{n_j \to \infty}
E_{a,\mathring g}&=2\alpha \cdot \lim_{n_j \to \infty}
\frac{1}{|\cI_{a,\mathring g}|}
    \sum_{i\in\cI_{a,\mathring g}} \Big(\mathring g(\bz[A_i])-p_{\mathring
g}(\bz[A_i])\Big)^2
=2\alpha \cdot \EE_{\bar\bZ \sim \cN(0,\bar\bSigma \otimes
\Id_a)} \Big[\Big(\mathring g(\bar\bZ)-p_{\mathring g}(\bar\bZ)\Big)^2\Big],
\end{align*}
which is strictly less than $2\iota$ by the bound $\alpha \leq 1$ and
the approximation guarantee
\eqref{eq:poly_approx_local} for $p_{\mathring g}$.
This again contradicts $E_{a,\mathring g}>2\iota$ for each $n_j$.
Thus (\ref{eq:Wassersteinclaim}) holds, concluding the proof.
\end{proof}

\subsection{Anisotropic functions}\label{sec:BCPAniso}

We recall the classes of polynomial and Lipschitz anisotropic functions
from Definitions \ref{def:polyaniso} and \ref{def:aniso}.

\begin{figure}[t]
    \centering
    \begin{tikzpicture}[scale = .9]
        \begin{scope}[every node/.style={circle,thick,draw}]
            \node (A) at (0,0) {$\bD^{(\sigma)}$};
            \node (B) at (0,2) {$\bK$};
            \node (C) at (2,2) {$\bK$};
            \node (D) at (2,0) {$\bK$};
            \node (E) at (-1.5,-1.5) {$\bK'$};
            \node (P) at (-3, 0) {$\bT^{(\sigma)}$};
        \end{scope}
        
        \begin{scope}[>={Stealth[black]},
                      every node/.style={fill=white,circle},
                      every edge/.style={draw=black,very thick},
                      line width = 1pt]
            \draw [-] (A) to (B);
            \draw [-] (A) to (C);
            \draw [-] (A) to (D);
            \draw [-] (A) to (E);
            \draw [-] (C) -- (3,3);
            \draw [-] (B) -- (0,3);
            \draw [-] (D) -- (3,0);
            \draw [-] (E) -- (-2,-2);
            \draw [-] (P) -- (-4,-1);
            \draw [-] (P) -- (-3, 1);
            \draw [-] (P) -- (-2, 0);
            \draw [-] (P) -- (-2, 1);
        \end{scope}

        \node at (-1.25, 0) {$=$};
    \end{tikzpicture}
    \caption{An example of a tensor $\bT \in \cT$ for the class of polynomial
anisotropic functions.}\label{fig:genAniso}
\end{figure}

\begin{proof}[Proof of Proposition \ref{prop:anisotropic}(a)]
Let $\cP=\bigsqcup_{t=0}^T \cP_t$, where $\cP_t$ consists of the functions
$p:\RR^{n \times t} \to \RR^n$. Consider any $p \in \cP_t$ given by
\[p(\bz_{1:t})=\bK' q(\bK^\top \bz_{1:t}),\]
where $q:\RR^{n \times t} \to \RR^n$ is separable with degree at most $D$ and
all entries bounded in magnitude by $B$. Then $q$ admits a representation
of the form (\ref{eq:tensorpolyrepr}),
\[q(\bz_{1:t})=\bD^{(0)}+\sum_{d=1}^D \sum_{\sigma \in \cS_{t,d}}
\bD^{(\sigma)}[\bz_{\sigma(1)},\ldots,\bz_{\sigma(d)},\cdot]\]
where each tensor $\bD^{(0)},\bD^{(\sigma)}$ has entries bounded in magnitude by
$B$ and is diagonal because $q$ is separable. So $p$ admits the representation
(\ref{eq:tensorpolyrepr}), where
\begin{equation}\label{eq:anisoT0}
\bT^{(0)}=\bK'\bD^{(0)}
\end{equation}
and $\bT^{(\sigma)}$ for each $\sigma \in \cS_{t,d}$
is a contraction of $\bD^{(\sigma)}$ with $\bK',\bK$ in each dimension,
having entries
\begin{equation}\label{eq:anisoT}
\bT^{(\sigma)}[i_1,\ldots,i_{d+1}]
=\sum_{j=1}^n \bD^{(\sigma)}[j,\ldots,j]
\bK[i_1,j] \ldots \bK[i_d,j]\bK'[i_{d+1},j].
\end{equation}
This is visualized in Figure \ref{fig:genAniso}. We let $\cT$ be the set of all
such tensors $\bT^{(0)},\bT^{(\sigma)}$ arising in this representation for all
$p \in \cP$. Then the cardinality $|\cT|$ is bounded independently of
$n$, by the boundedness of $|\cP|$ and of the degree of each $p \in \cP$.

By Definition \ref{def:BCP}, $\cT$ satisfies the BCP if
$\sup_\cL |\val_G(\cL)| \leq Cn$ for any connected bipartite multigraph
$G=(\cV_\Id \sqcup \cV_T,\cE)$ such that each vertex of $\cV_\Id$ has even degree,
where the supremum is over all $(\Id,\cT)$-labelings $\cL$ of $G$.
In light of the forms (\ref{eq:anisoT0}) and (\ref{eq:anisoT}), we see that any 
such value $\val_G(\cL)$ has a form
\begin{equation}\label{eq:diagval}
\val_G(\cL)=\sum_{\bi,\bj \in [n]^\cE}
\prod_{v \in \cV_T} \bD_v[j_e:e \sim v]
\prod_{u \in \cV_\Id} \Id^{\deg(u)}[i_e:e \sim u]
\prod_{e \in \cE} \bK_e[i_e,j_e]
\end{equation}
where $\bD_v$ is one of the above diagonal tensors $\bD^{(0)},\bD^{(\sigma)}$
for each $v \in \cV_T$, and $\bK_e$ is a matrix in $\cK$ for each $e \in \cE$.
Under condition (1) where $\|\bK\|_{\ell_\infty \to
\ell_\infty}=\max_i \sum_j |\bK[i,j]|$ and $\|\bK^\top\|_{\ell_\infty \to
\ell_\infty}$ are uniformly bounded by a constant over $\bK \in \cK$,
all tensors in (\ref{eq:diagval}) satisfy the property (\ref{eq:summabletensor}).
Then $\sup_{\cL} |\val_G(\cL)| \leq Cn$ by Lemma \ref{lem:summabletensor},
implying that $\cT$ satisfies the BCP.

Under condition (2), let $\bar\cT$ be an independent copy of $\cT$ where the
orthogonal matrices $\bO,\bU$ defining $\cK$ are replaced by independent copies
$\bar\bO,\bar\bU$. Given any $(\Id,\cT)$-labeling $\cL$ of
$G$, denote by $\bar\cL$ the labeling that replaces each label $\bT \in \cT$ by
its corresponding copy $\bar\bT \in \bar\cT$, and write $\EE$ for the expectation
over $\bO,\bar\bO,\bU,\bar\bU$. We claim that for any fixed connected
multigraph $G$ where all vertices of $\cV_\Id$ have even degree,
\begin{align}
\sup_\cL |\EE[\val_G(\cL)]| &\leq Cn\label{eq:firstmomentO}\\
\sup_\cL \EE\big[\big(\val_G(\cL)-\val_G(\bar\cL)\big)^4\big] &\leq Cn^2
\label{eq:fourthmomentO}
\end{align}
for a $G$-dependent constant $C>0$, where the suprema are over all
$(\Id,\cT)$-labelings $\cL$.
Assuming momentarily this claim, we then have by Markov's inequality and
Jensen's inequality that
\begin{align*}
\PP[|\val_G(\cL)| \geq (C+1)n]
&\leq \PP[|\val_G(\cL)-\EE \val_G(\cL)| \geq n]\\
&\leq \frac{\EE[(\val_G(\cL)-\EE \val_G(\cL))^4]}{n^4}
\leq \frac{\EE[(\val_G(\cL)-\val_G(\bar\cL))^4]}{n^4}
\leq \frac{C}{n^2}.
\end{align*}
The set of $(\Id,\cT)$-labelings of $G$ has cardinality bounded by a constant
independent of $n$, by the boundedness of $|\cT|$. Then taking a union bound,
$\PP[\sup_\cL |\val_G(\cL)|>C'n] \leq C'/n^2$ for a constant $C'>0$.
So by the Borel-Cantelli lemma, $\sup_\cL |\val_G(\cL)|<C'n$
almost surely for all large $n$, implying that the BCP holds almost surely
for $\cT$.

To conclude the proof, it remains to show
(\ref{eq:firstmomentO}--\ref{eq:fourthmomentO}). Since $\bO,\bar\bO,\bU,\bar\bU$
are assumed independent, whose densities with respect to Haar measure are
bounded above by a constant, by a change of measure it suffices
to show (\ref{eq:firstmomentO}--\ref{eq:fourthmomentO}) in the case where
$\bO,\bar\bO,\bU,\bar\bU$ are independent Haar-orthogonal matrices.
We provide an argument that extends the ideas of \cite[Appendix C]{Wang2024}
using the orthogonal Weingarten calculus: Fix any set $\cE$ of even cardinality,
and let $\bi,\bj \in [n]^\cE$ be any two index tuples. Let
$\bO \in \RR^{n \times n}$ be a Haar-distributed orthogonal matrix.
Then (c.f.\ \cite[Corollary 3.4]{collins2006integration})
\begin{equation}\label{eq:weingarten}
\EE \prod_{e \in \cE} \bO[i_e,j_e]
=\mathop{\sum_{\text{pairings } \pi,\pi' \in \sP}}_{\pi(\bi) \geq
\pi,\,\pi(\bj) \geq \pi'} \Wg_{n,\cE}(\pi,\pi')
\end{equation}
Here
\begin{itemize}
\item $\sP$ is the lattice of partitions of $\cE$ 
endowed with the partial ordering $\pi \geq \tau$ if $\tau$ refines $\pi$
(i.e.\ each block of $\pi$ is the union of one or more blocks of $\tau$).
\item $\pi,\pi'$ are pairings in $\sP$, i.e.\ partitions of $\cE$
where each block has size 2.
\item $\pi(\bi) \in \sP$ is the partition where $e,e'$ belong to
the same block of $\pi(\bi)$ if and only if $i_e=i_{e'}$.
Thus $\pi(\bi) \geq \pi$ for a pairing $\pi$
means $i_e=i_{e'}$ for each pair $(e,e') \in \pi$.
\item $\Wg_{n,\cE}(\pi,\pi')$ is the orthogonal Weingarten function, admitting
an asymptotic expansion
\begin{equation}\label{eq:weinasymp}
\Wg_{n,\cE}(\pi,\pi')=n^{-|\cE|/2-d(\pi,\pi')/2} \Big(\Wg_\cE^{(0)}(\pi,\pi')
-n^{-1}\Wg_\cE^{(1)}(\pi,\pi')+O(n^{-2})\Big)
\end{equation}
where $\Wg_\cE^{(0)}(\pi,\pi')$ and $\Wg_\cE^{(1)}(\pi,\pi')$ do not depend on
$n$, and $O(n^{-2})$ denotes an error at most $Cn^{-2}$ for a constant
$C \equiv C(|\cE|,\pi,\pi')>0$
and all large $n$. Here $d(\pi,\pi')$ is a metric on $\sP$ given by
\begin{equation}\label{eq:metricdef}
d(\pi,\pi')=|\pi|+|\pi'|-2|\pi \vee \pi'|
\end{equation}
where $|\pi|$ is the number of blocks of $\pi$, and $\pi
\vee \pi'$ is the join (i.e.\ least upper bound in $\mathscr{P}$).
For the equivalence between this and the $\ell(\cdot,\cdot)$ metric
of \cite{collins2006integration}, see \cite[Appendix C]{Wang2024}.
\item $\Wg_{n,\cE}(\pi,\pi'),\Wg_\cE^{(0)}(\pi,\pi'),\Wg_\cE^{(1)}(\pi,\pi')$
depend on $(\pi,\pi')$ only via the sizes of the blocks of $\pi \vee \pi'$.
Writing these sizes as $2k_1,2k_2,\ldots,2k_M$ (which must all be even),
\begin{align}
\Wg_\cE^{(0)}&=\prod_{m=1}^M (-1)^{k_m-1}c_{k_m-1},\label{eq:wein0}\\
\Wg_\cE^{(1)}&=\sum_{m=1}^M (-1)^{k_m-1}a_{k_m-1}
\mathop{\prod_{m'=1}^M}_{m' \neq m} (-1)^{k_{m'}-1}c_{k_{m'}-1},\label{eq:wein1}
\end{align}
where $c_k$ is the $k^\text{th}$ Catalan number, $a_k$ is the total area
under the set of all Dyck paths of length $k$, and we note that $\prod_{m=1}^M
(-1)^{k_m-1}=(-1)^{|\cE|/2-M}=(-1)^{d(\pi,\pi')/2}$. This form of $\Wg_\cE^{(0)}$ is
shown in \cite[Theorem 3.13]{collins2006integration}, of $\Wg_\cE^{(1)}$
in \cite[Theorem 3.13]{feray2012complete}, and we refer to
\cite[Theorem 4.6, Lemmas 4.12 and 4.13]{collins2017weingarten} for a summary.
\end{itemize}

To show (\ref{eq:firstmomentO}), further expanding $\bK_e=\bO\bD_e\bU^\top$,
we may express (\ref{eq:diagval}) as
\begin{equation}\label{eq:diagval1}
\val_G(\cL)=\sum_{\bi,\bj,\bk \in [n]^\cE}
\prod_{v \in \cV_T} \bD_v[j_e:e \sim v]
\prod_{u \in \cV_\Id} \Id^{\deg(u)}[i_e:e \sim u]
\prod_{e \in \cE} \bO[i_e,k_e]\bD_e[k_e,k_e]\bU[j_e,k_e].
\end{equation}
Let $\cE$ be the set of edges of $G$, which has even cardinality because each
vertex of $\cV_\Id$ has even degree.
Let $\sP$ be the lattice of partitions of $\cE$.
Let $\pi_T,\pi_\Id \in \mathscr{P}$ be the
two distinguished partitions where $e,e' \in \cE$ belong to the same block
of $\pi_T$ (or of $\pi_\Id$) if they are incident to the same vertex of $\cV_T$
(resp.\ of $\cV_\Id$); thus $|\pi_T|=|\cV_T|$ and $|\pi_\Id|=|\cV_\Id|$.
For each vertex $v \in \cV_T$, we write $e(v) \in \cE$
for an arbitrary choice of edge incident to this vertex.
Then, since $\bD_v$ and $\Id^{\deg(u)}$ are diagonal,
(\ref{eq:diagval1}) is further equivalent to
\begin{equation}\label{eq:diagvalfinal}
\val_G(\cL)=\mathop{\sum_{\bi,\bj,\bk \in [n]^\cE}}_{\pi(\bi) \geq \pi_\Id,
\,\pi(\bj) \geq \pi_T}
\prod_{v \in \cV_T} \bD_v[j_{e(v)},\ldots,j_{e(v)}]
\times
\prod_{u \in \cV_\Id} 1 \times \prod_{e \in \cE} \bO[i_e,k_e]\bD_e[k_e,k_e]\bU[j_e,k_e].
\end{equation}
Evaluating the expectations over $\bO$ and $\bU$ using (\ref{eq:weingarten}), 
noting that $\pi(\bj) \geq \pi_T$ and
$\pi(\bj) \geq \pi$ if and only if $\pi(\bj) \geq \pi_T \vee \pi$,
and similarly for $\pi(\bi)$ and $\pi(\bk)$, we have
\begin{align*}
\EE[\val_G(\cL)]&=\sum_{\text{pairings } \pi,\pi',\tau,\tau' \in \sP}
\Wg_{n,\cE}(\pi,\pi') \Wg_{n,\cE}(\tau,\tau')\\
&\hspace{1in}\mathop{\sum_{\bj \in [n]^\cE}}_{\pi(\bj) \geq \pi_T \vee \pi}
\prod_{v \in \cV_T} \bD_v[j_{e(v)},\ldots,j_{e(v)}]
\times
\mathop{\sum_{\bi \in [n]^\cE}}_{\pi(\bi) \geq \pi_\Id \vee \tau} 1
\times \mathop{\sum_{\bk \in [n]^\cE}}_{\pi(\bk) \geq \pi' \vee \tau'}
\prod_{e \in \cE} \bD_e[k_e,k_e].
\end{align*}
To show (\ref{eq:firstmomentO}), we will only use the bound
$|\Wg_{n,\cE}(\pi,\pi')| \leq
O(n^{-|\cE|/2-d(\pi,\pi')/2})$ implied by (\ref{eq:weinasymp}). Then,
identifying $\sum_{\bj \in [n]^{\cE}:\pi(\bj) \geq \pi_T \vee \pi}$ as a
summation over a single index $j \in [n]$ for each block of $\pi_T \vee \pi$,
and similarly for $\bi$ and $\bk$, and
applying the uniform boundedness of entries of $\bD_v$ and $\bD_e$, we have
for a constant $C>0$ and all large $n$,
\[|\EE[\val_G(\cL)]| \leq C\sum_{\text{pairings } \pi,\pi',\tau,\tau' \in \sP}
n^{-|\cE|/2-d(\pi,\pi')/2}\, n^{-|\cE|/2-d(\tau,\tau')/2}\,
n^{|\pi_T \vee \pi|}\,n^{|\pi_\Id \vee \tau|}\,n^{|\pi' \vee \tau'|}.\]
Recalling $|\pi_T|=|\cV_T|$, $|\pi_\Id|=|\cV_\Id|$,
and $|\pi|=|\pi'|=|\tau|=|\tau'|=|\cE|/2$ since these are pairings, we have
by definition (\ref{eq:metricdef}) of the metric $d(\cdot,\cdot)$ that
\begin{equation}\label{eq:piVEsize}
|\pi_T \vee \pi|=\frac{|\cV_T|+|\cE|/2-d(\pi_T,\pi)}{2},
\quad |\pi_\Id \vee \tau|=\frac{|\cV_\Id|+|\cE|/2-d(\pi_\Id,\tau)}{2},
\quad |\pi' \vee \tau'|=\frac{|\cE|-d(\pi',\tau')}{2}.
\end{equation}
We have also $|\pi_T \vee \pi_\Id|=1$ because $G$ is a connected graph, so
by the triangle inequality for $d(\cdot,\cdot)$,
\[d(\pi_T,\pi)+d(\pi,\pi')+d(\pi',\tau')+d(\tau',\tau)
+d(\tau,\pi_\Id) \geq d(\pi_T,\pi_\Id)=|\cV_T|+|\cV_\Id|-2.\]
Applying this above gives
$|\EE[\val_G(\cL)]| \leq
C'n^{-\frac{|\cE|}{2}-\frac{|\cE|}{2}+\frac{|\cV_T|+|\cE|/2}{2}+\frac{|\cV_\Id|+|\cE|/2}{2}+\frac{|\cE|}{2}-\frac{|\cV_T|+|\cV_\Id|-2}{2}}=C'n$
for a constant $C'>0$. This shows (\ref{eq:firstmomentO}).

To show (\ref{eq:fourthmomentO}), let $G^{(s)}=(\cV_\Id^{(s)} \sqcup \cV_T^{(s)},\cE^{(s)})$ for
$s=1,2,3,4$ denote four copies of $G$. Let $G^{\sqcup 4}=(\cV_\Id^{\sqcup 4}
\sqcup \cV_T^{\sqcup 4},\cE^{\sqcup 4})$ denote the (disconnected)
graph formed by their disjoint union. We write $\sP$ for the lattice of
partitions of the combined edge set $\cE^{\sqcup 4}$. Let
$\pi_T,\pi_\Id \in \sP$ be the partitions where $e,e' \in \cE^{\sqcup 4}$
belong to the same block of $\pi_T$ (or of $\pi_\Id$) if $e,e' \in \cE^{(s)}$
for the same copy $s \in \{1,2,3,4\}$ and are incident to the same vertex of
$\cV_T^{(s)}$ (resp.\ of $\cV_\Id^{(s)}$).
Thus $\pi_T \vee \pi_\Id$ has 4 blocks which are exactly
$\cE^{(s)}$ for $s=1,2,3,4$. Letting $e(v) \in \cE^{\sqcup 4}$ be an
arbitrary choice of edge containing each vertex $v \in \cV_T^{\sqcup 4}$,
and applying (\ref{eq:diagvalfinal}),
\begin{align*}
&(\val_G(\cL)-\val_G(\bar\cL))^4\\
&=\sum_{S \subseteq \{1,2,3,4\}} (-1)^{|S|}
\prod_{s \in S} \val_G(\cL) \prod_{s \notin S} \val_G(\bar\cL)\\
&=\sum_{S \subseteq \{1,2,3,4\}} (-1)^{|S|}
\mathop{\sum_{\bi,\bj,\bk \in [n]^{\cE^{\sqcup 4}}}}_{\pi(\bi) \geq \pi_\Id,
\,\pi(\bj) \geq \pi_T}
\bigg(\prod_{s \in S} \prod_{e \in \cE^{(s)}} \bO[i_e,k_e] 
\bU[j_e,k_e] \prod_{s \notin S}
\prod_{e \in \cE^{(s)}} \bar \bO[i_e,k_e] \bar \bU[j_e,k_e] \bigg)\\
&\hspace{1.5in}
\prod_{v \in \cV_T^{\sqcup 4}} \bD_v[j_{e(v)},\ldots,j_{e(v)}]
\times \prod_{u \in \cV_\Id^{\sqcup 4}} 1
\times \prod_{e \in \cE^{\sqcup 4}} \bD_e[k_e,k_e].
\end{align*}
We apply (\ref{eq:weingarten}) to take expectations over
$\bO,\bU$ and $\bar\bO,\bar\bU$ separately.
Let $\pi_S \in \sP$ be the partition with the two blocks
\[\cE_S \equiv \bigcup_{s \in S} \cE^{(s)},
\qquad \cE_{\bar S} \equiv \bigcup_{s \notin S} \cE^{(s)}\]
(or with a single block if either $\cE_S$ or $\cE_{\bar S}$ is empty).
The application of (\ref{eq:weingarten}) to $\bO,\bU$ enumerates
over four pairings of $\cE_S$, and the application of
(\ref{eq:weingarten}) to $\bar\bO,\bar\bU$ enumerates over four pairings of
$\cE_{\bar S}$, which we may combine into four
pairings $\pi,\pi',\tau,\tau'$ of $\cE^{\sqcup 4}$ that refine $\pi_S$.
For any such pairings $\pi,\pi'$,
we write $\Wg_{n,\cE_S}(\pi,\pi')$ for the Weingarten function of
the restrictions of $\pi,\pi'$ to $\cE_S$, as partitions of $\cE_S$. Then
\begin{align*}
&\EE(\val_G(\cL)-\val_G(\bar\cL))^4\\
&=\sum_{S \subseteq \{1,2,3,4\}} (-1)^{|S|}
\mathop{\sum_{\bi,\bj,\bk \in [n]^{\cE^{\sqcup 4}}}}_{\pi(\bi) \geq \pi_\Id,\,
\pi(\bj) \geq \pi_T}
\mathop{\sum_{\text{pairings } \pi,\pi',\tau,\tau' \in
\sP}}_{\pi,\pi',\tau,\tau' \leq \pi_S,\,\tau \leq \pi(\bi),\,\pi \leq
\pi(\bj),\,\tau',\pi' \leq \pi(\bk)}\\
&\hspace{1in} \Wg_{n,\cE_S}(\pi,\pi') \Wg_{n,\cE_S}(\tau,\tau')
\Wg_{n,\cE_{\bar S}}(\pi,\pi') \Wg_{n,\cE_{\bar S}}(\tau,\tau')\\
&\hspace{1in}
\times \prod_{v \in \cV_T^{\sqcup 4}} \bD_v[j_{e(v)},\ldots,j_{e(v)}]
\times \prod_{u \in \cV_\Id^{\sqcup 4}} 1
\times \prod_{e \in \cE^{\sqcup 4}} \bD_e[k_e,k_e]\\
&=\mathop{\sum_{\text{pairings } \pi,\pi',\tau,\tau'\in \sP}}
\Bigg(\mathop{\sum_{S \subseteq \{1,2,3,4\}}}_{\pi_S \geq \pi,\pi',\tau,\tau'}
(-1)^{|S|}
\Wg_{n,\cE_S}(\pi,\pi') \Wg_{n,\cE_S}(\tau,\tau')
\Wg_{n,\cE_{\bar S}}(\pi,\pi') \Wg_{n,\cE_{\bar S}}(\tau,\tau')\Bigg)\\
&\hspace{1in}
\mathop{\sum_{\bj \in [n]^{\cE^{\sqcup 4}}}}_{\pi(\bj) \geq \pi_T \vee \pi}
\prod_{v \in \cV_T^{\sqcup 4}} \bD_v[j_{e(v)},\ldots,j_{e(v)}]
\times
\mathop{\sum_{\bi \in [n]^{\cE^{\sqcup 4}}}}_{\pi(\bi) \geq \pi_\Id \vee \tau}
1 \times
\mathop{\sum_{\bk \in [n]^{\cE^{\sqcup 4}}}}_{\pi(\bk) \geq \pi' \vee \tau'}
\prod_{e \in \cE^{\sqcup 4}} \bD_e[k_e,k_e]\\
&\leq C\mathop{\sum_{\text{pairings } \pi,\pi',\tau,\tau' \in \sP}}
\Bigg|\underbrace{\mathop{\sum_{S \subseteq \{1,2,3,4\}}}_{\pi_S \geq
\pi,\pi',\tau,\tau'} (-1)^{|S|}
\Wg_{n,\cE_S}(\pi,\pi') \Wg_{n,\cE_S}(\tau,\tau')
\Wg_{n,\cE_{\bar S}}(\pi,\pi') \Wg_{n,\cE_{\bar
S}}(\tau,\tau')}_{:=W(\pi,\pi',\tau,\tau')}\Bigg|\\
&\hspace{1in}
\times n^{|\pi_T \vee \pi|}\, n^{|\pi_\Id \vee \tau|}\, n^{|\pi' \vee \tau'|}.
\end{align*}
Analogously to (\ref{eq:piVEsize}), we have
\[|\pi_T \vee \pi|=\frac{4|\cV_T|+2|\cE|-d(\pi_T,\pi)}{2},
\quad |\pi_\Id \vee \tau|=\frac{4|\cV_\Id|+2|\cE|-d(\pi_\Id,\tau)}{2},
\quad |\pi' \vee \tau'|=\frac{4|\cE|-d(\pi',\tau')}{2},\]
so the above gives
\begin{equation}\label{eq:fourthmomentgeneralbound}
\EE(\val_G(\cL)-\val_G(\bar\cL))^4
\leq C\sum_{\text{pairings } \pi,\pi',\tau,\tau' \in \sP}
|W(\pi,\pi',\tau,\tau')| \cdot
n^{2|\cV_T|+2|\cV_\Id|+4|\cE|-\frac{d(\pi_T,\pi)+d(\pi_\Id,\tau)+d(\pi',\tau')}{2}}.
\end{equation}
We recall that $|\pi_T \vee \pi_\Id|=4$, with the blocks
$\{\cE^{(s)}\}_{s=1}^4$. We consider three cases for $\pi,\pi',\tau,\tau'
\in \sP$:

{\bf Case 1:} $|\pi_T \vee \pi \vee \pi' \vee \tau' \vee \tau \vee \pi_\Id|
\leq 2$. Let $\pi|_S$ and $\pi|_{\bar S}$ denote the restrictions
of $\pi$ to $\cE_S$ and $\cE_{\bar S}$. We apply again the bound
$|\Wg_{n,\cE_S}(\pi,\pi')| \leq
O(n^{-|\cE_S|/2-d(\pi|_S,\pi'|_S)/2})$ from (\ref{eq:weinasymp}),
and similarly for $\cE_{\bar S}$.
Since $|\cE_S|+|\cE_{\bar S}|=4|\cE|$ and
$d(\pi|_S,\pi'|_S)+d(\pi|_{\bar S},\pi'|_{\bar S})=d(\pi,\pi')$ by definition (\ref{eq:metricdef}) of the metric
$d(\cdot,\cdot)$, this bound gives
\[|\Wg_{n,\cE_S}(\pi,\pi')\Wg_{n,\cE_{\bar S}}(\pi,\pi')|
\leq Cn^{-2|\cE|-d(\pi,\pi')/2},\]
and similarly for $\tau,\tau'$. Then
$|W(\pi,\pi',\tau,\tau')| \leq Cn^{-4|\cE|-d(\pi,\pi')/2-d(\tau,\tau')/2}$.
Applying this to
(\ref{eq:fourthmomentgeneralbound}),
\begin{equation}\label{eq:fourthmomentCase1}
\EE(\val_G(\cL)-\val_G(\bar\cL))^4
\leq Cn^{2|\cV_T|+2|\cV_\Id|-\frac{d(\pi_T,\pi)+d(\pi,\pi')+d(\pi',\tau')+
d(\tau',\tau)+d(\tau,\pi_\Id)}{2}}
\end{equation}
Here, the triangle inequality $d(\pi_T,\pi)+d(\pi,\pi')+d(\pi',\tau')+
d(\tau',\tau)+d(\tau,\pi_\Id) \geq d(\pi_T,\pi_\Id)$ is not tight,
because $\pi,\pi',\tau,\tau'$ are not all refinements of $\pi_T \vee \pi_\Id$.
We instead apply the following
observations about the metric $d(\cdot,\cdot)$:
\begin{itemize}
\item By the definition (\ref{eq:metricdef}), it is direct to check that
$d(\pi_1,\pi_2)=d(\pi_1,\pi_1 \vee \pi_2)+d(\pi_1 \vee \pi_2,\pi_2)$.
\item Applying this property and the triangle inequality,
\begin{align*}
&d(\pi_1,\pi_2)+d(\pi_2,\pi_3)\\
&=d(\pi_1,\pi_1 \vee \pi_2)+d(\pi_1 \vee \pi_2,\pi_2)
+d(\pi_2,\pi_2 \vee \pi_3)+d(\pi_2 \vee \pi_3,\pi_3)\\
&\geq d(\pi_1,\pi_1 \vee \pi_2)+d(\pi_1 \vee \pi_2,\pi_2 \vee \pi_3)
+d(\pi_2 \vee \pi_3,\pi_3)\\
&=d(\pi_1,\pi_1 \vee \pi_2)+d(\pi_1 \vee \pi_2,\pi_1 \vee \pi_2 \vee \pi_3)
+d(\pi_1 \vee \pi_2 \vee \pi_3,\pi_2 \vee \pi_3)
+d(\pi_2 \vee \pi_3,\pi_3)\\
&\geq d(\pi_1,\pi_1 \vee \pi_2 \vee \pi_3)
+d(\pi_1 \vee \pi_2 \vee \pi_3,\pi_3).
\end{align*}
\item Thus
\begin{align}
&d(\pi_1,\pi_2)+d(\pi_2,\pi_3)+\ldots+d(\pi_{k-1},\pi_k)\notag\\
&\hspace{1in}\geq d\Big(\pi_1,\bigvee_{i=1}^k \pi_i\Big)+d\Big(\bigvee_{i=1}^k \pi_i,\pi_k\Big)
=|\pi_1|+|\pi_k|-2\Big|\bigvee_{i=1}^k \pi_i\Big|.\label{eq:dlowerbound}
\end{align}
This may be shown by the above property and induction on $k$:
\begin{align*}
&d(\pi_1,\pi_2)+d(\pi_2,\pi_3)+\ldots+d(\pi_{k-1},\pi_k)\\
&\geq d(\pi_1,\pi_1 \vee \pi_2 \vee \pi_3)
+\underbrace{d(\pi_1 \vee \pi_2 \vee
\pi_3,\pi_3)+d(\pi_3,\pi_4)+\ldots+d(\pi_{k-1},\pi_k)}_{\text{apply induction
hypothesis}}\\
&\geq d(\pi_1,\pi_1 \vee \pi_2 \vee \pi_3)
+d\Big(\pi_1 \vee \pi_2 \vee \pi_3,\bigvee_{i=1}^k \pi_i\Big)
+d\Big(\bigvee_{i=1}^k \pi_i,\pi_k\Big)\\
&\geq d\Big(\pi_1,\bigvee_{i=1}^k \pi_i\Big)
+d\Big(\bigvee_{i=1}^k \pi_i,\pi_k\Big).
\end{align*}
\end{itemize}
Applying (\ref{eq:dlowerbound}) gives, under our assumption for Case 1,
\begin{align*}
d(\pi_T,\pi)+d(\pi,\pi')+d(\pi',\tau')+d(\tau',\tau)+d(\tau,\pi_\Id)
&\geq |\pi_T|+|\pi_\Id|-2|\pi_T \vee \pi \vee \pi' \vee \tau' \vee \tau \vee
\pi_\Id|\\
&\geq 4|\cV_T|+4|\cV_\Id|-4.
\end{align*}
Applying this to (\ref{eq:fourthmomentCase1}) shows
$\EE(\val_G(\cL)-\val_G(\bar\cL))^4 \leq Cn^2$ as desired.

{\bf Case 2:} $|\pi_T \vee \pi \vee \pi' \vee \tau' \vee \tau \vee \pi_\Id|=3$.
In this case we apply the leading order Weingarten expansion, by
(\ref{eq:weinasymp}),
\[\Wg_{n,\cE_S}(\pi,\pi')=n^{-\frac{|\cE_S|}{2}-\frac{d(\pi|_S,\pi'|_S)}{2}}
\Wg_{\cE_S}^{(0)}(\pi,\pi')+O(n^{-\frac{|\cE_S|}{2}-\frac{d(\pi|_S,\pi'|_S)}{2}-1}),\]
and similarly for $\cE_{\bar S}$ and $\tau,\tau'$. Then 
\begin{align*}
W(\pi,\pi',\tau,\tau')&=n^{-4|\cE|-\frac{d(\pi,\pi')}{2}-\frac{d(\tau,\tau')}{2}}\underbrace{\mathop{\sum_{S \subseteq \{1,2,3,4\}}}_{\pi_S \geq
\pi,\pi',\tau,\tau'}
(-1)^{|S|}\Wg_{\cE_S}^{(0)}(\pi,\pi')\Wg_{\cE_S}^{(0)}(\tau,\tau')
\Wg_{\cE_{\bar S}}^{(0)}(\pi,\pi')\Wg_{\cE_{\bar S}}^{(0)}(\tau,\tau')
}_{:=W^{(0)}(\pi,\pi',\tau,\tau')}\\
&\hspace{1in}+O(n^{-4|\cE|-\frac{d(\pi,\pi')}{2}-\frac{d(\tau,\tau')}{2}-1}).
\end{align*}
By the explicit form in (\ref{eq:wein0}), we see that
$\Wg_\cE^{(0)}(\pi,\pi')$ factorizes across blocks of $\pi \vee \pi'$, so
\[\Wg_{\cE_S}^{(0)}(\pi,\pi')\Wg_{\cE_{\bar S}}^{(0)}(\pi,\pi')
=\Wg_{\cE^{\sqcup 4}}^{(0)}(\pi,\pi')\]
which does not depend on $S$, and similarly for $\tau,\tau'$.
When $|\pi_T \vee \pi \vee \pi' \vee \tau' \vee \tau \vee \pi_\Id|=3$,
exactly two blocks $\cE^{(s)}$ of $\pi_T \vee \pi_\Id$ are merged in this
partition. Supposing without loss of generality that these are
$\cE^{(1)},\cE^{(2)}$,
then the summation over $S$ defining $W^{(0)}(\pi,\pi',\tau,\tau')$ is over all
subsets $S$ containing either both $\{1,2\}$ or neither $\{1,2\}$,
and we see that $\sum_{S \subseteq \{1,2,3,4\}:\pi_S \geq \pi,\pi',\tau,\tau'} (-1)^{|S|}=0$.
Thus $W^{(0)}(\pi,\pi',\tau,\tau')=0$, so
\[|W(\pi,\pi',\tau,\tau')| \leq Cn^{-4|\cE|-d(\pi,\pi')/2-d(\tau,\tau')/2-1}.\]
Under our assumption for Case 2 we have
\begin{align*}
d(\pi_T,\pi)+d(\pi,\pi')+d(\pi',\tau')+d(\tau',\tau)+d(\tau,\pi_\Id)
&\geq |\pi_T|+|\pi_\Id|-2|\pi_T \vee \pi \vee \pi' \vee \tau' \vee \tau \vee
\pi_\Id|\\
&=4|\cV_T|+4|\cV_\Id|-6,
\end{align*}
and applying these bounds in (\ref{eq:fourthmomentgeneralbound}) shows again
$\EE(\val_G(\cL)-\val_G(\bar\cL))^4 \leq Cn^2$.

{\bf Case 3:} $|\pi_T \vee \pi \vee \pi' \vee \tau' \vee \tau \vee \pi_\Id|=4$.
In this case we apply the sub-leading order Weingarten expansion,
by (\ref{eq:weinasymp}),
\[\Wg_{n,\cE_S}(\pi,\pi')=n^{-\frac{|\cE_S|}{2}-\frac{d(\pi|_S,\pi'|_S)}{2}}
\Wg_{\cE_S}^{(0)}(\pi,\pi')
-n^{-\frac{|\cE_S|}{2}-\frac{d(\pi|_S,\pi'|_S)}{2}-1}
\Wg_{\cE_S}^{(1)}(\pi,\pi')
+O(n^{-\frac{|\cE_S|}{2}-\frac{d(\pi|_S,\pi'|_S)}{2}-2}),\]
and similarly for $\cE_{\bar S}$ and $\tau,\tau'$. Then
\begin{align*}
W(\pi,\pi',\tau,\tau')
&=n^{-4|\cE|-\frac{d(\pi,\pi')}{2}-\frac{d(\tau,\tau')}{2}}
W^{(0)}(\pi,\pi',\tau,\tau')
-n^{-4|\cE|-\frac{d(\pi,\pi')}{2}-\frac{d(\tau,\tau')}{2}-1}
W^{(1)}(\pi,\pi',\tau,\tau')\\
&\hspace{1in}
+O(n^{-4|\cE|-\frac{d(\pi,\pi')}{2}-\frac{d(\tau,\tau')}{2}-2})
\end{align*}
where $W^{(0)}(\pi,\pi',\tau,\tau')$ is as defined in Case 2 above, and
\begin{align*}
&W^{(1)}(\pi,\pi',\tau,\tau')\\
&=\mathop{\sum_{S \subseteq \{1,2,3,4\}}}_{\pi_S \geq
\pi,\pi',\tau,\tau'} (-1)^{|S|}\Big[
\Big(\Wg_{\cE_S}^{(1)}(\pi,\pi')\Wg_{\cE_{\bar S}}^{(0)}(\pi,\pi')
+\Wg_{\cE_S}^{(0)}(\pi,\pi')\Wg_{\cE_{\bar S}}^{(1)}(\pi,\pi')\Big)
\Wg_{\cE_S}^{(0)}(\tau,\tau') \Wg_{\cE_{\bar S}}^{(0)}(\tau,\tau')\\
&\hspace{1in}
+\Big(\Wg_{\cE_S}^{(1)}(\tau,\tau')\Wg_{\cE_{\bar S}}^{(0)}(\tau,\tau')
+\Wg_{\cE_S}^{(0)}(\tau,\tau')\Wg_{\cE_{\bar S}}^{(1)}(\tau,\tau')\Big)
\Wg_{\cE_S}^{(0)}(\pi,\pi') \Wg_{\cE_{\bar S}}^{(0)}(\pi,\pi')\Big].
\end{align*}
Here, for 
$|\pi_T \vee \pi \vee \pi' \vee \tau' \vee \tau \vee \pi_\Id|=4$,
the blocks $\cE^{(s)}$ remain disjoint in this partition,
so the summations defining $W^{(0)}$ and $W^{(1)}$ are over all subsets
$S \subseteq \{1,2,3,4\}$. Then we still
have $\sum_{S \subseteq \{1,2,3,4\}:\pi_S \geq \pi,\pi',\tau,\tau'}
(-1)^{|S|}=0$, so $W^{(0)}(\pi,\pi',\tau,\tau')=0$ as in Case 2 above.
For $W^{(1)}$, letting
$2k_1,\ldots,2k_M$ be the sizes of the blocks of $\pi \vee \pi'$, we have
from (\ref{eq:wein0}) and (\ref{eq:wein1}) that
\[\Wg_{\cE_S}^{(0)}(\pi,\pi')\Wg_{\cE_{\bar S}}^{(1)}(\pi,\pi')
+\Wg_{\cE_S}^{(1)}(\pi,\pi')\Wg_{\cE_{\bar S}}^{(0)}(\pi,\pi')
=\sum_{m=1}^M (-1)^{k_m-1}a_{k_m-1}
\mathop{\prod_{m'=1}^M}_{m' \neq m} (-1)^{k_{m'}-1} c_{k_{m'}-1},\]
where the summands corresponding to blocks $m \in \{1,\ldots,M\}$ belonging to
$\cE_S$ come from the second term
$\Wg_{\cE_S}^{(1)}(\pi,\pi')\Wg_{\cE_{\bar S}}^{(0)}(\pi,\pi')$,
and those for blocks belonging to $\cE_{\bar S}$ come from the first term
$\Wg_{\cE_S}^{(0)}(\pi,\pi')\Wg_{\cE_{\bar S}}^{(1)}(\pi,\pi')$.
This quantity again does not depend on $S$, and similarly for $\tau,\tau'$.
Thus $W^{(1)}(\pi,\pi',\tau,\tau')=0$, so
\[|W(\pi,\pi',\tau,\tau')| \leq Cn^{-4|\cE|-d(\pi,\pi')/2-d(\tau,\tau')/2-2}.\]
Under our assumption for Case 3 we have
\[d(\pi_T,\pi)+d(\pi,\pi')+d(\pi',\tau')+d(\tau',\tau)+d(\tau,\pi_\Id)
\geq 4|\cV_T|+4|\cV_\Id|-8\]
(which coincides with the direct bound from the triangle inequality for
$d(\cdot,\cdot)$).
Applying these bounds in (\ref{eq:fourthmomentgeneralbound}) shows again
$\EE(\val_G(\cL)-\val_G(\bar\cL))^4 \leq Cn^2$. Thus (\ref{eq:fourthmomentO})
holds in all cases, concluding the proof.
\end{proof}

\begin{proof}[Proof of Proposition \ref{prop:anisotropic}(b)]
The ideas are similar to the proof of Proposition \ref{prop:local}(b), and we will omit details
to avoid repetition. Let $\cF=\bigsqcup_{t=0}^T \cF_t$, where $\cF_t$ consists
of the functions $f:\RR^{n \times t} \to \RR^n$.
Given any $C_0,\epsilon>0$,
we let $\zeta,\iota>0$ be constants depending on
$L,C_0,\epsilon$ to be specified later, and denote by $C,C'>0$ constants that do
not depend on $\zeta,\iota$.

To construct a set of polynomial anisotropic functions $\cP=\bigsqcup_{t=0}^T
\cP_t$ that verifies
condition (1) of Definition \ref{def:BCP_approx}, for each $f \in \cF_0$ we
include $p=f$ in $\cP_0$. For each $t=1,\ldots,T$ and $f \in\cF_t$, suppose
$f(\cdot)=\bK'g(\bK^\top \cdot)$ and $g=(\mathring g_i)_{i=1}^n$. We
construct an approximating $p \in \cP_t$ as follows:
\begin{enumerate}[(a)]
\item For each $i \in [n]$, let $\bK[i]$ denote the $i^\text{th}$ column of
$\bK$. Define the index set
\[\cI=\{i \in [n]:\|\bK[i]\|_2^2 \geq \zeta\}.\]
For each $i \in \cI$, define $\tilde g_i:\RR^t \to \RR$ by
\begin{equation}\label{eq:tilderingg}
\tilde g_i(\bx)=\mathring g_i(\|\bK[i]\|_2 \bx)
\end{equation}
Let $L'$ be a constant larger than $L \cdot \|\bK\|_\op$, and define
\[\cG=\{\tilde g:\RR^t\to\RR \text{ such that }
\tilde g \text{ is $L'$-Lipschitz with }
|\tilde g(0)| \leq L\}.\]
Then, since $\mathring g_i$ satisfies the Lipschitz property
(\ref{eq:anisolipschitz}), we have that
 $\tilde g_i \in \cG$ for each $i \in \cI$.
Let $\cN \subseteq \cG$ be a $\zeta$-net defined independently of $n$ for which,
for each $\tilde g \in \cG$, there exists $h\in\cN$ such that 
\begin{equation}\label{eq:lipschitz_net_approx_anisotropic}
    \sup_{\bx\in\RR^t:\|\bx\|_2^2\leq (1/\zeta)^2} |h(\bx)-\tilde g(\bx)|^2<\zeta.
\end{equation}
For each $i \in \cI$, let $h_i \in \cN$ be this approximation of $\tilde g_i$.
\item Now for each $h \in \cN$, let $\tilde q:\RR^t\to\RR$ be a
polynomial that approximates $h$ in the sense
\begin{equation}\label{eq:poly_approx_anisotropic}
    \EE_{\bZ \sim \bSigma}[(h(\bZ)-\tilde q(\bZ))^2]<\iota
\end{equation}
for every $\bSigma \in \RR^{t \times t}$ satisfying $\|\bSigma\|_\op<C_0$.
For each $h \in \cN$, we may construct this polynomial
$\tilde q$ independently of $n$
in the same manner as in Proposition \ref{prop:local}(b).
For each $i \in \cI$, let
$\tilde q_i:\RR^t \to \RR$ be this polynomial approximation of the function
$h_i$ from step (a), and define
\begin{equation}\label{eq:mathringq1}
\mathring q_i(\bx)=\tilde q_i\bigg(\frac{1}{\|\bK[i]\|_2}\bx\bigg)
\text{ for } i \in \cI.
\end{equation}
Thus $\tilde q_i(\bx)=\mathring q_i(\|\bK[i]\|_2\bx)$, paralleling
(\ref{eq:tilderingg}). We set
\begin{equation}\label{eq:mathringq2}
\mathring q_i(\bx)=\mathring g_i(0) \text{ for } i \notin \cI,
\end{equation}
$q=(\mathring q_i)_{i=1}^n$, and $p(\cdot)=\bK'q(\bK^\top \cdot)$, and we
include $p$ in $\cP_t$.
\end{enumerate}
Note that the degrees and coefficients of each $(\tilde q_i:i \in \cI)$ are
bounded by a constant independent of $n$.
Then, since $1/\|\bK[i]\|_2$ is bounded for all $i \in \cI$, the degrees and
coefficients of $q=(\mathring q_i)_{i=1}^n$ are also bounded by a constant
independent of $n$. Thus $\cP$ constructed in this way is a set of
polynomial anisotropic functions satisfying Definition \ref{def:polyaniso}.
Furthermore, $|\cP|=|\cF|$ which is finite and independent of $n$.
Thus $\cP$ is BCP-representable by Proposition \ref{prop:anisotropic}(a).

To analyze the approximation error, consider any $\bSigma \in \RR^{t \times t}$
with $\|\bSigma\|_\op<C_0$,
and let $\bZ \sim \cN(0,\bSigma \otimes \Id_n) \in\RR^{n\times t}$. Then
\begin{align*}
    &\frac{1}{n}\EE\left[\|f(\bZ)-p(\bZ)\|_2^2\right]
= \frac{1}{n} \EE\left[\|\bK'g(\bK^\top \bZ) - \bK'q(\bK^\top \bZ)\|_2^2\right]\\
    &\leq \frac{\|\bK'\|_\op^2}{n} \sum_{i=1}^n 
    \EE\left[|\mathring g_i(\bK[i]^\top \bZ)-\mathring q_i(\bK[i]^\top \bZ)|^2\right]\\
    &=\frac{\|\bK'\|_\op^2}{n} \sum_{i \in \cI}
    \EE\left[\left|\tilde g_i\bigg(\frac{\bK[i]^\top}{\|\bK[i]\|_2}\bZ\bigg)
-\tilde q_i\bigg(\frac{\bK[i]^\top}{\|\bK[i]\|_2}\bZ\bigg)\right|^2\right]
+\frac{\|\bK'\|_\op^2}{n} \sum_{i \notin \cI} \EE[|\mathring g_i(\bK[i]^\top \bZ)
-\mathring g_i(0)|^2]
\end{align*}
Applying the bound $\sup_{\bK \in \cK}\|\bK\|_\op<C$,
the Lipschitz property (\ref{eq:anisolipschitz}) for $\mathring g_i$,
the bound $\|\bSigma\|_\op<C_0$, and the condition $\|\bK[i]\|_2^2<\zeta$ for
all $i \notin \cI$, the second term is bounded by $C\zeta$. Then, also
decomposing the first term, we have
\begin{align}
\frac{1}{n}\EE\left[\|f(\bZ)-p(\bZ)\|_2^2\right]
&\leq \frac{C}{n} \sum_{i \in \cI}
    \EE\left[\left|\tilde g_i\bigg(\frac{\bK[i]^\top}{\|\bK[i]\|_2}\bZ\bigg)
-h_i\bigg(\frac{\bK[i]^\top}{\|\bK[i]\|_2}\bZ\bigg)\right|^2\right]\notag\\
&\hspace{1in}+\frac{C}{n} \sum_{i \in \cI}
    \EE\left[\left|h_i\bigg(\frac{\bK[i]^\top}{\|\bK[i]\|_2}\bZ\bigg)
-\tilde q_i\bigg(\frac{\bK[i]^\top}{\|\bK[i]\|_2}\bZ\bigg)\right|^2\right]
+C\zeta\label{eq:poly_approx_anisotropic_bound_1}
\end{align}
where $h_i$ is the above net approximation of $\tilde g_i$.
Here, $(\bK[i]/\|\bK[i]\|_2)^\top \bZ \in \RR^t$ has law $\cN(0,\bSigma)$. Then
the first term of \eqref{eq:poly_approx_anisotropic_bound_1} is at most
$C'\zeta$ by (\ref{eq:lipschitz_net_approx_anisotropic})
and the same argument as in Proposition \ref{prop:local}(b), while
the second term of \eqref{eq:poly_approx_anisotropic_bound_1} is at most
$C'\iota$ from the guarantee in \eqref{eq:poly_approx_anisotropic}.
Thus choosing $\zeta,\iota$ sufficiently small based on $\epsilon$ shows
    \[\frac{1}{n} \EE\left[\|f(\bZ)-p(\bZ)\|_2^2\right]<\epsilon,\]
verifying that condition (1) in Definition \ref{def:BCP_approx} holds.

Next, we verify condition (2) in Definition \ref{def:BCP_approx}. Let
$\cQ=\bigsqcup_{t=0}^T \cQ_t$ where $\cQ_t$ is the set
of all functions of the form $\bK'q(\bK^\top \cdot)$ where $\bK',\bK \in \cK$,
$q=(\mathring q_i)_{i=1}^n$ is a separable polynomial,
and $\mathring q_i:\RR^t \to \RR$
has all coefficients bounded by 1. For any $q_1,q_2 \in \cQ$ of bounded degrees,
$\cP \cup \{q_1,q_2\}$ is also BCP-representable.
Suppose $\bSigma \in \RR^{t \times t}$ (with $\|\bSigma\|_\op<C_0$) and
$\bz \in \RR^{n \times t}$ satisfy, for any
$q_1,q_2 \in \cQ_t$ of bounded degrees, almost surely
\begin{equation}\label{eq:q1q2aniso}
\lim_{n \to \infty} \frac{1}{n} q_1(\bz)^\top q_2(\bz)
-\frac{1}{n}\EE_{\bZ \sim \cN(0,\bSigma \otimes \Id_n)}
[q_1(\bZ)^\top q_2(\bZ)]=0.
\end{equation}
Similar to the above, we may bound $n^{-1}\|f(\bz)-p(\bz)\|_2^2$ as
\begin{align}
    \frac{1}{n}\|f(\bz)-p(\bz)\|_2^2 &\leq \frac{C}{n}\sum_{i \in \cI}
    \left(\tilde g_i\bigg(\frac{\bK[i]^\top}{\|\bK[i]\|_2}\bz\bigg)
-h_i\bigg(\frac{\bK[i]^\top}{\|\bK[i]\|_2}\bz\bigg)\right)^2\notag\\
    &\quad + \frac{C}{n}\sum_{i \in \cI}
\left(h_i\bigg(\frac{\bK[i]^\top}{\|\bK[i]\|_2}\bz\bigg)
- \tilde q_i\bigg(\frac{\bK[i]^\top}{\|\bK[i]\|_2}\bz\bigg)\right)^2
+\frac{C}{n}\sum_{i \notin \cI}
\Big(\mathring g_i(\bK[i]^\top\bz)-\mathring g_i(0)\Big)^2.
    \label{eq:poly_approx_anisotropic_bound_4}
\end{align}
The first and third terms may be bounded by $C'\zeta$ using 
(\ref{eq:q1q2aniso}), (\ref{eq:lipschitz_net_approx_anisotropic}),
and the same argument as in Proposition \ref{prop:local}(b).
The analysis for the second term is also similar to that in 
Proposition \ref{prop:local}(b):
For each function $h\in\cN$, define the index set
    \[\cI_h = \{i\in\cI: h_i=h\}.\]
For all $i\in\cI_h$, the polynomial approximation $\tilde q_i$ of $h_i$
is the same, and we denote this as $\tilde q_h$. Then the second term may be
decomposed as
\[C\sum_{h\in\cN} \underbrace{\frac{1}{n}\sum_{i\in\cI_h}
    \bigg(h\bigg(\frac{\bK[i]^\top}{\|\bK[i]\|_2}\bz\bigg)
    - \tilde q_h\bigg(\frac{\bK[i]^\top}{\|\bK[i]\|_2}\bz\bigg)\bigg)^2}_{:=E_h}.\]
We claim that for each $h \in \cN$,
$E_h<2\iota$ a.s.\ for all large $n$. If this does not
hold, we may consider a positive probability event where $E_h \geq 2\iota$
infinitely often, and (\ref{eq:q1q2aniso}) holds for $q_1,q_2$ in a 
suitably chosen countable subset of $\cQ_t$. We may pass to a subsequence
$\{n_j\}_{j=1}^\infty$ where $E_h \geq 2\iota$,
$|\cI_h|/n \to \alpha$, and
$\bSigma \to \bar\bSigma$. As in Proposition
\ref{prop:local}, if $\alpha=0$ then $E_h \to 0$, contradicting $E_h \geq
2\iota$. If $\alpha>0$, the convergence (\ref{eq:q1q2aniso}) over
a suitably chosen countable subset of $\cQ_t$ implies
the convergence in moments of the empirical distribution of
$\{\frac{\bK[i]^\top}{\|\bK[i]\|_2} \bz\}_{i\in\cI_h}$ to those of
$\cN(0,\bar\bSigma)$, and hence also
Wasserstein-$k$ convergence for any order $k \geq 1$.
Then since $h-\tilde q_h$ is of polynomial growth, this implies
    \[E_h \to \alpha \cdot \EE_{\bZ\sim\cN(0,\bar\bSigma)}
    \left[(h(\bZ)-\tilde q_h(\bZ))^2\right].\]
This limit is at most $\alpha \cdot \iota$ by \eqref{eq:poly_approx_anisotropic},
again contradicting $E_h \geq 2\iota$. Thus
$E_h<2\iota$ a.s.\ for all large $n$ as claimed.
Applying this for each $h \in \cN$ 
shows that the second term of \eqref{eq:poly_approx_anisotropic_bound_4} is at
most $C(\zeta)\iota$ a.s.\ for all large $n$. Then choosing $\zeta$
sufficiently small followed by $\iota$ sufficiently small ensures
$n^{-1}\|f(\bz)-p(\bz)\|_2^2<\epsilon$,
establishing condition (2) of Definition \ref{def:BCP_approx} and
completing the proof.
\end{proof}

\subsection{Spectral functions}\label{sec:BCPSpectral}

We consider the following class of polynomial spectral functions
paralleling Definition \ref{def:spectral}, where $\bTheta_* \in \RR^{M \times
N}$ has a form $r_0(\bG_*)$ for a function $r_0:[0,\infty) \to \RR$ applied
spectrally to a matrix $\bG_*$.

\begin{definition}\label{def:SpecPolyApp}
$\cP=\bigsqcup_{t=0}^T \cP_t$ is a set of {\bf polynomial spectral functions}
with shift $\bG_* \in \RR^{M \times N}$ if, for some constants $C,K,D>0$:
\begin{itemize}
\item For each $t=0,1,\ldots,T$ and each $p \in \cP_t$, there exist polynomial
functions $r_0,r_1,\ldots,r_K:[0,\infty) \to \RR$ and coefficients
$\{c_{ks}\}$ with $|c_{ks}|<C$ for which
\begin{equation}\label{eq:polyspec}
p(\bz_1,\ldots,\bz_t)=\sum_{k=1}^K \vec\bigg(r_k\bigg(\sum_{s=1}^t
c_{ks}\mat(\bz_s)+r_0(\bG_*)\bigg)\bigg)
\end{equation}
where $r_k(\cdot)$ is applied spectrally to the singular values of
its input as in (\ref{eq:spectralcalculus}).
\item For each $k=0,1,\ldots,K$, the above polynomial $r_k(\cdot)$ takes a form
$r_k(\cdot)=N^{1/2}\bar r_k(N^{-1/2}\,\cdot\,)$ where
$\bar r_k$ is an odd-degree polynomial given by
\begin{equation}\label{eq:rform}
\bar r_k(x)=\sum_{\text{odd } d=1}^D a_{kd}x^d
\end{equation}
with coefficients $\{a_{kd}\}$ satisfying $|a_{kd}|<C$.
\end{itemize}
\end{definition}

We note that since the inputs to $r_k(\cdot)$ will have operator norm on the
order of $N^{1/2}$, the scalings of $N^{-1/2}$ and $N^{1/2}$ defining
$r_k(\cdot)$ ensure that $\bar r_k(\cdot)$ defined via (\ref{eq:rform})
is applied to an input with operator norm of constant order.
We note also that the
restriction to odd-degree polynomials $r_k(\cdot)$ for approximating
$g_k(\cdot)$ here is without loss of
generality, as it is assumed in Definition \ref{def:spectral} that $g_k(0)=0$,
and the singular values passed as input to $g_k(\cdot)$ are always non-negative.
Thus we will extend $g_k(\cdot)$ to an odd function by $g_k(-x)=-g_k(x)$, and
take its resulting polynomial approximation to also be an odd function.

We show in Section \ref{sec:spectralrepr} that if the shift $\bG_* \equiv
\bG_*(n)$ has i.i.d.\ $\cN(0,1)$ entries, then any such set $\cP$ with bounded
cardinality is BCP-representable almost surely with respect to
$\{\bG_*(n)\}_{n=1}^\infty$. We then show in Section
\ref{sec:spectralapprox} that the Lipschitz spectral functions of
Definition \ref{def:spectral} are BCP-approximable via this polynomial class.

\subsubsection{BCP-representability}\label{sec:spectralrepr}

To describe a set of tensors representing the polynomial spectral functions of
Definition \ref{def:SpecPolyApp}, we will identify each index $i \in
[n]$ with its equivalent index pair $(j,j') \in [M] \times [N]$, and write
interchangeably
\[\bT[i_1,\ldots,i_k]=\bT[(j_1,j_1'),\ldots,(j_k,j_k')]\]
for a tensor $\bT \in (\RR^n)^{\otimes k} \equiv
(\RR^{M \times N})^{\otimes k}$. We represent the above class of
polynomial spectral functions by contractions of $\bG_*$ with tensors of the
following form.

\begin{definition}\label{def:Talt}
For each even integer $k \geq 2$, the {\bf alternating tensor of order
$\pmb{k}$} is the tensor $\bT_\alt^k \in (\RR^n)^{\otimes k}$ with entries
    \[\bT_\alt^k\ss{(j_1, j_1'), \ldots, (j_k, j_k')}=
    N^{1-k/2}\prod_{\text{odd } \ell \in [k]} \1\{j_\ell'=j_{\ell+1}'\}
\prod_{\text{even } \ell \in [k]} \1\{j_\ell=j_{\ell+1}\},\]
with the identification $j_{k+1} \equiv j_1$. 
\end{definition}

\begin{lemma}\label{lem:spectralBCP}
Let $\bG_* \equiv \bG_*(n) \in \RR^{M \times N}$ have i.i.d.\ $\cN(0,1)$ entries,
and let $\bT_\alt^2,\bT_\alt^4,\ldots,\bT_\alt^K$ be the alternating tensors up
to a fixed even order $K \geq 2$. If $MN=n$ and $M,N \leq C\sqrt{n}$ for a
constant $C>0$, then 
$\cT=\{\bG_*,\bT_\alt^2,\bT_\alt^4,\ldots,\bT_\alt^K\}$ satisfies
the BCP almost surely with respect to $\{\bG_*(n)\}_{n=1}^\infty$.
\end{lemma}

\begin{proof}
By Corollary \ref{cor:BCPgaussian}, it suffices to consider the set
$\cT=\{\bT_\alt^2,\ldots,\bT_\alt^K\}$ with $\bG_*$ removed and show that $\cT$
satisfies the BCP.

Consider any expression inside the supremum of (\ref{eq:BCPcondition}), where
each tensor $\bT_1,\ldots,\bT_m$
is given by $\bT_\alt^k$ for some even order $k \geq 2$. This takes the form
$n^{-1}|\val|$ for a value
\begin{equation}\label{eq:valalt}
\val=\sum_{i_1,\ldots,i_\ell=1}^n
\prod_{a=1}^m \bT_\alt^{k_a}[i_{\pi(k_{a-1}^++1)},\ldots,i_{\pi(k_a^+)}],
\end{equation}
so we must show for each fixed $m,\ell,k_1,\ldots,k_m$ and $\pi$ that
$|\val| \leq Cn$ for a constant $C>0$ and all large $n$.
Identifying each index $i \in [n]$ with its equivalent index pair $(j,j') \in [M]
\times [N]$ and applying the form of $\bT_\alt^k$ in Definition \ref{def:Talt},
we have
\begin{equation}\label{eq:valalt2}
\begin{aligned}
\val&=\sum_{j_1,\ldots,j_\ell=1}^M
\sum_{j_1',\ldots,j_\ell'=1}^N
\prod_{a=1}^m \Bigg(N^{1-k_a/2}\\
&\times \1\{j_{\pi(k_{a-1}^++1)}'=j_{\pi(k_{a-1}^++2)}'\}
\1\{j_{\pi(k_{a-1}^++3)}'=j_{\pi(k_{a-1}^++4)}'\}
\ldots \1\{j_{\pi(k_a^+-1)}'=j_{\pi(k_a^+)}'\}\\
&\times
\1\{j_{\pi(k_{a-1}^++2)}=j_{\pi(k_{a-1}^++3)}\}
\1\{j_{\pi(k_{a-1}^++4)}=j_{\pi(k_{a-1}^++5)}\}
\ldots \1\{j_{\pi(k_a^+)}=j_{\pi(k_{a-1}^++1)}\}\Bigg).
\end{aligned}
\end{equation}
Let us represent this value via a multigraph $G_\alt$ on the $\ell$ vertices
$\{v_1,v_2,\ldots,v_\ell\}$, with edges
$\cE=\cE_R \sqcup \cE_B$ having two colors red and blue. For each
equality constraint $\1\{j_a'=j_b'\}$ above, we add a red edge $(v_a, v_b)$ to
$\cE_{R}$; for each equality constraint $\1\{j_a = j_b\}$, we add a blue edge
$(v_a, v_b)$ to $\cE_{B}$. As an illustration, consider an example of
(\ref{eq:valalt})
with $\ell=4$ indices and $m=4$ tensors given by
\begin{equation}\label{eq:valtexample}
\val=\sum_{i_1,i_2,i_3,i_4=1}^n
\bT_\alt^2[i_1,i_2]\bT_\alt^2[i_1,i_2]
\bT_\alt^6[i_2,i_2,i_3,i_3,i_4,i_4]\bT_\alt^2[i_4,i_4].
\end{equation}
Then $G_\alt$ has 4 vertices $\{v_1,v_2,v_3,v_4\}$ corresponding to the 4
indices $i_1,i_2,i_3,i_4$. The first two tensors $\bT_\alt^2$ produce one red
edge and one blue edge each between $(v_1,v_2)$, the last
tensor $\bT_\alt^2$ produces one red and one blue self-loop on $v_4$,
and the tensor $\bT_\alt^6$ produces a red self-loop on each vertex
$v_2,v_3,v_4$ and a blue edge connecting each pair
$(v_2,v_3),(v_3,v_4),(v_4,v_2)$. The
resulting graph $G_\alt$ is depicted in Figure \ref{fig:Galt}.

\begin{figure}[t]
    \begin{tikzpicture}[ultra thick]
      \node (u1)[draw,circle, thick]  at (-3, 0) {$v_1$};
      \node (u2)[draw,circle, thick]  at (-1.5, 0) {$v_2$};
      \node (u3)[draw,circle, thick] at (0, 0) {$v_3$};
      \node (u4)[draw,circle, thick] at (1.5, 0) {$v_4$};

      \draw[red] (u1) to[bend right=30] (u2);
      \draw[blue] (u1) to[bend left=30] (u2);
      \draw[red] (u1) to[bend right=60] (u2);
      \draw[blue] (u1) to[bend left=60] (u2);
      \draw[red] (u2) to[out=270 - 30,in=270 + 30, looseness=8] (u2);
      \draw[red] (u3) to[out=270 - 30,in=270 + 30, looseness=8] (u3);
      \draw[red] (u4) to[out=270 - 30,in=270 + 30, looseness=8] (u4);
      \draw[blue] (u2) to[bend left=30] (u3);
      \draw[blue] (u3) to[bend left=30] (u4);
      \draw[blue] (u2) to[bend left=60] (u4);
      \draw[blue] (u4) to[out=90 - 30,in=90 + 30, looseness=8] (u4);
      \draw[red] (u4) to[out=270 - 45,in=270 + 45, looseness=12] (u4);
    \end{tikzpicture}
  \caption{An example of the graph $G_\alt$ representing the value
of the tensor contraction (\ref{eq:valtexample}).}\label{fig:Galt}
\end{figure}

Let $\c(G_{\alt,R})$ and $\c(G_{\alt,B})$ be the numbers of connected components
in the subgraphs of $G_\alt$ given by the red edges and blue edges,
respectively. Each red component corresponds to a distinct index $j' \in [N]$ of
(\ref{eq:valalt2}), and each blue component corresponds to a distinct index $j
\in [M]$. Thus
\[\val=N^{m-\sum_{a=1}^m k_a/2}
\cdot N^{\c(G_{\alt,R})} \cdot M^{\c(G_{\alt,B})}.\]
To bound this quantity, we claim the following combinatorial lemma, whose proof
we defer below.

\begin{lemma}\label{lemma:graphlemma}
Let $G=(\cV,\cE)$ be any multigraph with edges $\cE=\cE_R \sqcup
\cE_B$ of two colors red and blue. Suppose, in each subgraph $G_R$ or
$G_B$ of red or blue edges only, each vertex $v \in \cV$ has non-zero even
degree (where a self-loop contributes a degree of 2 to its vertex). Suppose
also that $\cE$ can be
decomposed as a union of $m$ edge-disjoint cycles $\cE=S_1 \sqcup \cdots \sqcup
S_m$, where each $S_a$ for $a=1,\ldots,m$ is a non-empty cycle containing an even number of edges
that alternate between red edges of $\cE_R$ and blue edges of $\cE_B$. Then the
numbers of connected components of $G_R,G_B,G$ satisfy
\begin{equation}\label{eq:CCbound}
\c(G_R)+\c(G_B) \leq \frac{|\cE|}{2}-m+2\c(G).
\end{equation}
\end{lemma}

We apply the lemma to $G_\alt$ constructed above:
Each vertex $v_b$ of $G_\alt$ has non-zero even
degree in each of the red and blue subgraphs $G_{\alt,R}$ and $G_{\alt,B}$,
because each appearance of the corresponding index $i_b$ in (\ref{eq:valalt})
contributes 1 to both the red and blue degrees of $v_b$, and each index
$i_1,\ldots,i_\ell$ appears a non-zero even number of times in
(\ref{eq:valalt}) by surjectivity of $\pi$ and the first condition of
Definition \ref{def:BCP}. Each
tensor $\{\bT_a:a=1,\ldots,m\}$ contributes an even-length cycle $S_a$ of edges
of alternating colors, so the decomposition $\cE=S_1 \sqcup \cdots \sqcup S_m$ holds with a number of cycles $m$ equal to the number of tensors.
The total number of edges $|\cE|$ of $G_\alt$ is the total order of all tensors
$\sum_{a=1}^m k_a$. Finally, $G_\alt$ is connected, for
otherwise there is a partition of the indices $i_1,\ldots,i_\ell$ corresponding
to two disjoint sets of tensors in (\ref{eq:valalt}), contradicting the second
condition of Definition \ref{def:BCP}. Thus $\c(G_\alt)=1$. Under the
given conditions for $M,N$, there exists a constant $C>0$ for which $M/N<C$ and
$N^2<Cn$. Thus, Lemma \ref{lemma:graphlemma} implies
\[\val \leq C N^{m-\sum_{a=1}^m k_a/2+\c(G_{\alt,R})+\c(G_{\alt,B})}
\leq CN^{2\c(G_\alt)} \leq C'n\]
for some constants $C,C'>0$, as desired.
\end{proof}

\begin{proof}[Proof of Lemma \ref{lemma:graphlemma}]
Let $\deg_{G_R}(v)$ and $\deg_{G_B}(v)$ denote the degrees of the vertex
$v \in \cV$ in
the subgraphs of red and blue edges only. Note that the assumptions of the lemma
imply $\deg_{G_R}(v)=\deg_{G_B}(v)$ (because each alternating cycle
$S_1,\ldots,S_m$ must contribute the same degree to $v$ in both the red and blue
subgraphs) which is non-zero and even for each $v \in \cV$.

    We induct on the total number of edges $|\cE|$, which must be even since
each cycle $S_1,\ldots,S_m$ is of even length.
    For the base case $|\cE|=2$, we must have $\cE=S$ for a single
    alternating cycle $S$, and $\cV=\{u\}$ and $S=((u,u),(u,u))$ for a single vertex
    $u$ in order for $\deg_{G_R}(v)=\deg_{G_B}(v) \geq 2$ to hold for all
vertices $v \in \cV$. In this case $\c(G_R)=\c(G_B)=\c(G)=1$, $|\cE|=2$, and
    $m=1$, so (\ref{eq:CCbound}) holds with equality.

    Consider now $|\cE| \geq 4$, and suppose by induction that the result
    holds when the total number of edges is at most $|\cE|-2$.
    Pick any vertex $u \in \cV$ and consider the following cases:
    \begin{itemize}
        \item Some alternating cycle, say $S_1$, has only two edges, both of
        which are self-loops on $u$: $S_1=\{(u,u),(u,u)\}$. Then consider
        $G'=(\cV',\cE')$ obtained from $G=(\cV,\cE)$ by removing these two edges
        from $\cE$, and also removing the vertex $u$ from $\cV$ if it appears on
        no other edge. Clearly $\deg_{G_R'}(v),\deg_{G_B'}(v)$ remain non-zero
and even for each remaining vertex
        $v \in \cV'$, each remaining $S_a \subset \cE'$ is
        a non-empty even alternating cycle, the number of such cycles
constituting $\cE'$ is now $m'=m-1$, and $|\cE'|=|\cE|-2$. Thus
        the induction hypothesis applied to $G'$ yields
        \begin{equation} \label{eq:induction1} 
            \c(G_R')+\c(G'_B) \leq \frac{1}{2}|\cE'|-m'+2\c(G') 
            =\frac{1}{2}|\cE|-m+2\c(G').
        \end{equation}
        If $u$ appears on another edge in $\cE$, then
        $\deg_{G_R'}(u)=\deg_{G_B'}(u)>0$ so $\c(G_R')=\c(G_R)$,
        $\c(G_B')=\c(G_B)$, and $\c(G')=\c(G)$. If $u$ appears only on these two
        edges of $\cE$ (meaning $u$ was its own connected component in $G$) then
        $\c(G_R')=\c(G_R)-1$, $\c(G_B')=\c(G_B)-1$, and $\c(G')=\c(G)-1$. In
        both cases, (\ref{eq:induction1}) implies that (\ref{eq:CCbound}) holds
        for $G$.
        \item Some alternating cycle, say $S_1$, has at least 4 edges including 
        a self-loop $(u,u)$:
        \[
            S_1=\Big\{(u,u),(u,u_3),(u_3,u_4),\ldots,(u_{2k},u)\Big\}.
        \]
        Then consider $G'=(\cV',\cE')$ obtained by merging $u$ and $u_3$ --- i.e.\
        replacing $u_3$ by $u$ in all edges of $\cE$ containing $u_3$ and then
        removing $u_3$ from $\cV$ --- and also replacing the edges of $S_1$ by
        $S_1'=\{(u,u_4),\ldots,(u_{2k},u)\}$ which removes the first two edges
        (now self-loops on $u$) from the cycle. Again
        $\deg_{G_R'}(v)=\deg_{G_B'}(v)$ remains non-zero and even for each $v
\in \cV'$, and $\cE'$ is comprised of $m'=m$ non-empty alternating cycles of even
length. We  have $|\cE'|=|\cE|-2$, so the induction hypothesis applied to $G'$ yields
        \begin{equation} \label{eq:induction2} 
            \c(G_R')+\c(G_B') \leq \frac{1}{2}|\cE'|-m'+2\c(G') 
            =\frac{1}{2}|\cE|-m+2\c(G')-1.
        \end{equation}
        Suppose (without loss of generality) $(u,u_3)$ is red.
        Then $\c(G')= \c(G)$ and $\c(G_R')=\c(G_R)$, whereas $
        \c(G_B') \in \{\c(G_B),\c(G_B)-1\}$ depending on whether $u$ and $u_3$ 
        belong to the same connected component of $G_B$.
        In particular $\c(G_B') \geq \c(G_B)-1$, so (\ref{eq:induction2}) 
        implies that (\ref{eq:CCbound}) holds for $G$.
        
        \item Some alternating cycle, say $S_1$, has at least 4 non-self-loop 
        edges incident to $u$:
        \[
S_1=\Big\{(u,u_2),(u_2,u_3),\ldots,(u_j,u),(u,u_{j+2}),\ldots,(u_{2k},u)\Big\}
        \]
        where $j$ is odd. Suppose $(u,u_2)$ is red and $(u_{2k},u)$ is blue;
        then $(u_j,u)$ is red and $(u,u_{j+2})$ is blue. Consider the graph
        $G'$ that merges $u$, $u_2$, and $u_{2k}$, and that also replaces
        the edges of $S_1$ by those of two alternating cycles
        \begin{align*}
            S_1'
            &=\Big\{(u,u_3),\ldots,(u_{j-1},u_j),(u_j,u)\Big\},\\
            S_1''
            &=\Big\{(u,u_{j+2}),(u_{j+2},u_{j+3}),\ldots,(u_{2k-1},u)\Big\}.
        \end{align*}
        This replaces the two red edges $(u_j,u),(u,u_2)$ (the latter now a
self-loop on $u$) by a single red
        edge $(u_j,u)$, and the two blue edges $(u_{2k},u),(u,u_{j+2})$ (the
former now a self-loop on $u$)
        by a single blue edge $(u,u_{j+2})$. Then $S'_1$ and
        $S''_1$ are both alternating cycles of non-zero even length, and $G'$
has $|\cE'|=|\cE|-2$ edges comprised of $m'=m+1$ alternating cycles. The
induction hypothesis applied to $G'$ yields
        \begin{equation}\label{eq:induction3} 
            \c(G_R')+\c(G_B') \leq \frac{1}{2}|\cE'|-m'+2\c(G') 
            =\frac{1}{2}|\cE|-m+2\c(G')-2.
        \end{equation}
        We have $\c(G')=\c(G)$, because all vertices connected to $u/u_2/u_{2k}$ 
        in $G$ remain connected to $u$ in $G'$.
        We have also $\c(G_R') \geq \c(G_R)-1$, because merging $(u,u_2)$
        does not change $\c(G_R)$, merging $(u,u_{2k})$ decreases $\c(G_R)$ by 
        at most
        1, and replacing $(u_j,u),(u,u_2)$ by the single edge $(u_j,u)$ and
replacing $(u_{2k},u),(u,u_{j+2})$ by the single edge $(u,u_{j+2})$ do not change $\c(G_R)$.
        Similarly, $\c(G_B') \geq \c(G_B)-1$, and applying these statements to 
        (\ref{eq:induction3}) shows that (\ref{eq:CCbound}) holds for $G$.
        
        \item Some alternating cycle, say $S_1$, has at least 4 non-self-loop 
        edges incident to $u$:
        \[
S_1=\Big\{(u,u_2),(u_2,u_3),\ldots,(u_j,u),(u,u_{j+2}),\ldots,(u_{2k},u)\Big\}
        \]
        where $j$ is even.
        Then we may split $S_1$ into the two cycles,
        \begin{align*}
            S_1'
            &=\Big\{(u,u_2),(u_2,u_3),\ldots,(u_j,u)\Big\}\\
            S_1''
            &=\Big\{(u,u_{j+2}),(u_{j+2},u_{j+3}),\ldots,(u_{2k},u)\Big\},
        \end{align*}
        both of which are of non-zero even length. This reduces to the final case below, which shows that in fact
        \[
            \c(G_R)+\c(G_B) \leq \frac{1}{2}|\cE|-(m+1)+2\c(G).
        \]
        \item Some two alternating cycles, say $S_1,S_2$, each contains at least 
        two consecutive non-self-loop edges incident to $u$, denoted by:
        \begin{align*}
            S_1
            &=\Big\{(u,u_2),(u_2,u_3),\ldots,(u_{2j},u)\Big\}\\
            S_2
            &=\Big\{(u,v_2),(v_2,v_3),\ldots,(v_{2k},u)\Big\}
        \end{align*}
        By reversing the orderings of the cycles, we may assume 
        $(u,u_2),(u,v_2)$ are red and $(u_{2j},u),(v_{2k},u)$ are blue.
        Consider the graph $G'=(\cV',\cE')$ obtained by replacing the edges of 
        $S_1 \sqcup S_2$ by
        \[
            S'=\Big\{(u_2,u_3),\ldots,(u_{2j-1},u_{2j}),(u_{2j},v_{2k}),
            (v_{2k},v_{2k-1}),\ldots,(v_3,v_2),(v_2,u_2)\Big\}
        \]
        and removing $u$ from $\cV$ if no other edge of $\cE$ except the above 
        four edges of $S_1,S_2$ is incident to $u$.
        This replaces the two red edges $(v_2,u),(u,u_2)$ by a single red edge 
        $(v_2,u_2)$, and the two blue edges $(u_{2j},u),(u,v_{2k})$ by a single 
        blue edge $(u_{2j},v_{2k})$.
        These actions do not change the degree of any vertex besides $u$, and
the red/blue degrees of $u$ are each decreased by 2.

        Note that $S'$ remains an alternating cycle of non-zero even length, so
$G'$ has $|\cE'|=|\cE|-2$ edges comprised of $m'=m-1$ alternating cycles.
        The induction hypothesis applied to $G'$ yields
        \begin{equation}\label{eq:induction5} 
            \c(G_R')+\c(G_B') \leq \frac{1}{2}|\cE'|-m'+2\c(G') 
            =\frac{1}{2}|\cE|-m+2\c(G').
        \end{equation}
        If $S'$ is disconnected from the component containing $u$ in $G'$, 
        then $\c(G')=\c(G)+1$.
        In this case the component of $G_R'$ containing $(v_2,u_2)$ is also 
        disconnected from the component of $G_R'$ containing $u$, so 
        $\c(G_R')=\c(G_R)+1$, and similarly $\c(G_B')=\c(G_B)+1$.
        Then applying these to (\ref{eq:induction5}) shows that (\ref{eq:CCbound}) 
        holds for $G$.
        If $u$ is no longer present in $G'$ or if $S'$ remains connected to the 
        component containing $u$ in $G'$, then $\c(G')=\c(G)$.
        In this case, we note simply that the above operation of replacing 
        $(v_2,u),(u,u_2)$ by $(v_2,u_2)$ and $(u_{2j},u),(u,v_{2k})$ by 
        $(u_{2j},v_{2k})$ cannot decrease $\c(G_R)$ or $\c(G_B)$, so 
        $\c(G_R') \geq \c(G_R)$ and $\c(G_B') \geq \c(G_B)$.
        Then applying these to (\ref{eq:induction5}) also shows that 
        (\ref{eq:CCbound}) holds for $G$.
    \end{itemize}
    Since $\deg_{G_R}(u)=\deg_{G_B}(u) \geq 2$, these cases exhaust all 
    possibilities for the vertex $u$.
    So (\ref{eq:CCbound}) holds for $G$, completing the induction.
\end{proof}

Using Lemma \ref{lem:spectralBCP}, we now verify that polynomial spectral
functions are BCP-representable.

\begin{lemma}\label{lem:BCPreprspectral}
Let $\cP=\bigsqcup_{t=0}^T \cP_t$ be a set of polynomial spectral functions
as given by Definition \ref{def:SpecPolyApp}, with shift $\bG_* \equiv \bG_*(n) \in \RR^{M
\times N}$ having i.i.d.\ $\cN(0,1)$ entries. Suppose $|\cP|<C$
for a constant $C>0$ independent of $n$. Then $\cP$ is BCP-representable almost
surely with respect to $\{\bG_*(n)\}_{n=1}^\infty$.
\end{lemma}

\begin{proof}
For any odd integer $d \geq 1$, consider the multivariate monomial
\[q(\bX_1,\ldots,\bX_d)=N^{1/2-d/2}\bX_1\bX_2^\top\ldots
\bX_{d-2}\bX_{d-1}^\top\bX_d.\]
Writing $\langle \cdot,\cdot \rangle$ for the Euclidean inner-product in
$\RR^n \equiv \RR^{M \times N}$, observe
for any $\bX_{d+1} \in \RR^{M \times N}$ that
\begin{align*}
\langle q(\bX_1,\ldots,\bX_d),\bX_{d+1} \rangle
&=N^{1/2-d/2}\Tr \bX_1\bX_2^\top \ldots \bX_d\bX_{d+1}^\top\\
&=N^{1/2-d/2}\sum_{j_1,\ldots,j_{d+1}=1}^M
\sum_{j_1',\ldots,j_{d+1}'=1}^N
\bX_1[j_1,j_1']\1\{j_1'=j_2'\}\bX_2[j_2,j_2']\1\{j_2=j_3\}\ldots\\
&\hspace{1.5in}\bX_d[j_d,j_d']\1\{j_d'=j_{d+1}'\}
\bX_{d+1}[j_{d+1},j_{d+1}']\1\{j_{d+1}=j_1\}\\
&=\sum_{i_1,\dots,i_{d+1}=1}^n \bT_\alt^{d+1}\ss{i_1, \dots, i_{d+1}}
\prod_{a=1}^{d+1}\bX_a\ss{i_a}.
\end{align*}
Thus
\[q(\bX_1,\ldots,\bX_d)=\bT_\alt^{d+1}[\bX_1,\ldots,\bX_d,\,\cdot\,].\]
In (\ref{eq:polyspec}), if each $\bar r_k(x)=x^{d_k}$ is a single monomial of
odd degree, then $r_k(x)=N^{1/2-d_k/2}\bar r_k(x)$, so this implies
\begin{align*}
p(\bz_1,\ldots,\bz_t)&=\sum_{k=1}^K
\bT_\alt^{d_k+1}\Bigg[\sum_{s=1}^t c_{ks}\bz_s+r_0(\bG_*),
\ldots,\sum_{s=1}^t c_{ks}\bz_s+r_0(\bG_*),\,\cdot\,\Bigg]\\
&=\sum_{k=1}^K \bT_\alt^{d_k+1}\Bigg[\sum_{s=1}^t
c_{ks}\bz_s+\bT_\alt^{d_0+1}[\bG_*,\ldots,\bG_*,\,\cdot\,],
\ldots,\sum_{s=1}^t c_{ks}\bz_s+\bT_\alt^{d_0+1}[\bG_*,\ldots,\bG_*,\,\cdot\,]
,\,\cdot\,\Bigg]
\end{align*}
Then multi-linearity of $\bT_\alt^{d_k+1}$ shows that $p(\bz_{1:t})$
takes the form (\ref{eq:tensorpolyrepr}) for tensors $\bT^{(0)},\bT^{(\sigma)}$
that are given by scalar multiples of contractions of
$\bT_\alt^{d_k+1}$, $\bT_\alt^{d_0+1}$, and $\bG_*$. Then again by
multi-linearity, the same holds true for any $p(\bz_{1:t})$ defined by
(\ref{eq:polyspec}) where $\bar r_k$ are given by general odd polynomials of the form (\ref{eq:rform}). Let $\cT$ be the set of all tensors arising
in this representation (\ref{eq:tensorpolyrepr}) for all polynomials $p \in
\cP$. Since the cardinality $|\cP|$ is bounded independently of $n$, so is
$|\cT|$. Each tensor in $\cT$ is a contraction of some number of tensors
$\{\bG_*,\bT_\alt^2,\ldots,\bT_\alt^{D+1}\}$ multiplied by a scalar that is
also bounded independently of $n$. By Lemma \ref{lem:spectralBCP},
$\{\bG_*,\bT_\alt^2,\ldots,\bT_\alt^{D+1}\}$ satisfies the BCP almost surely,
and hence by Lemma \ref{lemma:BCPcontraction} so does $\cT$. Thus $\cP$ is
almost surely BCP-representable.
\end{proof}

\subsubsection{BCP-approximability}\label{sec:spectralapprox}

We now prove Proposition \ref{prop:spectral} on the BCP-approximability of 
Lipschitz spectral functions. As a first step, we show that $\bG_*$ in Lemma
\ref{lem:BCPreprspectral} may be replaced by a matrix $\bX_*$ with the same
singular values as $\bG_*$, but with singular vectors satisfying the
conditions of Proposition \ref{prop:spectral}.

\begin{corollary}\label{cor:BCPreprspectral}
Let $\bTheta_*=\bO\bD\bU^\top \in \RR^{M \times N}$ be a shift matrix satisfying
the conditions of Proposition \ref{prop:spectral}.
Suppose $\bX_*=\bO\bS\bU^\top$ where $\bO$ and $\bU$ are the singular vector
matrices of $\bTheta_*$, and $\bS$ is independent of $(\bO,\bU)$ and equal in
law to the matrix of singular values (sorted in increasing order)
of $\bG_* \in \RR^{M \times N}$ having
i.i.d.\ $\cN(0,1)$ entries. Then Lemma \ref{lem:BCPreprspectral} holds also with
$\bX_*$ in place of $\bG_*$.
\end{corollary}
\begin{proof}
In the proof of Lemma \ref{lem:spectralBCP}, the BCP for
$\{\bG_*,\bT_\alt^2,\ldots,\bT_\alt^K\}$ follows from Corollary
\ref{cor:BCPgaussian}, which applies Wick's rule and Gaussian
hypercontractivity to verify that
\begin{equation}\label{eq:valprobboundGstar}
\PP[|n^{-1}\val(\bG_*)|>C] \leq C'e^{-cn^{1/m}}
\end{equation}
for some constants $C,C',c,m>0$, where $|n^{-1}\val(\bG_*)|$ is any expression
appearing inside the supremum of (\ref{eq:BCPcondition}) viewed as a function of
the Gaussian input $\bG_*$. Writing $\bG_*=\bO'\bS{\bU'}^\top$ for the singular
value decomposition of $\bG_*$,
we note that $\bO',\bS,\bU'$ are independent, and $\bO' \in
\RR^{M \times M}$ and $\bU' \in \RR^{N \times N}$ are Haar-distributed.
Then, by the given assumption that $\bO,\bU$ have bounded densities with respect
to Haar measure, (\ref{eq:valprobboundGstar}) implies also
\[\PP[|n^{-1}\val(\bX_*)|>C] \leq C'e^{-cn^{1/m}}\]
for the given matrix $\bX_*$ and a different constant $C'>0$.
Then the argument of Corollary
\ref{cor:BCPgaussian} implies that $\{\bX_*,\bT_\alt^2,\ldots,\bT_\alt^K\}$ also
satisfies the BCP almost surely, and hence Lemma \ref{lem:BCPreprspectral} holds
equally with $\bX_*$ in place of $\bG_*$.
\end{proof}

Next, we argue that the singular value matrix $\bD$ of $\bTheta_*$ 
may be approximated by $g(\bS)$ for some Lipschitz
function $g(\cdot)$ applied to the singular value matrix $\bS$ of Corollary
\ref{cor:BCPreprspectral}. The idea of the approximation is encapsulated in the
following lemma.

\begin{lemma}\label{lem:poly_approx_bounded_monotone}
Fix any constant $C_0>0$ and any probability distribution $\mu$ on an interval
$(a,b)$ with $0\leq a<b$, where $\mu$ has continuous and strictly positive
density on $(a,b)$. Then for any $\epsilon>0$, there exists a constant
$L_\epsilon>0$ such that the following holds:

Let $\cL_\epsilon$ be the set of functions $g:[a,b]\to[0,C_0]$ such that 
\[g(a)=0,\qquad |g(x)-g(y)| \leq L_\epsilon|x-y| \text{ for all } x,y \in
[a,b].\]
Let $s^{(j)}$ be the $j/M$-quantile of $\mu$, i.e.\ the value where
$\mu([a,s^{(j)}])=j/M$, for each $j=1,\ldots,M$.
Then for all large enough $M$ and for any $0\leq d^{(1)}\leq\cdots\leq d^{(M)}\leq C_0$, there exists $g\in\cL_\epsilon$ such that
    \[\frac{1}{M}\sum_{j=1}^M (g(s^{(j)})-d^{(j)})^2 \leq \epsilon.\]
\end{lemma}
\begin{proof}
Set $s^{(0)}=a$, and note that $s^{(M)}=b$.
For any $0\leq d^{(1)}\leq\cdots\leq d^{(M)}\leq C_0$, we construct $g$ as
follows: First let $g(s^{(0)})=g(a)=0$. Then for $j=1,\ldots,M$, fixing a small constant
$\iota>0$ to be determined later, let
    \[g(s^{(j)}) = \begin{cases}
        d^{(j)} & \text{if } d^{(j)} - g(s^{(j-1)}) \leq (s^{(j)}-s^{(j-1)})
\iota^{-1}, \\
        g(s^{(j-1)}) + (s^{(j)}-s^{(j-1)}) \iota^{-1} & \text{otherwise},
    \end{cases}\]
and let $g$ be the linear interpolation of the points
$(s^{(j)},g(s^{(j)}))$ for $j=0,\ldots,M$.
Note that $g$ is $\iota^{-1}$-Lipschitz, $g$ is monotonically increasing,
and $g(s^{(j)}) \leq d^{(j)}$ for all $j \in [M]$.

For some small $\delta\in(0,1)$ to be determined later, let
$j_0<j_1<\ldots<j_K$ be all indices in the range
$[\delta M,(1-\delta)M]$ for which $g(s^{(j)})=d^{(j)}$. 
Observe that $g(s^{(j)})=g(s^{(j-1)})+(s^{(j)}-s^{(j-1)})\iota^{-1}$
for each $j=\lceil \delta M \rceil,\ldots,j_0-1$, so
\[s^{(j_0-1)}-s^{(\lceil \delta M \rceil-1)}
=\iota[g(s^{(j_0-1)})-g(s^{(\lceil \delta M \rceil-1)})] \leq C_0\iota.\]
Since the density of $\mu$ is bounded above and below on compact sub-intervals
of $(a,b)$, there exist constants $C_\delta,c_\delta>0$ depending
on $\delta$ such that
\begin{equation}\label{eq:mudensitybound}
\mu(x)\in[c_\delta,C_\delta] \text{ for all }
x\in[s^{(\lceil \delta M \rceil-1)},s^{(\lfloor (1-\delta)M\rfloor+1)}].
\end{equation}
Thus
\begin{equation}\label{eq:j0bound}
\frac{j_0-\lceil \delta M \rceil}{M}
\leq c_\delta^{-1}(s^{(j_0-1)}-s^{(\lceil \delta M \rceil-1)})
\leq C_0c_\delta^{-1}\iota.
\end{equation}
By a similar argument,
\begin{equation}\label{eq:jKbound}
\frac{\lfloor (1-\delta)M \rfloor-j_K}{M}
\leq C_0c_\delta^{-1}\iota.
\end{equation}
We can then decompose the total error as
\begin{align}
    \frac{1}{M} \sum_{j=1}^M (g(s^{(j)}) - d^{(j)})^2 &= \frac{1}{M}
\sum_{j=1}^{j_0-1} (g(s^{(j)})-d^{(j)})^2 + \frac{1}{M} \sum_{k=1}^{K}
\sum_{j=j_{k-1}+1}^{j_k-1} (g(s^{(j)})-d^{(j)})^2\notag\\
    &\qquad + \frac{1}{M} \sum_{j=j_K+1}^M (g(s^{(j)})-d^{(j)})^2\notag\\
    &\leq 2C_0^2(\delta+C_0c_\delta^{-1}\iota)
+ \frac{1}{M} \sum_{k=1}^{K}\underbrace{\sum_{j=j_{k-1}+1}^{j_k-1} (g(s^{(j)})-d^{(j)})^2}_{A_k}\label{eq:lipschitz_approx_bound1}
\end{align}
where the inequality applies $d^{(j)}-g(s^{(j)})\in[0,C_0]$ for all $j\in[M]$
and the bounds (\ref{eq:j0bound}) and (\ref{eq:jKbound}) for $j_0,j_K$.

Now for each $k\in[K]$, $\{(s^{(j)}, g(s^{(j)}))\}_{j=j_{k-1}+1}^{j_k-1}$ are
points on the line between
$(s^{(j_{k-1})},g(s^{(j_{k-1})}))=(s^{(j_{k-1})},d^{(j_{k-1})})$ and
$(s^{(j_k-1)}, g(s^{(j_k-1)}))$ with slope $\iota^{-1}$.
Applying $g(s^{(j)}) \leq d^{(j)} \leq d^{(j_k)}$ for all
$j=j_{k-1}+1,\ldots,j_k-1$, we have
    \[A_k \leq \sum_{j=j_{k-1}+1}^{j_k-1} (d^{(j_{k})}-g(s^{(j)}))^2
    = \sum_{j=j_{k-1}+1}^{j_k-1} \Big(d^{(j_k)}-d^{(j_{k-1})}-(s^{(j)}-s^{(j_{k-1})})
\iota^{-1}\Big)^2.\]
Since $d^{(j_{k-1})}=g(s^{(j_{k-1})})$, $d^{(j_k)}=g(s^{(j_k)})$ and $g$ is
$\iota^{-1}$-Lipschitz, we have $d^{(j_k)}-d^{(j_{k-1})} \leq
(s^{(j_k)}-s^{(j_{k-1})})\iota^{-1}$.
Meanwhile, $d^{(j_k)}-d^{(j_{k-1})}\geq (s^{(j)}-s^{(j_{k-1})})\iota^{-1}$
for all $j=j_{k-1}+1,\ldots,j_k-1$.
Therefore, we can further bound $A_k$ as
\begin{align}
    A_k &\leq \sum_{j=j_{k-1}+1}^{j_k-1}
\Big((s^{(j_k)}-s^{(j_{k-1})})\iota^{-1} -
(s^{(j)}-s^{(j_{k-1})})\iota^{-1}\Big)^2
    = \iota^{-2} \sum_{j=j_{k-1}+1}^{j_k-1}  (s^{(j_k)}-s^{(j)})^2\notag\\
    &\leq \iota^{-2} (j_k-1-j_{k-1}) (s^{(j_k)}-s^{(j_{k-1})})^2 \label{eq:lipschitz_approx_bound2}
\end{align}
where the second inequality holds because $0<s^{(j_k)}-s^{(j)} < s^{(j_k)}-s^{(j_{k-1})}$ for all $j=j_{k-1}+1,\ldots,j_k-1$.
Next, observe that
\[d^{(j_k)}-d^{(j_{k-1})} \geq 
d^{(j_k-1)}-d^{(j_{k-1})}=(s^{(j_k-1)}-s^{(j_{k-1})}) \iota^{-1}\]
for all $k\in[K]$.
Applying this bound and (\ref{eq:mudensitybound}),
\begin{align*}
    \sum_{k=1}^{K}\frac{j_k-1-j_{k-1}}{M} &\leq C_\delta \sum_{k=1}^{K} (s^{(j_k-1)}-s^{(j_{k-1})})
\leq C_\delta\iota \sum_{k=1}^K (d^{(j_k)}-d^{(j_{k-1})})
    \leq C_0C_\delta \iota,\\
\max_{k=1}^K (s^{(j_k)}-s^{(j_{k-1})})
&\leq \max_{k=1}^K (s^{(j_k)}-s^{(j_k-1)})+\max_{k=1}^K
(s^{(j_k-1)}-s^{(j_{k-1})})
\leq c_\delta^{-1}M^{-1}+C_0\iota.
\end{align*}
Then applying these bounds to \eqref{eq:lipschitz_approx_bound2}
and \eqref{eq:lipschitz_approx_bound1}, we obtain 
\begin{equation}\label{eq:lipschitz_approx_bound3}
\frac{1}{M}\sum_{j=1}^M (g(s^{(j)})-d^{(j)})^2
\leq 2C_0^2(\delta+C_0c_\delta^{-1}\iota)
+C_0C_\delta\iota^{-1} (c_\delta^{-1}M^{-1}+C_0\iota)^2.
\end{equation}
Finally, for any target error level $\epsilon$, we can choose $\delta \equiv
\delta(\epsilon)$ small enough followed by $\iota \equiv \iota(\delta,\epsilon)$ small enough such that for all large $M$, the above
error is less than $\epsilon$. The Lipschitz constant $L_\epsilon$ is given by
$\iota^{-1}$, completing the proof.
\end{proof}

\begin{proof}[Proof of Proposition~\ref{prop:spectral}]
We may assume without loss of generality that $M \leq N$, hence $\delta=\lim_{n
\to \infty} M/N \in (0,1]$, and $\bD=\diag(d_1,\ldots,d_M)$ where $d_1 \leq \ldots \leq d_M$.
Let $\bX_*=\bO\bS\bU^\top$ be as defined in Corollary \ref{cor:BCPreprspectral},
where $\bS=\diag(s_1,\ldots,s_M)$ coincides with the singular values of a matrix
$\bG_* \in \RR^{M \times N}$ having i.i.d.\ $\cN(0,1)$ entries, and
$s_1 \leq \ldots \leq s_M$.
Let $\nu$ be the Marcenko-Pastur density with aspect ratio $\delta$,
which describes the asymptotic eigenvalue distribution of $\bG_*\bG_*^\top/N$,
and let $\mu$ be the density of $\sqrt{\lambda}$ when $\lambda \sim \nu$.
We note that $\mu$ is a continuous and strictly positive density on a single
interval of support $(a,b)$, where $a=0$ if $\delta=1$. Then
letting $s^{(j)}$ be the $j/M$-quantile of $\mu$,
the almost-sure weak convergence
$\frac{1}{M}\sum_{j=1}^M \delta_{s_j/\sqrt{N}} \to \mu$
(c.f.\ \cite{silverstein1995empirical}) implies the convergence of quantiles
\begin{equation}\label{eq:quantileconvergence}
\max_{j=1}^M |s_j/\sqrt{N}-s^{(j)}| \to 0 \text{ a.s.}
\end{equation}
Let $d^{(j)}=d_j/\sqrt{N}$.
By Lemma~\ref{lem:poly_approx_bounded_monotone}, for any $\epsilon>0$, there
exists an $n$-independent class $\cL_\epsilon$ of $L_\epsilon$-Lipschitz
functions $g:[a,b]\to[0,C]$ with $g(a)=0$ such that for some $\bar g_0 \in\cL_\epsilon$,
\begin{equation}\label{eq:lipschitz_approx_monotone}
    \frac{1}{M}\sum_{j=1}^M (\bar g_0(s^{(j)})-d^{(j)})^2 \leq \epsilon.
\end{equation}
For each $g \in \cL_\epsilon$ and any constant $B_\epsilon \geq b$,
we may extend $g$ to an odd function on $[-B_\epsilon,B_\epsilon]$ by setting
$g(x)=0$ for $x \in [0,a]$, $g(x)=g(b)$ for $x \in [b,B_\epsilon]$, and
$g(x)=-g(-x)$ for $x \in [-B_\epsilon,0]$.
By the Weierstrass approximation theorem, we may then construct an
$n$-independent net
$\cN_\epsilon$ of polynomial functions such that for any  $g\in\cL_\epsilon$,
there exists $r\in\cN_\epsilon$ for which
\begin{equation}\label{eq:rgapprox}
\max_{x\in[-B_\epsilon,B_\epsilon]}(g(x)-r(x))^2 \leq \epsilon.
\end{equation}
Replacing $r(x)$ by $(r(x)-r(-x))/2$, we may assume that each polynomial
function $r \in \cN_\epsilon$ is odd.
Then for the Lipschitz function $\bar g_0$ in \eqref{eq:lipschitz_approx_monotone}, the corresponding odd polynomial $\bar r_0\in\cN_\epsilon$ that approximates
$\bar g_0$ in the sense (\ref{eq:rgapprox}) further satisfies
    \[\frac{1}{M}\sum_{j=1}^M (\bar r_0(s^{(j)})-d^{(j)})^2 \leq 4\epsilon.\]
Set $r_0(\cdot)=N^{1/2}\bar r_0(N^{-1/2}\,\cdot\,)$. Then
$\|r_0(\bX_*)-\bTheta_*\|_\Fro=N^{1/2}\|\bar r_0(N^{-1/2}\bS)-N^{-1/2}\bD\|_\Fro$,
so this and (\ref{eq:quantileconvergence}) imply, almost surely for all large
$n$,
\begin{equation}\label{eq:poly_approx_monotone}
    \frac{1}{n}\|r_0(\bX_*)-\bTheta_*\|_\Fro^2 =
    \frac{1}{M}\sum_{j=1}^M (\bar r_0(s_j/\sqrt{N})-d_j/\sqrt{N})^2
    <5\epsilon.
\end{equation}

Now consider any $f \in \cF$, which by assumption takes a form
\[f(\bz_{1:t})=\sum_{k=1}^K \vec\bigg(g_k\bigg(\sum_{s=1}^t
c_{ks}\mat(\bz_s)+\bTheta_*\bigg)\bigg).\]
For any $\bSigma_t \in \RR^{t \times t}$ satisfying $\|\bSigma_t\|_\op<C_0$,
if $\bZ_{1:t} \sim \cN(0,\bSigma_t \otimes \Id_n)$, then
there is a constant $B>0$ such that
\begin{equation}\label{eq:inputopbound}
\max_{k=1}^K \bigg\|\sum_{s=1}^t
c_{ks}\mat(\bZ_s)+\bTheta_*\bigg\|_\op<B\sqrt{N}
\text{ a.s.\ for all large } n.
\end{equation}
For each $k=1,\ldots,K$, define $\bar g_k(\cdot)$ such that
$g_k(\cdot)=N^{1/2}\bar g_k(N^{-1/2}\,\cdot\,)$, and note that $\bar g_k$ is
also $L$-Lipschitz. In the definition of the above net $\cN_\epsilon$, we may
assume that $L_\epsilon$ is larger than this Lipschitz constant $L$,
and that $B_\epsilon$ is larger than this constant $B$. 
Let $\bar r_k \in \cN_\epsilon$ be the approximation for $\bar g_k$ satisfying
(\ref{eq:rgapprox}), set $r_k(\cdot)=N^{1/2}\bar r_k(N^{-1/2}\,\cdot\,)$, and
consider the polynomial approximation
    \[p(\bz_{1:t}) = \sum_{k=1}^K
\vec\bigg(r_k\bigg(\sum_{s=1}^t c_{ks} \mat(\bz_s) + r_0(\bX_*)\bigg)\bigg)\]
for $f$. Let $\cP$ be the set of polynomial spectral functions
consisting of this approximation for each $f \in \cF$. Then $\cP$ is
BCP-representable by Corollary \ref{cor:BCPreprspectral}. Furthermore, we have
\[\frac{1}{\sqrt{n}}\|f(\bZ_{1:t})-p(\bZ_{1:t})\|_2
 \leq \sum_{k=1}^K
\frac{1}{\sqrt{n}}\bigg\|g_k\bigg(\sum_{s=1}^t
c_{ks}\mat(\bZ_s) + \bTheta_*\bigg)-r_k\bigg(\sum_{s=1}^t
c_{ks}\mat(\bZ_s) + r_0(\bX_*)\bigg)\bigg\|_\Fro.\]
Since $g_k$ is $L$-Lipschitz and satisfies $g_k(0)=0$, the matrix function
given by applying $g_k$ spectrally to the singular values of its input
is also $L$-Lipschitz in the Frobenius norm
\cite[Theorem 1.1]{andersson2016operator}. Thus
\begin{equation}\label{eq:Thetaapprox}
\bigg\|g_k\bigg(\sum_{s=1}^t c_{ks}\mat(\bZ_s) + \bTheta_*\bigg)
-g_k\bigg(\sum_{s=1}^t c_{ks}\mat(\bZ_s) + r_0(\bX_*)\bigg)\bigg\|_\Fro
\leq L\|\bTheta_*-r_0(\bX_*)\|_\Fro
\leq C\sqrt{\epsilon n},
\end{equation}
the last inequality holding a.s.\ for all large $n$ by
\eqref{eq:poly_approx_monotone}. By the
approximation property (\ref{eq:rgapprox}) for $\bar g_k$ and $\bar r_k$
and the operator norm bound (\ref{eq:inputopbound}) where $B<B_\epsilon$, also
\begin{align}
&\bigg\|g_k\bigg(\sum_{s=1}^t
c_{ks}\mat(\bZ_s) + r_0(\bX_*)\bigg) - r_k\bigg(\sum_{s=1}^t
c_{ks}\mat(\bZ_s) + r_0(\bX_*)\bigg)\bigg\|_\Fro\notag\\
&=N^{1/2}\bigg\|\bar g_k\bigg(N^{-1/2}\bigg(\sum_{s=1}^t
c_{ks}\mat(\bZ_s) + r_0(\bX_*)\bigg)\bigg) - \bar r_k\bigg(N^{-1/2}\bigg(\sum_{s=1}^t
c_{ks}\mat(\bZ_s) + r_0(\bX_*)\bigg)\bigg)\bigg\|_\Fro\notag\\
&\leq N^{1/2} \cdot M^{1/2} \sqrt{\epsilon}=\sqrt{\epsilon n}\label{eq:grapprox}
\end{align}
a.s.\ for all large $n$. Combining (\ref{eq:Thetaapprox}) and
(\ref{eq:grapprox}),
\[\frac{1}{\sqrt{n}} \|f(\bZ_{1:t})-p(\bZ_{1:t})\|_2 \leq C'\sqrt{\epsilon}
\text{ a.s.\ for all large $n$}.\]
Applying the dominated convergence theorem, this implies
$n^{-1}\EE[\|f(\bZ_{1:t})-p(\bZ_{1:t})\|_2^2 \mid \bX_*]<C\epsilon$ for a
constant $C>0$ a.s.\ for all large $n$,
verifying the first condition of BCP-approximability.

For the second condition of BCP-approximability, let $\cQ=\bigsqcup_{t=0}^T
\cQ_t$ be the 
set of all polynomial functions of the form (\ref{eq:polyspec})
where $r_0(\cdot)$ is as defined above, $\{c_{ks}\}$ have the same uniform bound
as in $\cP$, and $r_k(\cdot)=N^{1/2}\bar r_k(N^{-1/2}\,\cdot\,)$ for some
monomial $\bar r_k(x)=x^{d_k}$ of odd degree $d_k \geq 1$. Then $\cP \cup
\{q_1,q_2\}$ continues to satisfy the BCP for any $q_1,q_2 \in \cQ$ of uniformly
bounded degrees. Let $\bz_{1:t}$ be any random vectors such that
\begin{equation}\label{eq:specQconverge}
n^{-1}q_1(\bz_{1:t})^\top q_2(\bz_{1:t})
-n^{-1}\EE[q_1(\bZ_{1:t})^\top q_2(\bZ_{1:t}) \mid \bX_*] \to 0 \text{ a.s.}
\end{equation}
for all $q_1,q_2 \in \cQ_t$ of uniformly bounded degrees. Applying
(\ref{eq:Thetaapprox}), for a constant $C_1>0$,
\begin{equation}\label{eq:fspectralapprox}
\frac{1}{n}\bigg\|f(\bz_1,\ldots,\bz_t)-\underbrace{\sum_{k=1}^K \vec\bigg(g_k\bigg(\sum_s c_{ks}
\mat(\bz_s)+r_0(\bX_*)\bigg)\bigg)}_{\tilde
f(\bz_{1:t})}\bigg\|_2^2<C_1\epsilon
\text{ a.s.\ for all large $n$}.
\end{equation}
Applying (\ref{eq:grapprox}) and the dominated convergence theorem, also
\begin{equation}\label{eq:fspectralapprox2}
\frac{1}{n}\,\EE[\|\tilde f(\bZ_{1:t})-p(\bZ_{1:t})\|_2^2 \mid
\bX_*]<C_1\epsilon \text{ a.s.\ for all large $n$}.
\end{equation}
Suppose by contradiction that there is a positive-probability event $\Omega$
on which
\begin{equation}\label{eq:spectralcontra}
n^{-1}\|f(\bz_{1:t})-p(\bz_{1:t})\|_2^2>5C_1\epsilon
\end{equation}
infinitely often. Let $\Omega'$ be the intersection of $\Omega$ with
the probability-one
event where (\ref{eq:fspectralapprox}) holds, and where
(\ref{eq:specQconverge}) holds for all $q_1,q_2$ in a suitably chosen countable 
subset of $\cQ$.
For any $\omega \in \Omega'$, we may pass to a subsequence
$\{n_j\}_{j=1}^\infty$ along which (\ref{eq:spectralcontra}) holds and where
the expectation over $\bZ_{1:t}$ of the empirical singular value distribution of
$N^{-1/2}(\sum_{s=1}^t c_{ks} \mat(\bZ_s)+r_0(\bX_*))$
converges weakly and in Wasserstein-$\ell$ to a limit $\nu_k$, for each
$k=1,\ldots,K$ and every order $\ell \geq 1$.
Then the statement (\ref{eq:specQconverge}) over a suitably chosen
countable subset of $\cQ$
implies that the singular value distribution of $N^{-1/2}(\sum_{s=1}^t c_{ks}
\mat(\bz_s)+r_0(\bX_*))$ converges weakly and in Wasserstein-$\ell$ to the same limit
$\nu_k$, for each $k=1,\ldots,K$ and $\ell \geq 1$. Since $(\bar g_k-\bar r_k)^2$
is a function of polynomial growth, this implies that
\[n_j^{-1}\|\tilde f(\bz_{1:t})-p(\bz_{1:t})\|_2^2
-n_j^{-1}\EE[\|\tilde f(\bZ_{1:t})-p(\bZ_{1:t})\|_2^2 \mid \bX_*] \to 0\]
along this subsequence $\{n_j\}_{j=1}^\infty$. Then combining
with (\ref{eq:fspectralapprox}) and (\ref{eq:fspectralapprox2}), we have
\[\limsup_{j \to \infty} n_j^{-1}\|f(\bz_{1:t})-p(\bz_{1:t})\|_2^2 \leq 4C_1\epsilon,\]
contradicting (\ref{eq:spectralcontra}).
So $n^{-1}\|f(\bz_{1:t})-p(\bz_{1:t})\|_2^2 \leq 5C_1\epsilon$
a.s.\ for all large $n$, showing
the second condition of BCP-approximability and completing the proof.
\end{proof}

\section{Auxiliary proofs}

\subsection{Tensor network representation of polynomial AMP}\label{sec:GenAlg}

We prove Lemma \ref{lem:tenUnroll} on the unrolling of polynomial AMP into
tensor network values. It is convenient to introduce the following
object which will represent the vector-valued iterates $\bu_1,\ldots,\bu_t$.

\begin{definition}
An {\bf open $\cT$-labeling} $\cL^*$ of a connected
ordered multigraph $G$ is an
assignment of a label $*$ to a vertex $v^* \in G$ with $\deg(v^*)=1$,
and a tensor label $\bT_v \in \cT$ 
to each remaining vertex $v \in \cV \setminus \{v^*\}$ such that
$\bT_v$ has order equal to $\deg(v)$.

The {\bf vector value} $\vecval_G(\cL^*) \in \RR^n$ of this open labeling is
the vector satisfying, for any $\bv \in \RR^n$,
\[\langle \vecval_G(\cL^*),\bv \rangle=\val_G(\cL^\bv)\]
where $\cL^\bv$ is the labeling of $G$ that completes $\cL^*$ by assigning
the label $\bv \in \RR^n$ to $v^*$.
\end{definition}

One may understand $v^*$ and the (unique) edge $e^*$ incident to $v^*$ as 
``placeholders'': the vector value $\vecval_G(\cL^*)$ is obtained by
contracting all tensor-tensor products represented by edges $\cE \setminus e^*$,
with the final index $i_{e^*} \in [n]$ associated to $e^*$ left unassigned.

\begin{lemma}\label{lem:tenUnrollVec}
Fix any $T \geq 1$.
Under the assumptions of Lemma \ref{lem:tenUnroll}, there exist constants
$C,M>0$, a list of connected ordered multigraphs $G_1,\ldots,G_M$
depending only on $T,D,C_0$ and
independent of $n$, and a list of open $(\cT \cup \bW)$-labelings
$\cL_1^*,\ldots,\cL_M^*$ of $G_1,\ldots,G_M$ and
coefficients $a_1,\ldots,a_M \in \RR$ with $|a_m|<C$, such that
\[\bu_T=\sum_{m=1}^M a_m \vecval_{G_m}(\cL_m^*).\]
\end{lemma}
\begin{proof}
By assumption, each function $f_0,\ldots,f_{T-1}$ admits a
representation
\[f_s(\bz_1,\ldots,\bz_s)=\bT_s^{(0)}
+\sum_{d=1}^D \sum_{\sigma \in \cS_{s,d}} \bT_s^{(\sigma)}[\bz_{\sigma(1)},
\ldots,\bz_{\sigma(d)},\,\cdot\,]\]
for some tensors $\bT_s^{(0)},\bT_s^{(\sigma)} \in \cT$. Then by definition of
the algorithm (\ref{eq:AMP}) and multi-linearity, there exists a constant $M>0$
(depending only on $(T,D)$) and coefficients $a_1,\ldots,a_M \in \RR$ for which
\[\bu_T=\sum_{m=1}^M a_m \bu_T^{(m)},\]
each $a_m$ is a product of a subset of the Onsager coefficients
$\{{-}b_{ts}\}_{s<t \leq T}$,
and each $\bu_T^{(m)}$ is the output of an iterative algorithm
with initialization $\bu_1^{(m)}=\bu_1$ and
\begin{align*}
\bz_t &\in \{\bW\bu_t,\bu_1,\ldots,\bu_{t-1}\}\\
\bu_{t+1} &\in \begin{cases}
\bT_t^{(0)} \\ \bT_t^{(\sigma)}[\bz_{\sigma(1)},\ldots,
\bz_{\sigma(d)},\cdot] \text{ for some } d \in [D] \text{ and }
\sigma \in \cS_{t,d}\end{cases}
\end{align*}
for $t=1,2,\ldots,T-1$. That is, in each iteration,
the algorithm is defined by a single (fixed) choice for $\bz_t
\in \{\bW\bu_t,\bu_1,\ldots,\bu_{t-1}\}$ and a non-linear function representable
by a single (fixed) tensor in $\cT$. Thus it suffices to
show that for any such algorithm and any $t \in \{1,\ldots,T\}$,
there exists a connected ordered multigraph $G$
independent of $n$ --- in fact, a tree rooted at $v^*$ ---
and an open $(\cT \cup \bW)$-labeling $\cL^*$ of $G$ for which
\begin{equation}\label{eq:singlenetworkrepr}
\bu_t=\vecval_G(\cL^*).
\end{equation}

This follows from an easy induction on $t$: For $t=1$,
$\bu_1$ is given by $\vecval_G(\cL^*)$ for a tree $G$
with root $v^*$ and a single edge connecting to a child with label
$\bu_1 \in \RR^n \cap \cT$. Assuming that
(\ref{eq:singlenetworkrepr}) holds for $s=1,\ldots,t-1$, let $(G_s,\cL_s^*)$
be the tree graph and open labeling for which
$\bu_s=\vecval_{G_s}(\cL_s^*)$, and let $d+1$ be the order of the tensor
$\bT_{t-1}^{(\sigma)}$ defining $\bu_t$. Then define a tree graph
$G$ with open labeling $\cL^*$ such that $G$ is rooted at $v^*$, and $v^*$ has
a single child $v$ labeled by $\bT_{t-1}^{(\sigma)}$, with $\deg(v)=d+1$ and ordered
edges $e_1,\ldots,e_{d+1}$ where the last edge $e_{d+1}$ connects to $v^*$.
For each other edge $e_i$ with $i \in [d]$:
\begin{itemize}
\item If $\bz_{\sigma(i)}=\bu_j$ for some $j \in [t-1]$, then the $i^\text{th}$
subtree
$v \overset{e_i}{-} T_i$ rooted at $v$ coincides with $(G_j,\cL_j^*)$ with $v$
replacing the root of $(G_j,\cL_j^*)$.
\item If $\bz_{\sigma(i)}=\bW\bu_j$ for $j=\sigma(i) \in [t-1]$, then
this $i^\text{th}$ subtree has a form
\[v \overset{e_i}{-} v_i \overset{e_i'}{-} T_i\]
where $v_i$ has $\deg(v_i)=2$ and label $\bW$, its first edge $e_i'$ connects to
$T_i$, and its second edge $e_i$ connects to $v$. The subtree $v_i
\overset{e_i'}{-} T_i$ coincides with $(G_j,\cL_j^*)$ with $v_i$
replacing the root of $(G_j,\cL_j^*)$.
\end{itemize}
It is readily checked from the definition of $\vecval$
and the inductive hypothesis $\bu_s=\vecval_{G_s}(\cL_s^*)$ for each $s \in
[t-1]$ that $\bu_t=\vecval_G(\cL^*)$, completing the induction and the proof.
\end{proof}

\begin{proof}[Proof of Lemma \ref{lem:tenUnroll}]
By Lemma \ref{lem:tenUnrollVec} and the given condition that $\phi_1,\phi_2$
are also $\cT$-representable, we have
\begin{align*}
\phi_1(\bz_1,\ldots,\bz_T)&=\sum_{m=1}^M a_m \vecval_{G_m}(\cL_m^*)\\
\phi_2(\bz_1,\ldots,\bz_T)&=\sum_{m'=1}^{M'} a_{m'}' \vecval_{G_{m'}'}({\cL_{m'}^*}')
\end{align*}
where $|a_m|,|a_{m'}'| \leq C$, $G_m,G_{m'}'$ are connected ordered multigraphs
(independent of $n$) with open labelings $\cL_m^*,{\cL_{m'}^*}'$, and
$C,M,M'>0$ are constants independent of $n$. Then
\[\phi(\bz_1,\ldots,\bz_T)=\frac{1}{n}
\phi_1(\bz_1,\ldots,\bz_T)^\top \phi_2(\bz_1,\ldots,\bz_T)
=\sum_{m=1}^M \sum_{m'=1}^{M'} \frac{a_m a_{m'}'}{n}
\langle \vecval_{G_m}(\cL_m^*),\vecval_{G_{m'}'}({\cL_{m'}^*}') \rangle.\]
The lemma then follows from the observation that for any two connected
ordered multigraphs $G_1,G_2$ with open labelings $\cL_1^*,\cL_2^*$, we have
\[\langle \vecval_{G_1}(\cL_1^*),\vecval_{G_2}(\cL_2^*) \rangle
=\val_G(\cL)\]
where $(G,\cL)$ is the tensor network obtained by removing the
distinguished vertex $v^*$ from both $G_1$ and $G_2$, and replacing the edge
$v_1-v^*$ in $G_1$ and the edge $v_2-v^*$ in $G_2$ by a single edge $v_1-v_2$.
(If $v_1-v^*$ is the $i^\text{th}$ ordered edge of $v_1$ in $G_1$ and $v_2-v^*$
is the $j^\text{th}$ ordered edge of $v_2$ in $G_2$,
then $v_1-v_2$ remains the $i^\text{th}$ ordered edge of
$v_1$ and $j^\text{th}$ ordered edge of $v_2$ in $G$.)
\end{proof}

\subsection{Extension to asymmetric AMP}\label{sec:Rectangle}

We prove Theorem \ref{thm:RectSE} on the extension of our main results to AMP
with an asymmetric matrix $\bW \in \RR^{m \times n}$.

\begin{proof}[Proof of Theorem \ref{thm:RectSE}]
We ``embed'' the asymmetric AMP algorithm (\ref{eq:RectAMP})
into the symmetric AMP algorithm \eqref{eq:AMP} by setting
\[\bW^\sym = \sqrt{\frac{m}{m+n}}
\begin{bmatrix} \bA & \bW \\ \bW^\top &\bB \end{bmatrix} \in \RR^{(n+m) \times
(n+m)}\]
where $\bA \in \RR^{m \times m}$ and $\bB \in \RR^{n \times n}$ have
independent Gaussian entries with mean 0 and variance $1/m$. Then
$\bW^\sym$ is a Wigner matrix of size $n+m$,
satisfying Assumption \ref{assump:Wigner}.

We consider the initialization
\[f_0^\sym(\cdot) \equiv \bu_1^\sym=\sqrt{\frac{m+n}{m}}\begin{bmatrix} 0\\
\bu_1 \end{bmatrix}\]
and the sequence of non-linear functions $f_t^\sym:\RR^{(n+m) \times t}
\to \RR^{n+m}$ given by
\begin{equation}\label{eq:fgEmbed}
\begin{aligned}
f_{2t-1}^\sym(\bz_{1:(2t-1)}^\sym) &= \sqrt{\frac{m+n}{m}} \begin{bmatrix}
f_t((\bz_{2j-1}^\sym[1:m])_{j=1}^t) \\ 0 \end{bmatrix},\\
f_{2t}^\sym(\bz_{1:2t}^\sym) &= \sqrt{\frac{m+n}{m}}
\begin{bmatrix} 0 \\ g_t((\bz_{2j}^\sym[(m+1):(m+n)])_{j=1}^t)\end{bmatrix}.
\end{aligned}
\end{equation}
We then consider the iterates of the symmetric AMP algorithm \eqref{eq:AMP},
\begin{equation}\label{eq:RectToSym}
\begin{split}
    \bz_t^\sym &= \bW^\sym \bu_t^\sym - \sum_{s=1}^{t-1} b_{ts}^\sym \bu_s^\sym\\
    \bu_{t+1}^\sym &= f_t^\sym(\bz_{1:t}^\sym)
\end{split}
\end{equation}
where $b_{ts}^\sym$ is as defined in Definition \ref{def:non_asymp_se} for the
function sequence $\{f_t^\sym\}_{t \geq 0}$. It is direct to verify that the
iterates of the asymmetric AMP algorithm (\ref{eq:RectAMP}) are embedded within
the iterates of this algorithm as
\[\bz_t=\bz_{2t-1}^\sym[1:m],
\quad \by_t=\bz_{2t}^\sym[(m+1):(m+n)],\]
\[\bu_t=\sqrt{\frac{m}{m+n}}\bu_{2t-1}^\sym[(m+1):(m+n)],
\quad \bv_t=\sqrt{\frac{m}{m+n}}\bu_{2t}^\sym[1:m],\]
and that the Onsager coefficients and state evolution covariances 
of Definition \ref{def:RectSE} are related to those of this symmetric AMP
algorithm (\ref{eq:RectToSym}) by
\[b_{ts}=\sqrt{\frac{m+n}{m}}b_{2t-1,2s}^\sym,
\qquad a_{ts}=\sqrt{\frac{m+n}{m}}b_{2t,2s-1}^\sym,\]
\begin{equation}\label{eq:SECovEmbed}
\bOmega_t=(\bSigma_{2t-1}^\sym[2j-1,2k-1])_{j,k=1}^t,
\qquad \bSigma_t=(\bSigma_{2t}^\sym[2j,2k])_{j,k=1}^t.
\end{equation}
Furthermore, for $i=1,2$, defining
\begin{equation}\label{eq:phipsiEmbed}
\begin{aligned}
\phi_i^\sym(\bz_{1:2T}^\sym) &= \sqrt{\frac{m+n}{m}} \begin{bmatrix}
\phi_i((\bz_{2j-1}^\sym[1:m])_{j=1}^T) \\ 0
\end{bmatrix},\\
\psi_i^\sym(\bz_{1:2T}^\sym) &= \sqrt{\frac{m+n}{m}}
\begin{bmatrix} 0 \\ \psi_i((\bz_{2j}^\sym[(m+1):(m+n)])_{j=1}^T)\end{bmatrix},
\end{aligned}
\end{equation}
and setting $\phi^\sym=(n+m)^{-1}(\phi_1^\sym)^\top(\phi_2^\sym)$ and
$\psi^\sym=(n+m)^{-1}(\psi_1^\sym)^\top(\psi_2^\sym)$, we have
\[\phi(\bz_{1:T})=\phi^\sym(\bz_{1:2T}^\sym),
\qquad \psi(\by_{1:T})=\psi^\sym(\bz_{1:2T}^\sym).\]
Thus Theorem \ref{thm:RectSE} follows directly from Theorems
\ref{thm:universality_poly_amp} and \ref{thm:main_universality} applied up to
iteration $2T$ of the symmetric AMP algorithm \eqref{eq:RectToSym}, provided
that the assumptions in Theorems
\ref{thm:universality_poly_amp} and \ref{thm:main_universality} hold.

To check these assumptions
in the polynomial AMP setting of Theorem \ref{thm:RectSE}(a), note that
$\bSigma_{2T}^\sym$ is block-diagonal with even rows/columns constituting one
block equal to $\bSigma_T$ and odd rows/columns constituting a second block
equal to $\bOmega_T$. Then $\lambda_{\min}(\bSigma_{2T}^\sym)>c$ by the given
conditions for $\bSigma_T$ and $\bOmega_T$, implying also that
$\lambda_{\min}(\bSigma_t^\sym)>c$ for each $t=1,\ldots,2T$. To apply
Theorem \ref{thm:universality_poly_amp}, it remains to check that
$\cF^\sym=\{f_0^\sym,\ldots,f_{2T-1}^\sym,\phi_1^\sym,\phi_2^\sym\}$ is
BCP-representable. As $\cG=\{g_0,\ldots,g_{T-1},\psi_1,\psi_2\}$ is
BCP-representable, there exists a set of tensors
$\cT^g=\bigsqcup_{k=1}^K \cT_k^g$ with $\cT_k^g \subset (\RR^n)^{\otimes k}$
that satisfies the BCP, 
for which each $g \in \cG$ admits the representation (\ref{eq:tensorpolyrepr})
with tensors in $\cT^g$. Similarly, there exists a set of tensors
$\cT^f=\bigsqcup_{k=1}^K \cT_k^f$ with $\cT_k^f \subset (\RR^m)^{\otimes k}$
that satisfies the BCP,
for which each $f \in \cF=\{f_1,\ldots,f_T,\phi_1,\phi_2\}$ admits
the representation (\ref{eq:tensorpolyrepr}) with tensors in $\cT^f$.
Let $\cT_k^{g,\sym} \subset (\RR^{m+n})^{\otimes k}$ be the embeddings of the
tensors $\cT_k^g$ into the diagonal block of $(\RR^{m+n})^{\otimes k}$
corresponding to the last $n$ coordinates $m+1,\ldots,m+n$, similarly let
$\cT_k^{f,\sym} \subseteq (\RR^{m+n})^{\otimes k}$ be the embeddings of
$\cT_k^f$ into the diagonal block of $(\RR^{m+n})^{\otimes k}$ corresponding to
the first $m$ coordinates $1,\ldots,m$, and define
\[\cT^\sym=\cT^{g,\sym} \sqcup \cT^{f,\sym}
=\bigsqcup_{k=1}^K \cT_k^{g,\sym} \sqcup \bigsqcup_{k=1}^K \cT_k^{f,\sym}.\]
Then each function $f \in \cF^\sym$ admits the representation
(\ref{eq:tensorpolyrepr}) with tensors in $\cT^\sym$. To see that $\cT^\sym$
satisfies the BCP, consider the expression on the left side of
(\ref{eq:BCPcondition}). If all tensors in this expression belong to
$\cT^{g,\sym}$, then (\ref{eq:BCPcondition}) holds by the BCP for
$\cT^g$. Similarly if all tensors belong to $\cT^{f,\sym}$, then
(\ref{eq:BCPcondition}) holds by the BCP for $\cT^f$.
If there is at least one tensor belonging to $\cT^{g,\sym}$ and at least
one belonging to $\cT^{f,\sym}$, then the second condition of Definition \ref{def:BCP} requires
that there is at least one index $i_j$ for some $j \in \{1,\ldots,\ell\}$
that is shared between a tensor $\bT_a \in \cT^{g,\sym}$ and a tensor
$\bT_b \in \cT^{f,\sym}$. Then the $\bT_a$ factor is 0 for all summands where
$i_j \in \{1,\ldots,m\}$, the $\bT_b$ factor is 0 for all summands where $i_j
\in \{m+1,\ldots,m+n\}$, so (\ref{eq:BCPcondition}) holds trivially as the left
side is 0. This verifies that $\cT^\sym$ satisfies the BCP. Then
$\cF^\sym$ is BCP-representable, and Theorem \ref{thm:RectSE}(a)
follows from Theorem \ref{thm:universality_poly_amp}.

For Theorem \ref{thm:RectSE}(b), we check the conditions of
Theorem \ref{thm:main_universality}: The boundedness and Lipschitz properties
(\ref{eq:Lipschitz}) for
$\cF^\sym=\{f_0^\sym,\ldots,f_{2T-1}^\sym,\phi_1^\sym,\phi_2^\sym\}$ follow
from the given property (\ref{eq:RectLipschitz}) for
$\cF=\{f_1,\ldots,f_T,\phi_1,\phi_2\}$ and
$\cG=\{g_0,\ldots,g_{T-1},\psi_1,\psi_2\}$. The condition that
$\lambda_{\min}(\bSigma_t^\sym[S_t,S_t])>c$ for the set of preceding iterates
$S_t$ on which $f_t^\sym$ depends, for each $t=1,\ldots,2T-1$, follows also
from the given conditions for $\bSigma_t,\bOmega_t$ and the above
identifications (\ref{eq:SECovEmbed}). For BCP-approximability of $\cF^\sym$,
fix any $C_0,\epsilon>0$, and let $\cP^g,\cQ^g$ and $\cP^f,\cQ^f$ be the sets
of polynomial functions guaranteed by Definition \ref{def:BCP_approx}
for the BCP-approximable families $\cG$ and $\cF$ respectively.
For each $p \in \cP^g$ where $p:\RR^{n \times t} \to \RR^n$,
consider the embedding $p^\sym:\RR^{(n+m) \times 2t} \to \RR^{n+m}$ given by
\[p^\sym(\bz_{1:2t}^\sym)
=\sqrt{\frac{m+n}{m}}\begin{pmatrix} 0 \\
p((\bz_{2j}^\sym[(m+1):(m+n)])_{j=1}^t) \end{pmatrix},\]
and for each $p \in \cP^f$ where $p:\RR^{m \times t} \to \RR^m$,
consider the embedding $p^\sym:\RR^{(n+m) \times (2t-1)} \to \RR^{n+m}$ given by
\[p^\sym(\bz_{1:(2t-1)}^\sym)
=\sqrt{\frac{m+n}{m}}\begin{pmatrix} p((\bz_{2j-1}^\sym[1:m])_{j=1}^t) \\ 0
 \end{pmatrix}.\]
Let $\cP^\sym=\bigsqcup_{t=0}^{2T} \cP_t^\sym$
be the set of such embeddings for all $p \in \cP^g$ and $p \in
\cP^f$, and define similarly $\cQ^\sym$ as the set of such embeddings
for all $q \in \cQ^g$ and $q \in \cQ^f$.
The preceding argument shows that $\cP^\sym \cup
\{q_1^\sym,q_2^\sym\}$ for any $q_1^\sym,q_2^\sym \in \cQ^\sym$ of uniformly
bounded degrees continues to satisfy the BCP. Then, in light
of (\ref{eq:fgEmbed}) and (\ref{eq:phipsiEmbed}), both conditions of Definition 
\ref{def:BCP_approx} hold for $\cF^\sym$ and the constants
$C_0,[(m+n)/m]\epsilon>0$, via these sets $\cP^\sym,\cQ^\sym$. Thus $\cF^\sym$ is
BCP-approximable, and Theorem \ref{thm:RectSE}(b)
follows from Theorem \ref{thm:main_universality}.
\end{proof}

\subsection{Auxiliary lemmas}

\begin{lemma}[Stein's Lemma]\label{lem:stein}
Let $\bX \sim \cN(0,\bSigma)$ be a multivariate Gaussian vector in $\RR^t$
with non-singular covariance $\bSigma \in \RR^{t \times t}$,
and let $g:\RR^t \to \RR$ be a weakly differentiable function such that
$\EE|\partial_j g(\bX)|<\infty$ for each $j=1,\ldots,t$. Then
\[\EE[\bX\,g(\bX)] = \bSigma \cdot \EE \nabla g(\bX)\]
\end{lemma}
\begin{proof}
See \cite[Lemma 6.20]{Feng2021}.
\end{proof}

\begin{lemma}[Wick's rule]\label{lem:wick}
Suppose $\bxi_1,\ldots,\bxi_t \in \RR^n$ are i.i.d.\ $\cN(0,\Id)$ vectors,
$\bT \in (\RR^n)^{\otimes d}$ is a deterministic tensor, and
$\sigma:[d] \to [t]$ is any
index map. Let $\pi=\{\sigma^{-1}(1),\ldots,\sigma^{-1}(t)\}$ be the partition
of $[d]$ where each block is the pre-image of a single index $s \in [t]$ under
$\sigma$, and let
\[\sP=\{\text{pair partitions } \tau \text{ of } [d]:\tau \leq \pi\}\]
be the set of pairings of $[d]$ that refine $\pi$ (where $\sP=\emptyset$ if any
block of $\pi$ has odd cardinality). Then
\[\EE\,\bT[\bxi_{\sigma(1)},\ldots,\bxi_{\sigma(d)}]
=\sum_{\tau \in \sP}\,\sum_{\bi \in [n]^d}
\bT[i_1,\ldots,i_d]\prod_{(a,b) \in \tau} \1\{i_a=i_b\}.\]
\end{lemma}
\begin{proof}
When $\bT$ has a single entry $(i_1,\ldots,i_d)$ equal to 1 and remaining
entries 0, we have
\[\EE\,\bT[\bxi_{\sigma(1)},\ldots,\bxi_{\sigma(d)}]
=\EE[\bxi_{\sigma(1)}[i_1]\ldots \bxi_{\sigma(d)}[i_d]]
=\sum_{\tau \in \sP} \prod_{(a,b) \in \tau} \1\{i_a=i_b\}\]
by \cite{isserlis1918formula}, and the result for general $\bT \in
(\RR^n)^{\otimes d}$ follows from linearity.
\end{proof}

\begin{lemma}[Gaussian hypercontractivity inequality]\label{lem:hypercontractivity}
Let $\bxi \in \RR^n$ have i.i.d.\ $\cN(0,1)$ entries. Then there are absolute
constants $C,c>0$ such that for any polynomial $p:\RR^n \to \RR$ of degree $k$
and any $t \geq 0$,
\[\PP[|p(\bxi)-\EE p(\bxi)| \geq t] \leq Ce^{-(\frac{ct^2}{\Var[p(\bxi)]})^{1/k}}.\]
\end{lemma}
\begin{proof}
See \cite[Theorem 1.9]{schudy2012concentration}.
\end{proof}

\begin{lemma}[Weighted uniform polynomial approximation]\label{lem:polyapprox}
Let $p_1,\ldots,p_d:[0,\infty) \to [0,\infty)$ be functions admitting the
representations
\[p_k(x_k)=p_k(1)+\int_1^{x_k} \frac{w_k(t_k)}{t_k}dt_k \text{ for all } x_k \geq 1\]
where $w_k(t_k)$ is nondecreasing and $\lim_{t_k \to \infty} w_k(t_k)=\infty$
for each $k=1,\ldots,d$, and that satisfy
$\int_1^\infty (p_k(x_k)/x_k^2)\,dx_k=\infty$ for each $k=1,\ldots,d$.
Let $q_1,\ldots,q_d:\RR \to [0,\infty)$ be continuous functions satisfying
$q_k(x_k) \geq p_k(|x_k|)$. If $f:\RR^d \to \RR$ is any
continuous function such that
\[\lim_{\|(x_1,\ldots,x_d)\|_2 \to \infty} \exp\Big({-}\sum_{k=1}^d
q_k(x_k)\Big)f(x_1,\ldots,x_d)=0,\]
then
\[\inf_Q \sup_{(x_1,\ldots,x_d) \in \RR^d}\bigg\{\exp\Big({-}\sum_{k=1}^d
q_k(x_k)\Big)|f(x_1,\ldots,x_d)-Q(x_1,\ldots,x_d)|\bigg\}=0\]
where $\inf_Q$ is the infimum over all polynomial functions $Q:\RR^d \to \RR$.
\end{lemma}
\begin{proof}
See \cite[Theorem 1]{dzhrbashyan1957weighted}.
\end{proof}

\subsection*{Acknowledgments}

This research was supported in part by NSF DMS-2142476 and a Sloan Research
Fellowship.

\bibliographystyle{plain}
\bibliography{reference}

\end{document}